\documentclass[11pt,leqno]{amsart}

\usepackage{graphics}
\usepackage{float}

\usepackage{hyperref}
\usepackage[all,cmtip]{xy}
\usepackage[active]{srcltx}
\usepackage{a4,latexsym}
\usepackage[utf8]{inputenc}
\usepackage{textcomp,fontenc}
\usepackage{exscale}
\usepackage{amsxtra,amsmath,amsfonts,amscd, amssymb, mathrsfs, amsthm}
\usepackage{color}

\newtheorem{thm}{Theorem}[section]
\newtheorem{prop}[thm]{Proposition}
\newtheorem{cor}[thm]{Corollary}
\newtheorem{lem}[thm]{Lemma}
\newtheorem{example}[thm]{Example}
\theoremstyle{definition}
\newtheorem{definition}[thm]{Definition}
\newtheorem{ex}[thm]{Example}
\newtheorem{rem}[thm]{Remark}
\newtheorem{art}[thm]{}

\numberwithin{paragraph}{section}
\numberwithin{equation}{thm}

\def\P{{\mathbb P}}
\def\N{{\mathbb N}}
\def\Z{{\mathbb Z}}
\def\Q{{\mathbb Q}}
\def\R{{\mathbb R}}
\def\C{{\mathbb C}}

\def\A{{\mathbb A}}
\def\B{{\mathbb B}}

\def\G{{\mathbb G}}
\def\L{{\mathbb L}}
\def\T{{\mathbb T}}

\def\KC{{\mathscr C}}
\def\KD{{\mathscr D}}
\def\KO{{\mathcal O}}

\newcommand{\metr}{{\|\hspace{1ex}\|}}
\newcommand{\Hom}{{\rm Hom}}

\def\an{{\rm an}}
\def\div{{\rm div}}
\def\trop{{\rm trop}}
\def\Trop{{\rm Trop}}

\newcommand{\Xan}{{X^{\rm an}}}

\newcommand{\Zan}{{Z^{\rm an}}}
\newcommand{\Tan}{{T^{\rm an}}}
\newcommand{\Uan}{{U^{\rm an}}}
\newcommand{\Acal}{{\mathscr A}}
\newcommand{\Bcal}{{\mathscr B}}
\newcommand{\Ccal}{{\mathscr C}}
\newcommand{\Dcal}{{\mathscr D}}

\newcommand{\Hcal}{{\mathscr H}}

\newcommand{\Lcal}{{\mathscr L}}
\newcommand{\Mcal}{{\mathscr M}}

\newcommand{\Ocal}{{\mathscr O}}

\newcommand{\Ucal}{{\mathscr U}}

\newcommand{\Vcal}{{\mathscr V}}
\newcommand{\Xcal}{{\mathscr X}}

\newcommand{\codim}{{\rm codim}}

\newcommand{\Div}{{\rm div}}
\newcommand{\cyc}{{\rm cyc}}
\newcommand{\Pic}{{\rm Pic}}

\newcommand{\Spec}{{\rm Spec}}
\newcommand{\Spf}{{\rm Spf}}

\newcommand{\ve}{{\varepsilon}}
\newcommand{\id}{{\rm id}}

\newcommand{\rk}{{\rm rk}}

\newcommand{\supp}{{\rm supp}}

\newcommand{\relint}{{\rm relint}}
\newcommand{\kcirc}{{ K^\circ}}
\newcommand{\ktilde}{{ \tilde{K}}}
\newcommand{\Xfrak}{{\mathfrak X}}
\newcommand{\Lfrak}{{\mathfrak L}}
\newcommand{\Lan}{{L^{\rm an}}}
\newcommand{\val}{{\rm val}}

\newcommand{\LC}{{\rm LC}}
\newcommand{\res}{{\rm res}}
\newcommand{\ps}{{\rm ps}}
\newcommand{\dpa}{{d'_{\rm P}}}
\newcommand{\dpb}{{d''_{\rm P}}}

\title{A tropical approach to non-archimedean Arakelov geometry}

\author{Walter Gubler}
\address{W. Gubler, Fakultät für Mathematik, Universit{\"a}t Regensburg, 
93040 Regensburg, Germany}
\email{walter.gubler@mathematik.uni-regensburg.de}

\author{Klaus K{\"u}nnemann}
\address{K. K{\"u}nnemann, Fakultät für Mathematik, Universit{\"a}t Regensburg, 
93040 Regensburg, Germany}
\email{klaus.kuennemann@mathematik.uni-regensburg.de}

\date{\today}

\begin{document}

\begin{abstract}
Chambert-Loir and Ducros have recently introduced a theory of real valued
differential forms {and currents} on Berkovich spaces.
In analogy to the theory of forms with logarithmic singularities,  we enlarge the space of
differential forms by 
so called $\delta$-forms on the non-archimedean analytification
of an algebraic variety. 
This extension is based on an intersection theory for tropical cycles with smooth
weights. 
We prove a generalization of the Poincar\'e-Lelong formula  which allows us to represent the first Chern current of a formally metrized 
line bundle by a $\delta$-form. We introduce the associated Monge-Amp\`ere measure $\mu$ 
as a wedge-power of this first Chern $\delta$-form and we show that $\mu$ is equal to the corresponding 
Chambert--Loir measure.
The $*$-product of Green currents is a crucial ingredient in the construction of the arithmetic intersection 
product. Using the formalism of $\delta$-forms, we obtain a non-archimedean analogue at least in the case 
of divisors. We use it to compute non-archimedean local heights of proper varieties.

\bigskip

\end{abstract}

\maketitle
\setcounter{tocdepth}{1}
\setcounter{section}{-1}

\tableofcontents

\section{Introduction}

Weil's adelic point of view was to compactify the ring of integers $\KO_K$ 
of a number field $K$ by the archimedean primes. 
Arakelov's brilliant idea was to  add metrics on the 
``fibre at infinity'' of a surface over $\KO_K$ which gave a good intersection theory for arithmetic divisors.
This Arakelov theory became popular after Faltings used it to prove Mordell's conjecture. 
A higher dimensional arithmetic intersection theory was developed by Gillet and Soul\'e. 
Their theory combines algebraic intersection theory on a regular model $\Xcal$ 
over $\KO_K$ with differential geometry on the associated complex 
manifold $\Xan$ of the generic fibre $X$ of $\Xcal$. 
Roughly speaking, an arithmetic cycle on $\Xcal$ is given by a pair 
$({\mathscr  Z},g_Z)$, where $\mathscr Z$ is a cycle on $\Xcal$ with 
generic fibre $Z$ and $g_Z$ is a  current on $\Xan$ satisfying the equation 
\[
dd^cg_Z =[\omega_Z]-\delta_Z
\]
for a smooth differential form $\omega_Z$ and the current of 
integration $\delta_Z$ over $\Zan$. The arithmetic intersection product 
uses the algebraic intersection product for algebraic 
cycles in the first component and the $*$-product of Green currents in 
the second component. 
This arithmetic intersection theory is nowadays called Arakelov theory. 
It found many nice applications as Faltings's proof of the Mordell--Lang 
conjecture for abelian varieties and the proof of Ullmo and Zhang of 
the Bogomolov conjecture for abelian varieties. 

It is an old dream to handle archimedean and non-archimedean places in a 
similar way. This means that we are looking for a description in terms of 
currents for the contributions of the non-archimedean places to 
Arakelov theory. Such a non-archimedean Arakelov theory at finite places was 
developed by Bloch--Gillet--Soul\'e relying strongly on the conjectured 
existence of resolution of singularities for models in mixed characteristics. 
The use of models has also another disadvantage since they are not suitable 
to describe canonical metrics as for line bundles on abelian varieties with 
bad reduction. 
A more analytic non-archimedean Arakelov theory was developed by 
Chinburg--Rumely and Zhang in the case of curves. 
A crucial role is played here by the reduction graph of the curve. 
Without any doubt, the latter should be replaced by the Berkovich analytic 
space associated to the curve and this was done by Thuillier in his thesis 
introducing a non-archimedean potential theory. 
Chambert--Loir and Ducros \cite{chambert-loir-ducros} introduced recently 
 differential forms and currents on Berkovich spaces. These 
provide us with a new tool to give an analytic 
description of non-archimedean 
Arakelov theory in higher dimensions.

We recall the definition of differential forms given in 
\cite{chambert-loir-ducros}. We restrict here to the algebraic
case. Let $U$ be an $n$-dimensional very affine open variety which means 
that $U$ has a closed embedding into a multiplicative torus 
$T = {\mathbb G}_m^r$ over a non-archimedean field $K$. 
By definition, such a field $K$ is endowed with a complete non-archimedean 
absolute value $|\phantom{a}|$. 
Let $t_1, \dots, t_r$ be the torus coordinates. Then we have the tropicalization
map $\trop:\Tan \rightarrow \R^r, \, t \mapsto (-\log|t_1|, \dots , -\log|t_r|)$.
By the Bieri--Groves theorem, the tropical variety $\Trop(U):=\trop(\Uan)$ is a 
finite union of $n$-dimensional polyhedra. More precisely, $\Trop(U)$ is an $n$-dimensional 
tropical cycle which means that $\Trop(U)$ is a polyhedral complex endowed with 
canonical weights satisfying a balancing condition. 
Let $x_1, \dots , x_r$ be the  coordinates on $\R^r$. 
Then Lagerberg's superforms on $\R^r$ are formally given by  
\[
\alpha = \sum_{|I|=p, |J|=q}  \alpha_{IJ} d'x_{i_1} \wedge \dots 
\wedge d'x_{i_p} \wedge d''x_{j_1} \wedge \dots \wedge d''x_{j_q}
\]
where $I$ (resp. $J$) consists of $i_1 < \dots < i_p$ 
(resp. $j_1 < \dots < j_q$), $\alpha_{IJ} \in C^\infty(\R^r)$. 
We have differential operators $d'$ and $d''$ on the space of superforms given by
\[d' \alpha :=  \sum_{|I|=p, |J|=q}  \sum_{i=1}^r \frac{\partial \alpha_{IJ}}{\partial x_i} d'x_i 
\wedge d'x_{i_1} \wedge \dots \wedge d'x_{i_p} \wedge d''x_{j_1} \wedge \dots \wedge d''x_{j_q}\]
and 
\[d'' \alpha :=  \sum_{|I|=p, |J|=q}  \sum_{j=1}^r \frac{\partial \alpha_{IJ}}{\partial x_j} {d''x_j}
\wedge d'x_{i_1} \wedge \dots \wedge d'x_{i_p} \wedge d''x_{j_1} \wedge \dots \wedge d''x_{j_q}.\]
They are the analogues of the differential operators $\partial$ and 
$\bar{\partial}$ in complex analysis. The space of superforms on $\R^r$ with 
the usual wedge-product is a differential bigraded $\R$-algebra with respect 
to $d'$ and $d''$. The space of supercurrents on $\R^r$ is given as the 
topological dual of the space of superforms.

Every superform $\alpha$ induces a differential form on $\Uan$ and two 
superforms $\alpha, \alpha'$ induce the same form if and only if they 
restrict to the same superform on $\Trop(U)$. 
In general, a differential form on an $n$-dimensional variety $X$ is given 
locally for the Berkovich analytic topology on very affine open subsets by Lagerberg's superforms 
which agree on common intersections (see \cite{gubler-forms} for more details). 
The wedge-product and the differential operators can be carried over 
to $\Xan$ leading to a  sheaf $A^{\cdot,\cdot}$ of differential forms on $\Xan$. 
Integration of superforms leads to integration of compactly supported 
$(n,n)$-forms on $\Xan$. 
The space of currents $D^{\cdot,\cdot}(\Xan)$  is defined as the topological 
dual of the space of compactly supported forms. 

A major result of Chambert--Loir and Ducros is the Poincar\'e--Lelong 
formula for the meromorphic section of a line bundle endowed with a 
continuous metric $\metr$. Note that in this situation, $c_1(L,\metr)$ is only 
a current, while a smooth metric allows to define the first Chern form in 
$A^{1,1}(\Xan)$. 
For a smooth metric, $c_1(L,\metr)^n$ is a form of top degree and hence 
defines a signed measure called the Monge--Amp\`ere measure of $(L,\metr)$. 
In arithmetics, metrics are often induced by proper algebraic models over the 
valuation ring. Such metrics are called algebraic. 
They are continuous on $\Xan$, but not smooth. This makes it difficult to 
define the Monge--Amp\`ere measure as a wedge product of currents. 
In the complex situation, one needs Bedford--Taylor theory to define such a 
wedge product. In the non-archimedean situation, Chambert--Loir and Ducros 
use an approximation process by smooth metrics to define this top-dimensional 
wedge product of first Chern currents. 

The main theorem in \cite{chambert-loir-ducros} shows that the 
Monge-Amp\`ere measure of a line bundle endowed with a formal metric is 
equal to the Chambert-Loir measure. 
The latter was introduced in \cite{chambert-loir-2006} before a definition 
of first Chern current was available. 
It is defined as a discrete measure on the Berkovich space using 
degrees of the irreducible components of the special fibre. 
Chambert-Loir measures play a prominent role in non-archimedean equidistribution results. 
For example, they occur in the non-archimedean version of Yuan's 
equidistribution theorem which has applications to the geometric 
Bogomolov conjecture.

In the thesis of Christensen \cite{christensen-2013} a 
different approach to a first Chern form was given. Christensen studied 
the example $E^2$ for a Tate elliptic curve E and he defined the first 
Chern form as a tropical divisor on the skeleton of $E^2$.  Then he showed 
that the  {$2$-fold} tropical self-intersection of this divisor gives the 
Chambert-Loir measure.

\vspace{2mm}
In this paper, we combine both approaches. We enrich the theory of differential forms given in 
\cite{chambert-loir-ducros} by enlarging the 
space of smooth forms to the space of $\delta$-forms. 
They behave as forms and they have the advantage that we can define 
a first Chern $\delta$-form for a line bundle endowed with a formal metric. 
This leads to a direct definition of the Monge--Amp\`ere measure as a 
wedge-product of $\delta$-forms and to an approach to non-archimedean 
Arakelov theory.

This will be explained in more detail now.   {
Throughout this paper $K$ denotes 
an algebraically closed field endowed with a non-trivial 
non-archimedean complete absolute value.} 
Note that this is no restriction of generality as for many problems including the ones discussed in this paper 
such a setup can always be achieved by base change. This is similar to the archimedean case where analysis is usually performed over the complex numbers.
For sake of simplicity, we assume in the introduction that tropical 
cycles have constant weights as usual in tropical geometry 
(see Section \ref{tropinttheo} for details and for a generalization 
to smooth weights). A $\delta$-preform on $\R^r$ is a supercurrent 
$\alpha$ on $\R^r$ of the form
\begin{equation} \label{deltapreform}
\alpha = \sum_{i \in I} \alpha_i \wedge \delta_{C_i} 
\end{equation}
for finitely many superforms $\alpha_i$ and tropical cycles $C_i$ on $\R^r$. 
Using the wedge-product of superforms and the stable intersection product 
of tropical cycles, we get a wedge-product of  
{$\delta$-preforms}. 
Since the supercurrents of integration $\delta_{C_i}$ are $d'$-closed 
and $d''$-closed, we can extend the differential operators $d'$ and 
$d''$ to $\delta$-preforms leading to a differential bigraded $\R$-algebra. 
We refer to Section \ref{algebradeltaforms} for precise definitions and 
generalizations allowing smooth tropical weights.

In Section \ref{Currents and delta-preforms on tropical cycles}, we extend all 
these notions from $\R^r$ to a fixed tropical cycle $C$ of $\R^r$. 
The balancing condition is equivalent to closedness of the supercurrent $\delta_C$ which means that $C$ is boundaryless in 
the sense that no boundary integral shows up in the theorem of Stokes over $C$. Therefore we may view a tropical cycle  as a 
combinatorial analogue of a complex analytic space. 
Using integration over $C$, we will see that a piecewise smooth form $\eta$ 
on the support of $C$ induces a supercurrent $[\eta]$ on $C$. 
We will apply this to a piecewise  {smooth} function $\phi$ on $C$. 
In tropical geometry, $\phi$ plays the role of a Cartier divisor on $C$ and 
has an associated tropical Weil divisor $\phi \cdot C$. 
The latter is also called the corner locus of $\phi$ as it is a tropical 
cycle of codimension $1$ with support equal to the singular locus of $\phi$. 
We will show in Corollary  \ref{tropicalpoincarelelong} the following 
{\it tropical Poincar\'e--Lelong formula}:

\begin{thm}\label{tpl} 
Let $\phi$ be a piecewise smooth function on $C$ and let  $\delta_{\phi\cdot C}$ 
be the supercurrent of integration over the corner locus $\phi\cdot C$. Then we have 
\begin{equation*}
d'd''[\phi]-[d'd''\phi]=\delta_{\phi\cdot C}
\end{equation*}
as supercurrents on $C$. 
\end{thm}

This is a statement about integration of superforms on tropical currents and 
its proof relies on Stokes theorem. In fact, we prove a more general statement  
in  Theorem \ref{deltatropicalpoincarelelong} involving integration 
of $\delta$-preforms on $C$. 

Let $X$ be an $n$-dimensional algebraic variety over $K$. 
We now define $\delta$-forms on $\Xan$ similarly as differential forms, but 
replacing superforms by the more general $\delta$-preforms. 
This means that a $\delta$-form 
is given locally with respect to the Berkovich analytic topology on very affine open subsets by pull-backs of 
$\delta$-preforms with respect to the tropicalization maps. 
The $\delta$-preforms have to agree on overlappings which involves a 
quite complicated restriction process which is explained in 
Section \ref{deltaalgvar}. 
Moreover, we will show that $\delta$-forms are bigraded, have a wedge 
product and differential operators $d',d''$ extending the corresponding 
structures for differential forms on $\Xan$. 
There is also a pull-back with respect to morphisms and so we see 
that $\delta$-forms behave as differential forms on complex manifolds.

In Section \ref{Integration of delta-forms}, we will study integration of 
compactly supported $\delta$-forms of bidegree $(n,n)$ on $\Xan$. 
To define the integral of such a $\delta$-form $\alpha$, we choose a 
dense open subset $U$ of $X$ with a closed embedding $U \hookrightarrow {\mathbb G}_m^r$ such that $\alpha$ is given on $\Uan$ by the 
pull-back of a $\delta$-preform $\alpha_U$ on $\R^r$ with respect to 
the tropicalization map $\trop_U:\Uan \rightarrow \R^r$. 
Using the corresponding tropical variety $\Trop(U)$, we set 
\[
\int_\Xan \alpha := \int_{|\Trop(U)|}\alpha_U.
\]
In Section \ref{delta-currents}, we introduce $\delta$-currents as 
continuous linear functionals on the space of compactly supported $\delta$-forms. 
By integration, every $\delta$-form $\alpha$ induces a $\delta$-current $[\alpha]$. 
Similarly, we get a current of integration $\delta_Z$ for every cycle $Z$ on $X$. 
As a major result, we will show in Corollary \ref{deltaform defines-measure} 
that $[\alpha]$ is a signed Radon measure on $\Xan$ for every $\alpha$ of 
bidegree $(n,n)$. 
We will deduce in Proposition \ref{current continuous function} that 
every continuous real function $g$ on $\Xan$ induces a $\delta$-current 
$[g]$ on $\Xan$ which is defined at a compactly supported $\delta$-form 
$\alpha$ of bidegree $(n,n)$ by integrating $g$ with respect 
to the corresponding Radon measure. 

Now let $f$ be a rational function on $X$ which is not identically zero. 
By integration again, we will get a $\delta$-current $[-\log|f|]$ on $\Xan$.

\begin{thm} \label{PL-formula}
Let $\cyc(f)$ be the Weil divisor associated to $f$. 
Then we have the  Poincar\'e--Lelong equation
\[
\delta_{{\cyc}(f)}=d'd''\bigl[
\log|f|\bigr]
\]
of $\delta$-currents on $\Xan$.
\end{thm}

The proof will be given given in \ref{Poincare-Lelong equation}. 
The Poincar\'e--Lelong equation of Chambert-Loir and Ducros is the special case of 
our formula where one evaluates the $\delta$-currents at differential forms. 
The generalization to $\delta$-forms is not obvious and needs a more tropical adaption of their 
beautiful arguments. 
In Section \ref{PoincareLelongFormula}, we will introduce the first 
Chern $\delta$-current $[c_1(L,\metr)]$ of a continuously metrized line 
bundle $(L, \metr)$ on $X$. 
As usual, we mean here continuity with respect to the Berkovich topology on $\Xan$. 
In Corollary \ref{PL for line bundles}, we will deduce from 
Theorem \ref{PL-formula} that a non-zero meromorphic section $s$ of $L$ 
satisfies the Poincar\'e--Lelong equation
\begin{equation} \label{PL for s}
d'd''[-\log\|s\|]=[c_1(L,\metr)]-\delta_{\cyc(s)}
\end{equation}
for $\delta$-currents on $\Xan$. 

In Section \ref{piecewice smooth and formal metrics on lb}, we will define 
piecewise smooth (resp. piecewise linear) metrics on $L$. 
We will show in Proposition \ref{equivalences for formal metrics} that a 
metric is piecewise linear if and only if it is induced by a formal model 
of the line bundle. 
In Section \ref{piecewise smooth forms and delta-metrics}, we will introduce 
piecewise smooth forms on $\Xan$. For a piecewise smooth metric $\metr$ on $L$, 
the first Chern $\delta$-current $[c_1(L,\metr)]$ has a canonical 
decomposition into a sum of a piecewise smooth form  and a residual current.
If $\metr$ is smooth, then $c_1(L,\metr)$ is a differential form on $\Xan$. 
We say that a piecewise smooth metric $\metr$ is a $\delta$-metric if the 
first Chern $\delta$-current $[c_1(L,\metr)]$ is induced by a 
$\delta$-form $c_1(L,\metr)$ (see Definition \ref{delta-metrics} for a 
more precise definition). 
In this situation, we call $c_1(L,\metr)$ the first Chern $\delta$-form 
of $(L,\metr)$. 
We will see in Remark \ref{smooth and pl are delta-metrics}
that every 
piecewise linear metric is a $\delta$-metric. 
Canonical metrics on line bundles exist on line bundles on abelian varieties, 
on line bundles which are algebraically equivalent to zero and on 
line bundles on toric varieties. 
It follows from our considerations in Section 
\ref{piecewice smooth and formal metrics on lb} that all these canonical 
metrics are $\delta$-metrics (see Example \ref{examples for delta-metrics}). 

In Section \ref{MAm}, we consider a proper algebraic variety $X$ over $K$ of 
dimension $n$ with a line bundle $L$ endowed with an algebraic metric $\metr$. 
This means that the metric is induced by an algebraic model of $L$.  
Based on the formal GAGA principle, we  show in Proposition \ref{algebraic and formal metrics} that an algebraic 
metric is the same as a formal metric and hence this is also the same as a piecewise linear metric.  
As a consequence, we note that $\metr$ is a $\delta$-metric and hence 
$c_1(L,\metr)$ is a well-defined $\delta$-form. 
We deduce that $c_1(L,\metr)^n$ is a $\delta$-form of bidegree $(n,n)$ 
on $\Xan$ which we may view as a signed Radon measure on $\Xan$ by the above. 
We call it the Monge--Amp\`ere measure associated to $(L,\metr)$. 
In \ref{Monge-Ampere vs Chambert-Loir} we are going to show the following result:
                                                                                                                                                                                                                                    
\begin{thm} \label{CLvsMA}
Under the assumptions above, the Monge--Amp\`ere measure associated to  
$(L,\metr)$ is equal to the Chambert--Loir measure associated to $(L, \metr)$. 
\end{thm}

As mentioned before, this theorem was first proved by Chambert--Loir and Ducros 
in a slightly different setting (for discrete valuations, but their method 
works also for algebraically closed fields). 
However, they have a different construction of the Monge--Amp\`ere measure. 
Since  algebraic metrics are usually not smooth, they have only a first 
Chern current $c_1(L,\metr)$ available. In general, the wedge product of 
currents is not well-defined. In the present situation, they can use a rather 
complicated approximation process by smooth metrics to make sense of the 
wedge product $c_1(L_1,\metr_1)^n$ as a current leading to their 
Monge--Amp\`ere measure.  Our Monge--Amp\`ere measure is defined 
directly as a wedge product of $\delta$-forms based on tropical 
intersection theory instead of the approximation process. 
This means that our proof is more influenced by tropical methods. 

In Section \ref{GC}, we define a Green current for a cycle $Z$ on the 
algebraic variety $X$ over $K$ as a $\delta$-current $g_Z$ such that 
\[d'd''g_Z = [\omega_Z] - \delta_Z\]
for a $\delta$-form $\omega_Z$ on $\Xan$. 
By the Poincar\'e--Lelong equation \eqref{PL for s}, a non-zero meromorphic 
section $s$ of $L$ induces a Green current
$g_Y:=-\log\|s\|$ for the Weil divisor $Y$ of $s$. 
Here, we assume that $\metr$ is a $\delta$-metric on the line 
bundle $L$ of $X$. In case of proper intersection, we define 
$g_Y * g_Z := g_Y \wedge \delta_Z +  {\omega_Y \wedge g_Z}$ as in the 
archimedean theory of Gillet-Soul\'e.  
It is an easy consequence of the Poincar\'e--Lelong equation that 
$g_Y * g_Z$ is a Green current for the cycle $Y\cdot Z$. 
We will show the usual properties for such $*$-products. 
Most difficult is the proof of the commutativity of the $*$-product of 
two Green currents for properly intersecting divisors. 
It relies on the study of piecewise smooth forms and the tropical 
Poincar\'e--Lelong formula in Theorem \ref{tpl}. 

In Section \ref{Local heights of varieties}, we  define the local height 
of a proper $n$-dimensional variety $X$ over $K$ with respect to properly 
intersecting Cartier divisors $D_0,\dots, D_n$ endowed with $\delta$-metrics 
on $O(D_0),\dots, O(D_n)$ as follows: Let $g_{Y_j}$ be the 
Green current for the Weil divisor $Y_j$ associated to $D_j$ as above, 
then the local height is given by 
\[\lambda_{\hat{D}_0, \dots, \hat{D}_n}(X):= g_{Y_0} * \dots * g_{Y_n}(1).\]
We will show that these local heights are multilinear and symmetric in the 
metrized Cartier divisors $\hat{D}_0, \dots, \hat{D}_n$, functorial 
with respect to morphisms and satisfy an induction formula useful 
to decrease the dimension of $X$. 
For algebraic metrics, local heights of proper varieties are also 
defined using intersection theory on a suitable 
proper model (see \cite[\S 9]{gubler-crelle}). 

\begin{thm} \label{localalgebraicheights}
Suppose that the metrics on $O(D_0),\dots, O(D_n)$ are all algebraic. 
Then the local height $\lambda_{\hat{D}_0, \dots, \hat{D}_n}(X)$ based on 
the $*$-product of Green currents is equal to the local height of $X$ given 
by intersection theory of divisors on $\kcirc$-models.
\end{thm}

The proof uses that the induction formula holds for both definitions of 
local heights and then Theorem \ref{CLvsMA} gives the claim (see 
\ref{comparision with old local heights} for more details and the proof). 

In the introduction, we have presented the whole theory of $\delta$-forms 
based on  $\delta$-preforms as in \eqref{deltapreform} 
using tropical cycles with constant weights. However, the theory can 
be extended to $\delta$-forms locally given by $\delta$-preforms allowing 
tropical cycles with smooth weights. This will be done throughout the whole 
paper which leads to slightly more 
complications, but it increases the class of $\delta$-metrics 
at the end which makes it worthwhile.
 {Observe that tropical cycles which arise as tropicalizations
from varieties always have {integer} weights. 
Therefore tropical cycles are always considered with constant weights when
they serve as in Section 
\ref{Currents and delta-preforms on tropical cycles} as underlying spaces for
supercurrents and $\delta$-preforms.}

The authors would like to thank  {José Ignacio Burgos Gil, Thomas Fenzl, 
Philipp Jell, Christian Vilsmeier, and Veronika Wanner} for helpful comments and the 
collaborative research center
SFB 1085 funded by the Deutsche For\-schungsgemeinschaft for its support. 
We also thank  {the anonymous referees for their careful reading and their} 
detailed comments.

\medskip
\centerline{\sc {Notations and }Terminology}  \nopagebreak
\medskip  

 {
Throughout this paper $K$ denotes an algebraically closed field
endowed with a complete non-trivial non-archimedean absolute value $|\phantom{a}|$,
valuation ring $\kcirc$,
and corresponding valuation $v=-\log |\phantom{a}|$. 
Let $\Gamma:=v(K^\times)$ be the value group.}

In $A \subset B$, $A$ is strictly smaller than $B$. 
 The complement of $A$ in $B$ is denoted by $B \setminus A$. 
The zero is included in $\N$ and in $\R_+$.

The group of multiplicative units in a ring $A$ with $1$ is denoted by $A^\times$. 
An (algebraic) variety over a field is an irreducible separated reduced scheme of finite type. 
The terminology from convex geometry is explained in the appendix.

\section{Tropical intersection theory with smooth weights}\label{tropinttheo}

In tropical geometry, a tropical cycle is given by a polyhedral complex whose 
maximal faces are weighted by integers satisfying a balancing condition along 
the faces of codimension $1$. In this section, we generalize the notion of 
a tropical cycle allowing smooth real functions on the maximal faces as weights.
This is similar to the tropical fans with polynomial weights 
introduced by Esterov and Fran{\c{c}}ois \cite{esterov, francois}.
{We generalize basic facts from stable tropical intersection theory 
and introduce the corner locus of a piecewise smooth function}.

Throughout this section $N$ and $N'$ denote free $\Z$-modules of finite rank 
$r$ and $r'$. We write $N_\R=N\otimes_\Z\R$ and $N_\R'=N'\otimes_\Z\R$.
Every integral $\R$-affine polyhedron $\sigma$ of dimension $n$ 
in the $\R$-vector space $N_\R=N\otimes_\Z\R$ determines an
affine subspace $\A_\sigma$  with
underlying vector space $\L_\sigma$ and a lattice
$N_\sigma=\L_\sigma\cap N$ in $\L_\sigma$
 {(see Appendix \ref{polyhedron})}. 
A smooth function $f:\sigma\to \R$ on a polyhedron $\sigma$ in $N_\R$ is 
the restriction of a smooth function on some open neighbourhood of 
$\sigma$ in $\A_\sigma$. For  further notation borrowed 
from convex geometry, we refer to Appendix \ref{convex geometry}.

\begin{definition}\label{polcomplex}
(i) A polyhedral complex $\KC$ of pure dimension $n$ is called {\it weighted
(with smooth weights)} if each polyhedron  $\sigma\in\KC_n$ is endowed
with a smooth weight function $m_\sigma:\sigma\to \R$. 
If all $m_\sigma$ are constant functions, then we call them 
\emph{constant weights}.
The {\it support} $|(\KC,m)|$ of a weighted polyhedral complex $(\KC,m)$
of pure dimension $n$ is the closed set
\[
|(\KC,m)|=\bigcup_{\sigma\in \KC_n}{\rm supp}\,(m_\sigma).
\]
The support $|\KC|$ of a polyhedral complex $\KC$ is the support of 
$(\KC,m)$ where $m=1$ is the trivial weight function.
We have $|(\KC,m)|\subseteq |\KC|$.

(ii) Let $C=(\KC,m)$ be an integral $\R$-affine polyhedral complex 
of pure dimension $n$ with smooth weights in $N_\R$.
For each codimension one face $\tau$ of a polyhedron $\sigma\in \KC_n$
we choose a representative $\omega_{\sigma,\tau}\in N_\sigma$
of the generator of the one-dimensional lattice $N_\sigma/N_\tau$
pointing in the direction of $\sigma$.
Then we say that the weighted polyhedral complex $C$ satisfies the 
\emph{balancing condition} if we have
\begin{equation}\label{balancingcondition}
\sum\limits_{\genfrac{}{}{0pt}{}{\sigma\in \KC_n}{{\sigma \succ \tau}}}
m_\sigma(\omega)\omega_{\sigma,\tau}\in \L_\tau
\end{equation}
for all $\tau \in \KC_{n-1}$ and all $\omega\in \tau$, where $\sigma \succ \tau$ means that $\tau$ is a face of $\sigma$.
This is a straightforward generalization of the balancing condition
for polyhedral complexes with integer weights \cite[13.9]{gubler-guide}.

(iii) A {\it tropical cycle 
$C=(\KC,m)$ of dimension $n$} in $N_\R$ is a weighted integral  
$\R$-affine  polyhedral complex of pure dimension $n$ which 
satisfies the balancing condition (\ref{balancingcondition}).
In the following, we identify two tropical cycles
$(\KC,m)$ and $(\KC',m')$ of dimension $n$ if $|(\KC,m)|=|(\KC',m')|$  and if $m_\sigma = m_{\sigma'}$
on the intersection of the relative interiors of $\sigma$ and $\sigma'$ for all  $\sigma\in \KC_n$ and $\sigma'\in \KC_n'$.
A tropical cycle $(\KC,m)$ whose underlying polyhedral complex is a rational polyhedral fan
 in the vector space $N_\R$ is 
called a {\it tropical fan}.

(iv) Let $C=(\KC,m)$ be a tropical cycle of  dimension $n$.
Given an integral $\R$-affine subdivision $\KC'$ of $\KC$,
there is a unique family of weight functions $m'$ such that $(\KC',m')$ 
is a tropical cycle and $m_\sigma|_{\sigma'}=m'_{\sigma'}$ holds for all
$\sigma'\in \KC_n'$ and $\sigma\in \KC_n$ such that $\sigma'\subseteq \sigma$.
If a tropical cycle $C$ in $N_\R$ is defined by a weighted  integral $\R$-affine polyhedral
complex $(\KC,m)$, we call $\KC$ a {\it polyhedral complex of definition
for the tropical cycle $C$}.

(v) The set of tropical cycles with smooth weights 
of pure dimension $n$ in $N_\R$ defines an 
abelian group ${\rm TZ}_n(N_\R)$ where the group law is given by the addition
of multiplicity functions on a common refinement of the integral $\R$-affine polyhedral
complexes.
We denote by ${\rm TZ}^k(N_\R)={\rm TZ}_{r-k}(N_\R)$ the {\it group of tropical
cycles of codimension $k$}.
\end{definition}

\begin{rem}[Reduction from smooth to constant weight functions]
\label{reductiontoclassicaltropicalvarieties}
In tropical geometry, one considers usually tropical cycles with integer 
weights. 
However it causes no problems to work instead with tropical
cycles with constant but not necessarily integer weights.
 
Many properties of these tropical cycles with integer or constant
weights extend even 
to tropical cycles with smooth weights by the following local argument 
in $\omega \in |\Ccal|$. We replace $\Ccal$ by the 
rational polyhedral fan of local cones 
in $\omega$ (see \ref{cones})
and we endow the local cone of $\sigma \in \Ccal_n$ by 
the constant weight $m_\sigma(\omega)$. 
By definition, these constant weights on the rational cones 
satisfy the balancing condition. 
 {We illustrate the use of this reduction process 
in \ref{intpairing}(ii).}
\end{rem} 

\begin{art} \label{stable troical intersection theory and FS}
In tropical geometry, there is a {\it stable tropical intersection product} of 
tropical cycles with integer weights. 
The astonishing fact is that this product is well-defined as a tropical 
cycle in contrast to algebraic intersection theory or homology, where an 
equivalence relation is needed. 
Constructions of a stable tropical intersection product  of tropical cycles 
with integer  weights have been given by Mikhalkin and Allermann-Rau. 
In both cases the construction is reduced to the case of tropical fans.
For tropical fans with integer weights, Mikhalkin \cite{mikhalkin-2006} 
uses the fan deplacement rule 
from \cite{fulton-sturmfels} whereas Allermann and Rau 
\cite{allermann-rau-2010} use reduction to the
diagonal and intersections with tropical Cartier divisors.  
It is shown in \cite[\S 5]{katz-2012}, \cite[Thm. 1.5.17]{rau-2009} that both definitions agree. 
This is based on a result of Fulton--Sturmfels \cite[Theorem 3.1] {fulton-sturmfels} which shows that the space of tropical fans,  with integer weights and with a given complete rational polyhedral fan $\Sigma$ as a polyhedral complex of definition, is canonically isomorphic to the Chow cohomology ring of the  complete toric variety $Y_\Sigma$ associated to $\Sigma$. Then the product in Chow cohomology leads to the stable intersection product of tropical fans with integer weights and the usual properties in Chow cohomology lead to corresponding properties in stable tropical intersection theory. By passing to a smooth rational polyhedral fan subdividing $\Sigma$, which means that  $Y_\Sigma$ is smooth, we may use the usual Chow groups instead of the Chow cohomology groups from \cite[Chapter 17]{fulton-intersection-theory}.
\end{art}

\begin{rem}[Stable tropical intersection theory]
\label{intpairing}
As an application of the reduction principle described in Remark \ref{reductiontoclassicaltropicalvarieties}, we  get a  stable tropical intersection theory for tropical cycles with smooth weights. The reduction process leads to tropical fans with constant weight functions. These weights are    not necessarily integers, but it is still possible to apply \ref{stable troical intersection theory and FS} by using Chow cohomology with real coefficients. We list here the main properties:

(i) There exists a natural bilinear pairing
\[
{\rm TZ}^k(N_\R)\times {\rm TZ}^l(N_\R)\longrightarrow 
{\rm TZ}^{k+l}(N_\R),\,\,
(C_1,C_2)\longmapsto
C_1\cdot C_2.
\]
which is called the {\it stable intersection product for tropical cycles}.
It is associative and commutative and respects supports in the sense that we have
$|C_1\cdot C_2|\subseteq |C_1|\cap |C_2|$.

(ii) Here  {comes the concrete construction} of the stable intersection product
for tropical cycles $C_1$ and $C_2$ of codimension $l_1$ and $l_2$ in $N_\R$
which is based on the fan displacement rule (see \cite[\S 4]{fulton-sturmfels}).
We choose a common polyhedral complex of definition $\KC$ for $C_1$ and $C_2$
and write $C_i=(\KC,m_i)$ $(i=1,2)$ for suitable families of weight functions
$m_i=(m_{i,\sigma})_{\sigma\in \KC^{l_i}}$.
Let $\KD$ denote the polyhedral subcomplex of $\KC$ which is generated by
$\KC^{l_1+l_2}$. We choose a generic vector $v\in N_\R$ for $\KD$, 
 a small $\varepsilon >0$
and equip $\KD$ with the family of weight functions $m=(m_\tau)_{\tau\in 
\KC^{l_1+l_2}}$ where $m_\tau:\tau\to \R$ is given by 
\begin{equation}\label{fandepruleg1}
m_\tau(\omega)=\sum\limits_{
\genfrac{}{}{0pt}{}{(\sigma_1,\sigma_2)\in
\KC^{l_1}\times \KC^{l_2}}{\genfrac{}{}{0pt}{}{
\tau=\sigma_1\cap\sigma_2}{\sigma_1\cap(\sigma_2+\varepsilon v)\neq \emptyset}}
}[N:N_{\sigma_1}+N_{\sigma_2}]\,
m_{1,\sigma_1}(\omega)\,m_{2,\sigma_2}(\omega).
\end{equation}
 {We will} show that $D=(\KD,m)$ is a tropical cycle whose construction is 
independent  of the choice of the generic vector $v$ and a sufficiently small   
$\varepsilon >0$.  {We use
$D$ as the definition} of the stable intersection product $C_1\cdot C_2$.

 {The proof is an illustration of the reduction to constant weights given in Remark \ref{reductiontoclassicaltropicalvarieties}. For $\omega \in |\KC|$, 
let $\KC_\omega$ be the rational polyhedral fan of local cones in $\omega$ of the polyhedra in $\Ccal$. First, we note that $\sigma \mapsto \rho := \LC_\sigma(\omega)$ is a bijective 
map from the set of polyhedra in $\KC$ containing $\omega$ onto $\Ccal_\omega$. 
For $i=1,2$ and $\sigma \in \KC_n$ with $\omega \in \sigma$, we endow the local cone $\rho =\LC_\sigma(\omega)$ with the 
constant weight $m_{i,\rho}(\omega):=m_{i,\sigma}(\omega)$. Since the weight functions $m_{i,\sigma}$  satisfy pointwise the balancing condition, we get a tropical fan $(\Ccal_\omega,m_i(\omega))$ 
with  real weights.}  

 {We claim that $m_\tau(\omega)$ from \eqref{fandepruleg1} is the same as the weight of the stable intersection product $(\Ccal_\omega,m_1(\omega)) \cdot (\Ccal_\omega,m_2(\omega))$ 
in $\tau\in 
\KC^{l_1+l_2}$   obtained from Chow cohomology as in \ref{stable troical intersection theory and FS}. To see this, note that for a generic vector $v \in N_R$, we choose 
 $\varepsilon >0$  so small that the condition $\sigma_1\cap(\sigma_2+\varepsilon v)\neq \emptyset$ is equivalent to $\rho_1\cap(\rho_2+v)\neq \emptyset$ for the corresponding cones $\rho_1,\rho_2$. Then \eqref{fandepruleg1} agrees 
with the formula in \cite[Theorem on p.~ 336]{fulton-sturmfels} for the product in Chow cohomology of proper toric varieties. By definition, the same formula is used for the stable intersection product of 
tropical fans with constant weights proving our local claim. It is well-known in tropical geometry and follows  from the comparison with Chow cohomology in \cite{fulton-sturmfels} that the stable tropical intersection product of tropical fans with real weights is independently defined from the choice of a generic vector $v$ and hence the definition of $D=(\KD,m)$ is independent of the choice of a generic vector $v\in N_\R$ and  a sufficiently small $\ve >0$.}

 {It is easily seen that the definition of $D$ is compatible with subdivisions and hence $C_1 \cdot C_2$ is a well-defined tropical cycle. The properties in (i) follow from the corresponding properties of the stable tropical intersection product of tropical fans with real weights.}

(iii) 
Let $F:N_\R'\to N_\R$ be an integral $\R$-affine map.
Let $C'=(\KC',m')$ be a weighted integral $\R$-affine polyhedral complex
in $N_\R'$ of pure dimension $n$.
After a suitable refinement we can assume that
\begin{equation}\label{intpairingg1}
F_*\KC':=\{F(\tau')\,|\, \exists \sigma'\in \KC_n' \,\mbox{ such that }\,
\tau' \preccurlyeq \sigma'\,\mbox{ and }\,F|_{\sigma'}\, \mbox{ is injective}\}
\end{equation}
is a polyhedral complex in $N_\R$.
We equip $F_*\KC'$ with the family of weight functions
\begin{equation}\label{intpairingg2}
m_{\nu}:\nu\to \R,\,\,
m_{\nu}(\omega)=\sum\limits_{
\genfrac{}{}{0pt}{}{\sigma' \in \KC_n'}{F(\sigma') =\nu}}
[N_\nu:\L_F(N'_{\sigma'})]m'_{\sigma'}\bigl((F|_{\sigma'})^{-1}(\omega)\bigr)
\end{equation}
for $\nu$ in $(F_*\KC')_n$ where $\L_F$ denotes the linear morphism 
defined by the affine morphism $F$.
The weighted integral $\R$-affine  polyhedral complex 
\[
F_*C'=(F_*\KC',m)
\]
in $N_\R$ of pure dimension $n$ is called the {\it direct image of $C'$ under $F$}.

(iv) 
Let $F:N_\R'\to N_\R$ be an integral $\R$-affine map.
There is a natural push-forward morphism
\[
F_*:{\rm TZ}_n(N_\R')\longrightarrow {\rm TZ}_{n}(N_\R),\,\,
C'\longmapsto F_*C',
\]
which satisfies
$|F_*C'|\subseteq F\bigl(|C'|\bigr)$.
Given a tropical cycle $C'$ in ${\rm TZ}_n(N_\R')$, we write $C'=(\KC',m')$
for a polyhedral complex of definition $\KC'$ such that
$F_*\KC'$ from (\ref{intpairingg1})
is a polyhedral complex in $N_\R$.
One defines the direct image $F_*C'=(F_*\KC',m)$ as in (iii) and 
verifies that $F_*C'$ is again a tropical cycle.
The formation of $F_*$ is functorial in $F$.
For further details see \cite[\S 7]{allermann-rau-2010} or
\cite[13.16]{gubler-guide}.

(v) 
Let $F:N_\R'\to N_\R$ be an integral $\R$-affine map.
There is a natural pull-back
\[
F^*:{\rm TZ}^l(N_\R)\longrightarrow {\rm TZ}^{l}(N_\R'),\,\,
C\longmapsto F^*(C)
\]
which satisfies
$|F^*C|\subseteq F^{-1}\bigl(|C|\bigr)$. 
The formation of $F^*$ is functorial in $F$. 
For a tropical cycle $C$ in ${\rm TZ}^l(N_\R)$, there is a complete polyhedral complex of definition
$\KC$  and we write $C=(\KC,m)$. 
After passing to a  subdivision,  there is a complete, integral $\R$-affine polyhedral complex $\KC'$ of $N_\R'$
such that for every $\gamma' \in \KC'$, there is a $\sigma \in \KC$ with   $F(\gamma') \subseteq \sigma$. 
We choose a generic vector $v  \in N_\R$ and a sufficiently small $\ve >0$. For $\gamma' \in (\KC')^l$, $\sigma \in \KC^l$ with $F(\gamma') \subseteq \sigma$ and $\sigma' \in (\KC')^0$ with $\gamma' \subseteq \sigma'$, we define
$$m_{\sigma',\sigma}^{\gamma'}:= 
\begin{cases} 
[N:\L_F(N')+N_\sigma] &\text{if $F(\sigma')$ meets $\sigma + \ve v$}\\
0                          &\text{otherwise.}
\end{cases}$$
These coefficients may depend on the choice of the generic vector $v$, but the following smooth weight function $m_{\gamma'}$ on $\gamma' \in (\KC')^l$ does not:
\begin{equation} \label{weight function for pull-back}
m_{\gamma'}(\omega'):= \sum_{\sigma',\sigma} m_{\sigma',\sigma}^{\gamma'} m_\sigma(F(\omega'))
\end{equation}
where $(\sigma',\sigma)$ ranges over all pairs in $(\KC')^0 \times \KC^l$ with $\gamma' \subseteq \sigma', F(\gamma') \subseteq \sigma$ and where $\omega' \in \gamma'$. 
By \cite[4.5--4.7]{fulton-sturmfels},  $(\KC')^{\geq l}$ equipped with the smooth weight functions $m_{\gamma'}$ is a tropical cycle in ${\rm TZ}^l(N_\R')$
which we define as $F^*(C)$. 

Let $p_1$ (resp. $p_2$) be the projection of $N_\R' \times N_\R$ to $N_\R'$ 
(resp. $N_\R$) and let $\Gamma_F$ be the graph of $F$ in $N_\R' \times N_\R$. 
Using the stable tropical intersection product from (ii) and \cite[4.5--4.7]{fulton-sturmfels}, we deduce 
\begin{equation} \label{graph formula}
F^*(C)=(p_1)_*(p_2^*(C)\cdot\Gamma_F).
\end{equation}
\end{rem}

\begin{prop}\label{tropintthprop}
Let $F:N_\R'\to N_\R$ be an integral $\R$-affine map.
\begin{enumerate}
\item[(i)]
For tropical cycles $C$ and $D$ on $N_\R$ we have
\[
F^*(C\cdot D)=F^*(C)\cdot F^*(D).
\]
\item[(ii)]
For tropical cycles $C$ on $N_\R$ and $C'$ on $N_\R'$ we have
\[
F_*\bigl(F^*(C)\cdot C'\bigr)=C\cdot F_*(C')
\]
\end{enumerate}
\end{prop}

\proof
We reduce as in
\ref{reductiontoclassicaltropicalvarieties} to the case where
our tropical cycles are tropical fans with constant weight functions. Since both sides of the claims are linear in the weights of the tropical fans, we may assume that the weights are integers. In this situation, the claims were proven by L. Allermann in  
\cite[Theorem 3.3]{allermann2012a}.
\qed

\begin{definition} \label{piecewise smooth function}
Let $\Omega$ be an open subset of an integral $\R$-affine polyhedral set $P$ in $N_\R$. We call $\phi:\Omega \rightarrow \R$ {\it piecewise smooth} 
 if there is an integral $\R$-affine polyhedral complex $\KC$ with support $P$ and smooth functions  $\phi_\sigma:\sigma \cap \Omega \to \R$  for every $\sigma \in \KC$ such that $\phi|_\sigma=\phi_\sigma$ on $\sigma \cap \Omega$. In this situation,  we call $\KC$ a {\it polyhedral complex of definition for the piecewise smooth function $\phi$}. We call $\phi$ {\it piecewise linear} if each $\phi_\sigma$ extends to an integral $\R$-affine function on $\A_\sigma$.
\end{definition}

\begin{rem} \label{tropical cycles of degree 0}
The balancing condition  \eqref{balancingcondition} for smooth 
weights shows easily that a tropical cycle 
of codimension $0$ in $N_\R$ is the same as a piecewise smooth function defined on 
the whole space $N_\R$. 
\end{rem}

\begin{prop} \label{extension of piecewise smooth}
Let $\phi$ be a piecewise smooth function on the open subset $\Omega$ of the integral $\R$-affine  
polyhedral set $P$ in $N_\R$ and let $\widetilde\Omega$ be any open subset of $N_\R$ with $\widetilde\Omega \cap P = \Omega$.
Then there is a piecewise smooth function on $\widetilde\Omega$ which restricts to $\phi$ on $\Omega$. 
\end{prop}

\begin{proof} 
We first show the claim in the special case when $\Omega=P$ is the support of a 
rational polyhedral fan $\KC$ of definition for $\phi$ and $\widetilde\Omega=N_\R$.  
After passing to a  subdivision of $\KC$, we can easily find a complete rational polyhedral fan $\KC'$ in $N_\R$ which contains
$\KC$. After suitable subdivisions of $\KC'$
(and $\KC$) we may furthermore assume that all cones in $\KC'$ are simplicial.
Now we will extend $\phi$ inductively by ascending dimension
from the cones in $\KC$ to the cones in $\KC'$.

Let $\sigma$ be a cone in $\KC'$ of dimension $m$.
We are looking for an extension $\tilde \phi$ of $\phi$ to $\sigma$.
By our inductive procedure, we can assume that $\phi$ is defined already 
on all  faces of codimension one of $\sigma$.
After a linear change of coordinates, we may assume that $\sigma$ is the 
standard cone $\R_+^m$ in $\R^m$.
Let us assume that $\phi$ is given on the face $\{x_i=0\}$ of $\sigma$ by 
the smooth function $\phi_i(x_1,\dots, x_{i-1},x_{i+1}, \dots ,x_m)$. 
For any  $1\leq i_1 < \dots < i_k \leq m$, the restriction of $\phi$ to the face 
$\{x_{i_1}=\dots=x_{i_k}=0\}$ of $\sigma$ is given by a smooth function 
$\phi_{i_1 \dots i_k}$ depending only on the coordinates $x_j$ with 
$j \not \in \{i_1,\dots , i_k\}$ which agrees with the restrictions of the 
functions $\phi_{i_1}, \dots, \phi_{i_k}$ to this face of codimension $k$. 
We consider all these functions $\phi_{i_1 \dots i_k}$ as functions on 
$\sigma$ depending only on the coordinates $x_j$. Then an elementary 
combinatorial argument shows that 
\[
\tilde{\phi}:=\sum_i \phi_i - \sum_{i<j} \phi_{ij} + \dots + 
(-1)^{k+1}\sum_{i_1 < \dots < i_k} \phi_{i_1 \dots i_k} \pm 
\dots +(-1)^{n+1}\phi_{1 \dots n}
\]
is a smooth extension of $\phi$ to $\sigma$. 

Finally, we prove the claim in general.  There is a finite open covering $(\Omega_i)_{i \in I}$ of $\Omega$ such that $\Omega_i=  \Omega_i \cap (\LC_{\omega_i}(P) + \omega_i)$ for the local cone $\LC_{\omega_i}(P)$ of $P$ at a suitable  $\omega_i$.
Let us choose an open covering $(\widetilde\Omega_i)_{i \in I}$ of $\widetilde\Omega$ such that $\widetilde\Omega_i\cap P = \Omega_i$. There is a partition of unity $(\rho_j)_{j \in J}$ on $\widetilde\Omega$  such that every $\rho_j$ has compact support in $\widetilde\Omega_{i(j)}$ for a suitable $i(j) \in I$. We choose  $\nu_j \in C^\infty(N_\R)$ with compact support in $\widetilde\Omega_{i(j)}$ such that $\nu_j \equiv 1$ on $\supp(\rho_j)$. 
Then the special case above shows that the piecewise smooth function $\nu_j \phi$ on $P$ has a piecewise smooth extension $\tilde\phi_j$ to $N_\R$. Even if $J$ is infinite, we note that only finitely many rational fans of definition occur  and the above construction gives piecewise smooth extensions $\tilde\phi_j$ with finitely many integral $\R$-affine polyhedral complexes of definition. By passing to a common refinement, we may assume that they are all equal to a complete integral $\R$-affine polyhedral complex $\KD$. We conclude that $\tilde\phi = \sum_{j \in J} \tilde\rho_j \tilde\phi_j$ is a piecewise smooth extension of $\phi$ to $\widetilde\Omega$ with $\KD$ as a polyhedral complex of definition. 
\end{proof}

\begin{rem} \label{variations of extensions} 
Let $\Sigma$ be a rational polyhedral fan of $N_\R$ and let $\phi:|\Sigma| \to \R$ be a piecewise linear function with polyhedral complex of definition $\Sigma$. Then $\phi$  is the restriction of 
a piecewise linear function on $N_\R$ with a complete rational polyhedral fan of definition.
The argument is a little different: 
By toric resolution of singularities, one can subdivide $\Sigma$ until we get a subcomplex of a smooth rational polyhedral fan $\Sigma'$ of $N_\R$ (see \ref{fans} for the connection to toric varieties). 
We may assume that $\phi(0)=0$. Let $\lambda$ be a primitive lattice vector 
contained in an edge of $\Sigma'$ with $\lambda \not \in |\Sigma|$
and let $\phi(\lambda )\in \Z$. Then there is a unique piecewise linear 
function $\phi'$ on $N_\R$ with $\phi'=\phi$ on $|\Sigma|$ and 
$\phi'(\lambda) = \phi(\lambda)$ for all primitive lattice 
vectors $\lambda$ as above.
\end{rem}

Similarly as in \cite{esterov,francois}, we introduce the corner locus of a piecewise 
smooth function.

\begin{definition}[Corner locus]\label{cornerlocus}
Let $C=(\KC,m)$ be a tropical cycle with smooth weights of dimension $n$. 
We consider a piecewise smooth function $\phi:|\KC| \to \R$ with  
polyhedral complex of definition $\KC$.   
Given $\tau \in \Ccal_{n-1}$ we choose for each $\sigma\in \KC_n$
with $\tau\prec \sigma$ an  $\omega_{\sigma,\tau}\in N_\sigma$ as 
in \ref{polcomplex}(ii). For $\omega$ in $\tau$,  we define
\[
\omega_\tau:=
\sum\limits_{\genfrac{}{}{0pt}{}{\sigma\in \KC_n}{\tau\prec \sigma}}
m_\sigma(\omega)\omega_{\sigma,\tau} \in \L_\tau.
\]
Note that $\omega_\tau$ depends on the choice of $\omega$. 
Viewing $\omega_{\sigma,\tau}$ (resp.\ $\omega_\tau$) 
as a tangential vector at $\omega$,   
we denote the corresponding derivative by  
$\frac{\partial \phi_\sigma}{\partial \omega_{\sigma,\tau}} {:=
\langle  d\phi_\sigma, \omega_{\sigma,\tau}\rangle}$ 
(resp.\ $ {\frac{\partial \phi_\tau}{\partial \omega_{\tau}}:=
\langle  d\phi_\tau, \omega_{\tau}\rangle}$).      
It is straightforward to check that the definition of the weight function
\begin{equation}\label{cornerlocusg1}
m_\tau:\tau \to \R,\,\,
m_\tau(\omega) := \Bigl(
\sum\limits_{\genfrac{}{}{0pt}{}{\sigma\in \KC_n}{\tau\prec \sigma}}
m_\sigma(\omega) 
\frac{\partial \phi_\sigma}{\partial \omega_{\sigma,\tau}}(\omega)\Bigr)
- \frac{\partial \phi_\tau}{\partial \omega_\tau}(\omega)
\end{equation}
does not depend on the choice of the $\omega_{\sigma,\tau}$.
The {\it corner locus $\phi \cdot C$ of $\phi$} 
is by definition the weighted polyhedral subcomplex $\Ccal'$ of $\Ccal$ 
generated by $\Ccal_{n-1}$ endowed with the smooth weight functions
$m_\tau$ defined in (\ref{cornerlocusg1}).
\end{definition}

\begin{rem}\label{connectiontodivisors}
Let $\phi:|\KC|\to \R$ be a piecewise linear function 
on a tropical cycle
$C=(\KC,m)$ with  integral weights. Then the corner locus 
$\phi\cdot \KC$ is a tropical cycle with integral weights which is the
tropical Weil divisor of $\phi$ on $C$ in the sense of Allermann and Rau 
\cite[6.5]{allermann-rau-2010}.
\end{rem}

Esterov showed in \cite[Theorem 2.7]{esterov} that the corner 
locus of a piecewise  {polynomial} function on a tropical cycle with polynomial 
weights is again a tropical 
cycle of the same kind. We have here a similar result for tropical cycles 
with smooth weights:

\begin{prop} \label{corner locus for tropical cycle}
The corner locus $\phi\cdot C$ of a piecewise smooth function 
$\phi:|\KC|\to \R$ on a tropical cycle $C=(\KC,m)$ of dimension $n$ 
is  a tropical cycle with smooth weights of dimension $n-1$. 
 {The corner locus is defined independently of
the choice of the polyhedral complex $\KC$ and  $\phi\cdot C$ depends only on
the function $\phi|_{|C|}$.}
\end{prop}

\proof
This follows  from Remark 
\ref{connectiontodivisors} as explained in \ref{reductiontoclassicaltropicalvarieties}. 
\qed

\begin{prop} \label{projection formula for corner loci}
Let $F:N'_\R \rightarrow N_\R$ be an integral $\R$-affine map. 
Let $\phi$ be a piecewise smooth function on an integral 
$\R$-affine  polyhedral complex $\Ccal$ on $N_\R$. 
Suppose that $C'=(\Ccal',m')$ is a tropical cycle on $N_\R'$ with 
smooth weights $m'$ such that $F(|\Ccal'|) \subseteq |\Ccal|$. 
Then we have the projection formula 
$F_*(F^*(\phi)\cdot C')=\phi\cdot {F_*(C')}$, where 
$F^*(\phi)$ is the piecewise smooth function on $|\Ccal'|$ 
obtained by $\phi \circ F$. 
 \end{prop}

\begin{proof} This follows locally as in 
\cite[Prop. 4.8]{allermann-rau-2010} 
using \ref{reductiontoclassicaltropicalvarieties},
 \ref{connectiontodivisors}, 
and a linearization procedure (see  proof 
of Proposition \ref{asscommcornerlocus}). 
\end{proof} 

\begin{prop}\label{asscommcornerlocus}
Let $C$ and $C'$ be tropical cycles on $N_\R$ 
with smooth weights. Let $\KC$ be a polyhedral complex of definition for
$C$ and $\phi:|\KC|\to \R$ and
$\psi:|\KC|\to \R$ piecewise smooth functions. 
Then we have the associativity law 
\begin{equation}\label{asscommcornerlocusg1}
\phi\cdot (C\cdot C')= (\phi\cdot C)\cdot C'
\end{equation} 
and the commutativity law 
\begin{equation}\label{asscommcornerlocusg2}
\phi\cdot (\psi\cdot C)= \psi\cdot (\phi\cdot C)
\end{equation} 
as identities of tropical cycles on $N_\R$.
\end{prop}

\begin{proof}
Using \ref{connectiontodivisors} it is shown in 
\cite[Lemma 9.7, Proposition 6.7]{allermann-rau-2010} 
that \eqref{asscommcornerlocusg1} and \eqref{asscommcornerlocusg2} 
hold for tropical cycles $C,C'$ with integral weights and 
piecewise linear functions $\phi,\psi:|\KC|\to \R$ with integral slopes.
As both sides of (\ref{asscommcornerlocusg1}) and (\ref{asscommcornerlocusg2})
are linear in weights
and slopes, both formulas extends by linearity to tropical cycles 
with constant weight functions and piecewise linear functions with 
arbitrary real slopes.

To reduce to the above situation, we will use the procedure described 
in \ref{reductiontoclassicaltropicalvarieties}. 
We may assume that $C$ and $C'$ are tropical cycles
of pure dimension $n$ and $n'$ respectively.
Let $\KC$ be an integral $\R$-affine polyhedral complex such that
$\KC_{\leq n}$ and $\KC_{\leq n'}$ are polyhedral complexes of definition
for $C$ and $C'$.
We write $C=(\KC_{\leq n},m)$, $C'=(\KC'_{\leq n'},m')$, and
$C\cdot C'=(\KC_{\leq l},m'')$ with $l:=n+n'-r$.
Given $\omega\in |\KC|$ we denote by $\KC_\omega$ the rational polyhedral fan 
of local cones of $\KC$ in $\omega$. 
There is a bijective correspondence between the polyhedra $\sigma\in \KC$ 
with $\omega\in \sigma$ and the cones $\sigma_\omega$ in $\KC_\omega$.
Each $\sigma\in \KC$ with $\omega\in \sigma$ determines a canonical isomorphism
of affine spaces
$I_\omega:\L_{\sigma_{\omega}}\stackrel{\sim}{\to} \A_\sigma$
with $I_\omega(0)=\omega$.
We obtain tropical fans with constant weight functions  
$C_\omega=(\KC_{\omega,\leq n},m(\omega))$, 
$C'_{\omega}=(\KC_{\omega,\leq n'},m'(\omega))$,
and $(C\cdot C')_\omega=(\KC_{\omega,\leq n},m''(\omega))$. We have
$C_\omega\cdot C'_\omega=(C\cdot C')_\omega$
by our construction of the stable tropical intersection product with smooth weights.
There is a unique piecewise linear function
$\phi_\omega:|\KC_\omega|\to \R$
such that for all $\sigma_\omega\in \KC_\omega$ the $\R$-linear function
$\phi_{\sigma_\omega}$ on $\L_{\sigma_\omega}$ determined by
$\phi_\omega|_{\sigma_\omega}=\phi_{\sigma_\omega}|_{\sigma_\omega}$ 
satisfies
\[
(d\phi_{\sigma_\omega})(0)=(I_\omega^*d\phi)(0)
\]
in $\L_{\sigma_\omega}^*$.
We write 
\begin{eqnarray*}
\phi\cdot (C\cdot C')&=&(\KC_{\leq  {l-1}},m_1),\\
(\phi\cdot C)\cdot C'&=&(\KC_{\leq  {l-1}},m_2),\\
\phi_\omega\cdot (C_\omega\cdot C_\omega')&=&(\KC_{\omega,\leq 
 {l-1}},m_{\omega,1}),\\
(\phi_\omega\cdot C_\omega)\cdot C_\omega'&=&(\KC_{\omega,\leq 
 {l-1}},m_{\omega,2}).
\end{eqnarray*}
The local nature of our definitions yields 
\[
m_{i,\sigma}(\omega)=m_{\omega,i,\sigma_\omega}(0)
\]
for $i=1,2$ and all $\sigma \in \KC_{\leq n+n'-r-1}$ with $\omega\in \sigma$.
Formula (\ref{asscommcornerlocusg1}) for
constant weight functions and piecewise linear functions with arbitrary real slopes gives 
$m_{\omega,1,\sigma_\omega}(0)=m_{\omega,2,\sigma_\omega}(0)$.
Hence $m_1=m_2$ and (\ref{asscommcornerlocusg1}) is proven in general.
The reduction of (\ref{asscommcornerlocusg2}) 
to the case of constant weight functions and piecewise linear functions
proceeds exactly in the same way.
\end{proof}

\begin{cor} \label{pullbackcornerlocus}
Let $F:N'_\R \rightarrow N_\R$ be an integral $\R$-affine map.
We consider a tropical cycle $C=(\KC,m)$ with smooth weights on $N_\R$ 
and a piecewise smooth function $\phi:|\KC|\to \R$.
We write $F^*C=(\KC',m')$ where $F(|\KC'|)\subseteq |\KC|$. 
Then $\phi$ induces a piecewise smooth function $F^*(\phi):|\KC'|\to \R$
and we have
\[
F^*(\phi)\cdot F^*(C)=F^*(\phi\cdot C),
\]
i.e. the formation of the corner locus is compatible with pull-back.
\end{cor}

\proof
Using \eqref{graph formula} giving pull-back as a stable intersection with the graph, 
the claim follows  
by applying Proposition \ref{projection formula for corner loci} and 
\eqref{asscommcornerlocusg1} in  Proposition \ref{asscommcornerlocus}.
\qed

\section{The algebra of delta-preforms}\label{algebradeltaforms}

In this section we define polyhedral supercurrents  
on an open subset $\widetilde\Omega$ in $N_\R$
for some free $\Z$-module $N$ of finite rank.
The polyhedral supercurrents are special supercurrents
in the sense of Lagerberg. We show that an analogue of
Stokes' theorem holds for polyhedral supercurrents with
respect to the polyhedral derivatives $\dpa$ and $\dpb$.
Then we introduce the algebra of $\delta$-preforms
on $\widetilde\Omega$ which is going to play a central role in this paper.
These $\delta$-preforms are  special 
polyhedral supercurrents defined by tropical cycles and superforms.
We show that $\delta$-preforms admit products and
pullback morphisms, satisfy a projection formula
and that the polyhedral derivative of a $\delta$-preform
coincides with its derivative in the sense of supercurrents.

Throughout this section $N$ and $N'$ denote free $\Z$-modules of finite rank 
$r$ and $r'$. We write $N_\R=N\otimes_\Z\R$ and $N_\R'=N'\otimes_\Z\R$. 
We refer to Appendix \ref{convex geometry} for the notation from convex geometry.

\begin{art}
Given an open subset $\widetilde\Omega$ in $N_\R$, we denote
by $A^{p,q}(\widetilde\Omega)$ the space of superforms of type
$(p,q)$ on $\widetilde\Omega$,
by $A_c^{p,q}(\widetilde\Omega)$ the space of superforms with compact support
of type $(p,q)$ on $\widetilde\Omega$, and by 
$D_{k,l}(\widetilde\Omega)=D^{r-k,r-l}(\widetilde\Omega)$ the 
space of supercurrents of type $(k,l)$ on $\widetilde\Omega$ in the 
sense of Lagerberg \cite{lagerberg-2012}
(see also \cite{chambert-loir-ducros, gubler-forms}).
We have seen in the introduction that $A:=\oplus_{p,q} A^{p,q}$ 
defines a sheaf of differential bigraded 
$\R$-algebras with respect to the differentials $d'$ and $d''$. 
The bigraded sheaf 
$D:=\oplus_{p,q} D^{p,q}$ contains $A$ as a bigraded subsheaf and has 
canonical differentials $d'$ and $d''$ extending those of $A$.

The sheaf $A^{p,q}$  comes with a natural operator 
$J^{p,q}:A^{p,q}\to A^{q,p}$ which extends to 
$J^{p,q}:D^{p,q}\to D^{q,p}$. 
The first one induces an involution $J:=\oplus_{p,q}J^{p,q}$ on $A$ which is determined by 
the fact that it is an endomorphism of sheaves of $A^{0.0}$-algebras 
and that $d' \circ J = J \circ d''$. The extension of $J$ to supercurrents 
is determined by 
$$\langle J(T) , \alpha \rangle = (-1)^r \langle T ,  J(\alpha) \rangle$$
for $\alpha \in A^{r-p,r-q}(\widetilde\Omega)$ and  
$T \in D^{p,q}(\widetilde\Omega)$.
Sections of $A^{p,p}$ (resp. $D^{p,p}$) which are invariant under the
action of $(-1)^p J^{p,p}$ are called {\it symmetric superforms}
({\it resp. symmetric supercurrents}).
Sections of $A^{p,p}$ (resp. $D^{p,p}$) which are invariant under the
action of $(-1)^{p+1} J^{p,p}$ are called {\it anti-symmetric superforms}
({\it resp. anti-symmetric supercurrents}).
\end{art}

\begin{art}\label{formsaffinespace} 
Let $\widetilde\Omega$ be an open subset of $N_\R$.
An integral $\R$-affine polyhedron $\Delta$ of dimension $n$ in
$N_\R$ determines a canonical calibration
\[
 {\mu_\Delta\in |\!\wedge^n\!\L_\Delta|
={\rm Or}\,(\A_\Delta)\times^{\pm 1}
\wedge^n\L_\Delta}
\]
as in \cite[ {(1.3.5)}]{chambert-loir-ducros}.
Given a superform $\alpha \in A_c^{n,n}(\widetilde\Omega)$ the integral
\[
\int_\Delta\alpha=\int_{N_\R}
\langle \alpha, \mu_\Delta\rangle
\]
was defined in \cite[{\S 1.5}]{chambert-loir-ducros} 
{(see also \cite[\S 3]{gubler-forms})}.
The polyhedron $\Delta$ determines a continuous functional
\begin{equation}\label{intcurrent}
A_c^{n,n}(\widetilde\Omega)\longrightarrow 
\R,\,\,\alpha\longmapsto
\int_\Delta\alpha
\end{equation}
and a symmetric supercurrent $\delta_\Delta\in D_{n,n}(\widetilde\Omega)$.

 {For $\Omega := \widetilde\Omega \cap \Delta$, we define  $A_\Delta^{p,q}(\Omega)$ as the space of superforms on the open subset $\widetilde\Omega \cap \relint(\Delta)$ 
of the affine space $\A_\Delta$ given by restriction of elements in $A^{p,q}(\widetilde\Omega)$.  
A partition of unity argument shows that this definition is independent of the 
choice of $\widetilde\Omega$.}

For a superform $\alpha\in A^{p,q}_\Delta(\widetilde\Omega \cap\Delta)$ 
the supercurrent
\[
\alpha\wedge \delta_\Delta\in D_{n-p,n-q}(\widetilde\Omega)
\]
is defined by $\langle \alpha\wedge \delta_\Delta,\beta\rangle=
\langle \delta_\Delta , \alpha\wedge \beta\rangle$
for all $\beta\in A_c^{n-p,n-q}(\widetilde\Omega)$.
\end{art}

\begin{definition}[Polyhedral supercurrents]\label{polyhedral current}
Let $\widetilde\Omega$ be an open subset of $N_\R$.
A supercurrent $\alpha\in D(\widetilde\Omega)$ is called {\it polyhedral}
if there exists an integral $\R$-affine polyhedral complex $\KC$ in $N_\R$ and
a family $(\alpha_\Delta)_{\Delta\in \KC}$ of
superforms $\alpha_\Delta\in A_\Delta(\widetilde\Omega\cap \Delta)$ 
such that
\begin{equation}\label{stdeltaform}
\alpha=\sum_{\Delta\in \KC}\alpha_\Delta\wedge
\delta_{\Delta}
\end{equation}
holds in $D(\widetilde\Omega)$.
In this case we say that $\KC$ is {\it a polyhedral complex of 
definition for $\alpha$}.
The {\it polyhedral derivatives} $\dpa(\alpha)$ and $\dpb(\alpha)$
of a polyhedral supercurrent (\ref{stdeltaform}) are the polyhedral supercurrents
defined by the formulas
\begin{eqnarray*}
\dpa(\alpha)=\sum_{\Delta\in \KC}d'(\alpha_\Delta)\wedge
\delta_{\Delta},\\
\dpb(\alpha)=\sum_{\Delta\in \KC}d''(\alpha_\Delta)\wedge
\delta_{\Delta}.
\end{eqnarray*}
\end{definition}

\begin{rem}
(i) We observe that the family of forms $(\alpha_\Delta)_{\Delta\in \KC}$ in
(\ref{stdeltaform}) is uniquely determined by $\alpha$ and $\KC$.
Furthermore the support $\supp(\alpha)$ of a polyhedral supercurrent $\alpha$ is the union of 
the supports of the forms $\alpha_\Delta$ for all $\Delta\in \KC$.

(ii) It is straightforward to check that the definitions of the polyhedral
derivatives $\dpa(\alpha)$ and $\dpb(\alpha)$ do not depend on the
choice of the polyhedral complex of definition $\KC$.

(iii) We do not claim that the polyhedral derivatives of a polyhedral supercurrent $\alpha$
coincide with derivative of a $\alpha$ in the sense of supercurrents.
In fact the derivatives of a polyhedral supercurrent in the sense of supercurrents 
are in general not even polyhedral.
\end{rem}

\begin{definition}\label{intdelta}
Let $\widetilde\Omega$ denote an open subset of $N_\R$.
Let $P \subseteq \widetilde{\Omega}$ be an integral $\R$-affine polyhedral set 
 in $N_\R$.
We choose an integral $\R$-affine polyhedral complex  {$\KC$}
in $N_\R$ whose support is $P$. 

(i) Let $\alpha\in D_{0,0}(\widetilde\Omega)$ be a polyhedral supercurrent
such that $\supp(\alpha) \cap P$ is compact.  
After suitable refinements, we may assume that $\alpha$ admits a polyhedral 
complex of definition  {$\KD$} such that $\KD$ is a subcomplex of $\KC$.
In this situation we write
\begin{equation}\label{intdeltag1} 
\alpha=\sum_{\Delta\in  {\KD}}\alpha_\Delta\wedge \delta_\Delta
\end{equation}
as in (\ref{stdeltaform}) and define the integral of $\alpha$ over $P$ as
\begin{equation}\label{intdeltag2} 
\int_{P}\alpha=\sum_{\Delta\in \KD}
\int_{\Delta} \alpha_\Delta.
\end{equation}

(ii) Let $\beta\in D_{1,0}(\widetilde\Omega)$ be a polyhedral supercurrent 
with $\supp(\beta) \cap P$ compact. Proceeding as in (i), 
 {we get $\beta=\sum_{\Delta\in  {\KD}}\beta_\Delta\wedge \delta_\Delta$ 
for a suitable subcomplex $\KD$ of $\KC$} and we define 
the integral of $\beta$ over the boundary of $P$ as
\begin{equation}\label{intdeltag3} 
\int_{\partial P}\beta=\sum_{\Delta\in \KD}
\int_{\partial\Delta} \beta_\Delta
\end{equation}
where the boundary integrals on the right are defined as in 
\cite[1.5]{chambert-loir-ducros} and \cite[2.6]{gubler-forms}.
We define the boundary integral (\ref{intdeltag3}) for
a polyhedral supercurrent $\beta\in D_{0,1}(\widetilde\Omega)$ with $\supp(\beta) \cap P$ compact by the same 
formula. 
\end{definition}

\begin{rem}\label{boundaryintdelta} 
(i) The definitions in (\ref{intdeltag2}) and (\ref{intdeltag3}) do not depend 
on the choice of the polyhedral complex $\KD$.

(ii) On the Borel algebra $\B(P)$, we get signed measures 
\[
\mu_{P,\alpha}:\B(P)\to \R,\,\,\,
\mu_{P,\alpha}(M)=\sum_{\Delta\in \KD}\int_{M\cap \Delta} \alpha_\Delta
\]
and 
\[
\mu_{\partial P,\beta}:\B(P)\to \R,\,\,\,
\mu_{\partial P,\beta}(M)=\sum_{\Delta\in \KD}\int_{M\cap \partial\Delta} 
\beta_\Delta.
\]

(iii) We recall from \ref{polyhedral set} that $\relint(P)$ denotes the 
set of regular points of a polyhedral set $P$. 
Then $\supp(\beta) \cap P \subseteq \relint(P)$ implies
\begin{equation}\label{stokesintdeltag3} 
\int_{\partial P}\beta=0
\end{equation}
as an immediate consequence of the definitions.
\end{rem}

\begin{prop}[Stokes' formula for polyhedral supercurrents]
\label{stokesforpolyhedralcurrents}
Let $\widetilde\Omega$ denote an open subset and $P$ an integral $\R$-affine polyhedral subset 
 in $N_\R$ with $P \subseteq\widetilde{\Omega}$. 
Then we have
\[
\int_P\dpa\alpha=\int_{\partial P}\alpha,\,\,\,
\int_P\dpb\beta=\int_{\partial P}\beta
\]
for all polyhedral supercurrents $\alpha\in D_{1,0}(\widetilde\Omega)$ 
and $\beta\in D_{0,1}(\widetilde\Omega)$ with $\supp(\alpha) \cap P$ and $\supp(\beta) \cap P$ 
compact.
\end{prop}

\proof
We choose a polyhedral complex of definition $\KC$ for $\alpha$ 
such that a subcomplex $\KD$ has support $P$. 
By linearity it is sufficient to treat the case 
$\alpha=\alpha_\Delta\wedge\delta_\Delta$ for a superform
$\alpha_\Delta\in A_c^{n-1,n}(\widetilde\Omega\cap \Delta)$ and $\Delta \in \KD_n$.
We get
\[
\int_P\dpa(\alpha)
=\int_Pd'(\alpha_\Delta)\wedge\delta_\Delta
=\int_\Delta d'(\alpha_\Delta)
=\int_{\partial\Delta}\alpha_\Delta
=\int_{\partial P}\alpha.
\]
using Stokes' formula for superforms on polyhedra 
(see \cite[(1.5.7)]{chambert-loir-ducros} or \cite[2.9]{gubler-forms}). 
The formula for $\beta$ follows in the same way.
\qed

\begin{rem} \label{superform wedge tropical cycle}
Let $\widetilde\Omega$ be an open subset of $N_\R$.
An integral $\R$-affine polyhedral complex  $C=(\KC,m)$ with smooth weights
of pure dimension $n$ and a superform $\alpha \in A^{p,q}(\widetilde\Omega)$
determine a polyhedral supercurrent
\[
\alpha\wedge \delta_C =\sum_{\Delta\in \KC_n}
\bigl(m_\Delta\cdot \alpha|_\Delta\bigr)\wedge \delta_\Delta
\in D_{n-p,n-q}(\widetilde\Omega).
\] 
We get in particular the polyhedral supercurrents $[\alpha]=\alpha\wedge \delta_{N_\R}
\in D_{r-p,r-q}(\widetilde\Omega)$
and $\delta_C=1\wedge \delta_C\in D_{n,n}(\widetilde\Omega)$.
\end{rem}

\begin{definition}[$\delta$-preforms]\label{defdeltaforms}
(i) Let $\widetilde\Omega$ be an open subset of $N_\R$.
A supercurrent $\alpha\in D_{r-p,r-q}(\widetilde\Omega)$ is called a {\it $\delta$-preform of 
type $(p,q)$} if there exist a finite set $I$, a family 
$(C_i)_{i\in I}$ of tropical 
cycles with smooth weights $C_i=(\KC_i,m_i)$ of codimension $n_i$ in $N_\R$,
and a family $(\alpha_i)_{i\in I}$ of 
superforms $\alpha_i\in A^{p_i,q_i}(\widetilde\Omega)$
such that $p_i+n_i=p$ and $q_i+n_i=q$ for all $i\in I$ and 
\begin{equation}\label{tropstandform}
\alpha=\sum_{i\in I}\alpha_i\wedge
\delta_{C_i}
\end{equation}
holds in $D_{r-p,r-q}(\widetilde\Omega)$.
The {\it support} of a $\delta$-preform is the support of its
underlying supercurrent.

(ii) The $\delta$-preforms  define a subspace $P^{p,q}(\widetilde\Omega)$ in 
$D_{r-p,r-q}(\widetilde\Omega)$.
We put 
\[
P^n(\widetilde\Omega)=\bigoplus_{p+q=n}P^{p,q}(\widetilde\Omega)
\]
and $P(\widetilde\Omega)=\oplus_{n\in \N}P^n(\widetilde\Omega)$.
We denote by $P_c(\widetilde\Omega)$ the subspace of
$P(\widetilde\Omega)$ given by the $\delta$-preforms with compact support.
A $\delta$-preform $\alpha \in P^{p,p}(\widetilde\Omega)$ of type $(p,p)$ is called 
{\it symmetric (resp. anti-symmetric)}, if the underlying supercurrent 
of $\alpha$ is symmetric (resp. anti-symmetric).

(iii) We say that a $\delta$-preform $\alpha$
has {\it codimension $l$}, if it admits a presentation (\ref{tropstandform})
where all the tropical cycles $\KC_i$ are of pure codimension $l$.
The $\delta$-preforms  {of type $(p+l,q+l)$} of codimension $l$ define a subspace
of $D^{p+l,q+l}(\widetilde\Omega)$ which we denote 
by $P^{p,q,l}(\widetilde\Omega)$.
As an immediate consequence of our definitions, we have the direct sum
\[
P^n(\widetilde\Omega)=\bigoplus_{p+q+2l=n}P^{p,q,l}(\widetilde\Omega).
\]
\end{definition}

\begin{ex} \label{preforms of degree 0} 
It follows from Remark \ref{tropical cycles of degree 0} that a 
$\delta$-preform  of codimension $0$ on $\widetilde\Omega$ is the same 
as a superform on $\widetilde\Omega$ with piecewise smooth coefficients.
\end{ex}

\begin{rem}\label{standardform}
Let 
\[
\alpha=\sum_{i\in I}\alpha_i\wedge
\delta_{C_i}\in P^{p,q,l}(\widetilde\Omega)
\]
be a $\delta$-preform as in (\ref{tropstandform}).
Let $\KC$ be a common polyhedral complex of definition for
the tropical cycles $(C_i)_{i\in I}$. 
Then the supercurrent $\alpha$ is polyhedral and $\KC$ is a polyhedral
complex of definition for $\alpha$.
In fact we have $C_i=(\KC,m_i)$ for suitable families of weight functions $m_{i,\Delta}$
on polyhedra $\Delta$ in $\KC_{r-l}$ and define
\[
\alpha_{\Delta}:=\sum_{i\in 
I}m_{i,\Delta}\cdot
(\alpha_i|_{\Delta})\in A_\Delta^{p,q}(\widetilde\Omega\cap |\Delta|).
\]
Then we get
\[
\delta_{C_i}=\sum_{\Delta\in \KC_{r-l}}
m_{i,\Delta}\wedge \delta_{\Delta}
\]
and
\[
\alpha=\sum_{\Delta\in \KC_{r-l}}\alpha_\Delta\wedge
\delta_{\Delta}.
\]
In order to compare  $\delta$-preforms
\[
\alpha=\sum_{i\in I}\alpha_i\wedge
\delta_{{C}_i},\,\,
\beta=\sum_{j\in J}\beta_j\wedge
\delta_{{D}_j}\in P^{p,q,l}(\widetilde\Omega)
\]
presented as in (\ref{tropstandform}),
we choose a common polyhedral complex of definition $\KC$ for 
the finite families  $(C_i)_{i\in I}$ and $(D_j)_{j\in J}$ 
of tropical cycles and obtain
\begin{equation}\label{equalcrit}
\alpha=\beta\,\,\Leftrightarrow
\alpha_{\Delta}=\beta_{\Delta}\,
\forall \Delta\in \KC_{r-l} .
\end{equation}
\end{rem}

\begin{prop}\label{deltaformsbasicprop}
Let $\widetilde\Omega$ denote an open subset of $N_\R$.
Presenting $\delta$-preforms as in (\ref{tropstandform}), we can
perform the following constructions:
\begin{enumerate}
\item[(i)]
We have a canonical $C^\infty(\widetilde{\Omega})$-linear map
\[
A^{p,q}(\widetilde\Omega)\longrightarrow 
P^{p,q,0}(\widetilde\Omega),\,\,\,
\alpha \longmapsto \alpha\wedge \delta_{N_\R}
\]
and a $C^\infty(N_\R)$-linear isomorphism
\[
{\rm TZ}^l{}(N_\R)
\stackrel{\sim}{\longrightarrow}  P^{0,0,l}(N_\R),\,\,\,
C\longmapsto 1\wedge \delta_C.
\]
\item[(ii)]
There are well-defined $C^\infty(\widetilde{\Omega})$-bilinear products
\begin{eqnarray*}
\wedge:P^{p,q,l}(\widetilde\Omega)\otimes_\R P^{p',q',l'}(
\widetilde\Omega)& \longrightarrow &
P^{p+p',q+q',l+l'}(\widetilde\Omega),\\
\Bigl(\sum_{i\in I}\alpha_i\wedge\delta_{{C}_i}\Bigr)\wedge
\Bigl(\sum_{j\in J}\beta_j\wedge\delta_{{D}_j}\Bigr)
&=&\sum_{(i,j)\in I\times J}(\alpha_i\wedge \beta_j)\wedge
\delta_{{C}_i\cdot {D}_j}.
\end{eqnarray*}

\item[(iii)]
An integral $\R$-affine map $F:N'_\R\to N_\R$ induces a natural pull-back 
\[
F^*:P^{p,q,k}(\widetilde\Omega)\longrightarrow 
P^{p,q,k}(\widetilde\Omega'),\,\,\,
\sum_{i\in I}\alpha_i\wedge
\delta_{{C}_i}\longmapsto
\sum_{i\in I}(F^*\alpha_i)\wedge
\delta_{F^*{C}_i}
\]
for any open subset $\widetilde\Omega'$ of $F^{-1}(\widetilde\Omega)$.
\item[(iv)]
The pull-back morphism $F^*$ in (iii)  
satisfies 
\[
F^*(\alpha\wedge \beta)=(F^*\alpha) \wedge (F^*\beta)
\]
for all  $\alpha,\beta\in P(\widetilde\Omega)$.
\end{enumerate}
\end{prop}

\proof
The proof of (i) is straightforward. 
For (ii), we have to show that the definition
\begin{equation} \label{wedge definition}
\alpha \wedge \beta := \sum_{(i,j)\in I\times J}(\alpha_i\wedge \beta_j)\wedge
\delta_{{C}_i\cdot {D}_j}
\end{equation}
is independent of the presentations 
\[
\alpha=\sum_{i\in I}\alpha_i\wedge\delta_{C_i}\in P^{p,q,l}(\widetilde\Omega),\,\,\,
\beta=\sum_{j\in J}\beta_j\wedge\delta_{D_j}\in P^{p',q',l'}(\widetilde\Omega)
\]
given as in  \eqref{tropstandform}. We choose a common 
 polyhedral complex of definition $\KC$ for all tropical cycles $C_i$ and $D_j$. 
Using Remark \ref{standardform}, we represent the $\delta$-preforms as 
polyhedral supercurrents
\begin{equation} \label{standard alpha beta}
\alpha = \sum_{\sigma \in \Ccal^l} \alpha_\sigma \wedge \delta_\sigma, \,\,\, 
\beta = \sum_{\sigma' \in \Ccal^{l'}} \beta_{\sigma'} \wedge \delta_{\sigma'}.
\end{equation} 
We choose a generic vector $v$ and $\ve >0$ as in \ref{intpairing}(ii).
From \eqref{wedge definition} and \eqref{fandepruleg1}, we deduce
\begin{equation}  \label{standard form and wedge}
\alpha\wedge \beta = 
\sum_{\tau \in \KC^{l+l'}} \sum_{\sigma,\sigma'}
[N:N_{\sigma}+N_{\sigma'}] \cdot \alpha_\sigma \wedge \beta_{\sigma'} \wedge \delta_\tau
\end{equation}
where $\sigma, \sigma'$ ranges over all pairs in $
\KC^{l}\times \KC^{l'}$ with $\sigma \cap \sigma' = \tau$ and $\sigma \cap (\sigma'+\ve v) \neq \emptyset$. 
Then (ii) follows from \eqref{standard form and wedge} and from the uniqueness of the representations in \eqref{standard alpha beta}. Bilinearity is obvious.

Similarly we show (iii).
Given a $\delta$-preform $\alpha$ as above, we have to prove that
\begin{equation}\label{deltaformsbasicpropg3}
F^*(\alpha):=\sum_{i\in I}(F^*\alpha_i)\wedge \delta_{F^*C_i}
\end{equation}
is independent of the representation of $\alpha$ in \eqref{standard alpha beta}. There is a complete, integral $\R$-affine polyhedral complex $\KC'$ of $N_\R'$
and a complete, common polyhedral complex of definition $\KC$ for all tropical cycles $C_i$ satisfying the following compatibility property: 
For every $\sigma' \in \KC'$, there is a $\sigma \in \KC$ with $F(\sigma') \subseteq \sigma$. Using the coefficients $m_{\sigma',\sigma}^{\gamma'}$ from \ref{intpairing}(v), 
we deduce from \eqref{deltaformsbasicpropg3} and \eqref{weight function for pull-back} that 
\begin{equation} \label{standard form and pull-back}
F^*(\alpha) =  \sum_{\gamma' \in (\KC')^{l}} \sum_{\sigma',\sigma} m_{\sigma',\sigma}^{\gamma'} \cdot F^*\alpha_\sigma \wedge \delta_{\gamma'}
\end{equation}
where $\sigma',\sigma$ ranges over all pairs in $(\KC')^0 \times \KC^l$ with $\gamma' \subseteq \sigma', F(\gamma') \subseteq \sigma$.
Then (iii) follows from \eqref{standard form and pull-back} and  uniqueness of the representation \eqref{standard alpha beta}.

Note that (iv) is a direct consequence of our definitions.
\qed

\begin{rem}\label{boundaryintdelta2} 
Let $P$ be an integral $\R$-affine polyhedral subset in $N_\R$
of dimension $n$.
Let $C=(\KC,m)$ be a tropical cycle with $|C|=P$ or $|\KC|=P$ and 
$\alpha\in P_c^{\cdot}(N_\R)$. Observe that
$\int_P \alpha$ is in general different from 
$\int_{N_\R}\alpha\wedge \delta_C$ as the latter integral takes the multiplicities of $C$ into account. 
\end{rem}

\begin{prop}[Projection formula] \label{projection formula for preforms}
Let $F:N_\R' \rightarrow N_\R$ be an integral $\R$-affine map
and $C'$ a tropical cycle of dimension $n$ on $N_\R'$. 
Let $P$ be an integral $\R$-affine polyhedral subset and $\widetilde\Omega$ an open subset of $N_\R$ with $P \subseteq \widetilde{\Omega}$.
Let $\alpha \in P(\widetilde\Omega)$ be a
$\delta$-preform such that 
$\supp(F^*(\alpha) \wedge \delta_{C'}) \cap F^{-1}(P)$ is compact. 
Then $\supp(\alpha \wedge \delta_{F_*(C')}) \cap P$ is compact. If $\alpha \in P^{n,n}(\widetilde\Omega)$, then  we have 
\begin{equation}\label{projection formula for preformsg1}
\int_{P} \alpha \wedge \delta_{F_*(C')} 
=\int_{F^{-1}(P)} F^*(\alpha) \wedge \delta_{C'},
\end{equation}
If  $\alpha \in P^{n-1,n}(\widetilde\Omega)$, 
then we have 
\begin{equation}\label{projection formula for preformsg2}
\int_{\partial P} \alpha \wedge \delta_{F_*(C')} 
=\int_{\partial F^{-1}(P)} F^*(\alpha) \wedge \delta_{C'}.
\end{equation}
\end{prop}

\proof
We consider first the case where $\alpha\in P^{n,n}(\widetilde\Omega)$.
We may assume without loss of generality that 
$\alpha\in P^{p,p,l}(\widetilde\Omega)$ where $n=p+l$.
We write
\[
\alpha=\sum_{i\in I}\alpha_i\wedge
\delta_{C_i}
\]
for suitable $\alpha_i\in A^{p,p}(\widetilde\Omega)$ and $C_i\in {\rm TZ}^l(N_\R)$ 
as in (\ref{tropstandform}). We get
\begin{eqnarray*}
\alpha \wedge \delta_{F_*(C')}
&=&\sum_{i\in I}\alpha_i \wedge \delta_{F_*(F^*C_i\cdot C')},\\
F^*(\alpha) \wedge \delta_{C'}
&=&\sum_{i\in I}F^*(\alpha_i) \wedge \delta_{F^*C_i\cdot C'}
\end{eqnarray*}
by the projection formula \ref{tropintthprop}(ii).
We choose common polyhedral complexes of definition $\KC'$ in $N'_\R$ for
$C'$ and $F^*C_i$ for all $i\in I$ and $\KC$ in $N_\R$ for
$F_*C'$ and $C_i$ for all $i\in I$. 
We may assume that $F_*(\KC')$ is a subcomplex of $\KC$. 
After further refinements we can find polyhedral subcomplexes 
$\KD'$ of $\KC'$ with support $F^{-1}(P)$ and $\KD$ of $\KC$ with support $P$.
Then  $F_*\KD'$ is a subcomplex of $\KD$ and we write
\begin{eqnarray}
\label{zwischenformel1}
\sum_{i\in I}\alpha_i \wedge \delta_{F_*(F^*C_i\cdot C')}
&=&\sum_{\sigma\in \KC_{p}}\alpha_\sigma \wedge \delta_{\sigma},\\
\label{zwischenformel2}
\sum_{i\in I}F^*(\alpha_i) \wedge \delta_{F^*C_i\cdot C'}
&=&\sum_{\sigma'\in \KC'_{p}}\alpha_{\sigma'} 
\wedge \delta_{\sigma'}.
\end{eqnarray}
Consider $\sigma\in \KC_{p}$. 
Given $\sigma'\in \KC'_p$ with $F(\sigma')=\sigma$ there 
is a unique form $\tilde\alpha_{\sigma'}\in A_\sigma(\sigma)$ such that 
$F^*(\tilde\alpha_{\sigma'})=\alpha_{\sigma'}$ in $A_{\sigma'}(\sigma')$.
From (\ref{intpairingg2}), 
(\ref{zwischenformel1}) and (\ref{zwischenformel2}) we get 
\begin{equation} \label{a-formula}
\alpha_\sigma=
\sum\limits_{
\genfrac{}{}{0pt}{}{\sigma' \in \KC'_{p}}{F(\sigma')=\sigma}}
\bigl[N_\sigma:\L_F(N'_{\sigma'})\bigr]\cdot\tilde\alpha_{\sigma'}.
\end{equation}
and $\alpha_{\sigma'}=0$ for all $\sigma'\in \KC'_p$ with ${\rm dim}\,F(\sigma')<p$. 
If $\sigma \in \KD_p$, which means $\sigma \subseteq P$, then only $\sigma' \in \KD_p'$ contribute to 
the sum in \eqref{a-formula}. Since the latter is equivalent to $\sigma' \subseteq F^{-1}(P)$, 
we deduce from \eqref{a-formula} and compactness of $\supp(F^*(\alpha) \wedge \delta_{C'}) \cap F^{-1}(P)$
that $\supp(\alpha \wedge \delta_{F_*(C')}) \cap P$ is compact. 
The above formulas show that 
\begin{equation*}
\int_{P}\alpha\wedge\delta_{F_*(C')}
=\sum_{\sigma \in \KD_{p}}\int_{\sigma}\alpha_\sigma
=\sum_{\sigma \in \KD_{p}}\sum\limits_{
\genfrac{}{}{0pt}{}{\sigma' \in \KC'_{p}}{F(\sigma')=\sigma}}
\bigl[N_\sigma:\L_F(N'_{\sigma'})\bigr]\int_\sigma\tilde\alpha_{\sigma'}
\end{equation*}
and hence the transformation formula of integration theory (see \cite[1.5.8]{chambert-loir-ducros}, \cite[Proposition 3.10]{gubler-forms}) gives 
\begin{equation*}
  \int_{P}\alpha\wedge\delta_{F_*(C')} =\sum_{\sigma \in \KD_{p}}\sum\limits_{
\genfrac{}{}{0pt}{}{\sigma' \in \KC'_{p}}{F(\sigma')=\sigma}}
\int_{\sigma'}\alpha_{\sigma'}
=\sum_{\sigma' \in \KD'_{p}}\int_{\sigma'}\alpha_{\sigma'}
=\int_{F^{-1}(P)}F^*(\alpha) \wedge \delta_{C'}.
\end{equation*}
This proves (\ref{projection formula for preformsg1}). 
Formula (\ref{projection formula for preformsg2}) is proved exactly 
in the same way using the transformation formula for boundary integrals in \cite[1.5.8]{chambert-loir-ducros}.
\qed

\begin{art}
Given a tropical cycle $C=(\KC,m)$ with constant weight functions, it follows
from Stokes' theorem that the supercurrent $\delta_C$ is closed under
$d'$ and $d''$ \cite[Prop. 3.8]{gubler-forms}.
The following Proposition shows that this is is no longer true for tropical 
cycles with smooth weights.
\end{art}

\begin{prop}\label{diffsmoothtropcyc}
Let $C=(\KC,m)$ be a tropical cycle with smooth weights
of pure dimension $n$ in $N_\R$.
Then we have 
\begin{eqnarray*}
d'\delta_C&=&d'm\wedge \delta_\KC,\\
d''\delta_C&=&d''m\wedge \delta_\KC
\end{eqnarray*} 
in $D_\cdot(N_\R)$ where the polyhedral supercurrent $d'm\wedge \delta_\KC$ is 
defined by
\[
\langle d'm\wedge \delta_\KC,\alpha\rangle =
\sum_{\sigma\in \KC_n}\int_\sigma d'm_\sigma\wedge \alpha
\]
and the supercurrent $d''m\wedge \delta_\KC$ is defined in an analog way.
\end{prop}

\proof
This is a direct consequence of Stokes' formula for superforms
on polyhedra 
\cite[(1.5.7)]{chambert-loir-ducros}, \cite[2.9]{gubler-forms}
and the balancing condition (\ref{balancingcondition}).
\qed

\begin{rem} \label{preforms are not d-invariant}
It follows from \ref{diffsmoothtropcyc} that the subspace $P^\cdot(N_\R)$
of $D^\cdot(N_\R)$ of $\delta$-preforms is not closed under the
differentials $d'$ and $d''$ in the sense of supercurrents.
We will address this problem again in  {\ref{differentiation and closedness}.} 
\end{rem}

\begin{prop}\label{polyhedralequalcurrent}
Let $\widetilde\Omega$ denote an open subset of $N_\R$. 
Then we have
\[
d'(\beta)=\dpa(\beta),\,\,\,d''(\beta)=\dpb(\beta)
\]
for all $\delta$-preforms $\beta\in P(\widetilde\Omega)$.
\end{prop}

\proof
It is sufficient to treat the case $\beta=\alpha\wedge\delta_C$ for a superform
$\alpha\in A^{p,q}(\widetilde\Omega)$ and a tropical cycle $C=(\KC,m)$ of
pure dimension $n$ on $N_\R$. 
We have
\[
\beta=\sum_{\sigma\in \KC_n}(m_\sigma\cdot \alpha|_\sigma)\wedge \delta_\sigma.
\]
From Proposition \ref{diffsmoothtropcyc} we get
\begin{equation*}
d'\beta
=d'\alpha\wedge \delta_C+(-1)^{p+q}\alpha\wedge d'm\wedge \delta_\KC
=\sum_{\sigma\in \KC_n}\bigl(m_\sigma\cdot d'\alpha|_\sigma +d'm_\sigma\wedge \alpha|_\sigma\bigr)\wedge \delta_\sigma.
\end{equation*}
Then Leibniz's rule shows 
\begin{equation*}
d'\beta
=\sum_{\sigma\in \KC_n}d'(m_\sigma\cdot \alpha|_\sigma)
\wedge \delta_\sigma
=\dpa(\beta)
\end{equation*}
which proves the first equality.
The second claim is proved similarly. 
\qed

\section{Supercurrents and delta-preforms on tropical cycles}
\label{Currents and delta-preforms on tropical cycles}

In this section, we introduce supercurrents and $\delta$-preforms on a given 
tropical cycle $C=(\KC,m)$ of pure dimension $n$  with {\it constant weight functions}.  
Similarly to complex manifolds, such tropical cycles have no boundary as $d'\delta_C=d''\delta_C=0$. 
In the applications, $C$ will be the tropical variety of a closed subvariety of a multiplicative torus. 
We build up on the  results in Section \ref{algebradeltaforms}. 
We will obtain the formulas of Stokes and Green. The main result is the tropical Poincar\'e--Lelong 
equation which will be  used in Section \ref{piecewise smooth forms and delta-metrics} 
for the first Chern $\delta$-current of a metrized line bundle.

\begin{art}\label{formspoyhedra} 
The space $A_\KC^{p,q}(\Omega)$ of $(p,q)$-superforms
on an open subset $\Omega$ in $|\KC|$ is defined  {as follows.
We choose an open subset $\widetilde\Omega$ of $N_\R$ such that
$\Omega=\widetilde\Omega\cap |\KC|$. Elements in  $A_\KC^{p,q}(\Omega)$ are represented by elements
in $A^{p,q}(\widetilde\Omega)$ where two such elements are identified if
they induce the same element in $A_\Delta^{p,q}({\Omega}\cap\Delta)$ 
(see \ref{formsaffinespace}) for all maximal polyhedra $\Delta$ in $\KC$.
A partition of unity argument shows that this definition is independent of the 
choice of $\tilde\Omega$.} 
Observe  {furthermore} that $A_\KC^{p,q}(\Omega)$ depends only on the support
$|\KC|$ of $\KC$.
We will often omit the polyhedral complex $\KC$ from our notation
and write simply $A^{p,q}(\Omega)$ instead of
$A_\KC^{p,q}(\Omega)$ when $\KC$ or at least $|\KC|$ is clear from the context.
The spaces $A^{p,q}(\Omega)$ define a sheaf on $|\KC|$. 
Hence the support of a superform in $A^{p,q}(\Omega)$ is defined as
a closed subset of $\Omega$. We denote by 
$A_{c}^{p,q}(\Omega)$ the space of superforms 
on $\Omega$ with compact support. 
\end{art}

\begin{definition}\label{currentspoylhedra} 
We define the {\it space of supercurrents} 
$D_{p,q}^\KC(\Omega)$ of
type $(p,q)$ on an open subset 
$\Omega$ in $|\KC|$ as follows. 
An element in $D_{p,q}^\KC(\Omega)$ is given by
a linear form
$T\in {\rm Hom}_\R(A^{p,q}_{c}(\Omega),\R)$
such that we can find an open set 
$\widetilde\Omega$ of $N_\R$ and a supercurrent 
$T' \in D_{p,q}(\widetilde\Omega)$ such that
$\Omega=\widetilde\Omega\cap |\KC |$ and
$T(\eta|_\Omega)=T'(\eta)$
for all $\eta\in A^{p,q}_c(\widetilde\Omega)$.
As in \ref{formspoyhedra} we often omit $\KC$ from  the notation 
and write $D_{p,q}(\Omega)$ instead of $D_{p,q}^\KC(\Omega)$. We also use the grading $D^{p,q}(\Omega):=D_{n-p,n-q}(\Omega)$. 
\end{definition}

\begin{rem} \label{currents and extension}
In the situation of \ref{currentspoylhedra} we fix an open subset 
$\widetilde\Omega$ of $N_\R$ with 
$\Omega = \widetilde\Omega \cap |\KC|$. 
 {It follows from a partition of unity argument that in the definition of 
$D_{p,q}(\Omega)$ we may use this $\widetilde\Omega$.
We may identify $D_{p,q}(\Omega)$ with a subspace of $D_{p,q}(\tilde\Omega)$ using 
the canonical map $T \mapsto {T'}$. Indeed, this map is well-defined and injective 
using that $T(\eta|_\Omega)=T'(\eta)$ holds
for all $\eta\in A^{p,q}_c(\widetilde\Omega)$.} 
Furthermore the differentials $d'$ and $d''$ on $D(\widetilde\Omega)$ induce 
 {well-defined} 
differentials $d'$ and $d''$ on $D(\Omega)$.

A polyhedral supercurrent $\alpha'$  on $\widetilde\Omega$ is in $D(\Omega)$ if and only 
if $\supp(\alpha')$ is contained in $\Omega$. The corresponding element $\alpha$ in $D(\Omega)$ 
is called a {\it polyhedral supercurrent on $\Omega$}. Using \ref{polyhedral current}, 
the polyhedral derivatives $\dpa \alpha$ and $\dpb \alpha$ are again polyhedral supercurrents 
on $\Omega$. Definition \ref{intdelta} yields  integrals $\int_P \alpha = \int_P \alpha'$ and boundary integrals 
$\int_{\partial P} \alpha =\int_{\partial P} \alpha'$ of polyhedral supercurrents $\alpha$ in $D_0(\Omega)$ (resp. $D_1(\Omega)$) over an integral 
$\R$-affine polyhedral subset $P$ of $\Omega$ provided that $\supp(\alpha) \cap P$ is compact.
\end{rem}

\begin{definition} \label{definitionpreformaufOmega}
Let $\Omega$ be an open 
subset of $|\KC|$ and consider an open subset 
$\widetilde\Omega$ of $N_\R$ 
with $\Omega=\widetilde\Omega\cap |\KC|$. For any 
 {$\delta$-preform}
$\tilde\alpha \in P(\widetilde\Omega)$  {on $\widetilde\Omega$}, the supercurrent 
$\tilde\alpha \wedge \delta_{C}$  on $\widetilde\Omega$ l
ies in the subspace 
$D(\Omega)$ of $D(\widetilde\Omega)$. 
We will denote the corresponding element 
in $D(\Omega)$ by $\tilde\alpha|_\Omega$. 
A supercurrent $\alpha \in D(\Omega)$ is called a {\it $\delta$-preform on $\Omega$} if there 
is an open subset $\widetilde\Omega$ of $N_\R$ with 
$\Omega=\widetilde\Omega\cap |\KC|$ and a
$\tilde\alpha \in P(\widetilde\Omega)$ with 
$\alpha =\tilde\alpha\wedge\delta_C$. 
The space of $\delta$-preforms on $\Omega$ is denoted by 
$P(\Omega)$ and the subspace of compactly supported $\delta$-
preforms is denoted by $P_c(\Omega)$. 
Note that these spaces depend also on the weights $m$ of 
the tropical cycle $C=(\KC,m)$ and not only on the open subset 
$\Omega$ of $|\KC|$. 
\end{definition}

\begin{rem}  \label{preform properties}
(i) A partition of unity argument again shows that $P(\Omega)$ is the 
image of the natural morphism
\[
P(\widetilde\Omega)\longrightarrow D(\Omega),\,\,\,
\tilde\alpha\longmapsto \tilde\alpha|_\Omega:=\tilde\alpha\wedge \delta_C
\]
for any open subset $\widetilde\Omega$ of $N_\R$ with 
$\Omega=\widetilde\Omega\cap |\KC|$.
We give $P(\Omega)$ the unique structure as a bigraded algebra such that the surjective map $P(\widetilde\Omega) \to P(\Omega)$ is a homomorphism of bigraded algebras. Similarly, we define the grading by codimension on $P(\Omega)$. For $\delta$-preforms $\alpha=\tilde\alpha\wedge \delta_C$ and
$\alpha'=\tilde\alpha'\wedge \delta_C$ on $\Omega$,  their
product is given by the formula
\[
\alpha\wedge\alpha'=\tilde\alpha\wedge\tilde\alpha'\wedge\delta_C.
\]

(ii) By Remarks \ref{standardform} and \ref{currents and extension}, $\alpha = \tilde\alpha\wedge \delta_C \in P(\Omega)$ is a polyhedral supercurrent on $\Omega$. After 
possibly passing to a subdivision of $\KC$, we have
$$\alpha=\sum_{\Delta\in \KC}\alpha_\Delta\wedge
\delta_{\Delta} \in D(\Omega)$$
with  $\alpha_\Delta \in A_\Delta(\Omega \cap \Delta)$. It follows from Proposition \ref{polyhedralequalcurrent} that 
\begin{equation} \label{polyhedral equals current}
\dpa \alpha = d'\alpha \quad \mbox{\rm and} \quad \dpb \alpha = d'' \alpha,
\end{equation}
where on the left hand sides we use the polyhedral derivative introduced in \ref{polyhedral current} and where on the right hand sides we use the derivative of currents in $D(\Omega)$.

(iii) Now we assume that $\alpha \in P^{n,n}(\Omega)$  and that $P$ is an integral $\R$-affine polyhedral subset of $\Omega$ such that 
$\supp(\alpha) \cap P$ is compact. By passing again to a subdivision, we may assume that $\KC$ has a subcomplex $\KD$ with $|\KD|=P$. 
 {Using the definition of the} integral of polyhedral 
supercurrents on $\Omega$ in Remark \ref{currents and extension}   and a decomposition of $\alpha$ as above {, \eqref{intdeltag2} gives}
\[
\int_P \alpha = \sum_{\Delta \in \KD} \int_\Delta \alpha_\Delta
\]
A similar formula holds for the  boundary integral $\int_{\partial P} \alpha$ for $\alpha \in P^{n-1,n}(\Omega)$ or $\alpha \in P^{n,n-1}(\Omega)$.
\end{rem}

\begin{prop}[Stokes' formula for $\delta$-preforms]
\label{stokesfordeltapreforms}
Let  $P$ be an integral $\R$-affine polyhedral subset 
of the open subset $\Omega$ of $|\KC|$. 
Then we have
\[
\int_Pd'\alpha=\int_{\partial P}\alpha,\,\,\,
\int_Pd''\beta=\int_{\partial P}\beta
\]
for all $\delta$-preforms $\alpha\in P^{n-1,n}(\Omega)$
and $\beta\in P^{n,n-1}(\Omega)$ with $\supp(\alpha) \cap P$ and $\supp(\beta) \cap P$ compact.
\end{prop}

\begin{proof}
This follows from Proposition \ref{stokesforpolyhedralcurrents} and 
\eqref{polyhedral equals current}.
\end{proof}

The following result will be important in the construction of $\delta$-forms on algebraic varieties.

\begin{lem}\label{erstesproduktlemma}
Let $\Omega$ be an open 
subset of $|\KC|$.
Given $d'$-closed (resp. $d''$-closed) $\delta$-preforms $\gamma$ and 
$\gamma'$ on $\Omega$, their product $\gamma\wedge\gamma'$ is again
$d'$-closed (resp. $d''$-closed).
\end{lem}

\proof
Consider an open subset $\widetilde\Omega$ of $N_\R$ 
with $\Omega=\widetilde\Omega\cap |\KC|$ and $d'$-closed $\delta$-preforms 
$\gamma$ and $\gamma'$ on $\Omega$. 
We may assume that $\gamma$ (resp. $\gamma'$) is of codimension $l$ (resp. $l'$) 
and that $\gamma$ has  degree $k+2l$ (resp. $k'+2l$).
We choose $\delta$-preforms $\tilde\gamma=\sum_i\alpha_i \wedge \delta_{C_i}$ 
and  $\tilde\gamma':= \sum_j \alpha_j' \wedge \delta_{C_j'}$ for superforms $\alpha_i \in A^k(\widetilde\Omega)$, 
$\alpha_j' \in A^{k'}(\widetilde\Omega)$ 
and tropical cycles $C_i=(\KC_i,m_i), C_{j}'=(\KC_j',m_j')$ 
of  codimension $l,l'$ with smooth weight functions such
that $\gamma=\tilde\gamma\wedge \delta_C$ and 
$\gamma'=\tilde\gamma'\wedge\delta_C$.
We have to show that the supercurrent
\[
\gamma\wedge\gamma'=\tilde\gamma\wedge\tilde\gamma'\wedge \delta_C\in D(\Omega)
\]
is $d'$-closed.
After suitable refinements we may assume that the polyhedral complexes 
$\KC_i$, $\KC_j'$ and $\KC$ are all subcomplexes of a complete integral 
$\R$-affine polyhedral complex $\Dcal$ in  {$N_{\R}$}. 
We  choose $v,w\in N_\R$ generic vectors 
in order to compute stable tropical intersection products as in \ref{intpairing} 
for tropical cycles with polyhedral complex of definition $\KD$.
We have $C_j'\cdot C=(\KD_{\leq n-l'},m_{j}'')$.
For $\rho\in \KD_{n-l'}$ and $\omega \in \rho$, we have
\[
m_{j\rho}''(\omega)=
\sum_{\rho = \sigma' \cap \Delta}
c_{\sigma'\Delta}\,
m_{j\sigma'}'(\omega)\,m_\Delta.
\]
for small $\varepsilon>0$, where $(\sigma',\Delta)$ ranges over $\KD^{l'}\times \KD_{n}$ and 
$c_{\sigma'\Delta}=[N:N_{\sigma'}+N_{\Delta}]$ if ${\sigma'\cap(\Delta+\varepsilon v)\neq \emptyset}$ and 
$c_{\sigma'\Delta}= 0$ else.
In the same way we write 
\[
C_i\cdot C_j'\cdot C=(\KD_{\leq n-l-l'},m_{ij}''').
\]
For $\tau\in \KD_{n-l-l'}$ and $\omega \in \tau$, we have
\[
m_{ij\tau}'''(\omega)=
\sum_{
\tau=\sigma\cap\rho}
c_{\sigma\rho}\,
m_{i\sigma}(\omega)\,m_{j\rho}''(\omega).
\]
for small $\varepsilon>0$, where $(\sigma,\rho)$ range over $\KD^{l}\times \KD_{n-l'}$ 
and $c_{\sigma\rho}=[N:N_{\sigma}+N_{\rho}]$ if 
$\sigma\cap(\rho+\varepsilon w)\neq \emptyset$ and $c_{\sigma\rho}=0$ else.
Combining the last two formulas, we get
\begin{equation} \label{newomega-1}
m_{ij\tau}'''(\omega)=
\sum_{
\tau=\sigma\cap\sigma'\cap \Delta}
c_{\sigma\sigma'\Delta}\,
m_{i\sigma}(\omega)\,m_{j\sigma'}'(\omega)\,m_\Delta
\end{equation}
where $(\sigma,\sigma',\Delta)$ range over $\KD^{l}\times \KD^{l'}\times \KD_n$ and
\begin{equation}\label{newomega0a}
c_{\sigma\sigma'\Delta}=c_{\sigma,\sigma'\cap\Delta}\cdot c_{\sigma'\Delta}.
\end{equation}
We observe that by associativity and commutativity
$C_i\cdot(C_j'\cdot C)=C_j'\cdot (C_i\cdot C)$. This implies
\begin{equation}\label{newomega0}
c_{\sigma\sigma'\Delta}=c_{\sigma',\sigma\cap\Delta}\cdot c_{\sigma\Delta}.
\end{equation}
Now we use $\dpa(\tilde\gamma \wedge \delta_{C})=d'\gamma=0$ in $D(\Omega)$. 
For every $\rho\in \Dcal_{n-l}$, we get
\begin{equation}\label{newomega1}
\sum_{i}\sum_{
\rho=\sigma\cap\Delta}
 c_{\sigma \Delta} \, d'( m_{i\sigma} m_\Delta\alpha_i) =0
\end{equation} 
on $\Omega \cap \rho$.  
Similarly, we use $\dpa(\tilde\gamma' \wedge \delta_{C})=d'\gamma'=0$.
For every  $\rho'\in \Dcal_{n-l'}$, this gives
\begin{equation}\label{newomega2}
\sum_{j}\sum_{
\rho'=\sigma'\cap\Delta}
 c_{\sigma' \Delta}\, d'( m_{j\sigma'}' m_\Delta \alpha_j') =0
\end{equation} 
on $\Omega \cap \rho'$. 
We have to show that
\begin{equation} \label{newomega1a}
d'(\gamma \wedge \gamma')=d'(\tilde\gamma \wedge \tilde\gamma' \wedge \delta_{C})
= \sum\limits_{ij} d'\bigl(\alpha_i \wedge \alpha_j' 
\wedge \delta_{C_i \cdot  C_j' \cdot C}\bigr)
\end{equation}
vanishes in $D(\Omega)$. Since $d'$ agrees with $\dpa$ on $\delta$-preforms, we deduce
\begin{equation} \label{Leibniz direct1}
d'(\alpha_i \wedge \alpha_j' \wedge \delta_{C_i \cdot  C_j' \cdot  C})
= \sum_{\tau } d'(m_{ij\tau}''' \alpha_i \wedge \alpha_j') \wedge \delta_\tau.
\end{equation}
By \eqref{newomega-1} and   Leibniz's rule, we can split this into the sum of 
\begin{equation} \label{Leibniz direct2}
\sum_{\tau }   
\sum_{
\tau=\sigma\cap\sigma'\cap \Delta}
c_{\sigma\sigma'\Delta}\,m_\Delta \,
d'(m_{i\sigma} \alpha_i)\wedge m_{j\sigma'}'\alpha_j' \wedge  \delta_\tau.
\end{equation}
and 
\begin{equation} \label{Leibniz direct3}
(-1)^k \sum_{\tau }   
\sum_{
\tau=\sigma\cap\sigma'\cap \Delta}
c_{\sigma\sigma'\Delta}\,m_\Delta \,
m_{i\sigma} \alpha_i \wedge d'(  m_{j\sigma'}'\alpha_j')     \wedge  \delta_\tau.
\end{equation}
Note that here and in the following, we use our standing assumption that the weight $m_\Delta$ of $C$ is constant. 
From (\ref{newomega0}) and (\ref{newomega1}) we get
\begin{eqnarray*}
&&\sum_{ij}\sum_\tau\sum_{\tau=\sigma\cap\sigma'\cap \Delta}
c_{\sigma\sigma'\Delta}m_\Delta d'(m_{i\sigma}\alpha_i)
\wedge m_{j\sigma'}'\alpha_j'\\
&=&\sum_j\sum_\tau\sum_{\tau =\sigma' \cap \rho}c_{\sigma'\rho}
\Bigl(\sum_{\rho =\sigma \cap \Delta}\sum_{i}c_{\sigma\Delta}
m_\Delta d'(m_{i\sigma}\alpha_i)\Bigr)
\wedge m_{j\sigma'}'\alpha_j'\\
&=&0.
\end{eqnarray*}
In the same way we get
\[
\sum_{ij}\sum_\tau\sum_{\tau=\sigma \cap\sigma' \cap\Delta}
c_{\sigma\sigma'\Delta}m_\Delta m_{i\sigma}\alpha_i
\wedge d'(m_{j\sigma'}'\alpha_j')=0
\]
from (\ref{newomega0a}) and (\ref{newomega2}).
These two equations and \eqref{newomega1a}--\eqref{Leibniz direct3} prove the vanishing of $d'(\gamma\wedge\gamma')$. In the same way,  
one derives $d''(\gamma\wedge\gamma')=0$ from the vanishing of 
$d''(\gamma)$ and $d''(\gamma')$.
\qed

\begin{cor}\label{omegaleibnizformel}
Let $\Omega$ be an open subset of $|\KC|$.
We consider $\beta=\eta\wedge \gamma\in P^k(\Omega)$ and 
$\beta'=\eta'\wedge \gamma'\in P^{k'}(\Omega)$ 
for superforms $\eta,\eta'\in A(\Omega)$ and $\delta$-preforms
$\gamma,\gamma'\in P(\Omega)$. If $d'\gamma=d'\gamma'=0$, then $d'\beta$ is again
a $\delta$-preform with 
\[
d'\beta=d'\eta\wedge \gamma \quad{\rm and} \quad  d'(\beta\wedge \beta')=d'\beta\wedge\beta'+(-1)^k
\beta\wedge d'\beta'.
\]
If $d'' \gamma = d'' \gamma' = 0$, then 
 {$d''\beta$ is again a $\delta$-preform with 
\[
d''\beta=d''\eta\wedge \gamma \quad{\rm and} 
\quad  d''(\beta\wedge \beta')=d''\beta\wedge\beta'+(-1)^k
\beta\wedge d''\beta'.
\]
}
\end{cor}

\proof
Given a superform $\eta\in A^{p}(\Omega)$ and a supercurrent $T\in D(\Omega)$,
we have
\begin{equation}\label{leibnizforcurrents}
d'(\eta\wedge T)=d'\eta\wedge T+(-1)^p
\eta\wedge d'T.
\end{equation}
This implies the first formula and hence $d'\beta$ is a preform. Combined with Lemma \ref{erstesproduktlemma}, 
we deduce the second formula as well. Similarly, we prove the corresponding claims for $d''$.
\qed

\begin{prop}[Greens' formula for $\delta$-preforms]
\label{greenfordeltapreforms}
 Let $\Omega$ be an open 
subset of $|\KC|$ and 
let $P$ be an integral $\R$-affine polyhedral subset  of $\Omega$. 
We consider symmetric $\delta$-preforms 
$\beta_i\in P^{p_i,p_i}(\Omega)$ for $i=1,2$ with $p_1+p_2=n-1$ such 
that $\beta_i=\eta_i\wedge \gamma_i$  
for superforms $\eta_i\in A(\Omega)$ and $\delta$-preforms
$\gamma_i\in P(\Omega)$ with  $d'\gamma=d'\gamma'=d'' \gamma = d'' \gamma' = 0$. 
Then we have
\[
\int_P
\bigl(\beta_1\wedge d'd''\beta_2-\beta_2\wedge d'd''\beta_1\bigr)
=\int_{\partial P}
\bigl(\beta_1\wedge d''\beta_2-\beta_2\wedge d''\beta_1\bigr)
\]
if we assume furthermore that $\supp(\beta_1) \cap \supp(\beta_2) \cap P$ is compact.
\end{prop}

\proof
The formula is obtained as in \cite[1.3.8]{chambert-loir-ducros}
as a direct consequence of 
Proposition \ref{stokesfordeltapreforms} and the Leibniz formula 
in Corollary \ref{omegaleibnizformel}.
\qed

\begin{definition}\label{piecewise smooth superforms}
Let $P$ be an integral $\R$-affine polyhedral 
subset in $N_\R$.
A {\it piecewise smooth superform} $\alpha$ on an open subset $\Omega$ of $P$ 
is given by an integral $\R$-affine polyhedral complex $\KD$ 
with support $P$ and smooth 
superforms $\alpha_\Delta \in A_\Delta(\Omega \cap \Delta)$ for every  $\Delta \in \KD$ 
such that $\alpha_\Delta$ restricts to $\alpha_\rho$ for every closed face 
$\rho$ of $\Delta$. 
In this case we call $\KD$ a {\it polyhedral complex of definition} for $\alpha$.
The {\it support} of a piecewise smooth superform $\alpha$ as above is
the union of the supports of the forms $\alpha_\Delta$ for all $\Delta$ in $\KD$.
We identify two superforms $\alpha,\alpha'$ on $\Omega$ if they have the same support 
and if $\alpha_\Delta = \alpha'_{\Delta'}$ on $\Delta \cap \Delta' \cap \Omega$ 
for all polyhedra $\Delta,\Delta'$ of the underlying polyhedral complexes $\KD,\KD'$.
\end{definition}

\begin{rem}[Properties of piecewise smooth superforms]  
\label{properties of piecewise smooth superforms}
Let $\Omega$ denote an open subset of an integral $\R$-affine polyhedral 
subset $P$ in $N_\R$.

(i) The space of piecewise smooth superforms on $\Omega$ is denoted by $PS(\Omega)$. 
It comes with a natural bigrading and has a natural wedge product. 
We conclude that  $PS^{\cdot,\cdot}(\Omega)$ is a bigraded  
$\R$-algebra which contains $A^{\cdot,\cdot}(\Omega)$ as a subalgebra.  
We denote by $PS_c^{\cdot,\cdot}(\Omega)$ the subspace of
$PS^{\cdot,\cdot}(\Omega)$ given by piecewise smooth superforms with
compact support.

(ii) Let $N'$ be also a free abelian group of finite rank and let 
$F:N_\R' \rightarrow N_\R$ be an integral $\R$-affine map. 
Suppose that $\Omega'$ is an open subset of an integral $\R$-affine 
polyhedral subset $Q$ in $N_\R'$ with $F(Q) \subseteq P$ and 
$F(\Omega') \subseteq \Omega$. 
For a piecewise smooth superform $\alpha$ on $\Omega$, 
 there is an integral $\R$-affine  polyhedral complex $\KD'$ 
 with $|\KD'|=Q$ and a polyhedral complex of definition $\KD$ for $\alpha$ such that for every $\Delta' \in \KD'$, there is a 
$\Delta \in \KD$ with $F(\Delta') \subseteq \Delta$. Then we define 
 a piecewise 
smooth superform $F^*(\alpha)=\alpha'$ on $\Omega'$ with $\KD'$ as a polyhedral complex of definition by 
setting $\alpha_{\Delta'}':=F^*(\alpha_\Delta) \in A_{\Delta'}(\Omega'\cap \Delta')$ for every $\Delta' \in \KD'$. 
In this way, we get a well-defined graded $\R$-algebra homomorphism 
\[
F^*:PS^{\cdot,\cdot}(\Omega) \rightarrow PS^{\cdot,\cdot}(\Omega').
\]
In particular, we can restrict $\alpha$ to an open subset of an integral $\R$-affine polyhedral subset of $P$.

(iii)  Let $\alpha \in PS^{p,q}(\Omega)$ be given by an integral $\R$-affine 
polyhedral complex $\KD$  and smooth superforms 
$\alpha_\Delta \in A^{p,q}(\Omega \cap \Delta)$ for every  $\Delta \in \KD$ . 
Then the superforms $d'\alpha_\Delta \in A_\Delta^{p+1,q}(\Omega \cap \Delta)$, with $\Delta$ 
ranging over $\KD$, define an element in $PS^{p+1,q}(\Omega)$ which 
we denote by $\dpa\alpha$. Similarly, we define 
$\dpb\alpha \in PS^{p,q+1}(\Omega)$. 
One verifies immediately that $PS^{\cdot,\cdot}(W)$ is a 
differential graded $\R$-algebra.
with respect to the differentials
\begin{equation}\label{differentiation of piecewise smooth}
\dpa:PS^{p,q}(\Omega)\to PS^{p+1,q}(\Omega),\,\,\,\,
\dpb:PS^{p,q}(\Omega)\to PS^{p,q+1}(\Omega).
\end{equation} 

(iv) The elements of $PS^{0,0}(\Omega)$ are the piecewise smooth functions on the open subset $\Omega$ of $P$ from  Definition \ref{piecewise smooth function}.
\end{rem}

\begin{art} \label{residue for tropical ps}
Now we apply the above to an open subset $\Omega$ of the polyhedral set $P:=|\KC|$ for the given tropical cycle $C = (\KC,m)$ with constant weight functions. 
A piecewise smooth superform $\alpha$ as above defines a polyhedral supercurrent
\[
[\alpha] := \sum_{\Delta \in \KC_n}\alpha_\Delta\wedge\delta_\Delta
\]
and the derivatives in (\ref{differentiation of piecewise smooth})
coincide -- as suggested by the notation -- with the polyhedral
derivatives introduced in \ref{polyhedral current}.
Note that these differentials of piecewise smooth superforms are not compatible with the corresponding differentials of the associated supercurrents. 
We define the {\it $d'$-residue of $\alpha$}  by 
\[
{\rm Res}_{d'}(\alpha) := d'[\alpha]-[\dpa\alpha].
\]
Similarly, we define residues with respect to the differential 
operators $d''$ and $d'd''$.
\end{art}

\begin{art} 
Given $\alpha \in PS(\Omega)$ and a polyhedral supercurrent $\beta$ on the open subset $\Omega$ of $|\KC|$, there is natural bilinear product $\alpha \wedge \beta$ which is defined as a polyhedral supercurrent on $\Omega$ as follows. After passing to a subdivision of $\KC$, we may assume that $\KC$ is a polyhedral complex of definition for $\alpha$ and $\beta$. Then $\KC$ is a polyhedral complex of definition for $\alpha \wedge \beta$ and for every $\Delta \in \KC$ we set 
$$(\alpha \wedge \beta)_\Delta := \alpha_\Delta \wedge \beta_\Delta \in A_\Delta(\Omega \wedge \Delta),$$
where $\alpha,\beta$ are given on $\Omega$ by $\alpha_\Delta,\beta_\Delta \in A_\Delta(\Omega \wedge \Delta)$. 
For $\alpha \in PS^k(\Omega)$, the Leibniz type formula
\begin{equation}\label{psproductsleibniz}
\dpa(\alpha\wedge\beta)=
\dpa\alpha\wedge \beta+(-1)^k\alpha\wedge \dpa \beta
\end{equation}
is a direct consequence of our definitions. An  {analogous} formula holds for $\dpb$.
\end{art}

There is no obvious product on the space of polyhedral currents
which extends the given products on the subspaces 
$P(\Omega)$ and $PS(\Omega)$. 
The next remark shows that such a product exists for a canonical subspace $PSP(\Omega)$ of the space of polyhedral currents.

\begin{rem} \label{PSP-forms}
The linear subspace  $PSP(\Omega)$ of $D(\Omega)$, generated by currents of the form $\alpha \wedge \beta$ with $\alpha \in PS(\Omega)$ and with $\beta \in P(\Omega)$, will play a role later. Note that $PSP(\Omega)$ has a unique structure as a bigraded differential $\R$-algebra with respect to $\dpa,\dpb$ extending the corresponding structures on $PS(\Omega)$ and $P(\Omega)$. To see that the wedge-product is well-defined, we can use the same arguments as for $P(\Omega)$. 
The crucial point is that for a piecewise smooth form $\alpha$ as in \ref{residue for tropical ps} and $\tau \preccurlyeq \Delta \in \KC$, the restriction of $\alpha_\Delta$ to $\tau$ is  $\alpha_\tau$. This allows to use the arguments in Proposition \ref{deltaformsbasicprop} which show that $\wedge$ is well-defined on $PSP(\Omega)$. 

If $F:N_\R' \rightarrow N_R$ is an integral $\R$-affine map and if $\widetilde\Omega'$ is an open subset of the preimage of the open subset $\widetilde\Omega$ of $N_\R$, then we have a unique pull-back $F^*:PSP(\widetilde\Omega) \rightarrow PSP(\widetilde\Omega')$ which extends the pull-back maps on piecewise smooth forms and on $\delta$-preforms and which is compatible with the bigrading and the wedge-product. Again, the argument is the same as in the proof of Proposition \ref{deltaformsbasicprop}. 
Moreover, it is clear that the projection formulas in Proposition \ref{projection formula for preforms} hold more generally for $\alpha \in PSP(\widetilde\Omega)$. 
\end{rem}

\begin{art} \label{corner locus and delta-preform}
Recall that $C=(\Ccal,m)$ is a tropical cycle on $N_\R$ of pure dimension $n$ 
and with constant weight functions. Let $\phi$ be a piecewise smooth function 
on $|\Ccal|$. We have seen in \ref{corner locus for tropical cycle} that the 
corner locus $\phi\cdot C$   is again a tropical cycle. It induces a polyhedral 
supercurrent $\delta_{\phi\cdot C} \in D_{n-1,n-1}(|\Ccal|)$ on $|\Ccal|$. 
By Proposition \ref{extension of piecewise smooth}, there is a piecewise 
smooth function $\tilde{\phi}$ on $N_\R$ extending $\phi$. We have
\[
\delta_{\phi\cdot C}=\delta_{\tilde{\phi}\cdot N_\R} \wedge \delta_C
\]
and hence $\delta_{\phi\cdot C}$ is a $\delta$-preform in $P^{1,1}(|\Ccal|)$. 
By Remark \ref{preform properties}, we obtain a $\delta$-preform 
$\delta_{\phi\cdot C} \wedge \beta \in P^{p,q,l+1}(|\KC|)$ for any 
$\beta \in P^{p,q,l}(|\Ccal|)$. 
\end{art}

The following {\it tropical Poincar\'e--Lelong formula} and
its Corollary \ref{tropicalpoincarelelong} compute the 
$d'd''$-residue of  
$\phi$. 

\begin{thm}\label{deltatropicalpoincarelelong}
We consider a $\delta$-preform $\omega \in P^{p,q,l}(|\KC|)$ such
that $d'\omega=0=d''\omega$. 
Let $\eta\in A^{n-p-l-1,n-q-l-1}(|\Ccal|)$ be a superform such that
$\beta=\eta\wedge\omega$ has compact support. Then we have
\begin{equation}\label{deltatropicalpoincarelelongg}
\int_{|\Ccal|} \phi\, d'd''\beta- \int_{|\Ccal|} \dpa \dpb\phi\wedge\beta
=\int_{|\Ccal|} \delta_{\phi\cdot  C} \wedge \beta
\end{equation}
where we use the integral of polyhedral supercurrents on $|\Ccal|$ defined in \ref{currents and extension}.
\end{thm}

\proof
We may assume after suitable refinements that $\KC$ is also a polyhedral
complex of definition for $\phi$ and $\omega$.
From (\ref{psproductsleibniz}) and \eqref{polyhedral equals current}, we get
\[
\dpb(\phi\, d'\beta)
+\dpa(\dpa\phi\wedge \beta)=
\phi\, d''d'\beta+\dpa\dpb\phi\wedge\beta 
= \dpa\dpb\phi\wedge\beta - \phi\, d'd''\beta.
\]
Let $P$ denote the polyhedral set $|\KC|$.
Stokes' formula for polyhedral supercurrents 
\ref{stokesforpolyhedralcurrents}
yields
\begin{equation} \label{int1}
\int_{P}\Bigl(\dpa\dpb\phi\wedge\beta - \phi\, d'd''\beta\Bigr)
=\int_{\partial P}\phi\wedge d'\beta+
\int_{\partial P}\dpb\phi\wedge \beta.
\end{equation}
We write 
\[
\omega=\sum_{i\in I}\omega_i\wedge\delta_{C_i}
\]
for tropical cycles $C_i=(\KC_{\leq n-l},m_i)$ 
with suitable smooth  weight functions $m_i$
and  superforms $\omega_i$. Then we have
\begin{eqnarray*}
\int_{\partial P}\phi\wedge d'\beta
&=&\sum_{i\in I}\int_{\partial P}\phi\wedge d'\eta\wedge 
\omega_i\wedge\delta_{C_i}\\
&=&\sum_{i\in I}\sum\limits_{\sigma\in \KC_{n-l}}
\int_{\partial\sigma}m_{i\sigma}\phi_\sigma\, 
d'\eta\wedge\omega_i.
\end{eqnarray*}
For each $\sigma\in \KC_{n-l}$ and each face 
$\tau\in \KC_{n-l-1}$
we choose an element $\omega_{\sigma,\tau}$ as in 
(\ref{balancingcondition}).
We observe that the elements $\omega_{\tau,\sigma}$ used in
\cite[2.8]{gubler-forms} to define the boundary integrals 
$\int_{\partial \sigma}$
satisfy $\omega_{\tau,\sigma}=-\omega_{\sigma,\tau}$.
The definition of the boundary integral uses the contraction 
$\langle \,.\,, \omega_{\tau,\sigma}\rangle_{\{n-l\}}$ of the involved superform of type 
$(n-l,n-l)$ given by inserting $\omega_{\tau,\sigma}$ for the $(n-l)$-th variable and leads to
$$\int_{\partial P}\phi\wedge d'\beta = -\sum_{i\in I}\sum\limits_{\tau\in \KC_{n-l-1}}
\sum\limits_{\genfrac{}{}{0pt}{}{\sigma
\in \KC_{n-l}}{\tau \prec \sigma}}\int_\tau
\langle m_{i\sigma}\phi_\sigma\,  d'\eta\wedge\omega_i,
\omega_{\sigma,\tau}\rangle_{\{n-l\}}.$$
Given $i\in I$ and $\tau\in \KC_{n-l-1}$, the balancing condition 
\eqref{balancingcondition} for $C_i$
gives us the vector field
\[
\omega_{i\tau}:=\sum\limits_{\genfrac{}{}{0pt}{}{\sigma
\in \KC_{n-l}}{\tau \prec \sigma}}
m_{i\sigma}\omega_{\sigma,\tau}:\tau\to \L_\tau.
\]
We observe that $\phi_\sigma|_\tau=\phi_\tau$ for all $\tau\prec \sigma$ 
yielding
$$
\int_{\partial P}\phi\wedge d'\beta
= -\sum_{i\in I}\sum\limits_{\tau\in \KC_{n-1}}\int_\tau
\langle \phi_\tau\, d'\eta\wedge\omega_i,\omega_{i\tau}\rangle_{\{n-l\}}
= 0
$$
as a superform contracted with a vector field with values in $\L_\tau$ restricts to zero on $\tau$.  
Using this in \eqref{int1}, we obtain
\begin{equation} \label{crucial residue formula}
\int_{|\KC|}\phi\, d'd''\beta-\int_{|\KC|} \dpa\dpb \phi\wedge\beta
=-\int_{\partial {|\KC|}}\dpb\phi\wedge \beta.
\end{equation}
Our claim is then a consequence of the following lemma.
\qed

\begin{lem}\label{tropicallemma} 
Let $\phi$ be a piecewise smooth function on $|\KC|$. For any $\delta$-preform $\beta\in P_c^{n-1,n-1}(|\Ccal|)$ with compact support, we have
\begin{equation}\label{tropicallemmag1}
\int_{\partial |\KC|}\dpb\phi\wedge \beta=-\int_{|\KC|}  \delta_{\phi\cdot C}\wedge \beta, \quad 
\int_{\partial |\KC|}\dpa\phi\wedge \beta=\int_{|\KC|}  \delta_{\phi\cdot C}\wedge \beta.
\end{equation}
\end{lem}

\proof
We prove only the first formula. The second formula follows by applying the first one to $J^*(\beta)$ and using 
the symmetry of the supercurrent of integration. 
We use the notations introduced in the proof of 
Theorem \ref{deltatropicalpoincarelelong}. 
We may assume that $\beta \in P^{n-l-1,n-l-1,l}(|\KC|)$ and that
$$\beta = \sum_{i \in I} \eta_i \wedge \delta_{C_i}$$
for tropical cycles $C_i=(\KC_{\leq n-l},m_i)$ 
with suitable smooth  weight functions $m_i$
and  superforms $\eta_i \in A^{n-l-1,n-l-1}(|\KC|)$. 
Since $\beta$ is a $\delta$-preform on $C$, we may assume that there is a tropical 
cycle $\widetilde{C}_i$ of codimension $l$ in $N_\R$ such that $C_i= \widetilde{C}_i . C$ for 
every $i \in I$. Recall that $\frac{\partial}{\partial\omega_{\sigma,\tau}}$ denotes the partial derivative along 
the tangential vector $\omega_{\sigma,\tau}$. 
An exercise in linear algebra gives
\[
\langle d'' \phi_\sigma\wedge\eta_i,
\omega_{\sigma,\tau}\rangle_{\{2n-2l-1\}}=
\frac{\partial \phi_\sigma}{\partial\omega_{\sigma,\tau}}\wedge\eta_i+
d'' \phi_\sigma\wedge\langle  \eta_i,\omega_{\sigma,\tau}
\rangle_{\{2n-2l-2\}}
\]
for all $i\in I$, $\sigma\in \KC_{n-l}$ and all faces $\tau$ of $\sigma$
of codimension one.
Furthermore one sees easily that
\[
\int_\tau d''\phi_\tau\wedge\langle \eta_i,
\omega_{i\tau}\rangle_{\{2n-2l-2\}}
=-\int_\tau\frac{\partial  \phi_\tau}{\partial \omega_{i\tau}}\wedge
\eta_i.
\]
Let $\phi\cdot C_i=(\KC_{\leq n-l-1},m_i)$ denote the corner locus of 
$\phi$ on $C_i$ introduced in \ref{cornerlocus}.
Using the last two formulas and the definition of the weight functions
$m_{i\tau}$ of the corner locus in (\ref{cornerlocusg1}), we get
\begin{eqnarray*}
&&\sum\limits_{\genfrac{}{}{0pt}{}{\sigma
\in \KC_{n-l}}{\tau \prec \sigma}}\int_\tau
\langle m_{i\sigma}
d''\phi_\sigma\wedge \eta_i,\omega_{\sigma,\tau}\rangle_{\{2n-2l-1\}}\\
&=&\sum\limits_{\genfrac{}{}{0pt}{}{\sigma
\in \KC_{n-l}}{\tau \prec \sigma}}\int_\tau \Bigl(m_{i\sigma}
\frac{\partial \phi_\sigma}{\partial\omega_{\sigma,\tau}}\wedge \eta_i
+d'' \phi_\tau\wedge
\langle  \eta_i,\sum\limits_{\genfrac{}{}{0pt}{}{\sigma
\in \KC_{n-l}}{\tau \prec \sigma}}m_{i\sigma}\omega_{\sigma,\tau}\rangle_{\{2n-2l-2\}}
\Bigr)\\
&=&\sum\limits_{\genfrac{}{}{0pt}{}{\sigma
\in \KC_{n-l}}{\tau \prec \sigma}}\int_\tau \Bigl(m_{i\sigma}
\frac{\partial \phi_\sigma}{\partial\omega_{\sigma,\tau}}\wedge \eta_i
-\frac{\partial  \phi_\tau}{\partial \omega_{i\tau}}\wedge
\eta_i\Bigr)\\
&=&\int_\tau \Bigl(
\sum\limits_{\genfrac{}{}{0pt}{}{\sigma
\in \KC_{n-l}}{\tau \prec \sigma}}m_{i\sigma}
\frac{\partial \phi_\sigma}{\partial\omega_{\sigma,\tau}}-
\frac{\partial  \phi_\tau}{\partial \omega_{i\tau}}
\Bigr)\wedge\eta_i\\
&=&\int_\tau m_{i\tau}\,\eta_i
\end{eqnarray*}
for all $i\in I$ and $\tau\in \KC_{n-l-1}$.
For the polyhedral set  $P:=|\KC|$, we have
\begin{eqnarray*} 
\int_{\partial P}\dpb\phi\wedge\beta
&=&\sum_{i\in I}\sum_{\sigma\in \KC_{n-l}}\int_{\partial\sigma}
m_{i\sigma}d''\phi_\sigma\wedge\eta_i\\
&=&-\sum_{i\in I}\sum_{\tau\in \KC_{n-l-1}}\sum_{
\genfrac{}{}{0pt}{}{\sigma
\in \KC_{n-l}}{\tau \prec \sigma}}\int_{\tau}
\langle m_{i\sigma}d''\phi_\sigma\wedge \eta_i,
\omega_{\sigma,\tau}\rangle_{\{2n-2l-1\}}\\
&=&-\sum_{i\in I}\sum\limits_{\tau\in \KC_{n-l-1}}
\int_\tau m_{i\tau}\eta_i\\
&=&-\sum_{i\in I}\int_P
\eta_i\wedge\delta_{\phi\cdot C_i}.
\end{eqnarray*}
We get 
$\delta_{\phi\cdot C_i}=\delta_{\phi\cdot \widetilde C_i\cdot C}=\delta_{\widetilde C_i}\wedge\delta_{\phi\cdot C}$
from Proposition \ref{asscommcornerlocus}. Hence
\[
\sum_{i\in I}\int_P\eta_i\wedge\delta_{\phi\cdot C_i}
=\int_P\Bigl(\sum_{i\in I}
\eta_i\wedge\delta_{\widetilde C_i}\Bigr)\wedge\delta_{\phi\cdot C}\\
=\int_P \delta_{\phi\cdot C}\wedge \beta
\]
yields our claim.
\qed

\begin{rem} \label{generalization of tropicallemma}
In the situation of Lemma \ref{tropicallemma} 
we consider a $\delta$-preform $\beta\in P^{n-1,n-1}(|\KC|)$ on $C$. However we do no longer assume that $\beta$ has 
compact support. Instead we  make the weaker assumption that 
the polyhedral supercurrents $\dpb\phi\wedge \beta \in D_{1,0}(|\KC|)$ 
(resp. $\dpa\phi\wedge \beta \in D_{0,1}(|\KC|)$) 
and
$\delta_{\phi\cdot C} \wedge  \beta \in D_{0,0}(|\KC|)$
have compact support.
Then the first (resp. second) formula in (\ref{tropicallemmag1}) still hold for $\beta$.
In order to prove this, one chooses a function $f\in A^0_c(|\KC|)$
which is equal to $1$ on the above compact supports and
applies Lemma \ref{tropicallemma} to $f\cdot\beta$.
\end{rem}

\begin{cor}\label{tropicalpoincarelelong} 
Let $C=(\KC,m)$ be a tropical cycle with constant weight functions 
of pure dimension $n$ on $N_\R$ and
$\phi:|\KC|\to \R$ a piecewise smooth function with corner locus $\phi\cdot C$. 
Then we have
\begin{equation}\label{tropicalpoincarelelongg1}
d'd''[\phi]-[\dpa\dpb\phi]=\delta_{\phi\cdot C}
\end{equation}
in $D_{n-1,n-1}^\KC(|\KC|)$.
\end{cor}

\proof
Both sides of (\ref{tropicalpoincarelelongg1}) have support in $|\KC|$. 
Hence it suffices to show that
\[
\bigl(d'd''[\phi]-[\dpa\dpb\phi]\bigr)(\alpha)=\delta_{\phi\cdot C}(\alpha)
\]
holds for all $\alpha\in A_{c}^{n-1,n-1}(|\KC|)$ and this is a special
case of Theorem \ref{deltatropicalpoincarelelong}.
\qed

\begin{cor}\label{weaktropicalpoincarelelong} 
Let $\phi:|\KC|\to \R$ a piecewise linear function on $C$.
Then we have
\begin{equation}
d'd''[\phi]=\delta_{\phi\cdot C}
\end{equation}
in $D_{n-1,n-1}^\KC(N_\R)$.
\end{cor}

\proof
This follows from Corollary \ref{tropicalpoincarelelong}.
\qed

\section{Delta-forms on algebraic varieties}\label{deltaalgvar}

Let $X$ be an algebraic variety over $K$  of dimension $n$ 
and $X^\an$ the associated Berkovich space.

We introduce the algebra $B(W)$ of $\delta$-forms on
an open subset $W$ of  $X^\an$. We use tropicalizations
as in \cite{chambert-loir-ducros} and \cite{gubler-forms}
to pullback algebras of $\delta$-preforms to $X^\an$.
After a suitable sheafification process we obtain the sheaves of
algebras $B$ and $P$ of 
$\delta$-forms and generalized $\delta$-forms.
We show that $B$ is a sheaf of bigraded differential
$\R$-algebras with respect to natural differentials 
$d'$ and $d''$.

\begin{art}\label{chart}
Consider a {\it tropical chart} $(V,\varphi_U)$ on $X$
as in \cite[4.15]{gubler-forms}. 
It consists of a very affine   Zariski open $U$ in 
$X$. Recall that $U$ is called {\it very affine} if $U$ has a closed 
immersion into a multiplicative torus. Then 
there is a canonical torus $T_U$ with cocharacter group 
\[
N_U={\rm Hom}_\Z(\KO(U)^\times/K^\times,\Z),
\]
and a canonical closed embedding
$\varphi_U:U\to T_U$, unique up to translation (see \cite[4.12, 4.13]{gubler-forms} for details). We get a tropicalization map
\[
{\rm trop}_U:U^\an
\stackrel{\varphi_U^\an}{\longrightarrow} 
T_U^\an\stackrel{\rm trop}{\longrightarrow} N_{U,\R}
\]
associated with $\varphi_U$. The second ingredient of a tropical chart is 
 an open subset 
$V\subseteq U^{\rm an}$ for which there is  an open subset 
$\widetilde\Omega$ of $N_{U,\R}$ with $V=\trop_U^{-1}(\widetilde\Omega)$.

The set ${\rm trop}_{U}(U^\an)$ is the support
of a canonical tropical cycle ${\rm Trop}\,(U)=({\rm Trop}\,(U),m_U)$
with integral weights. It is the {\it tropical variety} associated to the closed 
subvariety $U$ of $T_U$ 
equipped with its canonical tropical weights (see \cite[\S 3, \S 13]{gubler-guide}). 
Note that  
$V = \trop_U^{-1}(\Omega)$ for the open subset 
$\Omega := \widetilde\Omega \cap \Trop(U)$ of $\Trop(U)$. 
\end{art}

\begin{definition}\label{compatible-pair}
Let $f:X'\to X$ be a morphism of algebraic varieties over $K$.
We say that charts $(V,\varphi_{U})$ and $(V',\varphi_{U'})$
of $X$ and $X'$ respectively are {\it compatible with respect to $f$},
if we have $f(U')\subseteq U$ and $f^{\rm an}(V')\subseteq V$.
\end{definition}

\begin{art}\label{sturmfelstevelev}
Let $f:X'\to X$ be a morphism of algebraic varieties over $K$
Given compatible charts $(V,\varphi_{U})$ and $(V',\varphi_{U'})$
of $X$ and $X'$, we obtain a commutative diagram
\[
\xymatrix{
V'\ar[d]^{f^{\rm an}|_{V'}}\ar@{^{(}->}[r]& 
(U')^{\rm an}\ar[d]^{f^{\rm an}|_{U'_{\rm an}}}\ar[r]^{\varphi_{U'}}&
T_{U'}\ar[d]^{\psi}\ar[r]^{\trop}&
N_{U',\R}\ar[d]^{F}\\
V\ar@{^{(}->}[r]&  
U^{\rm an}\ar[r]^{\varphi_{U}}&
T_U\ar[r]^\trop&
N_{U,\R}}
\]
where $\psi:T_{U'} \rightarrow T_U$ is the canonical affine homomorphism 
of tori induced by $\Ocal^\times(U) \to \Ocal^\times(U')$ and where $F:N_{U',\R} \rightarrow N_{U,\R}$ is the induced  
canonical integral $\Gamma$-affine map. These maps are  unique up to translation, but this ambuigity will never play a role. If $\Omega'$ is the open subset of 
$\Trop(U')$ with $\trop_{U'}^{-1}(\Omega')=V'$, then $\Omega'\subseteq F^{-1}(\Omega)
\cap \Trop(U')$.

{We define ${\rm deg}\,(f)=[K(X'):K(X)]$ if
$f$ is dominant and} the  extension of function fields is finite.
Otherwise we set ${\rm deg}\,(f)=0$.
{Let $Y$ be the schematic image of $f$ and $f':X'\to Y$ the 
induced morphism.} 
Then a formula of Sturmfels and Tevelev which was generalized by
Baker, Payne and Rabinoff to the present setting gives
\begin{equation}\label{cor-sturmfels-tevelev}
F_*\,{\rm Trop}\,(U')={{\rm deg}\,(f')}\cdot{\rm Trop}\,(\overline{f(U')})
\end{equation} 
as an equality of tropical cycles (see \cite{sturmfels-tevelev},
\cite[\S 7]{baker-payne-rabinoff}
or \cite[Thm. 13.17]{gubler-guide}).
\end{art}

\begin{definition}\label{presheaf} 
Let  us consider a tropical chart $(V,\varphi_U)$ of $X$. As above, we 
consider the open subset $\Omega :=\trop_U(V)$ of $\Trop(U)$. 
We choose an open subset $\widetilde{\Omega}$ of $N_{U,\R}$ with 
$\Omega = \widetilde\Omega \cap \Trop(U)$ and a $\delta$-preform 
$\tilde\alpha \in P^{p,q}(\widetilde\Omega)$. 
For any  morphism $f:X'\to X$  {of varieties over $K$} 
and a tropical chart $(V',\varphi_{U'})$ of 
$X'$ compatible with $(V,\varphi_U)$,  we define 
$\Omega':=\trop_{U'}(V')$. 
We choose an open subset $\widetilde\Omega'$ of 
 {$F^{-1}(\widetilde\Omega)$} with $\widetilde\Omega' \cap \Trop(U')=\Omega'$. 
By Proposition \ref{deltaformsbasicprop}, we have 
$F^*(\tilde\alpha) \in P^{p,q}(\widetilde\Omega')$. 
We denote by $N^{p,q}(V,\varphi_U)$ the subspace given by elements
$\tilde\alpha\in P^{p,q}(\widetilde\Omega)$ such that we have
$F^*(\tilde{\alpha})|_{\Omega'}=0 \in P^{p,q}(\Omega')$ for 
all compatible pairs of charts as above (see \ref{definitionpreformaufOmega} 
for the definition of the restriction).
We define
\[
P^{p,q}(V,\varphi_{U}):=
P^{p,q}(\widetilde\Omega)/N^{p,q}(V,\varphi_U).
\]
A partition of unity argument shows that this definition is independent 
of the choice of $\widetilde\Omega$.
We call an element in $P^{p,p}(V,\varphi_{U})$ {\it symmetric
(resp. anti-symmetric)} if it can be represented
by a symmetric (resp. anti-symmetric) $\delta$-preform in 
 {$P^{p,p}(\widetilde\Omega)$}.
We define
\[
P^{p,q,l}(V,\varphi_{U}):=
P^{p,q,l}(\widetilde\Omega)/\bigl(P^{p,q,l}(\widetilde\Omega)\cap 
N^{p,q}(V,\varphi_U)\bigr).
\]
using the $\delta$-preforms on $\widetilde\Omega$ of codimension $l$ from Definition \ref{defdeltaforms}. 
\end{definition}

\begin{rem} \label{properties of delta-preforms on charts}
(i) The $\wedge$-product descends to the space 
\[
P(V,\varphi_U):= \bigoplus_{p,q\geq 0} P^{p,q}(V,\varphi_U)
\] 
and we get a bigraded anti-commutative
$\R$-algebra which contains $A(\Omega)$ as a bigraded subalgebra. 

(ii) If $(V',\varphi_{{U'}})$ and  $(V,\varphi_U)$ are compatible charts with 
respect to $f:X' \rightarrow X$
as in \eqref{compatible-pair}, then  we get a canonical bigraded homomorphism
\[
f^*:P(V,\varphi_{U})\to P(V',\varphi_{U'})
\]
of bigraded $\R$-algebras which is defined for $\alpha \in P^{p,q}(V,\varphi_U)$ 
as follows: By definition, $\alpha$ is 
represented by  some $\tilde{\alpha} \in P^{p,q}(\widetilde\Omega)$. 
Let $\Omega':=\trop_{U'}(V')$ and choose an open subset 
$\widetilde\Omega'$ of $F^{-1}(\widetilde\Omega)$ with $\Omega'= \widetilde\Omega' \cap \Trop(U')$. 
Then  we define $f^*(\alpha) \in P^{p,q}(V',\varphi_{U'})$ as the class of $F^*(\tilde{\alpha}) 
\in P^{p,q}(\widetilde\Omega')$. If  $X=X'$ and $f=\id$, then 
$(V',\varphi_{U'})$ is a tropical subchart of $(V,\varphi_U)$ and we write $\alpha|_{V'}$ 
for the pull-back of $\alpha \in P^{p,q}(V,\varphi_U)$.  

Note that the definition of $f^*(\alpha)$ does not depend on the choice of the 
representative $\tilde{\alpha}$. 

However, the 
elements of $P^{p,q}(V,\varphi_U)$ do not only depend on the {\it restriction} 
\begin{equation} \label{restriction of delta to Omega}
\alpha|_\Omega:= \tilde{\alpha}|_\Omega = \tilde\alpha \wedge \delta_{\Trop(U)}\in P^{p,q}(\Omega) \subseteq D^{p,q}(\Omega)
\end{equation}
{\it to $\Omega$} 
as  {Example \ref{rational conic} below shows that it} might happen that 
two different elements $\alpha, \beta \in P^{p,q}(V,\varphi_U)$ satisfy 
$\alpha|_\Omega=\beta|_\Omega \in P^{p,q}(\Omega)$. 
The purpose of our definition of $P(V,\varphi_U)$  
is to have a pull-back as above at hand. 
Here we use the fact that we always have a pull-back from tropical cycles on $N_{U,\R}$ to tropical
cycles on $N_{U',\R}$, but there  {is} a pull-back available from tropical 
cycles on $\Trop(U)$ to tropical cycles on $\Trop(U')$ only if these tropical 
varieties are smooth (see \cite{francois-rau}).   {To have a pull-back available, we  consider  all morphisms $f:X' \to X$ of varieties over $K$ in the definition of $N^{p,q}(V,\varphi_{U})$ and not only open immersions.}
\end{rem}

\begin{art} \label{differentiation and closedness}
 As mentioned already in \ref{preforms are not d-invariant},
we have the problem that the differential operators $d'$ 
and $d''$ are not defined on the algebra $P(V,\varphi_U)$.
For $\alpha$ in $P^{p,q}(V,\varphi_U)$ and every compatible tropical chart 
$(V',\varphi_{U'})$ with respect to $f:X' \to X$, we use the 
above notation. 
We get 
a $\delta$-preform $f^*(\alpha)|_{\Omega'}= F^*(\tilde{\alpha})|_{\Omega'} 
\in P^{p,q}(\Omega')$. 
Recall that $f^*(\alpha)|_{\Omega'}$ is a supercurrent on $\Omega'$. 
We differentiate it in the sense of supercurrents to get 
$d'[f^*(\alpha)|_{\Omega'}] \in D(\Omega')$, but it has  not to be a $\delta$-preform
on $\Omega'$. In the following construction, we pass to a convenient 
subalgebra of $P(V, \varphi_U)$ which is invariant under $d'$ and $d''$. 

In a preliminary step, we consider the elements $\omega$ of 
$P^{p,q}(V,\varphi_{U})$ (resp. $P^{p,q,l}(V,\varphi_{U})$)  
satisfying the {\it closedness condition} 
\begin{equation} \label{closedness condition}
d'[f^*(\omega)|_{\Omega'}]=d''[f^*(\omega)|_{\Omega'}]=0 \in D(\Omega')
\end{equation}
for every tropical chart $(V',\varphi_{U'})$ which is compatible with 
$(V,\varphi_U)$ with respect to $f:X' \rightarrow X$. These elements form a  
subspace $Z^{p,q}(V, \varphi_U)$ of $P^{p,q}(V,\varphi_{U})$ 
(resp. $Z^{p,q,l}(V, \varphi_U)$ of $P^{p,q,l}(V,\varphi_{U})$) and we define 
\[
Z(V,\varphi_U):=\bigoplus_{p,q\geq 0} Z^{p,q}(V,\varphi_U)
=\bigoplus_{p,q,l\geq 0} Z^{p,q,l}(V,\varphi_U)
\] 
as usual. 
\end{art}

\begin{prop} \label{properties of omega-preforms}
Using the notation above, $Z(V,\varphi_U)$ is a bigraded 
$\R$-sub\-al\-ge\-bra of $P(V,\varphi_U)$.  
\end{prop}

\begin{proof} 
The only non-trivial point is that $Z(V,\varphi_U)$ is closed under the 
$\wedge$-product.
This is a direct consequence of Lemma \ref{erstesproduktlemma} applied to 
$\delta$-preforms on the tropical cycle ${\rm Trop}\,(U')$ for any tropical 
chart $(V',\varphi_{U'})$ compatible with $(V,\varphi_{U})$. 
\end{proof}

\begin{example} \label{constant weights and omega-preforms} \rm
Every tropical cycle $C=(\Ccal,m)$ on $N_{U,\R}$ with {\it constant weight functions} 
induces an element in  $Z(V,\varphi_U)$. Indeed, if $(V',\varphi_{U'})$ is 
a tropical chart on $X'$ compatible with $(V,\varphi_U)$ as above, then 
$F^*(\delta_C) |_{\Omega'}$ is given by the restriction of 
$\delta_{F^*(C)\cdot \Trop(U')}$ to $\Omega'$.
Since $F^*(C)\cdot \Trop(U')$ is a tropical cycle with constant weight functions, the 
associated current is $d'$- and $d''$-closed \cite[Proposition 3.8]{gubler-forms}.
\end{example}

\begin{definition} \label{AZ preforms}
Let $AZ(V,\varphi_U)$ be the subalgebra of $P(V,\varphi_U)$ generated by 
$A(\Omega)$ and $Z(V,\varphi_U)$. An element  {$\beta \in AZ(V,\varphi_U)$} has the form
\begin{equation} \label{AZ-decomposition}
\beta = \sum_{i \in I} \alpha_i \wedge \omega_i
\end{equation}
for a finite set $I$ with all $\alpha_i \in A(\Omega)$ and 
$\omega_i \in Z(V,\varphi_U)$. We define
\[
d'\beta :=  \sum_{i \in I} d'(\alpha_i) \wedge \omega_i, \quad d''\beta 
:=  \sum_{i \in I} d''(\alpha_i) \wedge \omega_i.
\]
It follows from the closedness condition \eqref{closedness condition} that 
$d'\beta$ and $d''\beta$ are well-defined elements in $AZ(V,\varphi_U)$. 
By definition, we have 
\[
Z(V,\varphi_U) = \{\alpha \in AZ(V,\varphi_U) \mid d'(\alpha)=d''(\alpha)=0\}.
\]
An element in $AZ(V,\varphi_U)$ is called {\it symmetric (resp.
anti-symmetric)} if
it is symmetric (resp. anti-symmetric) in $P(V,\varphi_U)$.
\end{definition}

The following result shows that $AZ(V,\varphi_U)$ is a good analogue of the 
algebra of complex differential forms.

\begin{prop} \label{AZ preforms are graded algebra}
The space $AZ(V,\varphi_U)$ is a bigraded differential $\R$-algebra
with respect to the differentials $d'$ and $d''$. 
\end{prop}

\begin{proof} This follows easily from Leibniz's rule 
\ref{omegaleibnizformel}(ii)
and Proposition \ref{properties of omega-preforms}.
\end{proof}

\begin{prop} \label{pullback of AZ preforms}
Let $f:X' \rightarrow X$ be a morphism of varieties over $K$. 
Let $(V,\varphi_U)$ and $(V',\varphi_{U'})$ be tropical charts of $X$ and $X'$ 
respectively which are compatible with 
respect to $f$. Then the pull-back  homomorphism 
$f^*: P(V,\varphi_U) \rightarrow P(V',\varphi_{U'})$ maps $Z(V,\varphi_U)$ 
 to $Z(V',\varphi_{U'})$ and $AZ(V,\varphi_U)$ to $AZ(V',\varphi_{U'})$.
\end{prop}

\begin{proof} This follows directly from the definitions.
We leave the details to the reader. 
\end{proof}

\begin{prop} \label{finite sheaf property}
Let $(V,\varphi_U)$ be a tropical chart of $X$ and $\Omega :=\trop_U(V)$. 
Let $(\Omega_i)_{i \in I}$ be a finite open covering of $\Omega$. 
For $i\in I$, let $V_i := \trop_U^{-1}(\Omega_i)$ and let 
$\alpha_i \in P(V_i,\varphi_U)$. For all $i,j \in I$, we assume that 
$\alpha_i|_{{V_i \cap V_j}}=\alpha_j|_{{V_i \cap V_j}}$. 
Then there is a unique $\alpha \in P(V,\varphi_U)$ with 
$\alpha|_{V_i}=\alpha_i$ for every $i \in I$.
If $\alpha_i \in AZ(V_i,\varphi_U)$ for every $i\in I$ then
$\alpha \in AZ(V,\varphi_U)$.
\end{prop}

\begin{proof}
It is a straightforward consequence of our definitions that $\alpha$ is unique.
In order to construct $\alpha$ we choose for each $i\in I$ an open subset 
$\widetilde\Omega_i$ in $N_{U,\R}$ such that
$\widetilde\Omega_i\cap {\rm Trop}\,(U)=\Omega_i$ and a $\delta$-preform
$\tilde\alpha_i\in P(\widetilde\Omega_i)$ which represents $\alpha_i$.
Let $(\phi_i)_{i\in I}$ be a smooth partition of unity
on $\widetilde\Omega=\cup_{i\in I}\widetilde\Omega_i$
with respect to the covering $(\widetilde\Omega_i)_{i\in I}$.
Observe that we may choose the same index set $I$ as we do not require that
the $\phi_i$ have compact support.
Then $\tilde\alpha:=\sum_{i\in I}\phi_i\tilde\alpha_i\in P(\widetilde\Omega)$
induces by our assumptions the desired element $\alpha$ in $P(V,\varphi_U)$.
If 
\[
\alpha_i =\sum_{j\in I_i}\beta_{ij}\wedge\omega_{ij}\in AZ(V_i,\varphi_U)
\]
as in \eqref{AZ-decomposition},
we choose representatives $\tilde\beta_{ij}\in A(\widetilde\Omega_i)$ of
$\beta_{ij}\in A(\Omega_i)$ and
$\tilde\omega_{ij}$ in $P(\widetilde\Omega_i)$ of 
$\omega_{ij} \in Z(V,\varphi_U)$.
Then we may choose $\tilde\alpha_i$ as 
$\sum_{j\in I_i}\phi_i\beta_{ij}\wedge\tilde\omega_{ij}$
and
\[
\tilde\alpha=\sum_{i\in I}\phi_i\tilde\alpha_i=
\sum_{i\in I}\sum_{j\in I_i}\phi_i\beta_{ij}\wedge\tilde\omega_{ij}
\]
shows $\alpha\in AZ(V,\varphi_U)$ using finiteness of $I$.
\end{proof}

Recall that the tropical charts $(V,\varphi_U)$ of $X$ form a 
basis for $\Xan$ \cite[Proposition 4.16]{gubler-forms}. 
Hence we can use  the algebras $P(V,\varphi_U)$ and $AZ(V,\varphi_U)$ to define sheaves on $\Xan$ as follows: 

\begin{definition} \label{predefinition}
For a fixed open subset $W$ in $X^\an$,
the set of all tropical charts $(V,\varphi_U)$ on $X$ with $W\subseteq V$ is
ordered with respect to compatibility and forms a directed set. 
Then we get presheaves
\begin{equation}\label{predefinitiong}
W\mapsto \underrightarrow\lim\,P(V,\varphi_U),\,\,\,
W\mapsto \underrightarrow\lim\,AZ(V,\varphi_U)
\end{equation}
of real  {vector spaces} on $\Xan$, 
where the limit is taken over this directed set with respect
to the pull-back maps considered in Proposition \ref{pullback of AZ preforms}.
The associated sheaves $P$ and $B$ on $X^\an$ are by
definition the {\it sheaf of generalized $\delta$-forms}
and the {\it subsheaf of $\delta$-forms}. 
On an open subset $W$ of $X^{\rm an}$ the space of $\delta$-forms
\[
B(W)=\bigoplus_{p,q\geq 0}B^{p,q}(W)=\bigoplus_{p,q,l\geq 0}B^{p,q,l}(W)
\]
and the space of generalized $\delta$-forms
\[
P(W)=\bigoplus_{p,q\geq 0} P^{p,q}(W)=\bigoplus_{p,q,l\geq 0} P^{p,q,l}(W) 
\] 
carry natural gradings by the $(p,q)$-type of the underlying currents
and the codimension of the underlying tropical cycles (as defined in \ref{defdeltaforms} and \ref{presheaf}). The wedge product 
on the spaces $AZ(V, \varphi_U)$ (resp. $P(V,\varphi_U)$)  induces a product on $B(W)$ (resp. $P(W)$). 
Moreover, the differential operators $d',d''$ on $AZ(V, \varphi_U)$  induce differential operators 
$d',d''$ on $B(W)$.
The symmetric and anti-symmetric elements in $P(V,\varphi_U)$
define subsheaves of (generalized) symmetric and anti-symmetric 
$\delta$-forms in $B^{p,q}$ and $P^{p,q}$ for all $p,q\geq 0$.
 
\end{definition}

\begin{art} \label{delta-forms on X}
We conclude that a {\it $\delta$-form} 
$\beta$ of bidegree $(p,q)$ on 
an open subset $W$ of $\Xan$ is given by a covering $(V_i)_{i \in I}$ 
of $W$ by tropical charts $(V_i,\varphi_{U_i})$ of $\Xan$ and elements
 $\beta_i \in AZ^{p,q}(V_i,\varphi_{U_i})$
 such that 
\[
\beta_i|_{V_i \cap V_j}=\beta_j|_{V_i \cap V_j}
\]
holds for all 
$i,j \in I$.  If $\beta'$ is another $\delta$-form  of bidegree 
$(p,q)$ on $W$ given by $\beta_j' \in AZ^{p,q}(V_j',\varphi_{U'_j})$ 
with respect to the tropical charts $(V_j',\varphi_{U'_j})_{j \in J}$ 
covering $W$, then $\beta$ and $\beta'$ define 
the same $\delta$-forms  
if and only if  
\[
\beta_i|_{V_i \cap V_j'}=\beta_j'|_{V_i \cap V_j'}
\] 
holds for all $i \in I$ and $j\in J$. A similar description holds for generalized $\delta$-forms.
\end{art}

\begin{prop} \label{properties of delta-forms}
(i) The sheaves $P$ and $B$ are sheaves of 
bigraded anti-commu\-ta\-tive $\R$-algebras.

(ii) We have natural monomorphisms of sheaves of bigraded $\R$-algebras 
$A\to B$ and $B\to P$.

(iii) The differentials $d',d'':B\to B$ turn $(B,d',d'')$ into a sheaf 
of bigraded differential $\R$-algebras.
\end{prop}

\proof
Only the injectivity of the natural morphism $A\to B$ does not follow
directly from what we have shown before.
The injectivity of $A\to B$ can be checked on the presheaves \eqref{predefinitiong}.
For each tropical chart $(V,\varphi_U)$ of $X$
the natural map from $A(V)$ to $AZ(V,\varphi_U)$
is injective as the associated map 
$A(\Omega)\to AZ(V,\varphi_U)\to D(\Omega)$ for $\Omega={\rm trop}_U(V)$ 
is injective.
This yields directly our claim.
\qed

\begin{art} \label{pull-back of delta-forms}
Let $f:X' \rightarrow X$ be a morphism of varieties over $K$. For an open subset 
$W$ of $\Xan$ and an open subset $W'$ of $f^{-1}(W)$, we have a canonical 
pull-back morphism 
$f^*:P(W)\to P(W')$
which respects products and the bigrading. Furthermore it induces a  homomorphism $f^*:B(W)\to B(W')$  of bigraded $\R$-algebras 
which commutes with the differentials $d'$ and $d''$ on $B$. They are induced by the pull-back $f^*: P(V,\varphi_U) \rightarrow P(V',\varphi_{U'})$
 for compatible charts $(V',\varphi_{U'})$ on $W'$ and $(V,\varphi_U)$ on $W$ given in Proposition \ref{pullback of AZ preforms}.
\end{art}

\begin{lem}\label{phitrop lemma}  
Let $(V,\varphi_U)$ be a tropical chart on $X$.
Let $(V_i)_{i\in I}$ be an open covering of $V$ by tropical charts
$(V_i,\varphi_{U_i})$ on $X$ which are compatible with $(V,\varphi_U)$.
There are canonical integral $\Gamma$-affine morphisms $F_i:N_{U_i,\R}\to N_{U,\R}$
such that ${\rm trop}_U=F_i\circ {\rm trop}_{U_i}$.
We choose open subsets $\widetilde\Omega$ in $N_{U,\R}$ and 
$\widetilde\Omega_i$ in $F_i^{-1}(\widetilde\Omega)$ such that
$V={\rm trop}_U^{-1}(\widetilde\Omega)$ and $V_i={\rm trop}_U^{-1}(\widetilde\Omega_i)$
for all $i\in I$.
Let $\tilde\alpha_U\in P(\widetilde\Omega)$ be a $\delta$-preform.
Then $\tilde\alpha_U\wedge \delta_{{\rm Trop}\,(U)}$ vanishes in
$D(\widetilde\Omega)$ if $F_i^*(\tilde\alpha_U)\wedge \delta_{{\rm Trop}\,(U_i)}$ 
vanishes in $D(\widetilde\Omega_i)$ for every $i\in I$.
\end{lem}

\begin{proof}
We write $\tilde\alpha_U=\sum_{j\in J}\alpha_j\wedge\delta_{C_j}$ for
suitable superforms $\alpha_j\in A(\widetilde\Omega)$ and tropical cycles $C_j$.
We have $F_{i*}{\rm Trop}\,(U_i)={\rm Trop}\,(U)$ by \eqref{cor-sturmfels-tevelev}.
The projection formula \ref{tropintthprop} gives
\[
F_{i*}(F_i^*C_j\cdot {\rm Trop}\,(U_i))=C_j\cdot {\rm Trop}\,(U).
\]
By the same arguments as in the proof of Proposition \ref{projection formula for preforms}, the vanishing of
\[
F_i^*(\tilde\alpha_U)\wedge \delta_{{\rm Trop}\,(U_i)}
=\sum_{j\in J}F_i^*(\alpha_j)\wedge \delta_{F_i^*C_j\cdot {\rm Trop}\,(U_i)}
\]
in $D(\widetilde\Omega_i)$ for all $i\in I$ yields that
\[
\tilde\alpha_U\wedge \delta_{{\rm Trop}\,(U)}
=\sum_{j\in J}\alpha_j\wedge \delta_{F_{i*}(F_i^*C_j\cdot {\rm Trop}\,(U_i))}
\]
vanishes in $D(\widetilde\Omega)$.
\end{proof}

\begin{prop}\label{phitrop}
Given a tropical chart $(V,\varphi_U)$ on $X$, we have by
construction natural algebra homomorphisms
\begin{eqnarray*}
&\trop_U^*:P^{p,q}(V,\varphi_U)\longrightarrow P^{p,q}(V),&\\
\label{canonical trop maps g2}
&\trop_U^*:AZ^{p,q}(V,\varphi_U)\longrightarrow B^{p,q}(V)&
\end{eqnarray*}
for all $p,q\geq 0$. These maps are injective.
\end{prop}

\begin{proof}
We extend the argument in \cite[Lemme 3.2.2]{chambert-loir-ducros}.
It suffices to show that the upper map is injective. 
Let ${\rm trop}_U^*(\alpha_U)$ vanish for some 
$\alpha_U\in P(V,\varphi_U)$. We obtain an open covering $(V_i)_{i\in I}$ of
$V$ by tropical charts $(V_i,\varphi_{U_i})$ compatible with $(V,\varphi_U)$
such that $\alpha_U|_{V_i}=0$ in $P(V_i,\varphi_{U_i})$ for all $i\in I$.
Let $\alpha_U$ be induced by $\tilde\alpha_U \in P(\widetilde\Omega)$ for  some
open subset $\widetilde\Omega\in N_{U,\R}$ with 
$V={\rm trop}_U^{-1}(\widetilde\Omega)$.
We have to show that $\tilde\alpha_U\in N(V,\varphi_U)$.
Let $f:X'\to X$ be a morphism of varieties and 
$(V',\varphi_{U'})$ a
tropical chart on $X'$ which is compatible with $(V,\varphi_U)$. 
We obtain a canonical integral  $\Gamma$-affine morphism 
$F:N_{U',\R}\to N_{U,\R}$ such that 
${\rm trop}_U=F\circ {\rm trop}_{U'}$.
We choose an open subset $\widetilde\Omega'$ in 
$F^{-1}(\widetilde\Omega)$
such that $V'=\trop_U^{-1}(\widetilde\Omega')$.
We have to show that 
$F^*(\tilde\alpha_U)\wedge\delta_{\Trop(U')}$ 
vanishes in $D(\widetilde\Omega')$.

For every $i\in I$ we choose an open covering $(V_{ij}')_{j\in J_i}$
of $(f^{\rm an})^{-1}(V_i)\cap V'$ by tropical charts
$(V_{ij}',\varphi_{U_{ij}'})$ on $X'$ which are compatible with 
$(V',\varphi_{U'})$ and $(V_i,\varphi_{U_i})$.
For all $i\in I$ and $j\in J_i$ we obtain a commutative diagram
\[
\xymatrix{
(U_{ij}')^{\rm an}\ar[d]_{f|_{U_{ij}'}}\ar@{^{(}->}[r]\ar[drrrr]_{{\rm trop}_{U_{ij}'}}&
(U')^{\rm an}\ar[d]_{f|_{U'}}\ar[drrrr]_{{\rm trop}_{U'}}\\
U_i^{\rm an}\ar@{^{(}->}[r]\ar[drrrr]_{{\rm trop}_{U_i}}&
U^{\rm an}\ar[drrrr]_{{\rm trop}_{U}}&&&
N_{U_{ij}',\R}\ar[r]_{F_{ij}'}\ar[d]^{F_{ij}}&
N_{U',\R}\ar[d]^{F}\\
&&&&N_{U_i,\R}\ar[r]_{F_i}&N_{U,\R}.}
\]
of canonical maps. We choose an open subset $\widetilde\Omega'_{ij}$ in 
$(F_{ij}')^{-1}(\widetilde\Omega')\cap (F_{ij})^{-1}(\widetilde\Omega_i)$
such that $V_{ij}'={\rm trop}_{U'_{ij}}^{-1}(\widetilde\Omega'_{ij})$.
We have $(F_{ij}')^*F^*(\tilde\alpha_U)\wedge\delta_{{\rm Trop}\,(U_{ij}')}=0$
in $D(\widetilde\Omega_{ij}')$ by the commutativity of the above diagram
and the fact that $\alpha_U|_{V_i}=0$ in $P(V_i,\varphi_{U_i})$.
Now Lemma \ref{phitrop lemma} applied to $F^*(\tilde\alpha_U)$ on $(V',\varphi_{U'})$
and the covering $(V_{ij}')_{ij}$ of $V'$  
yields the vanishing of $F^*(\tilde\alpha_U)\wedge \delta_{{\rm Trop}\,(U')}$ 
in $D(\widetilde\Omega')$.
\end{proof}

\begin{art} \label{functions as delta-forms}
Let $W$ be an open subset of $\Xan$. By construction, the algebra 
$A^{\cdot,\cdot}(W)$ of differential forms on 
$W$ is a bigraded subalgebra of the algebra $B^{\cdot,\cdot}(W)$ 
of $\delta$-forms. In general, $A^{p,q}(W)$ is 
a proper subspace of $B^{p,q}(W)$. The situation in degree zero is quite 
different as we may identify $\delta$-forms
of degree $0$ with functions. We will show that 
\begin{equation} \label{delta-forms in degree 0}
A^{0,0}(W)=B^{0,0}(W).
\end{equation}
Clearly, this is a local statement and so we may consider a tropical 
chart $(V,\varphi_U)$ on $W$. It is enough to show 
\begin{equation} \label{AZ-forms in degree 0}
A^{0,0}(\Omega)=AZ^{0,0}(V,\varphi_U)
\end{equation}
for the open subset $\Omega := \trop_U(V)$ of $\Trop(U)$. 
Let $\widetilde\Omega$ be any open subset of $N_{U,\R}$. 
Since pull-back of functions is always well-defined, we may identify the 
elements of $P^{0,0}(V,\varphi_U)$ with some continuous 
functions on $\Omega$ and a partition of unity argument together with Example \ref{preforms of degree 0}  shows 
\begin{equation} \label{preforms and degree 0}
P^{0,0}(V,\varphi_U)= \{ \phi|_\Omega \mid \phi \in P^{0,0}(
\widetilde\Omega)\}=\{ \phi|_\Omega \mid \phi \in PS^{0,0}(
\widetilde\Omega)\}
\end{equation}
To prove \eqref{AZ-forms in degree 0}, it is enough to show that the 
elements 
of $Z(V,\varphi_U)$ are precisely the locally constant 
functions on $\Omega$. By \eqref{preforms and degree 0}, we have to show that 
$\phi|_\Omega$ is locally constant for any $\phi \in PS^{0,0}(\widetilde\Omega)$ with $\phi|_\Omega \in Z(V,\varphi_U)$. 
This means that $\phi$ is a continuous 
function on $\widetilde\Omega$ with an integral $\R$-affine complete polyhedral complex $\KC$ on $N_\R$ 
such that $\phi|_{\Omega \cap \Delta}$ is smooth for every $\Delta \in \KC$. 
By refinement, we may assume that a subcomplex $\KD$ of $\KC$ has support 
equal to $\Trop(U)$.  Then the 
closedness condition \eqref{closedness condition} yields that $[\phi|_\Omega]$ 
is $d'$- and $d''$-closed. We 
conclude that $\phi|_{\Omega \cap \Delta}$ is constant on every 
$\Delta \in \KD$. 
By continuity, we deduce that $\phi|_\Omega$ is locally 
constant proving the claim.
\end{art}

\begin{art}\label{delta-support}
Let $(V,\varphi_U)$ be a tropical chart on $X$ and $\Omega={\rm trop}_U(V)$.

(i) If $\Omega_0$ is an open subset of $\Omega$, then $V_0:=\trop_U^{-1}(\Omega_0)$ 
is an open subset of $V$ and $(V_0,\varphi_U)$ is a tropical chart 
of $X$. 
We say that 
$\alpha_U \in P(V,\varphi_U)$ vanishes on the open subset 
$\Omega_0$ if we have $\alpha_U|_{V_0}=0$ in $P(V_0,\varphi_U)$ (see Remark \ref{properties of delta-preforms on charts}). 
We define the {\it support of $\alpha_U\in P(V,\varphi_U)$} as 
the closed subset 
\[
\supp(\alpha_U)=
\Bigl\{\omega \in \Omega \Big| 
\genfrac{}{}{0pt}{}{\mbox{$\alpha_U$ does not vanish on any open}}{
\mbox{neighbourhood $\Omega_0$ of $\omega$ in $\Omega$}}
\Bigr\}
\]
in $\Omega$.

(ii) A (generalized) $\delta$-form $\alpha$ on an open subset
$W$ of $X^{\an}$ has a well defined support as a section of the sheaf $B^{p,q}$
(resp. $P^{p,q}$).
We denote by  $B^{p,q}_c$ 
(resp. by $P^{p,q}_c$) the subsheaves of forms with compact support.

(iii) Observe that compact support always implies proper support
in the sense of \cite[4.2.1]{chambert-loir-ducros} as our assumptions imply
that we have $\partial W=\emptyset$ for each open subset
$W$ of $X^{\an}$ (using that $\Xan$ is closed meaning that it has no 
boundary, see \cite[Theorem 3.4.1]{berkovich-book}).
\end{art}

\begin{prop}\label{delta-support on chart}
Let $(V,\varphi_U)$ be a tropical chart on $X$. 
Suppose that a generalized $\delta$-form $\alpha \in P(V)$ is 
given by $\alpha_U \in P(V,\varphi_U)$. 
Then $\alpha_U$ is uniquely determined and we have 
$\trop_U(\supp(\alpha))= \supp(\alpha_U)$. 
Furthermore $\alpha$ has compact support if and only if $\alpha_U$ has
compact support.
\end{prop}

\begin{proof}
This uniqueness follows from Proposition \ref{phitrop}.
The second statement follows from Proposition \ref{phitrop} by the same arguments as in \cite[Corollaire 3.2.3]{chambert-loir-ducros}. 
The last statement is a direct consequence of the
continuity and 
properness of the tropicalization map ${\rm trop}_U$
(see \cite[Remark in 2.3]{baker-payne-rabinoff}).
\end{proof}

\begin{example} \label{rational conic} \rm
 {We construct a tropical chart 
$(V,\varphi_U)$ and a non-zero {$\delta$-form} 
$\alpha \in AZ(V,\varphi_U)\setminus\{0\}$ 
with $\alpha|_\Omega=0$ for $\Omega := \trop_U(V)$. 
This example announced in Remark 
\ref{properties of delta-preforms on charts}
justifies the functorial definition of 
(generalized) delta-forms in Definition \ref{presheaf}.}

 {
We work over the ground field $K=\C_p$ 
for some prime number $p \neq 2,3$ 
and consider the affine curve $X$ in $\mathbb A_K^2$ defined
by the affine equation
\[
f(x,y)=xy+px^3+py^3.
\]
We consider the very affine open subset
$U=X\setminus (\{x=1\}\cup\{y=1\})$. 
The only singularity of the rational cubic $X$ is the origin $0=(0,0)$  
which is an ordinary double point. 
The normalization of $X$ may be seen as an open subset of $\mathbb P_K^1$ 
and can be obtained as the blowup of $X$ in $(0,0)$ as in \cite[Example I.4.9.1]{hartshorne77}. This description leads 
to a surjective morphism 
\[
\varphi\colon\mathbb P_K^1\setminus \{\xi_1,\xi_2,\xi_3\} \to X, 
\quad u \mapsto \biggl(x=\frac{-u}{p(1+u^3)},y=\frac{-u^2}{p(1+u^3)}\biggr)
\]
for a suitable affine coordinate $u$ on $\mathbb P_K^1$, where $\xi_i$ are the roots of $u^3+1=0$. 
It is clear that all $\xi_i$ have absolute value $1$ and we may choose  $\xi_1=-1$. 
Note that 
$\varphi^{-1}(\{x=1\})=\{\rho_1,\rho_2,\rho_3\}$ for the roots 
$\rho_i$ of $pu^3+u+p=0$ in $K$
and $\varphi^{-1}(X\setminus \{y=1\})=\{\rho_1^{-1},\rho_2^{-1},\rho_3^{-1}\}$.
Moreover, we have $\varphi^{-1}(0)=\{0,\infty\}$. }

 {The method of the Newton polygon  
\cite[Proposition II.6.3]{neukirch1999} shows that $pu^3+u+p=0$ has two roots of 
absolute value $|p|^{-1/2}$, say $\rho_2, \rho_3$, and one root $\rho_1$ of absolute value $|p|$.
We put 
\[
(\lambda_0,\lambda_1,\ldots,\lambda_8)=
(\xi_1,\xi_2,\xi_3,\rho_1,\rho_2,\rho_3,\rho_1^{-1},\rho_2^{-1},\rho_3^{-1})
\]
and get
\[
W:= \varphi^{-1}(U)=\mathbb P_K^1\setminus \{(\lambda_i:1)\mid i=0,\ldots,8\}.
\]
The abelian group $\mathcal O(W)^\times/K^\times$ is free of rank
eight with generators
$b_i=\frac{u-\lambda_i}{u+1}$, $i=1,\ldots,8$.
We deduce from \cite[Proposition 7.5.15]{liu-book} that 
$$ \mathcal O(U)^\times = \{f \in \mathcal O(W)^\times \mid f(0)=f(\infty)\}.$$
We conclude that $M_U:= \mathcal O(U)^\times/K^\times$ is a free abelian group of rank seven. }

 {
In the following, we would like to describe the canonical tropicalization 
${\Trop}(U)$ in the euclidean space $\R^7$  
given by chosing a basis in $M_U$. 
This is rather complicated and so we compute the tropicalization 
$\trop_{x-1,y-1}(U)$ in $\R^2$ using the tropicalization map 
\[
\trop_{x-1,y-1}\colon U^\an\longrightarrow \R^2,\,\,\,{q \mapsto 
(-\log |(x-1)(q)|, -\log |(y-1)(q)|}).
\]
This will be not enough for our purpose, but we will use the minimal skeleton $S(W)$ of $W$ for the computation and as $S(W)$ also covers ${\Trop}(U)$, we get a very good picture of the latter. This method to compute tropicalizations is due to 
\cite{baker-payne-rabinoff-2013,baker-payne-rabinoff} 
and we will refer to these papers for details of the following construction. 
{Skeleta are discussed in \cite{baker-payne-rabinoff-2013} and we refer to \cite[Corollary 4.23]{baker-payne-rabinoff-2013} for existence and uniqueness of the minimal skeleton $S(W)$ of the smooth curve $W$.} 
We recall that the skeleton $S(W)$ has a canonical retraction 
$\tau\colon (\mathbb P_K^1)^{\rm an} \to S(W)$ and hence $S(W)$ is a 
compact subset of $(\mathbb P_K^1)^{\rm an}$. 
Similarly as in the examples in \cite[Section 2]{baker-payne-rabinoff}, 
we describe the minimal skeleton $S(W)$ 
and the tropicalization $\Trop_{x-1,y-1}(U):=\trop_{x-1,y-1}(U^\an)$ in  Figure \ref{picture}. 
{Using \cite[Section 5]{gubler-rabinoff-werner},} there is a map $F:S(W) \to \Trop_{x-1,y-1}({U})$ 
with $F \circ \tau = \trop_{x-1,y-1} \circ \varphi^{\rm an}$ 
such that 
$F$ maps each segment  (resp.~leave) of $S(W)$  by an {integral $\Q$-affine map} 
 onto a segment (resp.~leave) of $\Trop_{x-1,y-1}(U)$. {One computes easily that these affine maps are all integral $\Q$-affine isomorphisms.} 
The {polyhedral set} $\Trop_{x-1,y-1}(U)$ carries a natural structure of
a tropical cycle \cite[Theorem 13.11]{gubler-guide}.
All weights are one if not indicated otherwise in Figure 1\footnote{Thanks to Christian Vilsmeier for drawing the {figure}.}.
For $r>0$, let $\zeta_r\in (\P_K^1)^\an$ be the 
supremum norm on the closed ball $\{|u|\leq r\}$
where $u$ denotes our distinguished affine coordinate on $\P_K^1$.}

\begin{figure}[H]\label{figure1}
\includegraphics{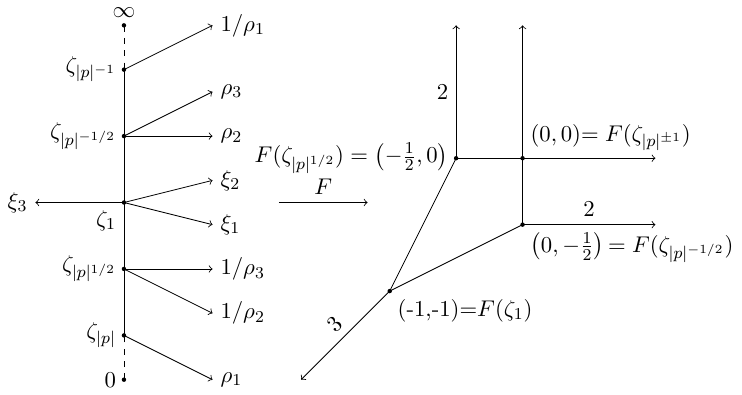}
\caption[Minimal skeleton $S(W)$ and $\trop_{x-1,y-1}(U)$]{Minimal skeleton $S(W)$ and $\Trop_{x-1,y-1}(U)$.}
\label{picture}
\end{figure}

 {Let $\widetilde\Omega := \{(x,y) \in \R^2 \mid x>-1/2, \, y > -1/2\}$  
and $\Omega := \widetilde\Omega \cap \Trop_{x-1,y-1}({U})$. 
Let $H\colon N_{U,\R} \to \R^2$ be the canonical 
affine map with $N_U$ the dual of $M_U$. 
Moreover, we have a canonical surjective map $G$ from the minimal 
skeleton $S(W)$ onto the canonical tropicalization $\Trop(U)$ 
which is affine on every segment and every leave of the minimal 
skeleton and such that $\trop_U \circ \varphi^{\rm an} = G \circ \tau$ for the canonical retraction $\tau$ onto the skeleton $S(W)$  (see \cite[Section 5]{gubler-rabinoff-werner}). 
Using our description of $\mathcal O(U)^\times$, we deduce that 
\[
{G(\zeta_{|p|^{-1}})}=G\circ \tau(\infty)=\trop_U \circ \varphi^{\rm an}(\infty)
\]
and
\[
{G(\zeta_{|p|})}=G\circ \tau(0)=\trop_U \circ \varphi^{\rm an}(0)
\]
are equal as the right hand sides are given in terms of units on $U$.}

 {Using that $F=H \circ G$,
we conclude that the fiber of the surjective map 
$H\colon \Trop(U)\to \trop_{x-1,y-1}(U)$ over $(0,0)$ is one single 
point and that $H$ maps $\Omega' := H^{-1}(\widetilde\Omega) \cap \Trop(U)$ 
homeomorphically and isometrically with respect to lattice 
length onto $\Omega$.
We express this fact by saying 
that $\Omega'$ is unimodular to $\Omega$. 
This is all we need in the following. }

 {
Now we consider the tropical chart $(V,\varphi_U)$ around the ordinary double point $0=(0,0)$ of $U$, where $V:=\trop_U^{-1}(\Omega)$. 
We consider the unique function $\tilde\phi$ on $\R^2$ which is linear on each quadrant with $\tilde\phi(1,0)=1$, $\tilde\phi(0,1)=-1$ and which is zero in the third quadrant. Let $ \phi$ be the restiction of $\tilde\phi$ to $\Omega$. 
Let $\phi':= \phi \circ H$ as a real function on $\Omega'$. 
It follows from the tropical Poincar\'e--Lelong formula in Theorem \ref{tpl} 
that $d'd''[\phi]$ is the supercurrent on $\Omega$ given by 
$\delta_{\phi \cdot \Trop_{x-1,y-1}(U)}$, where 
$\phi \cdot \Trop_{x-1,y-1}(U)$ is the corner locus of $\phi$. 
Similarly, $d'd''[\phi']=\delta_{\phi' \cdot \Trop(U)}$. 
It is clear that $\phi \cdot \Trop_{x-1,y-1}(U)$ is zero on 
$\Omega \setminus \{(0,0)\}$ as $\phi$ is linear there. 
By definition, the multiplicity of $\phi \cdot \Trop_{x-1,y-1}(U)$ in $(0,0)$ is the sum of the four outgoing slopes which is zero as well. 
We conclude that the corner locus  $\phi \cdot \Trop_{x-1,y-1}(U)$ is zero. 
Since we have shown that $\Omega'$ is unimodular to $\Omega$, we 
conclude that the corner locus $\phi' \cdot \Trop(U)$ is zero on $\Omega'$ as well.} 

{We note that the corner locus ${\tilde\phi' \cdot N_{U,\R}}$ of the 
function $\tilde\phi':=\tilde\phi \circ H$ on $N_{U,\R}$ induces a 
$\delta$-preform $\delta_{\tilde\phi' \cdot N_{U,\R}}$ on $N_{U,\R}$ 
which represents a $\delta$-form $\alpha$ on the tropical chart 
$(V,\varphi_U)$. 
We have $\alpha \in AZ^{1,1}(V,\varphi_U)\subset P^{1,1}(V,\varphi_U)$. 
It follows from Proposition \ref{asscommcornerlocus}  that
\begin{equation} \label{pull-back counterexample}
\alpha|_\Omega = \delta_{\tilde\phi' \cdot N_{U,\R}} \wedge 
\delta_{\Trop(U)} = \delta_{\phi' \cdot \Trop(U)}=0.
\end{equation}
Now let us consider the open ball $B:=\{|u| < |p|^{1/2}\}$ in $\mathbb P_K^1$. 
It is clear that $V'':= B \setminus \{\rho_1\}$ is mapped by $F$ to 
$\Omega \cap \{x=0\}$. 
The coordinate $w:=u-\rho_1$ on 
$U'':=\mathbb P_K^1 \setminus \{\rho_1,\infty\}$
induces {an isomorphism} $\varphi_{U''}\colon U''\to \G_m$.
Note that {$(V'',\varphi_{U''})$} is a tropical chart.
Indeed, we have $V''=\trop_{U''}^{-1}({\Omega''})$ 
and $\trop_{U''}(V'')= {\Omega''}$ {for $\Omega'':=]1/2,\infty[$.}
The tropical charts $(V'',\varphi_{U''})$ and $(V,\varphi_U)$ are compatible with 
respect to the morphism $\varphi$ and hence there is a canonical affine map 
$E:\R \to N_{U,\R}$ with $\trop_U \circ \varphi^{\rm an} = E \circ \trop_w$. 
We have
\[
\varphi^*(\alpha)|_{\Omega''}=E^*\bigl(\delta_{\tilde\phi' \cdot N_{U,\R}}\bigr)\big|_{\Omega''}
= \delta_{E^*(\tilde\phi')\cdot N_{U'',\R}}\big|_{\Omega''}.
\]
It is clear that $\phi'':=E^*(\tilde\phi')$ is a piecewise linear function on $\Omega''$ which is identically zero on ${]}1/2,1]$ and which has slope $1$ on $[1,\infty[$. It follows that 
$\varphi^*(\alpha)|_{\Omega''}= \delta_{{1}}$. 
We conclude that  $\alpha \in AZ^{1,1}(V,\varphi_U)$ is an example with $\alpha|_\Omega=0$, but {$\alpha\neq 0$ as} $\varphi^*(\alpha)|_{\Omega''} \neq 0$.}
 \end{example}

\section{Integration of delta-forms}\label{Integration of delta-forms}

We keep the notation and the hypotheses from the previous section. 
Our  goal is to introduce integration of generalized $\delta$-forms of top 
degree with compact support. We proceed as in \cite[5.13]{gubler-forms}. 
A crucial ingredient in our definition of the integral is
Lemma \ref{support and degree of delta-forms} which shows that 
the support of a generalized $\delta$-form of high degree
is always concentrated in points of high local dimensions.
This allows us to compute the integral with a single chart of integration. 
We obtain a well defined integral for generalized $\delta$-forms
which satisfies a projection formula and the theorem of Stokes.

\begin{art} \label{integration of P-forms}
Let $(V,\varphi_U)$ be a tropical chart of $X$.
As before we write $V={\rm trop}_U^{-1}(\widetilde\Omega)$ for some open subset
$\widetilde\Omega$ of $N_{U,\R}$ and $\Omega=\widetilde\Omega\cap {\rm Trop}\,(U)$.
 {Recall $n := \dim(X)$.} 
(i) An element $\alpha_U$ in $P(V,\varphi_U)$ is represented by
a $\delta$-preform $\tilde\alpha_U$ in $P(\widetilde\Omega)$
and determines a  $\delta$-preform 
\[
\alpha_U|_\Omega
=\tilde\alpha_U\wedge\delta_{{\rm Trop}\,(U)}
\in P(\Omega) \subseteq D(\Omega)
\]
on $\Omega$ as in \ref{restriction of delta to Omega} which 
does neither depend on the choice of $\tilde\alpha_U$ nor on 
the choice of $\widetilde\Omega$.  {Often, it is convenient 
to use the notation $\alpha_U|_{\Trop(U)}$ for $\alpha_U|_\Omega$.}

(ii) Given $\alpha_U$ in $P^{n,n}(V,\varphi_U)$ and an integral 
$\R$-affine polyhedral subset $P$ of $\Omega$
such that 
$P\cap \supp( {\alpha_U|_{\Omega}})$ is compact, we define
\[
\int_{P} \alpha_U
:=\int_{P} \tilde\alpha_U\wedge\delta_{{\rm Trop}\,(U)}
\]
where the righthand side is defined as in Remark \ref{preform properties}. 
As usual, we extend the integral by $0$ to the $\alpha_U$ of other bidegrees.

(iii) If $\alpha_U$ in $P^{n,n}(V,\varphi_U)$ and if the support of $\alpha_U|_\Omega$ is compact,  then we can 
consider $\alpha_U|_\Omega$ as a $\delta$-preform on $\Trop(U)$ with compact support 
 and we write
\[
\int_\Omega \alpha_U := \int_{|\Trop(U)|} \alpha_U|_\Omega
\]
again using Remark \ref{preform properties}.

(iv) Given $\alpha_U$ in $P^{n-1,n}(V,\varphi_U)$ or $P^{n,n-1}(V,\varphi_U)$
and an integral $\R$-affine polyhedral subset $P$ of $\Omega$
such that $P\cap \supp( {\alpha_U|_{\Omega}})$ is compact, 
we define 
\[
\int_{\partial P} \alpha_U
:=\int_{\partial P} \tilde\alpha_U\wedge\delta_{{\rm Trop}\,(U)}
\]
where the righthand side is defined in Remark \ref{preform properties}.
We extend the boundary integral by $0$ to the $\alpha_U$ of other bidegrees.
\end{art}

\begin{art} \label{setup for functoriality}
In the next result, we look at functoriality of the above 
integrals with respect to a morphism  $f:X'\to X$ of algebraic 
varieties over $K$. 
Let $(V,\varphi_U)$ be a tropical chart of $X$. 
Let $U'$ be a very affine open subset of $X'$ with 
$f(U') \subseteq U$. 
Recall that there is a canonical integral $\Gamma$-affine morphism 
$F:N_{U',\R}\to N_{U,\R}$ such that ${\rm trop}_U=F\circ {\rm trop}_{U'}$. 
We deduce easily that 
$(V':=({f^{\rm an}})^{-1}(V)\cap (U')^{\rm an},\varphi_{U'})$ 
is a tropical chart of $X'$ which is compatible with the tropical chart $(V,\varphi_U)$. 
Let $P$ be an integral $\R$-affine polyhedral subset of $\Omega :=\trop_U(V)$ and let $Q:=F^{-1}(P)\cap {\rm Trop}\,(U')$.
We consider $\alpha_U\in P(V,\varphi_U)$ and its pull-back
$f^*(\alpha_U)\in P(V',\varphi_{U'})$ (see Remark \ref{properties of delta-preforms on charts}).
In the following, 
we will use the degree of a morphism as introduced in \ref{sturmfelstevelev}. 
\end{art}

\begin{prop} \label{general independence of the chart}
Under the hypothesis of \ref{setup for functoriality}  {and with $n:=\dim(X)$}, we assume additionally that $Q\cap \supp(f^*(\alpha_U)|_{\Trop(U')})$ is compact. Then 
the following properties hold:
\begin{itemize}
\item[(i)] The set $P\cap {\rm supp}\,(\alpha_U|_{{\rm Trop}\,(U)})$ is compact.
\item[(ii)] If $\alpha_U$  is of bidegree $(n,n)$, then we have 
\begin{equation}\label{general independence of the chart g1}
\deg(f)\cdot \int_{P} \alpha_U=\int_{
Q} f^*(\alpha_{U}).
\end{equation}
\item[(iii)] If $\alpha_U$  is of bidegree 
$(n-1,n)$ or $(n,n-1)$, then we have 
\begin{equation}\label{general independence of the chart g2}
\deg(f) \cdot \int_{\partial P} \alpha_U=\int_{\partial
Q} f^*(\alpha_{U}).
\end{equation}
\end{itemize}
\end{prop}

\begin{proof}
We choose an open subset $\widetilde\Omega$ of $N_{U,\R}$ with $\Omega = \widetilde\Omega \cap \Trop(U)$. 
We write $V'={\rm trop}_U^{-1}(\widetilde\Omega')$ for some open subset
$\widetilde\Omega'$ of $N_\R'$.
Replacing $\widetilde\Omega'$ by $\widetilde\Omega'\cap F^{-1}(\widetilde\Omega)$,
we may assume that $\widetilde\Omega'$ is contained in $F^{-1}(\widetilde\Omega)$.
We write $\Omega'=\widetilde\Omega'\cap \Trop(U')$.
If $\alpha_U\in P(V,\varphi_U)$ is represented by some
element $\tilde\alpha_U\in P(\widetilde\Omega)$, then
$f^*(\alpha_U)$ 
is represented by the element $F^*(\tilde\alpha_U)$ in $P(\widetilde\Omega')$.
We obtain from (\ref{cor-sturmfels-tevelev}) and 
(\ref{projection formula for preformsg1}) that 
$P\cap {\rm supp}\,(\alpha_U|_{{\rm Trop}\,(U)})$ is compact. This proves (i). 

If $\alpha_U\in P^{n,n}(V,\varphi_U)$, then we obtain
\begin{equation}\label{general independence of the chart g3}
{\rm deg}\,(f)\,\int_{P}
\tilde\alpha_U\wedge\delta_{{\rm Trop}\,(U)}
=\int_{F^{-1}(P)}
F^*\tilde\alpha_U\wedge\delta_{{\rm Trop}\,(U')}\\
\end{equation}
if we combine (\ref{cor-sturmfels-tevelev}) with the projection
formula (\ref{projection formula for preformsg1}).
By definition, (\ref{general independence of the chart g1}) is a direct 
consequence of (\ref{general independence of the chart g3}).
Equation (\ref{general independence of the chart g2})
is derived in the same way from (\ref{cor-sturmfels-tevelev}) 
and (\ref{projection formula for preformsg2})
\end{proof}

Let $W$ denote an open subset of $X^{\rm an}$.
Note that a generalized $\delta$-form on $W$ is locally given by elements of 
$P(V,\varphi_U)$ for tropical charts $(V,\varphi_U)$. 
The following corollary will be crucial for the definition of the integral 
of generalized $\delta$-forms.

\begin{cor} \label{independence of the chart}
We consider very affine open subsets $U' \subseteq U$ in $X$.  
Let $\alpha = \trop_U^*(\alpha_U)$ for some
$\alpha_U\in P(U^{\rm an},\varphi_U)$.  
Then there is a unique $\alpha_{U'} \in P((U')^\an,\varphi_{U'})$ with 
$\alpha|_{(U')^{\rm an}}=\trop_{U'}^*(\alpha_{U'})$.  
If $\alpha$ is of bidegree $(n,n)$ and has compact support in $(U')^{\rm an}$, then we have 
\[
\int_{|\Trop(U)|} \alpha_U = \int_{|\Trop(U')|} \alpha_{U'}.
\]
\end{cor}

\begin{proof}
Let $F:N_{U',\R} \rightarrow N_{U,\R}$ 
be the canonical affine map with $\trop_U=F \circ  \trop_{U'}$ on 
$(U')^{\rm an}$. Then $\alpha|_{(U')^{\rm an}}$ is given by 
$\alpha_{U'}:=F^*(\alpha_U) \in P((U')^{\rm an},\varphi_{U'})$. 
This proves existence and uniqueness follows from \ref{phitrop}. 
To prove the last claim, we use (\ref{general independence of the chart g1}) 
for $f={\rm id}$, $P=|{\rm Trop}\,(U)|$ and $Q=|{\rm Trop}\,(U')|$.
\end{proof}

In the following result, we need the local invariant $d(x)$ for 
$x \in \Xan$ \cite[4.2]{gubler-forms}. This invariant was introduced in \cite[Chapter 9]{berkovich-book} and was extensively studied in \cite{ducros}. We note that $d(x) \leq m$ if $x$ belongs to a Zariski closed subset of dimension $m$ \cite[Proposition 9.1.3]{berkovich-book}.

\begin{lem} \label{support and degree of delta-forms}
Let $W$ be an open subset of $\Xan$ and let $\alpha \in P^{p,q}(W)$. 
If $x \in W$ satisfies $d(x) < \max(p,q)$, then $x \not \in \supp(\alpha)$.
\end{lem}

\begin{proof} 
The proof relies on a result of Ducros \cite[Th\'eor\`eme 3.4]{ducros} which says roughly
that in a sufficiently small analytic neighbourhood of $x$, 
the dimension of the  tropical variety is bounded by $d(x)$. 
The details are as follows. We choose a tropical chart $(V,\varphi_U)$ around 
$x$ such that $\alpha$ is induced by a $\delta$-preform  
$\sum_{i \in I} \alpha_i \wedge \delta_{C_i}$ on $N_{U,\R}$. 
By linearity, we may assume that $\alpha$ is induced by 
$\alpha_1\wedge \delta_{C_1}$ for a superform $\alpha_1$ in 
$ A^{p',q'}(N_{U,\R})$ and a tropical cycle $C_1$ of codimension 
$c:=p-p'=q-q' \geq 0$ in $N_{U,\R}$. By definition of a tropical chart, 
there is an open subset $\widetilde{\Omega}$ of $N_{U,\R}$ such that 
$V=\trop_U^{-1}(\widetilde{\Omega})$. 
By the mentioned result of Ducros (see also \cite[Proposition 4.14]{gubler-forms}), there is a compact neighbourhood $V_x$ of $x$ in $V$ 
such that $\trop_U(V_x)$ is a polyhedral subset of $N_{U,\R}$ with 
\begin{equation} \label{tropical dimension bound}
\dim(\trop_U(V_x))\leq  d(x) < \max(p,q).
\end{equation}

We will show that $\alpha|_{V_x}=0$. Let   $f:X' \rightarrow X$ be a morphism of algebraic varieties over $K$ and  $(V',\varphi_{U'})$ a tropical chart of $X'$ 
with $f^{\rm an}(V') \subseteq V_x$. By definition, we have $V'=\trop_{U'}^{-1}(\Omega')$ for the open subset $\Omega':=\trop_{U'}(V')$ 
of $\Trop(U')$. 
In this situation, we get a commutative diagram 
\begin{equation} \label{diagram of compatible charts}
\xymatrix{
V'\ar[d]^{f^{\rm an}|_{V'}}\ar@{^{(}->}[r]& 
(U')^{\rm an}\ar[d]^{f^{\rm an}|_{(U')^{\rm an}}}\ar[r]^{\varphi_{U'}}&
T_{U'}\ar[d]\ar[r]&
N_{U',\R}\ar[d]^{F}\\
V\ar@{^{(}->}[r]&  
U^{\rm an}\ar[r]^{\varphi_{U}}&
T_U\ar[r]&
N_{U,\R}}
\end{equation}
as before. 
To prove the claim, it is enough to show that 
$$f^*(\alpha)|_{\Omega'}:=F^*(\alpha_1) \wedge \delta_{F^*(C_1)}|_{\Omega'}=0,$$
or equivalently
\begin{equation} \label{goal to show}
F^*(\alpha_1) \wedge \delta_{C'} = 0 \in D(\Omega')
\end{equation}
for the tropical cycle  $C':=F^*(C_1)\cdot \Trop(U')$  of codimension $c$ 
in $\Trop(U')$.
We note that 
$\Omega'=\trop_{U'}(V')\subseteq F^{-1}(\trop_U(V_x))$.
Let $\Delta'$ be a maximal  polyhedron from $C'$ with $\Delta' \cap \Omega' \neq \emptyset$. 
Then $F(\Delta' \cap \Omega') \subseteq \trop_U(V_x)$ and hence 
\begin{equation} \label{consequence of diagram}
F^*(\alpha_1)|_{\Delta' \cap \Omega'}
= (F|_{\Delta'\cap \Omega'})^*(\alpha_1|_{F(\Delta') \cap \trop_U(V_x)}).
\end{equation}
We will show below that 
\begin{equation} \label{inequality for codimension}
\codim(F(\Delta' \cap \Omega'),\trop_U(V_x)) \geq c.
\end{equation}
Then \eqref{goal to show} follows from \eqref{consequence of diagram} by using \eqref{inequality for codimension} and \eqref{tropical dimension bound}. This proves $x \not \in \supp(\alpha)$. 

It remains to prove \eqref{inequality for codimension}. By definition of the stable tropical intersection product in \ref{intpairing}(ii), 
there are maximal polyhedra $\Delta_0'$ and $\Delta_1'$ of $\Trop(U')$ and $F^*(C_1)$, respectively, such that $\Delta'=\Delta_0' \cap \Delta_1'$. Moreover, 
$N_{\Delta_0',\R}$ and $N_{\Delta_1',\R}$ intersect transversely in $N_{U',\R}$ which means that 
\begin{equation} \label{transverse1}
 N_{\Delta_0',\R}+N_{\Delta_1',\R}=N_{U',\R}.
\end{equation}

Similarly, the definition of pull-back of tropical cycles in \eqref{intpairing}
(v) shows that there is a maximal polyhedron $\Delta_1$ of $C_1$ with $F(\Delta_1') \subseteq \Delta_1$ and such that 
\begin{equation} \label{transverse2}
 N_{\Delta_1,\R}+\L_F(N_{U',\R})=N_{U,\R}.
\end{equation}
It follows from \eqref{transverse1} and \eqref{transverse2} that $\L_F(N_{\Delta_0',\R})$ intersects $N_{\Delta_1,\R}$ transversely in $N_{U,\R}$. Since the codimension is decreasing under a surjective linear map, we get easily
\begin{equation*} 
\L_F(N_{\Delta',\R}) = \L_F(N_{\Delta_0',\R} \cap N_{\Delta_1',\R}) = \L_F(N_{\Delta_0',\R}) \cap N_{\Delta_1,\R}
\end{equation*}
and hence 
$$\codim(F(\Delta' \cap \Omega'),F(\Delta_0'\cap \Omega')) = \codim(\L_F(N_{\Delta',\R}), \L_F(N_{\Delta_0',\R}))=c$$
by transversality. Using $F(\Omega') \subseteq \trop_U(V_x)$, this proves \eqref{inequality for codimension}.  
\end{proof}

\begin{cor} \label{support corollary}
Let $W$ be an open subset of $\Xan$ and let $U$ be an open subset 
of $X$. If $\alpha \in P^{p,q}(W)$ with $\dim(X \setminus U) < \max(p,q)$, 
then $\supp(\alpha) \subseteq W \cap \Uan$.
\end{cor}

\proof If $x \in W \setminus \Uan$, then the assumptions yield 
$d(x) \leq \dim(X \setminus U) < \max(p,q)$ and hence the claim follows from 
Lemma \ref{support and degree of delta-forms}. \qed

\begin{prop} \label{single chart}
Let $\alpha \in P^{p,q}(\Xan)$ with compact support in the open subset $W$ of $\Xan$ . 
\begin{itemize}
\item[(a)] 
There is a non-empty tropical chart $(V,\varphi_U)$ with $\supp(\alpha) \cap \Uan \subseteq V \subseteq \Uan \cap W$ and  
$\alpha_U \in {P}^{p,q}(U^{\rm an},\varphi_U)$ such that 
$\alpha=\trop_U^*(\alpha_U)$ on $\Uan$.
\item[(b)] Given $U$, the element $\alpha_U$ in {\rm(a)} is unique.
\item[(c)] If $\alpha$ is a $\delta$-form, then we may choose $\alpha_U \in AZ^{p,q}(U^{\rm an},\varphi_U)$.
\item[(d)] If $\max(p,q)=\dim(X)$, then any non-empty very affine open subset $U$ of $X$ with $\alpha|_{\Uan}=\trop_U^*(\alpha_U)$ for some $\alpha_U \in {P}^{p,q}(U^{\rm an},\varphi_U)$ satisfies automatically $\supp(\alpha) \subseteq \Uan$. Moreover, $\alpha_U$ has always compact support in $\Trop(U)$.
\end{itemize}
Explicitly, if $\supp (\alpha)$ is covered by non-empty tropical charts $(V_i,\varphi_{U_i})_{i=1,\ldots,s}$ in $W$  and if
$\alpha$ is given on $V_i$ by $\alpha_i \in {P}^{p,q}(V_i,\varphi_{U_i})$, 
then  any non-empty very affine open subset $U$ of $U_1 \cap \ldots \cap U_s$ and $V=(V_1 \cup \ldots \cup V_s) \cap \Uan$  fit in (a).
\end{prop}

\begin{proof}
Since the support of $\alpha$ is a compact subset of $W$, it is covered by tropical charts $(V_i,\varphi_{U_i})_{i=1,\ldots,s}$ describing $\alpha$ as above. Compactness again shows that for any $i=1,\dots,s$, there is a relatively  compact open subset $\Omega_i'$ of $\Omega_i$ with corresponding open subset $V_i':=\trop_{U_i}^{-1}(\Omega_i')$ of $V_i$ such that $\supp(\alpha) \subseteq V_1' \cup \ldots \cup V_s'$. 
Let us consider a non-empty very 
affine open subset  $U$ of $U_1 \cap \dots \cap U_s$  of $X$ 
and the open subsets 
$$
V':= \Uan \cap \bigcup_{i=1}^s V_i' \subseteq V:= \Uan \cap \bigcup_{i=1}^s V_i
$$ 
of $W \cap \Uan$. We have to show that $V$ and $U$ satisfy (a). 
Let $F_i:N_{U,\R} \to N_{U_i,\R}$ be the canonical integral $\Gamma$-affine map induced by the inclusion $U \subseteq U_i$ (see \ref{sturmfelstevelev}). Then the open subsets 
$$
\Omega':= \Trop(U) \cap \bigcup_{i=1}^s F_i^{-1}(\Omega_i') \subseteq \Omega:= \Trop(U) \cap \bigcup_{i=1}^s F_i^{-1}(\Omega_i) 
$$
of $\Trop(U)$ satisfy $V = \trop_U^{-1}(\Omega)$ and $V'=\trop_U^{-1}(\Omega')$ which means that $(V',\varphi_U)$ and $(V,\varphi_U)$ are compatible tropical charts of $X$ contained in $W$.  
Note that the tropical chart $(V_i \cap \Uan, \varphi_U)$ is compatible with
$(V_i,\varphi_{U_i})$ and hence $\alpha$ is given on 
 {$(V_i\cap \Uan,\varphi_U)$}
by $\alpha_i'=\alpha_i|_{V_i \cap \Uan} \in {P}(V_i \cap \Uan,\varphi_U)$.
Using that $\Omega_i'$ is relatively compact in $\Omega_i$, we deduce that the closure $S$ of $\Omega'$ in $\Trop(U)$ is contained in $\Omega$. We set $V'':=\trop_U^{-1}(\Trop(U) \setminus S)$ leading to the tropical chart $(V'',\varphi_U)$. Since $\alpha$ has compact support in $W$, we may view $\alpha$ as an element of $P^{p,q}(\Xan)$. By construction, we have $\supp(\alpha) \cap \Uan \subseteq V'$. Using that $V'$ and $V''$ are disjoint,  we deduce that   $\alpha$ is given on the tropical chart $(V'',\varphi_U)$ by $0 \in {P}^{p,q}(V'',\varphi_{U})$. We note that the tropical charts $(V_i \cap \Uan,\varphi_U)_{i=1,\dots,s}$ and $(V',\varphi_U)$ cover $\Uan$ and hence we may apply the glueing from 
Proposition \ref{finite sheaf property} to get the desired $\alpha_U \in P^{p,q}(\Uan,\varphi_U)$ from (a).

Uniqueness in (b) follows from Proposition \ref{phitrop}. If $\alpha \in B_c^{p,q}(\Xan)$, then we may choose $\alpha_i \in {AZ}^{p,q}(V_i,\varphi_{U_i})$ and hence we get  (c).

If $\max(p,q)=\dim(X)$, then (d) follows from 
 Corollary \ref{support corollary} and 
Proposition \ref{delta-support on chart}. 
\end{proof}

\begin{definition} \label{chart of integration}
Let $W$ be an open subset of $\Xan$ and let 
$\alpha \in {P}_c^{n,n}(W)$, where $n:=\dim(X)$. 
We may view $\alpha$ as a generalized $\delta$-form on 
$\Xan$ with compact support contained in $W$. 
A non-empty very affine open subset $U$ of $X$ is called a 
{\it very affine chart of integration for $\alpha$} if 
$\alpha|_\Uan=\trop_U^*(\alpha_U)$ for some 
$\alpha_U \in {P}^{n,n}(U^{\rm an},\varphi_U)$. 
By Proposition \ref{single chart}, a chart of integration exists,  
$\alpha_U$ is unique and has compact support in $\Trop(U)$. 
We define {\it the integral of $\alpha$ over $W$} by
\[
\int_W \alpha := \int_{|\Trop(U)|} \alpha_U
\]
where the righthand side is defined in \ref{integration of P-forms}. 
As usual, we extend the integral by $0$ to generalized $\delta$-forms of other bidegrees.
\end{definition}

\begin{prop} \label{integration well-defined}
Let $W$ be an open subset of $\Xan$ and $\alpha \in {P}_c^{n,n}(W)$  {as above.}
 
(i) If $\supp(\alpha)$ is covered by finitely many non-empty tropical charts $(V_i,\varphi_{U_i})$ such that $\alpha$ is given on any $V_i$ by $\alpha_i \in {P}^{n,n}(V_i,\varphi_{U_i})$, then $U:= \bigcap_{i} U_i$ is a very affine chart of integration for $\alpha$.  

(ii) The definition of the integral $\int_W \alpha$ given in 
\ref{chart of integration} does not depend on the choice 
of the very affine chart of integration for $\alpha$. 

(iii) The integral defines a linear map
$\int_W\colon P_c^{n,n}(W)\to \R$.

(iv) If $f:X'\to X$ is a proper morphism of degree 
${\rm deg}\,(f)$ then the projection formula
\begin{equation}\label{integration well-definedg1}
{\rm deg}\,(f)\cdot \int_W\alpha=\int_{(f^{\rm an})^{-1}(W)}f^*\alpha
\end{equation}
holds for all $\alpha\in P_c^{n,n}(W)$.
\end{prop}

\begin{proof}
The explicit description of $U$ in Proposition \ref{single chart} proves (i). 
We show (ii).
Let $U$ be a very affine chart of integration for $\alpha$. Then every non-empty 
very affine open subset $U'$ of $U$ is a very affine chart of integration and it is 
enough to show that $U'$ leads to the same integral. By uniqueness in Proposition \ref{single chart},
the pull-back of $\alpha_U$ with respect to the canonical affine map $F: N_{U',\R}\rightarrow N_{U,\R}$
is equal to $\alpha_{U'}$ and the claim follows from Corollary \ref{independence of the chart}.

Claim (iii) is a direct consequence of our definitions.
To prove (iv), we may assume that $\dim(X')=\dim(X)=n$. 
We choose a very affine chart of integration $U$ for $\alpha$ 
and a non-empty very affine open subset $U'$ of $X'$ with 
$f(U') \subseteq U$. 
Note that $f^*(\alpha)$ is given on $(U')^{\rm an}$ by 
$f^*(\alpha_U) \in P^{n,n}(U',\varphi_{U'})$ constructed in 
Remark \ref{properties of delta-preforms on charts}. 
Since $f^{\rm an}$ is proper as well, the support of 
$f^*(\alpha)$ is compact. 
We conclude that $U'$ is a very affine chart of integration 
for $f^*(\alpha)$ and
$$
\int_{(f^{\rm an})^{-1}(W)}f^*\alpha = \int_{|\Trop(U')|} f^*(\alpha_U).
$$
The projection formula in (iv) is now a direct consequence
of \eqref{general independence of the chart g1}. 
\end{proof}

In our setting, we have the following version of the 
{\it theorem of Stokes}.

\begin{thm}\label{stokes-for-var}
For $\alpha\in B_c^{2n-1}(X)$
we have
\[
\int_\Xan d'\alpha= \int_\Xan d''\alpha = 0.
\]
\end{thm}

\begin{proof}
By Proposition \ref{single chart}, there is a non-empty very affine open subset $U$ of $X$ such that $\supp(\alpha) \subseteq\Uan$ and  
$\alpha_U \in {AZ}_c^{2n-1}(U^{\rm an},\varphi_U)$ such that 
$\alpha|_\Uan=\trop_U^*(\alpha_U)$. Then $U$ is a chart of integration for $d'\alpha$ and $d''\alpha$ using $d'\alpha_U$ and $d''\alpha_U$ on the tropical side for integration. 
The claim follows from Stokes' formula for $\delta$-preforms on $\Trop(U)$ (see Proposition \ref{stokesfordeltapreforms}) using that boundary integrals 
$\int_{\partial |\Trop(U)|}$ vanish as $\Trop(U)$
satisfies the balancing condition. 
\end{proof}

\section{Delta-currents}\label{delta-currents}

In this section, we  define $\delta$-currents on an open subset $W$ of 
$\Xan$ for an $n$-dimensional algebraic variety $X$ over $K$.  
We  proceed similarly as in the case of manifolds in differential geometry  
endowing some specific subspaces of the 
space $B_c(W)$ of $\delta$-forms with compact support in $W$ with the 
structure of a locally convex topological vector space.
{Then} we  define a $\delta$-current as a linear functional on $B_c(W)$ 
with continuous restrictions to all these subspaces.

\begin{art}\label{topology on AZ}
Let $(V,\varphi_U)$ be a tropical chart of $X$ with 
$V \subseteq W$ and let $\Omega := \trop_U(V)$ be as usual.
We recall from \ref{AZ preforms} that an element 
$\beta \in AZ(V,\varphi_U)$ has the form
\begin{equation} \label{AZ-decomposition2}
\beta = \sum_{j \in J} \alpha_j \wedge \omega_j \in P(V,\varphi_U)
\end{equation}
for a finite set $J$, $\alpha_j \in A(\Omega)$ and 
$\omega_j \in Z(V,\varphi_U)$. 

Now we {\it fix} the family ${\omega}_J:=(\omega_j)_{j \in J}$ and 
define $AZ(V,\varphi_U,{\omega}_J)$ as the subspace 
of $AZ(V,\varphi_U)$ given by all elements $\beta$ with a 
decomposition \eqref{AZ-decomposition2} for suitable 
$\alpha_j \in A(\Omega)$. For every $s \in \N$ and 
every compact subset $C$ of $\Omega$, we have the usual 
seminorms $p_{C,s}$ on $A(\Omega)$ measuring uniform convergence 
on $C$ of the derivatives of 
the coefficients of the superforms up to order $s$ 
(see for example \cite[(17.3.1)]{dieudonne-III}). 
We get seminorms $p_{C,s,{\omega}_{J}}$ on 
$AZ(V,\varphi_U,{\omega}_J)$ by defining
\[
p_{C,s,{\omega}_J}(\beta) 
:= \inf \Bigl\{ \max_{j \in J} p_{C,s}(\alpha_j) \,\Big|\, \beta 
= \sum_{j \in J} \alpha_j \wedge \omega_j , \, \alpha_j \in A(\Omega)
\Bigr\}.
\]
Letting $s \in \N$ and the compact subset $C$ of $W$ vary, we get a 
structure of a locally convex topological vector space 
on $AZ(V,\varphi_U,{\omega}_J)$.  
\end{art}

\begin{art} \label{topology on subspaces of delta-forms}
A $\delta$-form $\beta$ on $W$ is given by a covering 
$(V_i,\varphi_{U_i})_{i \in I}$ of $W$ by tropical charts  and by 
$\beta_i \in AZ(V_i,\varphi_{U_i})$ such 
that $\beta|_{V_i}=\trop_{U_i}^*(\beta_i)$ for every $i \in I$. 
Using \ref{topology on AZ}, we have a finite tuple $\omega_{J_i}$ 
of elements in $AZ(V_i,\varphi_{U_i})$ such that 
$\beta_i \in AZ(V_i,\varphi_{U_i},\omega_{J_i})$ for every $i \in I$. 
Now  we fix the covering by tropical charts
and all ${\omega_{J_i}}$ and we define 
$B(W;V_i,\varphi_{U_i},{\omega_{J_i}}:i \in I)$
to be the subspace of $B(W)$ given by the elements $\beta$ such 
that $\beta|_{V_i}=\trop_{U_i}^*(\beta_i)$ for some 
$\beta_i \in AZ(V_i,\varphi_{U_i},{\omega_{J_i}})$ 
and for every $i \in I$.
We endow $B(W;V_i,\varphi_{U_i},\omega_{J_i}:i \in I)$ with the 
coarsest structure of a locally convex topological vector space such that the 
canonical linear maps 
\[
B(W;V_i,\varphi_{U_i},\omega_{J_i}:i \in I) \rightarrow 
AZ(V_i,\varphi_{U_i},{\omega_{J_i}})
\]
are continuous for every $i \in I$. 
An element $\beta \in  B(W;V_i,\varphi_{U_i},{\omega_{J_i}}:i \in I)$  
given as above is mapped to $\beta_i$ which is well-defined by Proposition \ref{phitrop}. 

For a compact subset $C$ of $W$, we consider the subspace 
$B_C(W;V_i,\varphi_{U_i},{\omega_{J_i}}:i \in I)$ of 
$B(W;V_i,\varphi_{U_i},{\omega_{J_i}}:i \in I)$ given by the 
$\delta$-forms with compact support in $C$. 
We endow it with the induced structure of a 
locally convex topological vector space.
\end{art}

\begin{definition}\label{def-delta-current}
A {\it $\delta$-current} on $W$ is a real linear functional 
$T$ on $B_c(W)$ such that the restriction of $T$ to 
$B_C(W;V_i,\varphi_{U_i},{\omega_{J_i}}:i \in I)$
is continuous for every compact subset $C$ of $W$, 
for every covering $(V_i,\varphi_{U_i})_{i \in I}$ of $W$ by tropical charts 
and for every finite tuple ${\omega_{J_i}}$ of elements 
in $Z(V_i,\varphi_{U_i})$. We denote the space of $\delta$-currents
on $W$  by $E(W)$.
A $\delta$-current is called {\it symmetric (resp. anti-symmetric)}
if it vanishes on the subspace of anti-symmetric (resp. symmetric)
$\delta$-forms in $B_c(W)$.
\end{definition}

\begin{art}\label{sheaf-property}
Let $W$ be an open subset of $X^\an$.
Using that $B_c(W)= \oplus_{p,q} B_c^{p,q}(W)$ is bigraded,
we get $E(W)=  \oplus_{r,s} E_{r,s}(W)$ as a bigraded $\R$-vector space, 
where a $\delta$-current in $E_{r,s}(W)$ acts trivially on every $B_c^{p,q}(W)$ 
with $(p,q) \neq (r,s)$. We set $E^{p,q}(W):=E_{n-p,n-q}(W)$.
The definition of $\delta$-currents in \ref{def-delta-current} is local 
and hence $E_{\cdot,\cdot}$ is a sheaf of bigraded real vector spaces on $\Xan$. 
This follows from standard arguments using partition of unity 
if $W$ is paracompact, and follows in 
general from the fact that every compact subset $C$ of $W$ has a paracompact open neighbourhood in $W$
by \cite[Lemme (2.1.6)]{chambert-loir-ducros}. 
The argument is similar as in \cite[Lemme 4.2.5]{chambert-loir-ducros} and we leave the details to the reader.

There is a product 
\begin{equation}
B^{p,q}(W)\times E^{p',q'}(W)\longrightarrow E^{p+p',q+q'}(W),\,\,
(\alpha,T)\longmapsto \alpha\wedge T 
\end{equation}
such that
\[
\langle \alpha\wedge T,\beta\rangle=(-1)^{(p+q)(p'+q')}T(\alpha\wedge\beta)
\]
for each $\beta\in B_c^{n-p-p',n-q-q'}(W)$.
\end{art}

\begin{prop} \label{currents and Zariski dense}
Let $U$ be a  Zariski open subset of $X$ and let $W$ be an open subset of $\Xan$.  
If $\codim(X\setminus U,X) > \min(p,q)$, then 
 $E^{p,q}(W \cap \Uan)=E^{p,q}(W)$. 
\end{prop}

\begin{proof} 
Corollary \ref{support corollary}  shows 
that every $\delta$-form on $W$ of bidegree $(n-p,n-q)$ has support in $W \cap \Uan$. 
We conclude that every $\delta$-current $T$ in $E^{p,q}(W \cap \Uan)$ is  a linear functional  on $B_c^{n-p,n-q}(W)$. 
It remains to prove that the restriction of $T$ to $B_C^{n-p,n-q}(W;V_i,\varphi_{U_i},{\omega_{J_i}}:i \in I)$ is continuous for every  
compact subset $C$ of $W$, for every covering $(V_i,\varphi_{U_i})_{i \in I}$ of $W$ by tropical charts 
and for every finite tuple ${\omega_{J_i}}$ of elements 
in $Z(V_i,\varphi_{U_i})$.  

We consider the set $S$ of $x \in W$ for which there is $i \in I$ and a compact neighbourhood $V_x$ of $x$ in $W \cap V_i$ with
\begin{equation} \label{S1}
\dim(\trop_{U_i}(V_x)) < \max(n-p,n-q).
\end{equation}
Note that $\trop_{U_i}(V_x)$ is a polyhedral subset of $N_{U,\R}$ by  \cite[Th\'eor\`eme 3.2]{ducros}. Obviously, $S$ is an open subset of $W$. 
It follows from the proof of Lemma \ref{support and degree of delta-forms} that $S$ is disjoint from the support of any {$\delta$-form in $B^{n-p,n-q}(W;V_i,\varphi_{U_i},{\omega_{J_i}}:i \in I)$}. We conclude that 
\begin{equation} \label{S2}
B_C^{n-p,n-q}(W;V_i,\varphi_{U_i},{\omega_{J_i}}:i \in I) = B_D^{n-p,n-q}(W;V_i,\varphi_{U_i},{\omega_{J_i}}:i \in I)
\end{equation}
for the compact subset $D := C \setminus S$ of $C$. 
By the proof of Lemma \ref{support and degree of delta-forms} again, every $x \in \Xan \setminus \Uan$  satisfies 
$$d(x) \leq  \dim(X \setminus U) < \max(n-p,n-q)$$ 
and  has a compact neighbourhood $V_x$ contained in some $V_i$ with \eqref{S1}. This proves $D \subseteq W \cap \Uan$. 
Using  \eqref{S2} and  
$T \in E^{p,q}(W \cap \Uan)$, we get
the continuity of the restriction of $T$ to  $B_C^{n-p,n-q}(W;V_i,\varphi_{U_i},{\omega_{J_i}}:i \in I)$. 
\end{proof}

\begin{prop} \label{current associated to generalized delta-form}
A generalized $\delta$-form $\eta\in P^{p,q}(W)$ 
determines a $\delta$-current $[\eta]\in E^{p,q}(W)$
such that
\[
\bigl\langle[\eta],\beta\rangle=\int_W\eta\wedge\beta
\]
for each $\beta\in B_c^{n-p,n-q}(W)$.
\end{prop}

\begin{proof} 
We have to show that the restriction of $[\eta]$ to every subspace 
\[
B_C^{n-p,n-q}(W;V_i,\varphi_{U_i},{\omega_{J_i}}:i \in I)
\] 
as in Definition 
\ref{def-delta-current} is continuous. By passing to a refinement of the covering by tropical charts,
we may assume that $\eta$ is given on $V_i$ by $\eta_i \in P^{p,q}(V_i,\varphi_{U_i})$ for every $i \in I$. 
Since $C$ is compact, there is a finite subset
$I_0$ of $I$ such that $\bigcup_{i \in I_0}V_i$ covers $C$.  
By Proposition \ref{integration well-defined}(i), we may use 
$U:=\bigcap_{i \in I_0} U_i$ as very affine chart of integration for any  
$\gamma \in B_C^{n,n}(W;V_i,\varphi_{U_i},{\omega_{J_i}}:i \in I)$. 

Similarly as in the proof of Proposition \ref{currents and Zariski dense}, 
we consider the set $S$ of $x \in W$ for which there is $i \in I$ and a 
compact neighbourhood $V_x$ of $x$ in $W \cap V_i$ with
\begin{equation} \label{S11}
\dim(\trop_{U_i}(V_x)) < n.
\end{equation}
It follows again from  the proof of Lemma \ref{support and degree of delta-forms} 
that the open subset $S$ of $W$ is disjoint from the support of 
$\eta \wedge \beta \in P^{n,n}(W)$ for any 
$\beta \in B^{n-p,n-q}(W;V_i,\varphi_{U_i},{\omega_{J_i}}:i \in I)$ and that the 
compact set $D := C \setminus S$ is contained in  $W \cap \Uan$. 

By definition, $\beta \in B_C^{n-p,n-q}(W;V_i,\varphi_{U_i},{\omega_{J_i}}:i \in I)$ is given on $V_i$ by 
$\beta_i = \sum_{j \in J_i} \alpha_{ij} \wedge \omega_{ij}$ 
with $\alpha_{ij} \in A(\Omega_i)$ and $\omega_{ij} \in Z(V_i,\varphi_{U_i})$, 
where $\Omega_i:=\trop_{U_i}(V_i)$. 
For $i \in I_0$, let $F_i:N_{U,\R} \rightarrow N_{U_i,\R}$ 
be the canonical affine map  with $\trop_{U_i}=F_i \circ \trop_U$ on $\Uan$ and 
let $\Omega_i':=F_i^{-1}(\Omega_i)\cap \Trop(U)=\trop_U(V_i \cap \Uan)$. 
The definition of $\int_W\eta\wedge\beta$ uses that $\eta \wedge \beta$ is 
given on $\Uan$ by a unique $\gamma_U \in P^{n,n}(U^\an,\varphi_U)$ 
(see Definition \ref{chart of integration}). 
Moreover, {Proposition \ref{single chart}} shows   
that $\gamma_U$ has  compact support in $\bigcup_{i \in I_0} \Omega_i'$ 
and that $\gamma_U$ is characterized by the restrictions 
$$
\gamma_U|_{V_i \cap \Uan}=\sum_{j \in J_i} \eta_i \wedge \alpha_{ij} 
\wedge \omega_{ij}|_{V_i \cap \Uan} \in P^{n,n}(V_i,\varphi_{U_i})
$$
for every $i \in I_0$.  
Recall that $D$ is a compact subset of $W \cap \Uan$ with $\supp(\gamma) \subseteq D$. By Proposition \ref{delta-support on chart}, $\trop_U(D)$ is a compact set of $\Trop(U)$ containing the support 
of $\gamma_U$. Then there is an integral $\R$-affine polyhedral subset $P$ of $\Trop(U)$ with $\trop_U(D) \subseteq P$ and hence we have 
\begin{equation} \label{integration and P}
\langle [\eta], \beta \rangle = \int_\Xan \eta \wedge \beta = \int_{|\Trop(U)|} \gamma_U = \int_{P} \gamma_U.
\end{equation}
We use now that $P$ is independent of the choice of $\beta \in B_C^{n,n}(W;V_i,\varphi_{U_i},{\omega_{J_i}}:i \in I)$. If all the $\alpha_{ij}$ are small with respect to the supremum-norm 
(of the coefficients), then a partition of unity argument on $\Trop(U)$ shows that \eqref{integration and P} 
 is small proving the desired continuity. 
\end{proof}

\begin{rem} \label{remarks to currents and delta-forms}
{The maps $P^{p,q}(W)\rightarrow E^{p,q}(W)$ induce a map of sheaves
$P^{p,q}\to E^{p,q},\,\alpha\mapsto [\alpha]$}
which fits into a commutative diagram
\begin{equation}\label{sheaf-propertyg1}
\xymatrix{
A^{p,q}\ar@{^{(}->}[r]\ar@{^{(}->}[d]^{[\,.\,]_D}&
B^{p,q}\ar@{^{(}->}[r]\ar[d]^{[\,.\,]}&
P^{p,q}.\ar[dl]^{[\,.\,]}\\
D^{p,q}&E^{p,q}\ar[l]&}
\end{equation} 
There is an induced map $P^{p,q}(W)\rightarrow D^{p,q}(W)$.
For $\beta \in P^{p,q}(W)$, we denote the associated current 
in $D^{p,q}(W)$ by $[\beta]_D$.
\end{rem}

There is no a priori reason that the canonical map from $\delta$-forms to 
currents or $\delta$-currents is injective. 
However, we have the following functorial criterion:

\begin{prop} \label{functorial criterion}
Let $W$ be an open subset of $\Xan$ and let $\alpha, \beta \in P^{p,q}(W)$. 
Then $\alpha = \beta$ if and only if 
$[f^*(\alpha)]_D=[f^*(\beta)]_D \in D^{p,q}(W')$ for all morphisms 
$f:X' \rightarrow X$ from algebraic varieties $X'$ over $K$ and for all 
open subsets $W'$ of $(X')^{\rm an}$ with $f(W') \subseteq W$.
\end{prop}

\begin{proof}
If $\alpha = \beta$, then all pull-backs and also their associated currents 
are the same. Conversely, we assume that the associated currents of all 
pull-backs are the same for $\alpha$ and $\beta$. There is an open covering 
$(V_i)_{i \in I}$ of $\Xan$ by tropical charts $(V_i,\varphi_{U_i})$ such 
that $\alpha, \beta$ are given on $V_i$ by 
$\alpha_i,\beta_i \in P^{p,q}(V_i,\varphi_{U_i})$. 
Let $f:X' \rightarrow X$ be a 
morphism of varieties over $K$ and let $(V',\varphi_{U'})$ be a tropical 
chart of $X'$ which is compatible with $(V_i,\varphi_{U_i})$. 
Let $\Omega'$ denote the open subset $\trop_{U'}(V')$ of $\Trop(U')$.
It follows from our definitions that $\alpha_i=\beta_i$ in
$P(V_i,\varphi_{U_i})$ if we show 
$f^*(\alpha_i)|_{\Omega'}=f^*(\beta_i)|_{\Omega'} \in D^{p,q}(\Omega')$ for all morphisms
$f$ and all charts $(V',\varphi_{U'})$ compatible with $(V_i,\varphi_{U_i})$.
By assumption, we have $[f^*(\alpha)]_D=[f^*(\beta)]_D$ in $D^{p,q}(V')$. 
We conclude that 
$f^*(\alpha_i)|_{\Omega'}= f^*(\beta_i)|_{\Omega'} \in P^{p,q}(\Omega') \subseteq D^{p,q}(\Omega')$ 
and get $\alpha_i= \beta_i \in P(V_i,\varphi_{U_i})$ proving the claim. 
\end{proof}

\begin{art} \label{differentiation of delta-currents}
As usual, we define the linear differential operators 
$d':E^{p,q}(W) \rightarrow E^{p+1,q}(W)$ 
and $d'':E^{p,q} \rightarrow E^{p,q+1}(W)$  by
$$\langle d'T , \beta \rangle := (-1)^{p+q+1} \langle T, d' \beta \rangle , \quad 
\langle d''T , \beta \rangle := (-1)^{p+q+1} \langle T, d'' \beta \rangle.$$
Note that $d'$ and $d''$ induce continuous linear maps on the locally convex 
topological vector spaces  introduced in  
\ref{topology on subspaces of delta-forms} 
and hence it is easy to check that $d'$ and $d''$ are well-defined on 
$\delta$-currents. Moreover, the natural maps from \ref{remarks to currents and delta-forms} fit 
into commutative diagrams 
\begin{equation}\label{differentiation of delta-currentsg1}
\xymatrix{
B^{p,q}\ar[r]^{[\,.\,]}\ar@{->}[d]^{d'}&E^{p,q}\ar@{->}[d]^{d'}&&
B^{p,q}\ar[r]^{[\,.\,]}\ar@{->}[d]^{d''}&
E^{p,q}\ar@{->}[d]^{d''}\\
B^{p+1,q}\ar[r]^{[\,.\,]}&E^{p+1,q}&&
B^{p,q+1}\ar[r]^{[\,.\,]}&
E^{p,q+1}
}
\end{equation}
of sheaves. As usual, we define $d:=d' + d''$ 
also on $E$. 
\end{art}

\begin{art} \label{pushforward of delta-currents}
If $f:X' \rightarrow X$ is a proper morphism of algebraic varieties, 
then we get a {\it push-forward} $f_*:E_{r,s}(f^{-1}(W)) \rightarrow E_{r,s}(W)$ as follows: 
For $T' \in E_{r,s}(f^{-1}(W))$, the push-forward is the $\delta$-current on $W$ given by
$$ \langle f_*(T), \beta \rangle := \langle T, f^*(\beta) \rangle $$
for $\beta \in B_c^{r,s}(W)$. 
It is easy to see that pull-back of $\delta$-forms induces continuous linear 
maps between appropriate locally convex topological vector spaces 
defined in \ref{topology on subspaces of delta-forms} and hence the proper 
push-forward of $\delta$-currents is well-defined. 
\end{art}

\begin{ex} \label{current of integration}
In \ref{chart of integration}, we have introduced $\int_\Xan \beta$ 
for $\beta \in P_c^{n,n}(\Xan)$. 
Setting $\langle \delta_{X}, \beta\rangle:=\int_\Xan \beta$, 
we get the $\delta$-current 
$\delta_{X}=[1] \in E^{0,0}(\Xan)$. 
We call it the {\it $\delta$-current of integration} along $X$. 
Using linearity in the components and \ref{pushforward of delta-currents}, 
we get a $\delta$-current of integration $\delta_Z$ for 
every algebraic cycle $Z$ on $X$.
\end{ex}

\begin{prop}
Let $f:X'\to X$ be a proper morphism of algebraic varieties and 
let $Z'$ be a $p$-dimensional algebraic cycle on $X'$. Then we have
the equality $f_*\delta_{Z'}=\delta_{f_*Z'}$ in $E_{p,p}(X^{\rm an})$.
\end{prop}

\proof
This is a direct consequence of the projection formula
\eqref{integration well-definedg1}.
\qed

\begin{prop}\label{integral is continuous}
Let $W$ be an open subset of $X^{\rm an}$. We equip the space
$C_c(W)$ of continuous functions $f:W\to \R$ with compact support
with the
supremum norm $|\phantom{a}|_W$ and its subspace $A^0_c(W)$ of smooth functions
with compact support with the induced norm.
Then for each $\alpha\in P_c^{n,n}(W)$ the map
\[
A^0_c(W)\longrightarrow \R,\,\,
f\mapsto \int_Wf\cdot \alpha
\]
is continuous and extends in a unique way to a continuous
map $C_c(W)\to \R$.
\end{prop}

\begin{proof} 
We may assume that $\alpha$ is of codimension $l$. 
We observe that the Stone-Weierstraß Theorem
\cite[Prop. (3.3.5)]{chambert-loir-ducros} implies that
$A_c^0(W)$ is a dense subspace of $C_c(W)$.
Consider $f\in A_c^0(W)$ and $\alpha\in P^{n,n}_c(W)$.
Our claims are obvious once we have obtained  
a bound $C_\alpha$ such that the inequality
\begin{equation}\label{basic integral estimate}
\Bigl|\int_Wf\cdot \alpha\Bigr|\leq  C_\alpha \cdot |f|_W
\end{equation}
holds. 
{We are going to prove this inequality in four steps.}

\vspace{2mm}\noindent
{\it First step: The definition of the bound $C_\alpha$.}

\vspace{2mm}\noindent
We fix a very affine chart of integration $U$ for $\alpha$ which 
means that there is $\alpha_U \in P_c^{n,n}(U^\an,\varphi_U)$ with $\trop_U^*(\alpha_U)=\alpha$ and we set $N:=N_{U}$. 
Then $\alpha_U$ is represented by a $\delta$-preform $\tilde{\alpha}_U \in P_c^{n,n}(N_\R)$ of the form
\begin{equation} \label{representation of the lift}
\tilde{\alpha}_U = \sum_\sigma \alpha_\sigma \wedge \delta_\sigma
\end{equation}
as a polyhedral supercurrent, where $\sigma$ ranges over $\KC^l$ for a 
complete integral $\R$-affine polyhedral complex $\KC$ of $N$ and where 
$\alpha_\sigma \in A_c^{n-l,n-l}(\sigma)$.
The definition of the bound $C_\alpha$ will depend on the choice of $U$ 
and of the lift $\tilde{\alpha}_U$, but not on the choice of $\KC$. 
The restriction $\alpha_{\sigma \tau}$ of $\alpha_\sigma$ to an 
$n-l$-dimensional face $\tau$ of $\sigma$ is an element of $A_c^{n-l,n-l}(\tau)$. 
As this is a superform of top-degree, we have a well-defined compactly 
supported superform $|\alpha_{\sigma \tau}|$ of degree $(n-l,n-l)$ with 
continuous coefficient on $\tau$.
This single coefficient is independent 
of the choice of an integral base of $\L_\tau$ and it is given by the 
absolute value of the coefficient of $\alpha$. 
After passing to a refinement, we may assume that $\Trop(U)$ is given by 
the tropical cycle $(\Ccal_{\leq n},m)$. Then we define
\begin{equation} \label{definition of the norm}
C_\alpha := \sum_{ {(\Delta, \sigma)}} [N:N_\Delta + N_\sigma] 
m_\Delta \int_\tau |\alpha_{\sigma \tau}|,
\end{equation}
where $(\Delta,\sigma)$ ranges over all elements of $\Ccal_n \times \Ccal^l$ such that {$\L_\Delta + \L_\sigma = N_\R$} and such that $\tau:=\Delta \cap \sigma$ is $n-l$-dimensional. Here, the integral of a superform of top-degree with continuous coefficient is defined as 
 in \cite[(1.2.2),(1.4.1)]{chambert-loir-ducros}.

\vspace{2mm}\noindent{\it Second step: A first estimate for the integral.}

\vspace{2mm}\noindent
By definition of a smooth function, there is a covering of $W$ by tropical charts $(V_j',\varphi_{U_j'})_{j \in J}$ such that $f|_{V_j'}=\trop_{U_j'}^*(\phi_j')$ for smooth functions $\phi_j'$ on 
the open subsets $\Omega_j'=\trop_{U_j'}(V_j')$ of $\Trop(U'_j)$. Any given compact subset $C$ of $W$ containing the support of $\alpha$ will be covered by  $(V_j')_{j \in J_0}$ for a finite subset $J_0$ of $J$. 
By Proposition \ref{integration well-defined}, 
$U':=U \cap  \bigcap_{j \in J_0} U_j' $ is a very affine chart of integration 
for $\alpha$ and for $f \alpha$. 
Let $N':=N_{U'}$ and let $F:N_\R' \rightarrow N_\R$ be the canonical integral $\R$-affine map. Since the restriction map $\Ocal(U)^\times \rightarrow \Ocal(U')^\times$ is injective, it follows that $F$ is surjective. After refining $\KC$, there is  a complete integral $\R$-affine polyhedral complex $\KC'$ on $N'_\R$ such that $\Trop(U')=(\KC_{\leq n}',m')$ and such that $\Delta:=F(\Delta') \in \KC$ for every $\Delta' \in \KC'$. 

Note that $(V':=\bigcup_{j \in J_0} V_j \cap (U')^\an, \varphi_{U'})$ is a tropical chart of $W$ containing $C \cap (U')^\an$ and the support of $\alpha$ by Corollary \ref{support corollary}. The pull-backs of the functions $\phi_j'$ with respect to the canonical affine maps $F_j:N'_\R \rightarrow N_{U'_j,\R}$ glue to a well-defined smooth function $f_{U'}$ on $\Omega':= \trop_{U'}(V')$. By definition, we have
\begin{equation} \label{integration of f alpha}
\int_W f \alpha = \int_{|\Trop(U')|} f_{U'} F^*(\tilde{\alpha}_U)|_{\Trop(U')}.
\end{equation}
Using that $F$ is surjective, we deduce from \eqref{representation of the lift} and \eqref{standard form and pull-back} that 
\begin{equation*} 
F^*(\tilde{\alpha}_U) =  \sum_{\sigma'} [N:\L_F(N')+N_\sigma] \cdot F^*\alpha_\sigma \wedge \delta_{\sigma'},
\end{equation*}
where $\sigma'$ ranges over all elements of $(\KC')^l$ such that $\sigma := F(\sigma')$ is of codimension $l$ in $N$. We choose a generic vector $v' \in N'_\R$. It follows from  \eqref{standard form and wedge} that 
\begin{equation*}  
F^*(\tilde{\alpha}_U)|_{\Trop(U')}= 
\sum_{\tau'} \sum_{ {(\Delta', \sigma')}}
[N':N'_{\Delta'}+N'_{\sigma'}]   [N:\L_F(N')+N_\sigma]  m_{\Delta'} F^*\alpha_\sigma \wedge \delta_{\tau'},         
\end{equation*}
where $\tau'$ ranges over $\KC_{n-l}'$ and $(\Delta',\sigma')$ ranges over all pairs in $\KC_n' \times (\KC')^l$ such that $\tau' = \Delta' \cap \sigma'$ and such that $\Delta' \cap (\sigma' + \ve v') \neq \emptyset$ for all sufficiently small $\ve >0$. 
Additionally, we assume that $\sigma:= F(\sigma')$ is of codimension $l$ in $N_\R$ as above. 
By degree reasons, we may restrict the sum to those 
$\tau'$ with $\tau:=F(\tau')$ of dimension $n-l$. 
Note that this is equivalent to restrict our attention to those 
$\Delta'$ with $\Delta:=F(\Delta')$ of dimension $n$. 
Since $\alpha$ has support in $V'$, the restriction of 
$F^*\alpha_\sigma$ to $\sigma'$ has support in $\Omega' \cap \sigma'$.
By \eqref{integration of f alpha}, we have
\begin{equation*} 
\int_W f\alpha 
= \sum_{\tau'} \sum_{ {(\Delta', \sigma')}}
[N':N'_{\Delta'}+N'_{\sigma'}]   [N:\L_F(N')+N_\sigma]  m_{\Delta'} \int_{\tau'} f_{U'} F^*\alpha_\sigma .
\end{equation*}
We deduce the following bound
\begin{equation} \label{upper bound 1}
\begin{split}
&\left|\int_W f\alpha \right|  \\
& \leq |f|_W
 \sum_{\tau'} \sum_{ {(\Delta', \sigma')}}
[N':N'_{\Delta'}+N'_{\sigma'}]   [N:\L_F(N')+N_\sigma] m_{\Delta'}  \int_{\tau'} | F^*\alpha_{\sigma \tau}|.
\end{split}
\end{equation}
The transformation formula shows
$$ \int_{\tau'} | F^*\alpha_{\sigma \tau}|=[N_\tau:\L_F(N'_{\tau'})] \int_\tau |\alpha_{\sigma \tau}|$$
and hence the sum in \eqref{upper bound 1} is equal to 
\begin{equation} \label{upper bound 2}
 \sum_{\tau'} \sum_{ {(\Delta', \sigma')}}
[N_\tau:\L_F(N'_{\tau'})] [N':N'_{\Delta'}+N'_{\sigma'}]   [N:\L_F(N')+N_\sigma]    m_{\Delta'}  \int_{\tau} |\alpha_{\sigma \tau}|.
\end{equation}

\vspace{2mm}\noindent
{\it Third step: The following basic lattice index identity holds:
\begin{equation} \label{basic lattice index identity}
\begin{split}
&[N_\tau:\L_F(N'_{\tau'})] [N':N'_{\Delta'}+N'_{\sigma'}]   [N:\L_F(N')+N_\sigma]  \\ & = [N : N_\Delta + N_\sigma] [N_\Delta : \L_F(N_{\Delta'})].
\end{split}
\end{equation}}

\vspace{2mm}\noindent
In the basic lattice index identity \eqref{basic lattice index identity}, $(\Delta',\sigma')$ is a pair in  $\KC_n' \times (\KC')^l$ such that $ \Delta' \cap (\sigma'+\ve v') \neq \emptyset$ for $\ve>0$ sufficiently small and such that $\sigma:=F(\sigma')$ is of codimension $l$ in $N$.  
We have also used $\Delta := F(\Delta')$, $\tau':= \Delta' \cap \sigma'$ and $ \tau := F(\tau')$.
Since $F$ is a surjective integral $\R$-affine map, all lattice 
indices in the claim of the third step are finite. 
Setting $P':=N'_{\Delta'}$ and $Q:=N_\sigma$,  
the basic lattice identity \eqref{basic lattice index identity} 
follows from the projection formula for lattices in 
Lemma \ref{projection formula for lattices} below.

\vspace{2mm}\noindent
{\it Fourth step:
The desired inequality \eqref{basic integral estimate} holds.}

\vspace{2mm}\noindent
To prove \eqref{basic integral estimate}, we note that $v:=F(v')$ is a generic vector for $\Ccal$. We have $\tau = \Delta \cap \sigma$ and $\Delta \cap (\sigma + \ve v) \neq \emptyset$.  
The basic lattice index identity \eqref{basic lattice index identity} yields that the sum in \eqref{upper bound 2} is equal to
\begin{equation} \label{upper bound 3}
\sum_{\tau'} \sum_{ {(\Delta', \sigma')}}
 [N : N_\Delta + N_\sigma] [N_\Delta : \L_F(N'_{\Delta'})]
m_{\Delta'}  \int_{\tau} |\alpha_{\sigma \tau}|.
\end{equation}
The Sturmfels--Tevelev multiplicity formula  \eqref{cor-sturmfels-tevelev} gives
$$\sum_{\Delta'} [N_\Delta : \L_F(N'_{\Delta'})] m_{\Delta'} = m_\Delta,$$
where $\Delta'$ ranges over all elements of $\KC'_n$ mapping onto a given $\Delta \in \KC_n$. 
Using this, one can show that \eqref{upper bound 3} is equal to
\begin{equation} \label{upper bound 4}
\sum_{\tau} \sum_{ {(\Delta, \sigma)}}
 [N : N_\Delta + N_\sigma] 
m_{\Delta}  \int_{\tau} |\alpha_{\sigma \tau}|,
\end{equation}
where the sum is over all pairs $(\Delta,\sigma) \in \KC_n \times \KC^l$ such that $\Delta \cap (\sigma + \ve v) \neq \emptyset$
and $\tau=\Delta \cap \sigma$. 
Now \eqref{basic integral estimate} follows from \eqref{definition of the norm}--\eqref{upper bound 4}.
\end{proof}

The basic lattice index identity \eqref{basic lattice index identity} is a special case of  the following {\it projection formula for lattices.} 
Note that it is stronger as the projection formula for tropical cycles in Proposition \ref{tropintthprop}. The latter would not give the required bound in the fourth step above.

\begin{lem} \label{projection formula for lattices}
 Let $F:N' \rightarrow N$ be a  homomorphism of free abelian groups of finite 
rank and let $P' \subseteq N'$, $Q \subseteq N$ be  subgroups. 
We assume that $\rk(F(N'))=\rk(N)=\rk(F(P')+Q)$. Then we have the equality
\begin{equation} \label{projection formula for lattices2}
\begin{split}
&[F(P')_\R \cap Q : F(P'\cap F^{-1}(Q))][N':P'+F^{-1}(Q)][N:F(N')+Q] \\ & = [N : F(P')_\R \cap N+Q]  [F(P')_\R \cap N : F(P')]
\end{split}
\end{equation}
where all involved lattice indices are finite.
\end{lem}

\begin{proof} 
The assumptions show easily that all lattice indices are finite. 
Using $F(P'\cap F^{-1}(Q))=F(P') \cap Q$ and  the isomorphism theorem 
$A/(A \cap B) \cong (A+B)/B$ for abelian groups, we get
$$
(F(P')_\R \cap Q )/ F(P'\cap F^{-1}(Q))\cong (F(P')_\R \cap Q + F(P'))/F(P').
$$
Similarly,  $F(P')_\R \cap Q + F(P')=F(P')_\R \cap (F(P')+Q)$ yields 
$$
(F(P')_\R\cap N)/(F(P')_\R \cap Q + F(P'))\cong (F(P')_\R \cap N + Q)/(F(P')+Q).
$$
Multiplying \eqref{projection formula for lattices2} by $[F(P')_\R \cap N + Q:F(P')+Q]$, the above two isomorphisms show that 
the claim is equivalent to 
\begin{equation} \label{projection formula for lattices3}
[N':P'+F^{-1}(Q)][N:F(N')+Q]  
= [N : F(P')+Q]. 
\end{equation} 
Using $F(P')+Q \cap F(N')=(F(P')+Q) \cap F(N')$, we have
$$N'/(P'+F^{-1}(Q)) \cong F(N')/ (F(P')+Q \cap F(N')) \cong (F(N')+Q) / (F(P')+Q)$$
and hence \eqref{projection formula for lattices3} holds. This proves the claim.
\end{proof}

We recall that on a locally compact Hausdorff space $Y$, 
the Riesz representation theorem gives a bijective correspondence 
between positive (resp. signed) Radon 
measures on $Y$ and positive (resp. bounded) linear functionals on the 
space of continuous real functions with compact support on $Y$ 
endowed with the supremum norm.

\begin{cor}\label{deltaform defines-measure}
Let $W$ be an open subset of $X^{\rm an}$. 
For each $\alpha\in P_c^{n,n}(W)$ there is a unique signed Radon measure
$\mu_\alpha$ on $W$ such that
\begin{equation} \label{measure identity}
\int_W f\cdot\alpha=\int_Wf\,d\mu_\alpha
\end{equation}
for all smooth functions $f$ on $W$ with compact support. 
\end{cor}

\proof
This is a consequence of Proposition \ref{integral is continuous} and Riesz's representation 
theorem.
\qed

\begin{prop}\label{current continuous function}
Let $W$ be an open subset of $X^{\rm an}$ and let $f$ be a continuous function on $W$. 
Then the map 
$$[f]: B_c^{n,n}(W) \rightarrow \R,\quad \alpha \mapsto \int_Wf \,d\mu_\alpha$$
is a $\delta$-current in {$ E^{0,0}(W)$}. 
\end{prop}

\begin{proof}
The integral is well-defined by Corollary \ref{deltaform defines-measure} using 
that $\supp(\alpha)$ is compact. 
Obviously, $[f]$ is a linear map. We have to show that the restriction of $[f]$ to 
any {subspace} $B^{n,n}_C(W;V_i,\varphi_{U_i},{\omega_{J_i}}:i \in I)$ as in 
\ref{topology on subspaces of delta-forms} 
is continuous. For $i \in I$, let $\Omega_i := \trop_{U_i}(V_i)$. 
For every $x \in C$, there is $i(x) \in I$ with $x \in V_{i(x)}$. 
We choose a polytopal neighbourhood $P_{i(x)}$ of $\trop_{U_{i(x)}}(x)$ in 
$N_{U_{i(x),\R}}$ such that 
$P_{i(x)} \cap \Trop(U_{i(x)}) \subseteq \Omega_{i(x)}$ and 
we denote the interior of  $P_{i(x)}$ by $Q_{i(x)}$. 
There is a finite set $Y$ of $X$ such that the open sets 
$\trop_{U_{i(x)}}^{-1}(Q_{i(x)})$, $x \in Y$, cover the compact set $C$. 
By Proposition \ref{integration well-defined}, 
$U:= \cap_{x \in Y} U_{i(x)}$ works as a very affine chart of integration for 
every $\alpha \in B^{n,n}_C(W;V_i,\varphi_{U_i},{\omega_{J_i}}:i \in I)$. Then we 
have $\alpha_U \in AZ_c^{n,n}(U,\varphi_U)$ with $\trop_U^*(\alpha_U)=\alpha$.
By the Sturmfels--Tevelev multiplicity formula \eqref{cor-sturmfels-tevelev} and 
by degree reasons, 
one can show that $\alpha_U$ has support in the compact subset 
\[
C_U=\bigcup_{x \in Y} \bigcup_{\Delta_{i(x)} } \Delta_{i(x)} \cap  F_{i(x)}^{-1}( P_{i(x)})
\]
of $\Trop(U)$, where $\Delta_{i(x)}$ ranges over all $n$-dimensional faces of $\Trop(U)$ such that $ \Delta_{i(x)} \cap  F_{i(x)}^{-1}( P_{i(x)})$ is mapped onto an $n$-dimensional face of $\Trop(U_{i(x)})$ by the canonical affine map $F_{i(x)}:N_{U,\R} \to N_{U_{i(x)},\R}$.  
Using the supremum semi\-norm $|f|_C$ on $C$, we get
\begin{equation}\label{basic integral estimate2}
\Bigl|\int_Wf \,d\mu_\alpha\Bigr|\leq  C_\alpha \cdot |f|_C.
\end{equation} 
To see this, we note first that $\supp(\mu_\alpha) \subseteq C$. 
There is a smooth function $g$ on $W$ with $0 \leq g \leq 1$ with 
$g \equiv 1$ on $C$ and with compact support in a sufficiently small 
neighbourhood of $C$ \cite[Corollaire 3.3.4]{chambert-loir-ducros}. 
Then \eqref{basic integral estimate2} follows from applying 
\eqref{basic integral estimate} to compactly supported smooth approximations 
of $fg$ using the Stone--Weierstra\ss -theorem in 
\cite[Corollaire 3.3.4]{chambert-loir-ducros}.

Now we deduce the claim from  the definition of the bound 
$C_\alpha$ in \eqref{definition of the norm}: We set $i:=i(x)$ 
for $x \in Y$ and we may 
assume that $\alpha|_{V_i}$ is given by 
$$\sum_{j \in J_i} \alpha_{ij} \wedge \omega_{ij}$$
for $\alpha_{ij} \in A(\Omega_i)$ with small supremum seminorm 
$p_{P_i \cap \Trop(U_i), 0}(\alpha_{ij})$ of the coefficients. 
Noting that the $\omega_{ij}$ are fixed, this yields that every  
$\alpha_{\sigma\tau}$ in \eqref{definition of the norm} has small coefficient. 
Using that only the compact subset $C_U \cap \tau$ matters for integration, 
we deduce that $C_\alpha$ is 
small and hence \eqref{basic integral estimate2} shows that $[f]$ is continuous.
\end{proof}

\section{The Poincar\'e--Lelong formula and first Chern delta-currents} \label{PoincareLelongFormula}

The Poincar\'e--Lelong formula in complex analysis is of fundamental 
importance for Arakelov theory. 
Chambert--Loir and Ducros have shown in 
\cite[\S 4.6]{chambert-loir-ducros} that 
the Poincar\'e--Lelong formula holds 
 {as an identity between currents {on Berkovich spaces} while 
Theorem \ref{Poincare-Lelong equation} below enhances {the} 
Poincar\'e--Lelong formula as an equality of $\delta$-currents.}
We use the Poincar\'e--Lelong formula to define the first Chern 
$\delta$-current of a continuously metrized line bundle.

\begin{art} \label{definition of the currents in P-L}
Let $X$ be a variety over $K$ of dimension $n$ and let 
$f \in K(X) \setminus \{0\}$. In Example \ref{current of integration}, we  
have introduced the  $\delta$-current of integration $\delta_{X}$  
leading to the definition of the $\delta$-current $\delta_Z$ for any cycle 
$Z$ on $X$. Using that for the Weil divisor $\cyc(f)$ of $f$, we get a $\delta$-current $\delta_{\cyc(f)}$ on $\Xan$.

On the other hand, the complement $U$ of the support of the principal Cartier 
divisor $\Div(f)$ is an open dense subset of $X$. 
By Proposition \ref{currents and Zariski dense}, we get the $\delta$-current $[\log|f|] \in E^{0,0}(\Uan)=E^{0,0}(\Xan)$. 
\end{art}

\begin{thm} \label{Poincare-Lelong equation}
For a non-zero rational function  $f$ on $X$, the Poincar\'e--Lelong equation
\[
\delta_{\cyc(f)}=d'd''\bigl[
\log|f|\bigr]
\]
holds in $E^{1,1}(X^\an)$.
\end{thm}

\begin{proof} 
The proof is similar as in \cite[\S 4.6]{chambert-loir-ducros}, 
but it is more on the tropical side as we do not have integrals 
of $\delta$-forms 
over analytic subdomains at hand. 
We will first do some reduction steps and then we will introduce some 
notation which allows us to use results from \cite{chambert-loir-ducros}. 
The claim is local on $\Xan$ and so we may 
assume that $X= \Spec(A)$ and $f \in A$. 
The latter induces a morphism $f:X \rightarrow {\mathbb A}^1$. 
We may assume that the morphism is not constant as otherwise all 
terms are $0$. 
Since $A$ is a domain, the property $S_1$ of Serre is satisfied. 

Let us recall some results from \cite{chambert-loir-ducros} before we start
the actual proof. 
Let $W$ be an affinoid subdomain of $\Xan$ and let $g:W \rightarrow \Tan$ 
be an analytic map for $T = {\mathbb G}_m^r$. 
Following \cite{chambert-loir-ducros}, 
we call such a map to a torus an {\it analytic moment map}. 
We obtain a continuous map 
$$g_\trop := \trop \circ g:W \rightarrow \R^r.$$ 
We get an analytic map 
$h:=(f,g):W \rightarrow ({\mathbb A}^1)^{\rm an} \times \Tan$. 
We denote the fibre of $W$ over $t \in ({\mathbb A}^1)^{\rm an}$ by $W_t$ 
with respect to the restriction of $f$ to $W$. 
If $I$ is an interval in $(0,\infty)$, 
then $W_I := |f|^{-1}(I) \cap W$. 
We observe that $W_t$ and $W_I$ carry natural structures of analytic spaces of 
dimension $n-1$ and $n$ respectively.
It follows from general results of Ducros \cite[Th\'eor\`eme 3.2]{ducros} that the sets 
$g_{\trop}(W_t)$ and $h_{\trop}(W_I)$ are integral $\R$-affine polyhedral
sets of dimension less or equal to $n-1$ and $n$ respectively.
These polyhedral sets can be equipped with natural integral weights.
A construction of these so called tropical weights can be found in
\cite[\S 7]{gubler-forms} or in \cite[3.5]{chambert-loir-ducros}
in the language of calibrations.
We observe that the tropical weights take the multiplicities of 
irreducible components into account.
The $k$-skeleton of a polyhedral set $P$ of dimension at most $k$ is by 
definition the union of all $k$-dimensional polyhedra contained in $P$.
By \cite[Proposition 4.6.6]{chambert-loir-ducros}, there exist 
a real number $r >0$ and an 
integral $\R$-affine polyhedral complex $\Ccal$ in $\R^r$ of pure 
dimension $n-1$ with integer weights $m$ such that all polyhedra
in $\KC$ are polytopes
with the following properties:
\begin{itemize}
\item[(a)] 
for every $t$ in the closed ball in $({\mathbb A}^1)^{\rm an}$ with 
center $0$ and radius $r$, the $(n-1)$-skeleton of $g_{\rm trop}(W_t)$ endowed 
with the canonical tropical weights is equal to $(\Ccal,m)$;
\item[(b)] 
for every closed interval $I \subset (0,r]$, the $n$-skeleton of 
$h_{\rm trop}(W_I)$ endowed with the canonical analytic tropical weights is 
equal to $(-\log(I),1) \times (\Ccal,m)$ as a product of weighted polyhedral 
complexes. 
\end{itemize}
In fact, Chambert--Loir and Ducros formulated this crucial result in terms of 
canonical calibrations instead of analytic tropical weights. We refer to 
\cite[\S 7]{gubler-forms} for the definition and translation of these 
equivalent notions. 
The analytic space $W_0$ coincides with the closed analytic subspace of $W$
determined by the effective Cartier divisor ${\rm div}\,(f|_W)$.
Using (a) for $t=0$, we see that $(\KC,m)$ is equal to the $(n-1)$-skeleton of 
$g_{\rm trop}(\Div(f|_W))$ as a weighted polyhedral complex.

After recalling these preliminary results, we start with the proof. 
Since the $\delta$-currents $\delta_{\cyc(f)}$ and $d'd''[\log|f|]$ are 
symmetric, it is enough to check the Poincar\'e--Lelong equation by 
evaluating at a symmetric $\alpha \in B_c^{n-1,n-1}(\Xan)$. 
The $\delta$-form $\alpha$ is given by  tropical charts 
$(V_i,\varphi_{U_i})_{i \in I}$ covering $\Xan$ and symmetric 
$\alpha_i \in AZ^{n-1,n-1}(V_i,\varphi_{U_i})$. 
Since $\alpha$ has compact support, there are finitely many $i$ such that 
the corresponding $V_i$'s cover $\supp(\alpha)$. 
In the following, we restrict our attention to these finitely many $i$'s and 
we number them by $i=1, \dots, m$. 

Let us consider the very affine open subset  
$U:= U_1 \cap \dots \cap U_m \setminus \supp(\Div(f))$ of $X$. 
Let $G_i:N_{U,\R} \rightarrow N_{U_i,\R}$ (resp. $F: N_{U,\R} \rightarrow \R$) 
be the canonical affine map compatible with $\trop_U$ and $\trop_{U_i}$ 
(resp. $-\log|f|$). Let $x_0$ be the coordinate on $\R$ and let 
$H_i:=(F,G_i):N_{U,\R}\to \R \times N_{U_i,\R}$.

For every $x \in \supp(\alpha)$, there is an $i \in \{1, \dots , m\}$ such 
that $x \in V_i$. We choose an integral $\Gamma$-affine polytope $\Delta_i$ 
of maximal dimension in $N_{U_i,\R}$ containing $\trop_{U_i}(x)$ in its interior. 
We may assume that $\Delta_i \cap \Trop(U_i) \subseteq\trop_{U_i}(V_i)$. 
Then $W_i := \trop_{U_i}^{-1}(\Delta_i)$ is an affinoid subdomain  of 
$\Xan$ with $x \in {\rm Int}(W_i)$. 
Renumbering the covering and using again compactness of $\supp(\alpha)$, 
we may assume that $i$ does not depend on $x$ which means that the interiors 
of the affinoid subdomains $W_1, \dots , W_m$ cover $\supp(\alpha)$. 
Note that $W:= \bigcup_{i=1}^m W_i$ is a compact analytic subdomain of $\Xan$.

For every non-empty subset $E$ of $\{1, \dots , m\}$, the set 
$W_E:= \bigcap_{i \in E} W_i$ is affinoid (using that $\Xan$ is separated). 
Note that $U_E:= \bigcap_{i \in E} U_i$ is  very affine and we set 
$V_E:= \bigcap_{i \in E} V_i$. 
We choose $r>0$ sufficiently small such that (a) and (b) above hold for 
every $W_E$ and moment map $g_E:=\varphi_{U_E}$.  
Note that the union of the integral $\Gamma$-affine polyhedral sets
\begin{equation} \label{polydeco 1}
\trop_U(W_i \cap \Uan)= \Trop(U) \cap G_i^{-1}(\Delta_i)  \quad (i=1, \dots, m)
\end{equation}
is equal to $\trop_U(W \cap \Uan)$. 
For every subset $E$ of $\{1, \dots , m\}$, we have a integral $\Gamma$-affine polyhedral set
\begin{equation} \label{polydeco 2}
\trop_U(W_E \cap \Uan)= \Trop(U) \cap \bigcap_{i \in E} G_i^{-1}(\Delta_i)
=\bigcap_{i \in E} \trop_U(W_i \cap \Uan) .
\end{equation}
For  $V:= \bigcup_i V_i \cap \Uan$, it follows from Corollary 
\ref{support corollary}  that $(V,\varphi_U)$ is a tropical chart containing 
the support of $d''\alpha$.  
The $\delta$-form $\alpha$ is represented on $V$ by  
$\alpha_U \in AZ^{n-1,n-1}(V,\varphi_U)$, i.e. $\alpha=\trop_U^*(\alpha_U)$ on $V$. 
In fact, we have seen in Proposition  \ref{single chart} that $\alpha_U$ extends by $0$ to an element of 
$AZ^{n-1,n-1}(U^{\rm an},\varphi_U)$, but the support of this extension 
is not necessarily compact. 
We conclude that $U$ is a very affine chart of integration  for 
$\log|f|\,d'd''\alpha$ and that 
\begin{equation} \label{PL1}
\bigl\langle d'd''[\log|f|], \alpha \bigr\rangle 
= -\int_{\trop_U(V)} F^*(x_0) d'd'' \alpha_U.
\end{equation}
The minus sign comes from the use of tropical coordinates  
$\trop_U^*(F^*(x_0))=-\log|f|$ as remarked above. 
Corollary \ref{support corollary} shows that the support of $d''\alpha$ 
does not meet $f^{-1}(0)$. 
Since the support of $d''\alpha$ is compact,
there is a positive $s <r $ such $|f(x)| > s$ for every $x \in \supp(d'' \alpha)$.
We consider the analytic subdomain $W(s):=\{x\in W \mid |f(x)| \geq s\}$ of $W$, 
the affinoid subdomain $W_i(s):=\{x \in W_i \mid |f(x)| \geq s \}$ of $W_i$ 
and the affinoid subdomain $W_E(s):=\{x \in W_E \mid |f(x)| \geq s \}$ of $W_E$. 
It follows from \eqref{polydeco 1} and \eqref{polydeco 2} that their 
tropicalizations are integral $\R$-affine polyhedral  sets such that the union of all
\begin{equation} \label{polydeco 3}
\trop_U(W_i(s) \cap \Uan)
= \Trop(U) \cap G_i^{-1}(\Delta_i)\cap F^{-1}((-\infty,-\log s])
\end{equation}
for $i=1,\dots, m$ is equal to $\trop_U(W(s) \cap \Uan)$ and such that 
\begin{equation} \label{polydeco 4}
\trop_U(W_E(s) \cap \Uan)
= \Trop(U) \cap \bigcap_{i\in E}G_i^{-1}(\Delta_i)\cap F^{-1}((-\infty,-\log s]).
\end{equation}
In the following, we use integrals and boundary integrals of $\delta$-preforms 
over integral $\R$-affine polyhedral  sets as introduced in Definition \ref{intdelta}, 
Remark \ref{preform properties} and \ref{integration of P-forms}. 
By the choice of $s$,  we have $\supp(d''\alpha) \subseteq W(s) \cap \Uan$. 
We conclude that $\supp(d''\alpha_U)  \subseteq \trop_U(W(s) \cap \Uan)$ and hence
\begin{equation} \label{PL2}
\int_{\trop_U(V)} F^*(x_0) d'd'' \alpha_U  
=  \int_{\trop_U(W(s)\cap \Uan)} F^*(x_0) d'd'' \alpha_U.
\end{equation}
By Green's formula (see Proposition \ref{greenfordeltapreforms}) and using $d'd''F^*(x_0)=0$, 
the integrals in \eqref{PL2} are equal to
\begin{equation} \label{PL4}
\int_{\partial (\trop_U(W(s)\cap \Uan))}
\bigl(F^*(x_0)  d''\alpha_U   -  d''(F^*(x_0) ) \wedge \alpha_U\bigr).
\end{equation}
By construction and \eqref{polydeco 1}, we have 
$$\supp(\alpha_U)  \subseteq \relint(\trop_U(W \cap \Uan)).$$
By the choice of $s$, it follows that $\supp(d''\alpha_U )$ is contained in 
$\relint(\trop_U(W(s)\cap \Uan))$. 
It follows from Remark  \ref{boundaryintdelta}(iii) applied to the integral $\R$-affine polyhedral  
set $\trop_U(W(s)\cap \Uan)$ that 
\begin{equation} \label{PL5}
\int_{\partial (\trop_U(W(s)\cap \Uan))} F^*(x_0)  d''\alpha_U  =0.
\end{equation}
{Combining \eqref{PL1} and \eqref{PL2}--\eqref{PL5}
with \eqref{PLL1} below,
we get
\begin{equation} \label{PL10}
\langle d'd''[\log|f|], \alpha \rangle =  \langle \delta_{\cyc(f)},\alpha \rangle 
\end{equation}
proving the claim.} 
\qed

\begin{lem}\label{PLL}
{In the situation of the proof of Theorem
\ref{Poincare-Lelong equation} above, we have
\begin{equation} \label{PLL1}
\int_{\partial (\trop_U(W(s)\cap \Uan))}
d''(F^*(x_0))  \wedge \alpha_U=\langle\delta_{{\rm cyc}\,(f)},
\alpha\rangle.
\end{equation}}
\end{lem}

\proof
For integers $\ell \geq 1$, there are $\varphi_\ell \in C^\infty(\R)$ 
with $0 \leq \varphi_\ell \leq 1$, $\varphi_\ell(t)=1$ for 
$t \leq -\log(s)-1/\ell$ and $\varphi_\ell(t)=0$ for $t \geq -\log(s)-1/(2\ell)$. 
By construction,  $\supp\,(\varphi_\ell(F^*(x_0)) d''(F^*(x_0)) \wedge \alpha_U$ 
is contained in the relative interior of $\trop_U(W(s)\cap \Uan)$ and hence
\begin{equation} \label{PL6}
\int_{\partial (\trop_U(W(s)\cap \Uan))} 
\varphi_\ell(F^*(x_0)) d''(F^*(x_0)) \wedge \alpha_U =0
\end{equation}
as above. Setting $\psi_\ell := 1 - \varphi_\ell$, it follows from 
{\eqref{PL6} that the left hand side in \eqref{PLL1}} is equal to
\begin{equation} \label{PL7}
\int_{\partial(\trop_U(W(s)\cap \Uan))}\psi_\ell(F^*(x_0))d''(F^*(x_0))\wedge \alpha_U .
\end{equation}
Now we use the additivity of measures from Remark \ref{boundaryintdelta}(ii). 
The decomposition  \eqref{polydeco 3} of the polyhedral set 
$\trop_U(W(s) \cap \Uan)$ and equation \eqref{polydeco 4} show that 
\eqref{PL7} is equal to
\begin{equation} \label{PL3}
 \sum_{j=1}^m (-1)^{j+1} \sum_{|E|=j} \int_{\partial(\trop_U(W_E(s) \cap \Uan))} 
\psi_\ell(F^*(x_0)) d''(F^*(x_0)) \wedge  \alpha_U.
\end{equation}
We fix $i \in E$. Let 
$G_E:N_{U,\R} \rightarrow N_{U_E,\R}$ and $G_{E,i}:N_{U_E,\R} \rightarrow N_{U_i,\R}$ 
be the canonical affine maps which are compatible with the given moment maps. 
Let  us consider the closed embedding 
\[
h_E:=(f,g_E)=(f,\varphi_{U_E}):
U_E\setminus \Div(f)
\rightarrow {\mathbb G}_m \times T_{U_E}.
\] 
inducing the tropical variety $h_{E,\rm trop}(U_E \setminus \Div(f))$ which we view as a tropical cycle on $\R \times N_{U_E,\R}$. The affine maps  $H_E:=(F,G_E):N_{U,\R} \rightarrow \R \times N_{U_E,\R}$ (resp. 
$H_{E,i}:= \id_\R \times G_{E,i}:\R \times N_{U_E,\R} \rightarrow \R \times N_{U_i,\R}$) 
are compatible with the moment maps $\varphi_U$ and $h_E$ (resp. $h_E$ and $h_i$). 
The Sturmfels--Tevelev multiplicity formula shows 
that \begin{equation} \label{ST}
h_{E,\rm trop}(U_E \setminus \Div(f))=(H_E)_*(\Trop(U))
\end{equation}(see \cite[Proposition 4.11]{gubler-forms} for the required generalization of  \eqref{cor-sturmfels-tevelev}). 
For $\alpha_E:= \alpha_i|_{V_E}\in AZ^{n-1,n-1}(V_E,\varphi_{U_E})$, we have 
$\alpha|_{V_E}=\trop_{U_E}^*(\alpha_E)$ and the definition of $\alpha_E$ does 
not depend on the choice of $i \in E$. 
In the following, the weighted integral $\R$-affine polyhedral  complex 
$\Sigma_E(s):=h_{E,\rm trop}(W_E(s))$  in $\R \times N_{U_E,\R}$ plays a crucial role. 
Note that we have 
\begin{equation} \label{polydeco 5}
\Sigma_E(s)=h_{E,\rm trop}(U_E \setminus \div(f)) \cap \bigcap_{i \in E} 
H_{E,i}^{-1}((-\infty,-\log s] \times \Delta_i).
\end{equation}
Let $P_E:\R\times N_{U_E,\R}\to N_{U_E,\R}$ denote the canonical projection.
By definition, the element $\alpha_E$ of $AZ^{n-1,n-1}(V_E,\varphi_{U_E})$ is represented by a $\delta$-preform $\tilde\alpha_E$ on an open subset $\widetilde\Omega_E$ of $N_{U_E,\R}$ with $\widetilde\Omega_E \cap \Trop(U_E) = \trop_{U_E}(V_E)$. Recall from \eqref{restriction of delta to Omega}, that $\alpha_U|_{\Omega}=\tilde\alpha_U \wedge \delta_{\Trop(U)}$ denotes the $\delta$-preform on  $\Omega:=\Trop_U(V)$ induced by $\alpha_U$.  Using $\alpha_U=G_E^*(\alpha_E)$, 
we get
\[
\alpha_U|_\Omega=G_E^*(\alpha_E)|_\Omega=G_E^*(\tilde\alpha_E)\wedge\delta_{{\rm Trop}\,(U)}
=H_E^*P_E^*(\tilde\alpha_E)\wedge\delta_{{\rm Trop}\,(U)}.
\]
We consider the coordinate $x_0$ on $\R$ also as a function on $\R\times N_{U_E,\R}$.
Using $\trop_U(W(s)\cap \Uan)=H_E^{-1}(\Sigma_E(s)) \cap \Trop(U)$ 
and \eqref{ST}, the projection formula 
\eqref{projection formula for preformsg2} shows that 
\begin{eqnarray} 
\nonumber&&\int_{\partial(\trop_U(W_E(s)\cap \Uan))}  \psi_\ell(F^*(x_0)) d''(F^*(x_0)) \wedge  \alpha_U\\  
&&=\int_{\partial(\trop_U(W_E(s)\cap \Uan))}   H_E^*(\psi_\ell(x_0) d''x_0) 
\wedge H_E^*P_E^*(\tilde\alpha_E)\wedge\delta_{{\rm Trop}\,(U)}\\
\nonumber&&=\int_{\partial(\Sigma_E(s))}   \psi_\ell(x_0) d''x_0 
\wedge P_E^*(\tilde\alpha_E)\wedge\delta_{h_{E,\rm trop}(U_E \setminus \Div(f))}
\end{eqnarray}
By construction of the functions $\varphi_\ell$, we have 
\begin{equation} \label{limit of integrals}
\begin{split}
&\lim_{\ell \to \infty}\int_{\partial(\Sigma_E(s))} \psi_\ell(x_0)d''x_0\wedge
P_E^*(\tilde\alpha_E)\wedge\delta_{h_{E,\rm trop}(U_E \setminus \Div(f))}\\
&= \int_{\Sigma_E(s)\cap \{x_0=-\log|s|\}}  
P_E^*(\tilde\alpha_E)\wedge\delta_{h_{E,\rm trop}(U_E \setminus \Div(f))}.
\end{split}
\end{equation}
By \eqref{polydeco 5} and \cite[\S 7]{gubler-forms}, the analytic tropical 
weights on the $n$-skeleton of the tropicalization $\Sigma_E(s)$ of the 
affinoid domain $W_E(s)$ are the same as the tropical weights induced by 
$h_{E,\rm trop}(U_E \setminus \Div(f))$. 
Using that $s <r$ and $I:=[s,r]$, it follows from (a) and (b) that the 
$n$-skeletons of 
$\Sigma_E(I):=\{\omega \in \Sigma_E(s) \mid x_0(\omega) \in -\log(I)\} $ and 
$-\log(I) \times \trop_{U_E}(\Div(f)\cap W_E)$ are equal even as a product 
of weighted polyhedral complexes if we endow $-\log(I)$ with weight $1$. 
Note that these tropicalizations can differ from the $n$-skeletons only 
inside the relative boundary. 
As we have some flexibility in the choice of the polyhedra $\Delta_i$ and 
in the choice of $s$, we may assume that 
$\Sigma_E(I)= -\log(I) \times \trop_{U_E}(\Div(f)\cap W_E)$ and that 
this is of pure dimension $n$. 
We conclude that \eqref{limit of integrals} is equal to 
\begin{equation} \label{PL8}
\int_{\trop_{U_E}(\Div(f)\cap W_E \cap U_E^{\rm an} )} \alpha_E .
\end{equation}
Using \eqref{PL7}--\eqref{PL8}, it follows that the left hand side of \eqref{PLL1} 
is equal to
\[
\sum_{j=1}^m (-1)^{j+1} \sum_{|E|=j} \int_{\trop_{U_E}(\Div(f)\cap W_E \cap U_E^{\rm an}  )} \alpha_E .
\]
Let $Y$ be an irreducible component of $\Div(f)$ and let 
$E_Y := \{i \in \{1, \dots, m\} \mid U_i \cap Y \neq \emptyset \}$. 
Then we use the very affine open subset $U_{E_Y}$ to compute the following 
integrals over $Y$ by going the above steps backwards:
\begin{equation} \label{PL9}
\begin{split}
&\sum_{j=1}^m (-1)^{j+1} \sum_{|E|=j} \int_{\trop_{U_E}(Y\cap W_E \cap U_E^{\rm an})} \alpha_E\\
&= \int_{\trop_{U_{E_Y}}(Y \cap W \cap U_{E_Y}^{\rm an})}\alpha_{U_{E_Y}}  = \int_Y \alpha,
\end{split}
\end{equation}
where we have used in the last step that $W$ covers $\supp(\alpha)$. 
Using linearity in the irreducible components $Y$ 
(see \cite[Remark 13.12]{gubler-guide}), 
we get equation \eqref{PLL1}. 
\end{proof}

\begin{rem}\label{PLR}
Let $f$ denote a regular function on the affine variety $X$.
The proof of Lemma \ref{PLL} given above shows that
equation \eqref{PLL1} holds more 
generally for any generalized $\delta$-forms $\alpha$ on $X^\an$ 
with compact support.
If we permute the roles of $d'$ and $d''$, we obtain by
the same argument that
\begin{equation} \label{PLL2}
-\int_{\partial (\trop_U(W(s)\cap \Uan))}
d'F^*(x_0)  \wedge \alpha_U=\langle\delta_{{\rm cyc}(f)},
\alpha\rangle
\end{equation}
holds for all generalized $\delta$-forms $\alpha\in
P_c^{n-1,n-1}(X^\an)$. 
An elegant way to deduce \eqref{PLL2} is to apply \eqref{PLL1} for $J^*(\alpha)$ and to use 
symmetry of the $\delta$-current of integration.
\end{rem}

\begin{art}\label{analyticddc}
Let $\varphi$ denote an invertible analytic function on some open subset $W$ of
$X^\an$. 
Given $x\in W$ there exists by \cite[Proposition 7.2]{gubler-forms}
an open subset $U$ of $X$, an algebraic moment map $f:U\to\G_m$
and an open neighbourhood $V$ of $x$ in $U^\an\cap W$ such that
$-\log |\varphi|$ and $-\log |f|$ agree on $V$.  
It follows that the function $-\log |\varphi|$ belongs to $A^0(W)$ and we get
\begin{equation}\label{analyticddcg1}
d'd''[-\log |\varphi|]=-[d'd''\log |\varphi|]=0
\end{equation}
from \eqref{differentiation of delta-currentsg1} and the trivial
case of the Poincar\'e--Lelong formula where $f$ is invertible.
\end{art}

\begin{art}\label{definitionmetric}
Let $L$ be a line bundle on $X$ and let $W$ be an open subset of $\Xan$. 
We fix an open covering $(U_i)_{i\in I}$
of $X$, a family $(s_i)_{i\in I}$
of  {nowhere vanishing sections} $s_i\in \Gamma(U_i,L)$, and the $1$-cocyle
$(h_{ij})$ with values in $\KO_X^\times$ determined by $s_j=h_{ij}s_i$.
Recall that a {\it continuous metric $\|\,\,\|$
on $L$} over $W$ is given by a family $(\rho_i)_{i\in I}$
of continuous functions $\rho_i: U_i^\an \cap W \to \R$ such that
$\rho_j=|h_{ij}|\rho_i$ on $(U_i\cap U_j)^\an \cap W$ for all $i,j\in I$.
An analytic section $s\in \Gamma(V,L^\an)$ on some open subset $V$ of
$W$ determines as follows a continuous function $\|s\|:V\to \R$.
We write $s=f_is_i$ for some analytic function $f_i$ on $V\cap U_i^\an$ and define
$\|s\|=|f_i|\cdot\rho_i$ on $V\cap U_i^\an$.
Observe that we have $\rho_i=\|s_i\|$ on $U_i^\an\cap W$.
\end{art}

\begin{art} \label{first Chern current}
Let $L$ be a line bundle on $X$ endowed with a continuous metric $\metr$ over the open subset $W$ of $\Xan$. 
Then we define the {\it first Chern current associated to the metrized line 
bundle $(L|_W, \metr)$} as the $\delta$-current $[c_1(L|_W,\metr)] \in E^{1,1}(W)$ 
given locally on $W \cap \Uan$   
by $d'd''[-\log \|s|_{\Uan \cap W}\|]$ for any trivialization $U$ of $L$ with 
 {nowhere vanishing} section $s \in \Gamma(U,L)$. 
Here, we have used that a continuous function defines a 
$\delta$-current as explained in Proposition \ref{current continuous function}.
Since $d'd''[-\log|\varphi|]=0$ for an invertible analytic function 
$\varphi$, the $\delta$-current $[c_1(L|_W,\metr)]$ 
is well-defined on $W$ and we may even use analytic trivializations 
in the definition. Obviously, the formation of the first Chern current is 
compatible with tensor product of metrized line bundles as usual. 

If the metric is smooth then $[c_1(L,\metr)]$ is 
the current associated to 
the first Chern form $c_1(L|_W,\metr)$ defined in 
\cite{chambert-loir-ducros}. 
In general, the notion $c_1(L|_W,\metr)$ has no 
meaning as a form and we use 
brackets  to emphasize that $[c_1(L|_W,\metr)]$ 
is a $\delta$-current. 
In Section \ref{piecewise smooth forms and delta-metrics}, 
we will introduce   
metrics for which  $c_1(L|_W,\metr)$ has a meaning as 
a $\delta$-form.
\end{art}

\begin{cor} \label{PL for line bundles}
Let $L$ be a line 
bundle on $X$ endowed with a continuous metric $\metr$ over the open subset $W$ of $\Xan$. 
For every non-trivial
meromorphic section $s$ of $L$ with associated Weil divisor $Y$, 
the equality
\[
[c_1(L|_W,\|.\|)]=
 {d'd''\bigl[{-}\log \|s|_W\|\bigr]}+
\delta_{Y}|_W
\]
holds in $E^{1,1}(W)$.
\end{cor}

\begin{proof} This can be checked locally on a trivialization $U$ of $L$ 
with  {a nowhere vanishing} $s_U \in \Gamma(U,L)$. Then there is a rational function 
$f$ on $X$ with $s=fs_U$ and hence 
\[
d'd''[-\log\|s|_{W \cap \Uan}\|]
=d'd''[-\log\|s_U|_{W \cap \Uan}\|]+d'd''[-\log|f|_{W \cap \Uan}|].
\] 
Using the definition of $c_1(L|_{W \cap \Uan},\|.\|)$ for the first 
summand and using Theorem \ref{Poincare-Lelong equation} for the 
second summand,  we get the claim.
\end{proof}

\section{Piecewise smooth and formal metrics on line bundles} \label{piecewice smooth and formal metrics on lb}

In this section, $X$ is an algebraic variety over $K$.  
In the following, we consider an open subset $W$ of $\Xan$. 

We first introduce piecewise smooth functions and piecewise linear functions on $W$. This leads to corresponding notions for metrics on line 
bundles. We prove that a piecewise linear metric is the same as a formal metric. We show that canonical metrics in various situations are piecewise smooth.

\vspace{2mm}
In Definition \ref{piecewise smooth function},  we have defined piecewise 
smooth functions on an open subset  of an integral $\R$-affine polyhedral set. 
Using tropicalizations and viewing tropical varieties as polyhedral sets, 
we will define piecewise smooth functions on $W$ as follows:

\begin{definition} \label{piecewise smooth}
A function $f: W \rightarrow \R$ is called {\it piecewise smooth}  if for 
every $x \in W$ there is a tropical chart $(V,\varphi_U)$ such that $V$ is 
an open neighbourhood of $x$ in $W$ and such that there is a  piecewise 
smooth  function $\phi$ on $\trop_U(V)$ with $f = \phi \circ \trop_U$ on $V$. 
\end{definition}

In a similar way, we will define a piecewise linear function on $W$. 
We recall from Definition \ref{piecewise smooth function} that we have defined 
piecewise linear functions on integral $\R$-affine polyhedral complexes. 
As we are working with a variety over a valued field, we will take  
the value group $\Gamma$ into account and we will require 
{in the definition of piecewise linear functions} additionally 
that the underlying polyhedral complex and the restriction of the functions 
are both integral $\Gamma$-affine. Note however that in Definition 
\ref{piecewise smooth}, the underlying polyhedral complex for $\phi$ is   
only assumed to be integral $\R$-affine.

\begin{definition} \label{piecewise linear on X}
A function $f: W \rightarrow \R$ is called {\it piecewise linear}  if for 
every $x \in W$ there is a tropical chart $(V,\varphi_U)$ such that $V$ is 
an open neighbourhood of $x$ in $W$ and a  real function $\phi$ on 
$\trop_U(V)$ with $f = \phi \circ \trop_U$ on $V$. 
We require that there is an integral $\Gamma$-affine polyhedral complex 
$\Sigma$ in $N_{U,\R}$ with $\trop_U(V) \subseteq |\Sigma|$ such that $\phi$ 
is the restriction of a function on $|\Sigma|$ with integral $\Gamma$-affine 
restrictions to all faces of $\Sigma$. 
\end{definition}

\begin{art} \label{properties of piecewise smooth}
The space of piecewise  smooth  functions on $W$ is an $\R$-subalgebra  
of the $\R$-algebra of continuous functions on $W$. It contains all smooth 
functions on $W$.  
The space of piecewise linear functions on $W$ is closed under 
forming $\max$ and $\min$. Moreover, it is a subgroup of the space of 
piecewise smooth functions on $W$ with respect to addition. 
If $\varphi:X' \rightarrow X$ is a morphism and $W'$ is an open subset 
of $(\varphi^{\rm an})^{-1}(W)$, then for every piecewise smooth 
(resp. piecewise linear) function $f$ on $W$, the restriction of 
$f \circ \varphi$ to $W'$ is a piecewise smooth (resp. piecewise linear) function on $W'$.
\end{art}

In the following result, we need the $\rm G$-topology on $W$. 
It is a Grothendieck topology build up from analytic subdomains of $W$ and 
it is closely related to the Grothendieck topology of the underlying rigid analytic space 
(\cite[\S 1.3, \S 1.6]{berkovich-ihes}). 

\begin{prop} \label{G-characterization}
Let $f:W \rightarrow \R$ be a continuous function. Then $f$ is piecewise smooth 
(resp. piecewise linear) if and only if there is a $\rm G$-covering 
$(W_i)_{i \in I}$ by analytic (resp. strict analytic) subdomains $W_i$ of 
$W$ and analytic moment maps $\varphi_i:W_i \rightarrow (T_i)^{\rm an}$ to 
tori $T_i := \Spec(K[M_i])$ such that $f=\phi_i \circ \varphi_{i, \rm trop}$ 
on $W_i$ for a smooth (resp. integral $\Gamma$-affine) function 
$\phi_i:N_{i,\R} \rightarrow \R$, where $N_i := \Hom(M_i,\Z)$ as usual.
\end{prop}

\begin{proof} First, we assume that $f$ is piecewise smooth 
(resp.\ piecewise linear). 
For any $x \in W$, there is a tropical chart $(V,\varphi_U)$ 
in $W$ containing $x$  such that $f= \phi \circ \trop_U$ on 
$V$ for a  piecewise 
smooth (resp.\ integral $\Gamma$-affine function) $\phi$ on 
the open subset $\Omega:=\trop_U(V)$ of $\Trop(U)$. 
There are finitely many integral $\R$-affine (resp. 
$\Gamma$-affine) polytopes $\Delta_i$ in $N_{U,\R}$ 
containing $\trop_U(x)$ such that $\bigcup_i \Delta_i$ 
is a neighbourhood of $\trop_U(x)$ in $\Omega$ and such 
that $\phi|_{\Delta_i}=\phi_i|_{\Delta_i}$ for a smooth 
(resp.\ integral $\Gamma$-affine) function 
$\phi_i:N_{U,\R} \to \R$. Note that the affinoid 
(resp.\ strictly affinoid) 
subdomains  $W_i(x):=\trop_U^{-1}(\Delta_i)$ of $W$ 
contain $x$ and cover a neighbourhood of $x$. 
Letting $x$ vary over $W$, we get a $\rm G$-covering 
of $W$ with the desired properties.

To prove the converse, we assume that $f$ is given on 
a $\rm G$-covering 
$(W_i)_{i \in I}$ of $W$ by smooth (resp. integral 
$\Gamma$-affine) 
functions $\phi_i:N_{i,\R} \rightarrow \R$ with respect 
to analytic moment 
maps $\varphi_i:W_i \rightarrow (T_i)^{\rm an}$.
Piecewise smoothness (resp. 
piecewise linearity) is a local condition and 
so we have to check that $f$ 
is piecewise smooth in a neighbourhood of $x \in \Xan$. 
There is a finite $I_0 \subseteq I$ such that the 
sets $(W_i)_{i \in I_0}$ cover a
sufficiently small strict affinoid neighbourhood 
$W'$ of $x$ in $W$. 
By shrinking $W$, we may assume that  
$x \in W_i$ for every $i \in I_0$. 
In the following, we restrict our attention to 
elements $i \in I_0$. 
The definition of an analytic (resp. of a strict analytic) 
domain shows that 
we may assume that all $W_i':=W_i \cap W'$ are affinoid 
(resp. strict affinoid) 
subdomains of $W$. 
Any analytic function on a neighbourhood of $x$ in $W_i'$  
can be approximated uniformly on a sufficiently small 
neighbourhood of $x$ by 
rational functions on $X$. By shrinking $W$ again, this 
shows that we may assume that  $\varphi_i|_{W_i'}$ is 
induced by the 
restriction of an algebraic moment map 
$\varphi_i':U_i \rightarrow T_i$ for a 
dense open subset $U_i$ of $X$ with $W_i' \subseteq (U_i)^{\rm an}$ (see  \cite[Prop. 7.2]{gubler-forms} for a similar argument).
Similarly, we may assume that  there are affinoid coordinates $(x_{ij})_{j \in J_i}$ on $W_i'$ which extend to rational functions on $X$.   Clearly, we may assume that $|x_{ij}(x)|=1$ for $i \in I_0$ and $j \in J_i$. There is a tropical chart $(V,\varphi_U)$ with $x \in V \subseteq W'$, $U \subseteq  \bigcap_{i \in I_0} U_i$ and such that all the functions  $x_{ij}$ are in $\Ocal(U)^\times$. We may assume that $\trop_U(x)=0$ and hence there is an open neighbourhood $\widetilde{\Omega}$  of $0$ in $N_{U,\R}$ with $V=\trop_U^{-1}(\widetilde{\Omega})$. By \cite[4.12, Proposition 4.16]{gubler-forms}, $\varphi_i'|_U$ is the composition of an affine homomorphism $\psi_i:T_U \rightarrow T_i$ with $\varphi_U$. 
By shrinking $V$ and using the Bieri--Groves theorem \cite[Thm. 3.3]{gubler-guide}, we may assume that there are finitely many rational cones $(\Delta_j)_{j \in J}$ in $N_{U,\R}$ such that 
\begin{equation} \label{local cones in omega}
\trop_U(V) = \widetilde{\Omega} \cap \bigcup_{j \in J} \Delta_j.
\end{equation}
For every $i \in I_0$ and every $j \in J_i$,  there is a linear form $u_{ij} \in M_U$ with  $-\log|x_{ij}|=u_{ij}\circ \trop_U$ on $\Uan$. The definition of affinoid coordinates yields 
\begin{equation} \label{affinoid as pullback}
W_i' \cap \Uan = \trop_U^{-1}(\sigma_i )
\end{equation}
for  $\sigma_i:=\{\omega \in N_{U,\R}  \mid u_{ij}(\omega) \geq r_{ij} \, \forall j \in J_i\}$ and suitable $r_{ij} \in \R$. Note that $\sigma_i$ is an integral $\R$-affine polyhedron. In the piecewise linear case, we may choose always $r_{ij}=0$ and hence $\sigma_i$ is a rational cone. Using that the sets $W_i' \cap \Uan$ cover $V$ and equations \eqref{local cones in omega}, \eqref{affinoid as pullback}, we get  the decomposition $(\sigma_i \cap \Delta_j \cap \widetilde{\Omega})_{i \in I_0, j \in J}$ of $\trop_U(V)$. On  $\sigma_i \cap \Delta_j \cap \widetilde{\Omega}$, we choose the smooth (resp. integral $\Gamma$-affine) function $\phi_{ij}':=\phi_i \circ \psi_i$. Using \eqref{affinoid as pullback}, we see that  
these functions paste to a continuous piecewise smooth (resp. continuous piecewise linear) function $\phi'$ on $\trop_U(V)$ with $\phi' \circ \trop_U=f$ on $V$. This proves easily that $f$ is piecewise smooth (resp. piecewise linear) on $W$. 
\end{proof}

\begin{definition} \label{piecewise smooth metrics}
Let $L$ be a line bundle on $X$ and let $W$ be an open subset of $\Xan$. 
A metric $\metr$ on ${L|_W}$ is called 
{\it piecewise smooth} (resp. {\it piecewise linear}) if for 
every $x \in W$, 
there is a tropical chart $(V,\varphi_U)$ with $x \in V \subseteq W$ 
and a  {nowhere vanishing} section $s \in \Gamma(U,L)$  such 
that $-\log \|s|_{V} \|$ is piecewise smooth (resp. piecewise linear) 
on $V$. 
\end{definition}

\begin{art} \label{remarks for ps metrics}
Since $-\log |f|$ is smooth for an invertible regular function $f$ and even the 
pull-back of a linear function with respect to a suitable tropicalization, {the
last definition} does neither depend on the choice of the trivialization $s$ 
nor on the choice of the tropical chart $(V,\varphi_U)$. 
Moreover, we may also use 
analytic 
trivializations in the definition. By Proposition \ref{G-characterization}, 
the definition of a piecewise linear metric agrees with the definition of 
$\rm PL$-metrics in \cite[\S 6.2]{chambert-loir-ducros}. 

Note that every piecewise linear metric is piecewise smooth.
It follows from \ref{properties of piecewise smooth} that 
every piecewise smooth metric is continuous, that the tensor product of piecewise 
linear (resp. piecewise smooth) metrics is again a piecewise linear (resp. 
piecewise smooth) metric and that the dual metric of a  a piecewise linear (resp. 
piecewise smooth) metric is  piecewise linear (resp. piecewise smooth). Moreover, 
the pull-back of a piecewise linear (resp. piecewise smooth) metric on $L|_W$ with 
respect to a morphism  $\varphi:X' \rightarrow X$ is a piecewise linear (resp. 
piecewise smooth) metric on $\varphi^*(L)|_{W'}$ for any open subset $W'$ of 
$\varphi^{-1}(W)$.
\end{art}

\begin{art} \label{formal models}
Recall that  $\kcirc$ is the valuation ring of the given non-archimedean absolute value $|\phantom{a}|$ on $K$.
 Raynaud introduced an {\it admissible formal scheme over $\kcirc$} 
as  a formal scheme $\Xcal$ over the valuation ring $\kcirc$ which is locally isomorphic to 
$\Spf(A)$ for a flat $K^\circ$-algebra $A$  of topologically finite type over  
$\kcirc$ (see \cite[\S 1]{bosch-luetkebohmert-1} 
for details). For simplicity, we require additionally that $\Xcal$ has a locally 
finite atlas of admissible affine formal schemes over $\kcirc$.  
Then $\Xcal$ has a generic fibre $\Xcal_\eta$ (resp. a special fibre 
$\Xcal_s$) which is a paracompact strictly analytic Berkovich space over 
$K$ (resp. an algebraic scheme over the residue field $\ktilde$) locally 
isomorphic to $\Mcal(\Acal)$ (resp. $\Spec(A \otimes_{\kcirc} \ktilde$)) 
for the strict affinoid algebra $\Acal := A \otimes_\kcirc K$ (see 
\cite[\S 1.6]{berkovich-ihes} for the equivalence to rigid analytic 
spaces over $K$ with an affinoid covering of finite type). 

A {\it formal $\kcirc$-model  of $X$} is an admissible formal 
scheme $\Xcal$ over $\kcirc$ with an isomorphism 
$\Xcal_\eta \cong \Xan$. For a line bundle $L$ on $X$, we define a 
{\it formal $\kcirc$-model of $L$} as a line bundle $\Lcal$ on a formal 
$\kcirc$-model $\Xcal$ of $X$ with an isomorphism $\Lcal_\eta \cong \Lan$
{over $\Xcal_\eta\cong \Xan$}. 
For simplicity, we usually identify $\Lcal_\eta$ with $\Lan$. 
\end{art}

\begin{art} \label{formal metrics}
Let $L$ be a line bundle on $X$. A {\it formal metric} on $L$ is a metric 
$\metr_\Lcal$ associated to a formal $\kcirc$-model $\Lcal$ of $L$ in the 
following way: {If $\Lcal$ admits a formal trivialization over
$\Ucal$ and if  $s \in \Gamma(\Ucal, \Lcal)$ corresponds under this 
trivialization to the function 
$\gamma \in \Ocal_\Xcal(\Ucal)$,} then $\|s(x)\|=|\gamma(x)|$ for all 
$x \in \Ucal_\eta$. This definition is independent of the choice of the 
trivialization and shows immediately that formal metrics are continuous. 
The tensor product and the pull-back of formal metrics are again formal metrics. 
\end{art}

\begin{prop} \label{existence}
Every line bundle $L$ on $X$ has a formal $\kcirc$-model and hence a formal metric.
\end{prop}

\begin{proof} This follows as in 
\cite[Prop. 7.6]{gubler-crelle} based on the theorem of Raynaud that every 
paracompact analytic space has a formal $\kcirc$-model 
(see \cite[Theorem 8.4.3]{bosch-lectures-2015}). 
The argument for paracompact strictly $K$-analytic spaces was first given 
in \cite[Proposition 6.2.13]{chambert-loir-ducros}.
\end{proof}

\begin{prop} \label{uniqueness of line bundle}
Let $\metr$ be a formal metric on the line bundle $L$ on $X$. 
Then there is an admissible formal $\kcirc$-model $\Xcal$ of $X$ with 
reduced special fibre $\Xcal_s$ and a $\kcirc$-model $\Lcal$ of $L$ on 
$\Xcal$ such that $\metr = \metr_\Lcal$. Moreover, the invertible sheaf 
associated to $\Lcal$ is always canonically isomorphic to the sheaf  
on $\Xcal$ given by 
$\{s \in \Gamma(L,\Ucal_\eta) \mid ||s(s)|| \leq 1 \; \forall x \in \Ucal_\eta\}$ 
on a formal open subset $\Ucal$ of $\Xcal$.
\end{prop}

\begin{proof} This follows as in \cite[Lemma 7.4 and Proposition 7.5]{gubler-crelle}. 
\end{proof}

\begin{prop} \label{equivalences for formal metrics} 
Let $\metr$ be a metric on the  line bundle $L$ on $X$. 
Then the following properties are equivalent:
\begin{itemize}
\item[(a)]  $\metr$ is a formal metric;
\item[(b)]  $\metr$ is a piecewise linear metric;
\item[(c)]  there is a $\rm G$-covering $(W_i)_{i \in I}$ of $\Xan$ 
by strict analytic subdomains $W_i$ of $\Xan$ and trivializations 
$s_i \in \Gamma(W_i,\Lan)$ with $\|s_i(x)\|=1$ for all $x \in W_i$ and all $i \in I$.
\end{itemize}
\end{prop}

\begin{proof} If we use again  Raynaud's 
theorem to generalize to paracompact $\Xan$, 
the equivalence of (a) and (c) is proved as in 
\cite[Lemma 7.4 and Proposition 7.5]{gubler-crelle}. 
The implication (a) $\Rightarrow$ (c) can also be found in 
\cite[Exemple 6.2.10]{chambert-loir-ducros}.
It remains to see the equivalence of (b) and (c). Suppose that (b) holds. 
Then there is a locally finite covering of $X$ by trivializations $U_i$ of 
$L$ such that $-\log \|s_i\|$ is piecewise linear on $(U_i)^{\rm an}$ for 
every $i \in I$. 
By Proposition \ref{G-characterization}, there is a $\rm G$-covering 
$W_{ij}$ of $(U_i)^{\rm an}$ and analytic moment maps 
$\varphi_{ij}: W_{ij} \rightarrow (T_{ij})^{\rm an}$ such that 
$-\log\|s_i\|=\phi_{ij} \circ \varphi_{ij,\trop}$ on $W_{ij}$ for 
an integral $\Gamma$-affine function $\phi_{ij}$ on $N_{ij,\R}$. 
The definition of integral $\Gamma$-affine functions shows that there 
is an invertible analytic function $\gamma_{ij}$ on $W_{ij}$ such that 
$\|s_i\|=|\gamma_{ij}|$ on $W_{ij}$. 
Using the trivialization $\gamma_{ij}^{-1}s_i$ on $W_{ij}$, we 
get (b) $\Rightarrow$ (c). 
The converse is an immediate application of 
Proposition \ref{G-characterization}.
\end{proof}

\begin{art} \label{algebraic metrics}
If $X$ is proper over $K$, than an {\it algebraic $\kcirc$-model} of $X$ is an 
integral  scheme $\Xfrak$ which is of finite type, flat and proper over 
$\kcirc$ and with a fixed isomorphism between the generic fibre $\Xcal_\eta$ 
and $X$. We use the isomorphism to identify $\Xfrak_\eta$ and $X$.
An {\it algebraic $\kcirc$-model} of $L$ is a line bundle 
$\Lfrak$ on an algebraic $\kcirc$-model $\Xfrak$ of $X$ together with a fixed
isomorphism between $\Lfrak_\eta$ and $L$. We define an {\it algebraic metric} on $L$ as in \ref{formal metrics} by using an algebraic $\kcirc$-model $\Lfrak$ of $L$. 
\end{art}

\begin{prop} \label{algebraic and formal metrics}
On a line bundle on a proper variety over $K$, a metric is algebraic if and only if it is formal. 
\end{prop}

\begin{proof} 
Passing to the formal completion along the special fibre, it is clear that every algebraic metric is a formal metric. 
Using \cite[Proposition 10.5]{gubler-pisa}, the converse is true if $X$ is projective. The same argument shows that the converse is also true for proper $X$ if the formal GAGA theorem in \cite[Theorem 5.1.4]{ega-3-1} holds over $\kcirc$ and if $X$ has an algebraic $\kcirc$-model. In  \cite[Theorem 5.1.4]{ega-3-1}, the base has to be  noetherian and hence it applies only for discrete valuation rings. The required generalization is now given in 
\cite[Theorem I.10.1.2]{fujiwara-kato-1}.
The existence of an algebraic $\kcirc$-model follows from Nagata's compactification theorem. This was proved by Nagata in the noetherian case and proved by Conrad in general (based on notes of Deligne, see Temkin's paper 
\cite{temkin-2011} for another proof and references).
\end{proof}

\begin{cor} \label{formal vs algebraic}
Let $L$ be a line bundle on a proper variety over $K$. Then $L$ has an algebraic metric. 
\end{cor}

\begin{proof}
This follows from Proposition \ref{existence} and Proposition \ref{algebraic and formal metrics}.
\end{proof}

We will show now that many important metrics are piecewise smooth.

\begin{ex} \label{canonical metric on abelian variety}
Let $L$ be a line bundle on the abelian variety $A$ over $K$. 
Choosing a rigidification of $L$ at $0 \in A$ and assuming 
$L$ symmetric (resp. odd), the theorem of the cube allows to identify 
$[m]^*(L)$ with $L^{\otimes m^2}$ (resp. $L^{\otimes m}$). 
There is a unique continuous metric $\metr_{\rm can}$ on $\Lan$ with 
$[m]^*\metr_{\rm can}=\metr_{\rm can}^{\otimes m^2}$ (resp. 
$[m]^*\metr_{\rm can}=\metr_{\rm can}^{\otimes m}$) for all 
$m \in \Z$. In general, $L^{\otimes 2}$ is the tensor product 
of a symmetric and an odd line bundle, unique up to $2$-torsion 
in $\Pic(A)$, and hence we get a {\it canonical metric 
$\metr_{\rm can}$}  on $L$ which is unique up to multiples 
from $|K^\times|$ if we vary rigidifications. 
We claim that $\metr_{\rm can}$ is locally {on $X^\an$} the tensor product of a smooth 
metric and a piecewise linear metric. In particular, we deduce that 
$\metr_{\rm can}$ is a piecewise smooth metric. 

To prove the claim, we use the Raynaud extension  of $A$ to describe the 
canonical metric on $L$ (see  \cite[\S 4]{gubler-compositio} for details). 
{The Raynaud extension} is an exact sequence 
{\begin{equation} \label{Raynaud extension 2}
1 \longrightarrow \Tan \longrightarrow E \stackrel{q}{\longrightarrow} 
B^{\rm an} \longrightarrow 0
\end{equation}}
of {commutative} analytic groups, where $T=\Spec(K[M])$ is a 
multiplicative torus of rank $r$  and $B$ is an abelian variety of good reduction. 
Moreover, there is a lattice $P$ in $E$ with $E/P = A^{\rm an}$. 
{More precisely $P$} is a discrete subgroup of $E(K)$ which is mapped 
{by a canonical map }$\val:E \rightarrow N_\R$ isomorphically onto a 
complete lattice {$\Lambda$} of $N_\R$, where $N$ is the dual of $M$. 
The map $\val$ is locally over $B$ a tropicalization. Note that the Raynaud 
extension is algebraizable, but the quotient homomorphism 
$p:E \rightarrow A^{\an}$ is only { defined in the analytic category}. 
 
Let $\Bcal$ be the abelian scheme over $\kcirc$ with generic fibre $B$. 
{By {\it loc. cit.} there exists} a line bundle $\Hcal$ on $\Bcal$ such 
that $q^*((\Hcal_\eta)^{\rm an})= p^*(\Lan)$. Here and in the following, we 
use rigidified line bundles to identify isomorphic line bundles. 
Then $q^*\metr_\Hcal$ is a formal metric on $p^*(\Lan)$. 
On $p^*(\Lan)$, we have a {canonical} $P$-action $\alpha$ over 
{the canonical action of $P$ on $E$ by translation}. 
{By {\it loc. cit.}} there is a {$1$-cocycle $Z$ in 
$Z^1(P,(\R^\times)^E)$} such that
\begin{equation} \label{action and metric}
\left(q^*\|\alpha_\gamma(w)\|_\Hcal\right)_{\gamma \cdot x} = 
Z_\gamma(x)^{-1} \cdot \left(q^*\|w \|_\Hcal\right)_x
\end{equation}
for all $\gamma \in P$, $x \in E$ and $w \in (p^*L^{\rm an})_x$. 
The cocycle $Z$ depends only on the map $\val$ which means that there is a 
unique function $z_\lambda: N_\R \rightarrow \R$ with
\[
z_\lambda(\val(x))=- \log(Z_\gamma(x)) \quad (\gamma \in P, x \in E, \lambda=\val(\gamma)).
\]
Moreover, there is a canonical symmetric bilinear form 
$b:\Lambda \times \Lambda \rightarrow \Z$ associated to $L$ such that 
\[
z_\lambda(\omega) = z_\lambda(0) + b(\omega,\lambda) \quad (\omega \in N_\R, \lambda \in \Lambda).
\]
The cocycle condition 
\[
Z_{\rho \gamma}(x)=Z_\rho(\gamma x) Z_\gamma(x) \quad (\rho, \gamma \in P, x \in E)
\]
shows that 
\[
z_{\lambda + \mu}(0)=z_\lambda(0)+z_\mu(0) + b(\lambda, \mu) \quad (\lambda, \mu \in \Lambda)
\]
which means that $\lambda\mapsto z_\lambda(0)$ is a quadratic function 
{on $\Lambda$}. 
There is a unique extension to a quadratic function $q_0:N_\R \rightarrow \R$. 
We define a metric $\metr$ on $p^*(\Lan)$ by 
$\metr := e^{-q_0 \circ \val} q^*\metr_\Hcal$. 
Using \eqref{action and metric} and that $q_0$ is a quadratic function with 
associated bilinear form $b$, it follows easily that $\metr$ descends to the 
canonical metric on $L$. We conclude from the descent {with respect to the
local isomorphism $p$} that the canonical 
metric on $L$ is locally on $A^{\rm an}$ the tensor product of a smooth 
metric with a piecewise linear metric. This proves the claim.
\end{ex}

\begin{ex} \label{algebraically equivalent to zero} 
{Let $L$ be a line bundle on a proper smooth algebraic variety 
over $K$ which is algebraically equivalent to zero.
Let $A$ denote the Albanese variety of $X$ 
(see \cite[Exp. 236, Thm. 2.1, Cor. 3.2]{grothendieck-fga}).
We fix some $x\in X(K)$ and obtain a universal morphism
$\psi:X\to A$ from $X$ to the abelian variety $A$ with $\psi(x)=0$.
Furthermore $L$ is in a canonical way the pullback of an odd line
bundle on $A$ along $\psi$.
It follows that $L$ carries a canonical metric $\metr_{\rm can}$, unique 
up to multiples from $|K^\times|$.
}
By \cite[Example 3.7]{gubler-compositio}, there is an integer 
$N \geq 1$ such that $\metr_{\rm can}^{\otimes N}$ is an algebraic metric 
and hence piecewise linear. We conclude that $\metr_{\rm can}$ is a piecewise 
smooth metric.
\end{ex}

\begin{ex} \label{toric varieties}
Let $L$ be a line bundle on a complete toric variety $X$ over $K$. 
Similarly as in the case of abelian varieties and using rigidifications, 
we have $[m]^*(L)=L^{\otimes m}$ and there is a unique metric 
$\metr_{\rm can}$ on $L$ with $ {[m]^*\metr_{\rm can}}=\metr_{\rm can}^{\otimes m}$ 
for all integers $m \in \Z$ (see \cite[Section 3]{maillot-toric}). 
There is a canonical algebraic $\kcirc$-model $\Xcal$ of $\kcirc$ and a 
canonical algebraic $\kcirc$-model $\Lcal$ by using the same 
rational polyhedral fan and the same piecewise linear function. 
Since $\metr_{\rm can} = \metr_\Lcal$, the canonical metric on $L$ is 
algebraic and hence a piecewise linear metric.
\end{ex}

\begin{art} \label{trivial ps}
Finally, we consider the case where our  {variety $X$ is defined over a
ground field $F$ which} is equipped with
the trivial valuation. 
If $L$ is a line bundle on $X$, then we choose an algebraically 
closed extension field  {$K$} endowed with a non-trivial complete absolute value 
extending the trivial absolute value of  {$F$. Then $F\subseteq K^\circ$} and
the line bundle 
 {$L \otimes_F K^\circ$ on $X \otimes_F K^\circ$} is a canonical 
algebraic  {$K^\circ$}-model of the line bundle  {$L_K$ on $X_K$}. 
We conclude that $L$ has a canonical metric 
$\metr_{\rm can}$. 

The metric $\metr_{\rm can}$ has the following intrinsic description. 
Let $U=\Spec(A)$ be an affine open subset of $X$ which is a 
trivialization of $L$ given by the nowhere vanishing section 
$s \in \Gamma(U,L)$. We consider the formal affinoid subdomain 
$U^\circ := \{x \in \Uan \mid |f(x)| \leq 1 \; \forall f \in A\}$ 
of $\Xan$. Note that $U^\circ$ is the set of points in $\Uan$ with 
reduction contained in $U$ (see \cite[\S 4]{gubler-guide} for more 
details). It follows that $\|s(x)\|_{\rm can}=1$ for all $x \in U^\circ$. 
Since $X$ is proper, such trivializations $U^\circ$ cover $\Xan$ leading 
to a description of $\metr_{\rm can}$ which is independent of  {$K$}. 

For simplicity, we have considered only varieties in this paper. 
We may also consider continuous metrics on $L^{\rm an}$ for a line bundle 
over a separated scheme $X$ of finite type over the ground field  {$F$}.
For such schemes $X$, the intrinsic description above shows in particular 
that we still have a 
canonical metric $\metr_{\rm can}$ on  $L$  in the case of a 
trivially valued  {$F$}.
\end{art}

\section{Piecewise smooth forms and delta-metrics}  \label{piecewise smooth forms and delta-metrics}

We consider again an algebraic variety $X$ over $K$ of dimension $n$.
In this section, we first study piecewise smooth forms on an open subset $W$ 
of $\Xan$. 
This  {leads} to a decomposition of the first Chern current of a piecewise 
smoothly metrized line bundle $(L|_W,\metr)$ into the sum of a piecewise smooth 
form and a residual current. 
We show that the residual current is induced by a generalized $\delta$-form. 
If the first Chern current of $(L|_W,\metr)$ is induced by a $\delta$-form on $W$, 
then $\metr$ is called a $\delta$-metric and the $\delta$-form is called the 
first Chern $\delta$-form. 
We show that many important metrics are $\delta$-metrics. 
In the following sections, we will use $\delta$-metrics for our approach 
to non-archimedean Arakelov theory.

\begin{art} \label{recall of ps superforms}
In Definition \ref{piecewise smooth superforms}, we have defined the space 
$PS(\Omega)$ of piecewise smooth superforms on an open subset $\Omega$ of a polyhedral subset. 
If $(V,\varphi_U)$ is a tropical chart, then we apply this definition for the 
open subset $\Omega := \trop_U(V)$ of {$\Trop(U)$}. 
If $\alpha \in PS(\Omega)$ and $(V',\varphi_{U'})$ is a tropical chart with 
$V' \subseteq V$ and $U' \subseteq U$, then we define $\alpha|_{V'}$ as the 
piecewise smooth form on $\Omega':=\trop_{U'}(V')$ given by pull-back 
of $\alpha$ with respect to the canonical affine map $N_{U',\R} \rightarrow N_{U,\R}$. 
\end{art}

\begin{definition} \label{piecewise smooth forms on X}
A {\it piecewise smooth form} on an open subset $W$ of $\Xan$ may be defined 
in a similar way as a differential form in $A(W)$: 
A piecewise smooth form $\alpha$ is  given by an open covering 
$(V_i,\varphi_{U_i})_{i \in I}$ of $W$ by tropical charts and piecewise smooth superforms 
$\alpha_i$ on $\Omega_i:=\trop_{U_i}(V_i)$ such that 
$\alpha_i|_{V_i \cap V_j}= \alpha_j|_{V_i \cap V_j}$ for all $i,j \in I$. 
A superform $\alpha'$ given by the covering 
$(V_j',\varphi_{U_j'})_{j \in J}$ and piecewise smooth superforms 
$\alpha_j'$ on $\Omega_j':= \trop_{U_j'}(V_j')$ will be identified with $\alpha$ 
if and only if $\alpha_i|_{V_i \cap V_j'}= \alpha_j'|_{V_i \cap V_j'}$ for 
every $i\in I$ and every $j \in J$.
\end{definition}

\begin{art} \label{properties of ps forms on X}
We denote the space of piecewise smooth forms on $W$ by $PS(W)$. 
It comes with a bigrading and {is canonically equipped with a $\wedge$-product}. 
We conclude easily that $PS^{\cdot,\cdot}(W)$ is a bigraded  $A^{\cdot,\cdot}(W)$-algebra on $\Xan$. 
It is clear that $PS^{0,0}(W)$ is the space of piecewise smooth functions on $W$. 
{It coincides with the space $P^{0,0}(W)$ of generalized
$\delta$-preforms of degree zero.
The equality
\begin{equation}
PS^{0,0}(W)=P^{0,0}(W)
\end{equation} 
is in fact a direct consequence of \eqref{preforms and degree 0}.}

If $\varphi:X' \rightarrow X$ is a morphism of algebraic varieties over $K$, then 
the pull-back of piecewise smooth superforms from 
\ref{piecewise smooth superforms} carries over to define a {\it pull-back} 
$f^*:PS^{p,q}(W) \rightarrow PS^{p,q}(W')$ for any open subset $W'$ of 
$(X')^{\rm an}$ with $f(W') \subseteq W$. 
In the special case of $X=X'$, $f=\id$ and $W'$ an open subset of $W$, 
we denote the pull-back by $\alpha|_{W'}$ and call it the {\it restriction of $\alpha$ to $W'$}.
\end{art}

\begin{art} \label{differentiation of ps forms on X}
In \ref{differentiation of piecewise smooth}, we have introduced differentials 
of piecewise smooth forms on open subsets of polyhedral sets. 
If $\alpha \in PS^{p,q}(W)$ is given as in Definition 
\ref{piecewise smooth forms on X}, then the {polyhedral} differential 
$\dpa\alpha \in PS^{p+1,q}(W)$ is 
locally defined by $\dpa\alpha_i \in PS^{p+1,q}(\Omega_i)$. 
Similarly, we define $\dpb\alpha \in PS^{p,q+1}(W)$. 
Then $PS^{\cdot,\cdot}(W)$ is a differential graded $\R$-algebra with respect 
to the {polyhedral} differentials $\dpa$ and $\dpb$. 
\end{art}

\begin{art} \label{psp on W}
The bigraded differential $\R$-algebras $PS(W)$ of piecewise smooth forms and 
$P(W)$ of generalized $\delta$-forms are not directly comparable except that 
they contain both $A(W)$ as a bigraded differential $\R$-subalgebra. 
We construct a bigraded differential $\R$-algebra {$PSP(W)$} containing both 
spaces as follows. 

Recall from Remark \ref{PSP-forms} that we have obtained a bigraded 
differential $\R$-algebra $PSP(\widetilde\Omega)$ with respect to 
$\dpa,\dpb$ for any open subset $\widetilde\Omega$ of $N_\R$. 
We repeat now the construction of generalized $\delta$-forms in 
\S \ref{deltaalgvar} building up on the spaces $PSP(\widetilde\Omega)$ 
instead of $P(\widetilde\Omega)$. 
This leads first to {spaces} $PSP(V,\varphi_U)$ for tropical charts 
$(V,\varphi_U)$ of $X$ and then to the desired {space} $PSP(W)$. 
Note that $PSP(W)$ is a differential bigraded $\R$-algebra with respect to the 
{polyhedral} differential operators $\dpa$ and $\dpb$ which extends the 
corresponding structure on the subalgebra $P(W)$. 
To see that $PS(W)$ is a graded subalgebra of $PSP(W)$, we use the obvious 
generalization of Proposition \ref{extension of piecewise smooth} 
{from piecewise smooth functions} to piecewise smooth forms. 
Obviously, $PSP(W)$ is generated by the subalgebras $PS(W)$ and $P(W)$. 
Moreover, the {polyhedral} differentials $\dpa,\dpb$ agree with 
the corresponding differential operators on $PS(W)$.

All properties {of} generalized $\delta$-forms from \S \ref{deltaalgvar} 
and \S \ref{Integration of delta-forms} extend immediately to the sheaves $PSP$. 
Hence we have an integral $\int_W \alpha $ for any $\alpha \in PSP_c^{n,n}(W)$. 
As a special case, we obtain such an integral for a piecewise smooth form 
with compact support on $W$. 
As in \ref{sheaf-property}, this leads to a $\delta$-current 
$[\alpha] \in E^{p,q}(W)$ for any $\alpha \in PSP^{p,q}(W)$. 
In particular, this applies to a piecewise smooth $\alpha$.
\end{art}

\begin{rem} \label{residues for PSP and PS-forms}
Note that the {polyhedral} differential $\dpa \alpha$ of a piecewise smooth form $\alpha$ or more generally of  any $\alpha \in PSP(W)$ is not compatible 
with the corresponding differential of the associated $\delta$-current. 
We define the {\it $d'$-residue} 
by 
\[{\rm Res}_{d'}(\alpha) := d'[\alpha]-[\dpa\alpha].\]
Similarly, we define residues with respect to $d''$ and $d'd''$.
\end{rem}

\begin{art} \label{ps-part and residual part}
Now we consider a line bundle  $L$ on $X$ endowed with a piecewise smooth 
metric $\metr$ over the open subset $W$ of $\Xan$. We are going to obtain a canonical decomposition of the 
Chern current $[c_1(L|_W,\metr)] \in E^{1,1}(\Xan)$ (see \ref{first Chern current}) 
into a {\it piecewise smooth part} $c_1(L|_W,\metr)_{\rm ps}\in PS^{1,1}(W)$ 
and a {\it residual part $[c_1(L|_W,\metr)]_{\rm res}\in E^{1,1}(W)$}.

Let {$(U,s)$ be a trivialization of $L$, i.e. $U$ is an open subset of $X$ 
and $s$ is a nowhere vanishing section in $\Gamma(U,L)$}. 
Then $-\log \|s\|$ is a piecewise smooth function on $\Uan\cap W$ and hence 
$-\dpa\dpb\log \|s|_{\Uan \cap W}\| \in PS^{1,1}(\Uan\cap W)$. Note that this piecewise smooth form 
is independent of the choice of $s$ by the same argument as in 
\ref{first Chern current} and hence we obtain a globally defined element of $PS^{1,1}(W)$ 
which we denote by $c_1(L|_W,\metr)_{\rm ps}$. 
Recall from \ref{psp on W} that we denote the associated $\delta$-current on $W$ by $[c_1(L|_W,\metr)_{\rm ps}]$. 
The same argument shows that the residues ${\rm Res}_{d'd''}(-\log\|s|_{\Uan \cap W}\|)$ 
paste together to give a global $\delta$-current $[c_1(L|_W,\metr)]_{\rm res}\in E^{1,1}(W)$ and we have
\begin{equation} \label{decomposition into ps and res}
[c_1(L|_W,\metr)] = [c_1(L|_W,\metr)_{\rm ps}] + [c_1(L|_W,\metr)]_{\rm res}.
\end{equation}
\end{art}

\begin{prop} \label{residual Chern form} 
Let $\metr$ be a piecewise smooth metric on $L|_W$. 
Then there is a unique $\beta \in P^{1,1}(W)$ with 
\[[\varphi^*(\beta)]=[c_1(\varphi^*(L)|_{W'},\varphi^*\metr)]_\res \in E^{1,1}(W')\]
for every morphism $\varphi:X' \rightarrow X$ from any algebraic variety $X'$ over $K$ and for every open subset $W'$ of $\varphi^{-1}(W)$. 
The generalized $\delta$-form $\beta$ has codimension $1$ (see Definition \ref{predefinition}) 
and will be denoted 
by $c_1(L|_W,\metr)_\res$. 
\end{prop}

\begin{proof}
Note that uniqueness follows from Proposition \ref{functorial criterion}. 
By definition of a piecewise smooth metric, there is an open covering 
$(V_i)_{i \in I}$ of $W$ by tropical charts $(V_i, \varphi_{U_i})$, 
nowhere vanishing sections $s_i \in \Gamma(V_i,L^{\rm an})$ and piecewise smooth 
functions $\phi_i$ on $\Omega_i := \trop_{U_i}(V_i)$ with 
$-\log \|s_i\| = \phi_i \circ \trop_{U_i}$ on $V_i$.
Passing to a refinement of the open covering, we may assume that $\phi_i$ is 
defined on $\Trop(U_i)$.  By Proposition \ref{extension of piecewise smooth}, 
there is a piecewise smooth 
function $\tilde{\phi}_i$ on $N_{U_i,\R}$ restricting to $\phi_i$. 
By Lemma \ref{corner locus for tropical cycle}, the corner locus 
$C_i := \tilde{\phi}_i\cdot N_{U_i,\R}$  of $\tilde{\phi}_i$ is a tropical cycle 
of codimension $1$. 

The  $\delta$-preform $\delta_{C_i}$ represents an element $\beta_i \in P(V_i,\varphi_{U_i})$ 
of codimension $1$ (see \ref{presheaf}). 
We have seen in \ref{properties of delta-preforms on charts} that there is a 
pull-back $f^*(\beta_i) \in P^{1,1}(V',\varphi_{U'})$ for every morphism 
$f:X' \rightarrow X$ of algebraic varieties over $K$ and every tropical chart 
$(V',\varphi_{U'})$ of $X'$ compatible with $(V_i,\varphi_{U_i})$. 
For the open subset $\Omega':=\trop_{U'}(V')$ of $\Trop(U')$, 
we have $f^*(\beta_i)|_{\Omega'} \in P^{1,1}(\Omega') \subseteq D^{1,1}(\Omega')$ 
(see \ref{restriction of delta to Omega}). 
Let $F:N_{U',\R} \rightarrow N_{U_i,\R}$ be the canonical affine map with 
$\trop_{U_i} = F \circ \trop_{U'}$ on $(U')^{\rm an}$. 
By Proposition \ref{asscommcornerlocus} and Corollary \ref{pullbackcornerlocus}, $F^*(C_i)\cdot \Trop(U')$ is the 
corner locus of $\phi':=\phi_i \circ F|_{\Trop(U')}$ and hence we get 
\[f^*(\beta_i)|_{\Omega'}= F^*(\delta_{C_i}) \wedge \delta_{\Trop(U')} 
= \delta_{\phi'\cdot \Trop(U')} \in P^{1,1}(\Omega').\]
Together with the tropical Poincar\'e--Lelong formula 
(Corollary \ref{tropicalpoincarelelong}), we get 
\begin{equation} \label{tropical PL for covering}
 f^*(\beta_i)|_{\Omega'} + [\dpa\dpb\phi']=  d'd''[\phi'] \in  D^{1,1}(\Omega').
\end{equation}
It follows  from \eqref{tropical PL for covering} that $f^*(\beta_i)|_{\Omega'}$ 
is independent of all choices. 
This yields that  $\beta_i|_{V_i \cap V_j}= \beta_j|_{V_i \cap V_j}$ for all $i,j \in I$. 
We get a well-defined generalized $\delta$-form $\beta \in P^{1,1}(W)$ of 
codimension $1$ given by $\beta_i \in P^{1,1}(V_i,\varphi_{U_i})$ on $V_i$ for 
every $i \in I$. 

It remains to check that 
$[\varphi^*(\beta)]=[c_1(\varphi^*(L)|_{W'},\varphi^*\metr)]_\res$ for  every 
morphism $\varphi:X' \rightarrow X$ and every open subset 
$W'$ of $\varphi^{-1}(W)$. 
This has to be tested on $\alpha \in B_c^{n-1,n-1}(W')$. 
The claim is local and a partition of unity argument in a paracompact open neighbourhood of $\supp(\alpha)$ shows that we may 
assume $\supp(\alpha) \subseteq \varphi^{-1}(V_i)$ for some $i \in I$. 
There are finitely many tropical charts $(V_j',\varphi_{U_j'})_{j \in J}$ 
in $W'$ which cover $\supp(\alpha)$ such that $\alpha$ is given on every 
$V_j'$ by {$\alpha_j \in AZ^{n-1,n-1}(V_j',\varphi_{U_j'})$}. 
We choose a non-empty very affine open subset $U'$ of $X'$ contained in 
every {$U_j'$ and in  $\varphi^{-1}(U_i)$}. 
By Proposition \ref{integration well-defined}, $U'$ is a very affine chart of integration  for both 
$\varphi^*(\beta) \wedge \alpha$ and $d'd''\alpha$. 
By construction, $V':=\bigcup_{j \in J} V_j'\cap \varphi^{-1}(V_i)  
\cap (U')^{\rm an}$ and $\varphi_{U'}$ form a tropical chart in $W'$. 
By Proposition \ref{single chart}, $\alpha$ is given on $\Omega':=\trop_{U'}(V')$ by 
{$\alpha_{U'} \in AZ^{n-1,n-1}(V',\varphi_{U'})$}. 
In the following, we will use only the $\delta$-preform 
$\alpha' \in P^{n-1,n-1}(\Omega')$ induced by $\alpha_{U'}$. 
For the tropical cycle $C':=\Trop(U')$ and the canonical affine map 
$F:N_{U',\R}\rightarrow N_{U_i,\R}$, it follows as above that 
$\varphi^*(\beta)$ is given on $V'$ by the element in 
$P^{1,1}(V',\varphi_{U'})$ represented by  
$\delta_{(\phi_i\circ F)\cdot N_{U',\R}} \in P^{1,1}(N_{U',\R})$. 
For {$\phi':=\phi_i\circ F|_{\Trop(U')}$}, we have seen that 
\[
\varphi^*(\beta_i)|_{\Omega'}=(\delta_{(\phi_i\circ F)\cdot N_{U',\R}})|_{\Omega'}
= \delta_{\phi'\cdot C'} \in P^{1,1}(\Omega').
\]
Note that $\supp(\alpha) \subseteq \bigcup_{j \in J} V_j'\cap \varphi^{-1}(V_i)$. 
We deduce from the generalizations of Corollary \ref{support corollary} and Proposition \ref{delta-support on chart} 
to $PSP$-forms (see \ref{psp on W})
that the currents  
$\dpa\phi' \wedge \alpha', \dpb \phi' \wedge \alpha', 
d'd''\alpha',\alpha' \wedge \delta_{\phi'\cdot C'}$ 
have compact support in $\Omega'$. 
{We write $C'=(\Ccal',m')$ for an integral $\Gamma$-affine 
polyhedral complex $\Ccal'$ and a family of integral weights $m'$.}
To prove $[c_1(\varphi^*(L)|_{W'},\varphi^*\metr)]_\res= [\varphi^*(\beta)]$, 
we have to show that
\begin{equation} \label{generalized tropical PL}
{\int_{|\Ccal'|} \phi' \wedge d'd''\alpha' 
= \int_{|\Ccal'|}\delta_{\phi'\cdot C'}\wedge \alpha' 
+ \int_{|\Ccal'|} \dpa\dpb\phi' \wedge \alpha'}
\end{equation}
holds. If $\alpha'$ has compact support in $\Omega'$, then this follows from 
the  tropical Poincar\'e--Lelong formula \eqref{tropical PL for covering}. 
In general, we still can deduce  from the proof of the tropical Poincar\'e--Lelong 
formula in {Theorem \ref{deltatropicalpoincarelelong}}
the formula \eqref{crucial residue formula} which here reads as
\[
{\int_{|\Ccal'|} \phi' \wedge d'd''\alpha' 
= -\int_{\partial |\Ccal'|} \dpb\phi' \wedge \alpha' + 
\int_{|\Ccal'|} \dpa\dpb \phi' \wedge \alpha'}
\]
as we have used only that $d'\alpha'$ and $\dpb \phi' \wedge \alpha'$ have 
compact support. 
Now \eqref{generalized tropical PL} follows from Lemma \ref{tropicallemma} 
and Remark \ref{generalization of tropicallemma} using additionally that   
$\alpha' \wedge \delta_{\phi'\cdot C'}$ has compact support.
\end{proof}

\begin{definition}  \label{delta-metrics}
A  metric $\metr$ on $L|_W$ is called a {\it $\delta$-metric} if for every 
$x \in W$, there are a tropical chart $(V,\varphi_U)$ such that 
$x \in V \subseteq W$ and a piecewise smooth function $\phi$
on $\Trop(U)$ satisfying the following properties: 
\begin{itemize}
\item[(i)] There is a nowhere vanishing section $s$ of $L$ over $U$ such that  
$\phi \circ \trop_U = -\log \|s\|$ on $V$.
\item[(ii)] There is a superform $\gamma$ on $N_{U,\R}$ of bidegree $(1,1)$ 
with piecewise smooth coefficients such that $\dpa\dpb\phi$ and $\gamma|_{\Trop(U)}$ 
agree on the open subset $\trop_U(V)$ of $\Trop(U)$. 
\end{itemize}
\end{definition} 

\begin{rem} \label{remarks for delta-metrics}
Condition (i) just means that the metric is piecewise smooth. 
Note that a superform on $N_{U,\R}$ with piecewise smooth coefficients 
is the same as a $\delta$-preform on $N_{U,\R}$ of codimension $0$ (see 
Example \ref{preforms of degree 0}). Using \ref{ps-part and residual part}, 
we deduce easily that (ii) is equivalent to the 
condition that $[c_1(L|_W,\metr)_\ps]$ is the $\delta$-current associated to a 
generalized $\delta$-form on $W$ (of codimension $0$).
\end{rem}

\begin{prop} \label{characterization of Chern delta-form}
Let $\metr$ be a piecewise smooth metric on $L|_W$. 
Then $\metr$ is a $\delta$-metric if and only if there is a $\beta \in B^{1,1}(W)$ with 
\[
[\varphi^*(\beta)]=[c_1(\varphi^*(L)|_{W'},\varphi^*\metr)] \in E^{1,1}((W')
\]
for every morphism $\varphi:X' \rightarrow X$ from any algebraic variety $X'$ over $K$ and for every open subset $W'$ of $\varphi^{-1}(W)$. 
\end{prop}

\begin{proof}
Suppose that $\metr$ is a $\delta$-metric. 
By Remark \ref{remarks for delta-metrics}, there is $\gamma \in P^{1,1}(W)$ 
of codimension $0$ such that $[c_1(L|_W,\metr)_\ps]=[\gamma]$. 
Since $\gamma$ is of codimension $0$, we may handle $\gamma$ as a piecewise 
smooth form and hence we get
\[
[\varphi^*(\gamma)]=[\varphi^*(c_1(L|_W,\metr)_\ps)]
=[c_1(\varphi^*(L)|_{W'},\varphi^*\metr)_\ps] \in E^{1,1}(W').
\]
Proposition \ref{residual Chern form} yields that 
$\beta:=c_1(L,\metr)_\res + \gamma \in P^{1,1}(W)$ and that 
\[
[\varphi^*(\beta)]= [c_1(\varphi^*(L)|_{W'},\varphi^*\metr)_\res] + [\varphi^*(\gamma)]
=[c_1(\varphi^*(L)|_{W'},\varphi^*\metr)] \in E^{1,1}(W')
\]
as claimed. It remains to show that $\beta \in B^{1,1}(W)$. 
Let $(V,\varphi_U)$ be a tropical chart in $W$ and let $\phi$ be a piecewise 
smooth function on $\Trop(U)$ as in Definition \ref{delta-metrics} such that 
$\beta|_V$ is given by $\beta_V \in P^{1,1}(V,\varphi_U)$. 
For every tropical chart $(U',\varphi_{U'})$ of an algebraic variety 
$X'$ over $K$ compatible with $(V,\varphi_U)$ with respect to the morphism 
$f:X' \rightarrow X$ and for $\Omega':=\trop_{U'}(V')$, the last display yields 
\begin{equation}\label{property of beta}
[f^*(\beta_V)|_{\Omega'}]=  d'd''[\phi \circ F] \in D^{1,1}(\Omega'),
\end{equation}
where $F:N_{U',\R} \rightarrow N_{U,\R}$ is the canonical affine map. 
Since this supercurrent is $d'$-closed and $d''$-closed on $\Omega'$, we 
conclude that  $\beta$ is given on $V$ by an element of $Z(V,\varphi_U)$. 
This shows $\beta \in B^{1,1}(W)$.

To prove the converse, we just use that $[c_1(L|_W,\metr)]=[\beta]$ for some 
$\beta \in P^{1,1}(W)$. By Proposition \ref{residual Chern form},  
$[c_1(L|_{W},\metr)_\ps]$ is the $\delta$-current associated to 
$\beta - c_1(L|_{W},\metr)_\res \in P^{1,1}(W)$. By 
Remark \ref{remarks for delta-metrics}, $\metr$ is a $\delta$-metric. 
\end{proof}

\begin{definition} \label{first Chern delta-form}  \label{properties of Chern delta-form}
Let $\metr$ be a $\delta$-metric on $L|_W$. 
By Proposition \ref{functorial criterion}, the $\delta$-form  $\beta$ in 
Proposition \ref{characterization of Chern delta-form} is unique. 
We call it the {\it first Chern $\delta$-form} of $(L|_W,\metr)$ and we 
denote it by $c_1(L|_W,\metr)$.
\end{definition}

\begin{art} \label{summary of Chern delta-construction}
We summarize the above constructions and definitions. 
A metric $\metr$ on $L|_W$ is 
a $\delta$-metric if and only if every $x \in W$ is contained 
in a tropical 
chart $(V,\varphi_U)$ in $W$ with a piecewise smooth function 
$\phi$ on $N_{U,\R}$ and 
a nowhere vanishing section $s$ of $L$ over $U$ such that
$$-\log \|s\| = \phi \circ \trop_U$$
on $V$ and such that
$$\dpa \dpb (\phi|_{\trop_U(V)})= \gamma|_{\trop_U(V)}$$
for a superform $\gamma$ on $N_\R$ of bidegree $(1,1)$ with piecewise 
smooth coefficients. Then 
the restriction of $c_1(L|_W,\metr)_{\rm res}$ to $V$ is represented by 
the $\delta$-preform $\delta_{\phi \cdot N_{U,\R}}$ on $N_{U,\R}$, 
$c_1(L|_W,\metr)_{\rm ps}|_V$ is given by $\gamma$ and  
 $c_1(L|_W,\metr)|_V$ is represented by the $\delta$-preform 
$\gamma + \delta_{\phi\cdot N_{U,\R}}$ on $N_{U,\R}$. 
A piecewise linear metric is a $\delta$-metric as we can choose 
$\phi$ integral $\Gamma$-affine (use 
Remark \ref{variations of extensions}) and $\gamma=0$. 
\end{art}

\begin{art}\label{first Chern delta-form properties}
By construction, the {$\delta$-current} associated to $c_1(L|_W,\metr)$ is 
equal to the first Chern current $[c_1(L|_W,\metr)]$ defined in 
\ref{first Chern current} which explains the notation used there.
{It is an immediate consequence of equation \eqref{property of beta}
that the first Chern $\delta$-form $c_1(L|_W,\metr)$ is $d'$-closed and $d''$-closed.}

To be a $\delta$-metric is a local property and respects isometry. 
The tensor product of $\delta$-metrics is again a $\delta$-metric and the dual 
metric of a $\delta$-metric is also a $\delta$-metric. 
If a positive tensor power of a metric $\metr$ on $L|_W$ is a $\delta$-metric, 
then $\metr$ is a $\delta$-metric. 
It is easy to see that the first Chern $\delta$-form $c_1(L|_W,\metr)$ is 
additive in terms of isometry classes $(L|_W,\metr)$ for $\delta$-metrics $\metr$.
\end{art}

\begin{prop} \label{pull-back and Chern form}
Let $\varphi:X' \rightarrow X$ be a morphism of algebraic varieties and 
let $L$ be a line bundle on $X$ endowed with a $\delta$-metric $\metr$ 
over the open subset $W$ of $\Xan$. 
Then $\varphi^*\metr$ is a $\delta$-metric on $\varphi^*(L)|_{W'}$ and we have
\begin{equation} \label{pull-back equality}
c_1(\varphi^*(L)|_{W'}, \varphi^*\metr)=\varphi^*c_1(L|_W,\metr) {\in B^{1,1}(W')}
\end{equation}
for any open subset $W'$ of $\varphi^{-1}(W)$.
\end{prop}

\begin{proof}
This follows from \ref{remarks for ps metrics} and 
Proposition \ref{characterization of Chern delta-form}.
\end{proof}

\begin{rem} \label{smooth and pl are delta-metrics}
Smooth metrics and piecewise linear metrics are $\delta$-metrics 
which is clear from the definitions. It follows 
from Proposition \ref{equivalences for formal metrics} that every 
formal metric is a $\delta$-metric. In particular, every algebraic metric 
on a line bundle of a proper variety is a $\delta$-metric.
\end{rem}

\begin{ex} \label{examples for delta-metrics}
All the canonical metrics in Example \ref{canonical metric on abelian variety}, 
in Example \ref{algebraically equivalent to zero} and in Example 
\ref{toric varieties} are $\delta$-metrics. Indeed, a positive tensor power 
of such a metric is locally the tensor product of a formal metric with a 
smooth metric and hence the claim follows from 
Remark \ref{smooth and pl are delta-metrics}. 
\end{ex}

\section{Monge-Amp\`ere measures} \label{MAm}

We have seen in the previous section that formal metrics are $\delta$-metrics 
giving rise to a first Chern $\delta$-form. 
The formalism of $\delta$-forms allows us to define the Monge-Amp\`ere measure
as a wedge-product of first Chern $\delta$-forms.
We recall that Chambert--Loir has introduced discrete measures for formally 
metrized line bundles on a proper variety which are important 
for non-archimedean equidistribution. 
The main result of this section shows that the Monge--Amp\`ere measure is equal to 
the Chambert--Loir measure.

In this section $X$ is a proper algebraic variety over $K$ of dimension $n$.

\begin{art} \label{Monge-Ampere}
Let $\overline{L_1}, \dots , \overline{L_n}$ be line bundles on $X$  
endowed with $\delta$-metrics. 
Then the wedge-product $c_1(\overline{L_1}) \wedge  \dots \wedge c_1(\overline{L_n})$ of the first Chern $\delta$-forms is 
a $\delta$-form of bidegree $(n,n)$. 
By Corollary \ref{deltaform defines-measure}, the $\delta$-current 
associated to  a $\delta$-form on $\Xan$ of type $(n,n)$ extends 
to a bounded linear functional on the space of continuous functions 
and defines a signed Radon measure on $\Xan$. 
The signed measure associated to 
$c_1(\overline{L_1}) \wedge  \dots \wedge c_1(\overline{L_n})$ is called 
the {\it Monge-Amp\`ere measure} and is denoted by 
\[
{\rm MA}\bigl(c_1(\overline{L_1}), \dots ,c_1(\overline{L_n})\bigr).
\]
\end{art}

\begin{prop} \label{projection formula for Monge-Ampere}
If $\varphi: X' \rightarrow X$ is a  morphism of $n$-dimensional proper  
varieties over $K$, then the following projection formula holds:
\[
\varphi_*{\rm MA}\bigl(c_1(\varphi^*\overline{L_1}), \dots , c_1(\varphi^*\overline{L_n})\bigr)
=\deg(\varphi) {\rm MA}\bigl(c_1(\overline{L_1}), \dots ,c_1(\overline{L_n})\bigr).
\]
\end{prop}

\begin{proof} 
{The Stone-Weierstraß Theorem
\cite[Prop. (3.3.5)]{chambert-loir-ducros} implies that
$A^0(\Xan)$ is a dense subspace of $C(\Xan)$. 
For functions in $A^0(\Xan)$ the desired equality 
follows from Proposition \ref{pull-back and Chern form} and 
from the projection formula for $\delta$-forms \eqref{integration well-definedg1}.
This yields our claim.}
\end{proof}

\begin{prop} \label{total mass of Monge-Ampere}
If $X$ is a proper variety of dimension $n$, then the total mass of 
${\rm MA}\bigl(c_1(L_1, \metr_1), \dots , c_1(L_n, \metr_n)\bigr)$ is equal to 
$\deg_{L_1, \dots, L_n}(X)$.
\end{prop}

\begin{proof} 
This follows similarly as in 
\cite[Proposition 6.4.3]{chambert-loir-ducros}. 
They handled there only the case of smooth metrics, but our formalism of $\delta$-forms allows us to obtain this result more generally for $\delta$-metrics.
\end{proof}

We  recall the crucial properties of Chambert-Loir's measures. 
They were introduced in a slightly different setting 
by Chambert--Loir in  \cite{chambert-loir-2006}.

\begin{prop} \label{Chambert-Loir's measures}
There is a unique way to associate to any $n$-dimensional proper variety $X$ over $K$ and to any family of formally metrized line bundles $\overline{L_1}, \dots, \overline{L_n}$ on $X$  a  {signed Radon} measure $\mu=\mu_{\overline{L_1}, \dots, \overline{L_n}}$ on $X^{\rm an}$ such that the following properties hold:
\begin{itemize}
\item[(a)] The measure $\mu$ is multilinear and symmetric in $\overline{L_1}, \dots, \overline{L_n}$.
\item[(b)] If $\varphi:Y \rightarrow X$ is a morphism of $n$-dimensional  proper varieties over ${K}$, 
then the following projection formula holds:
\[
\varphi_* \left( \mu_{\varphi^* \overline{L_1}, \dots , \varphi^* \overline{L_n}} \right) = 
\deg(\varphi) \mu_{\overline{L_1}, \dots, \overline{L_n}}.\]
\item[(c)] If $\Xcal$ is a formal $\kcirc$-model of $X$ with reduced 
 {special fibre $\Xcal_s$} and 
if the metric of $\overline{L_j}$ is induced by a formal $\kcirc$-model $\Lcal_j$ of $L_j$ on $\Xcal$ for every $j=1, \dots, n$, then
\[\mu= 
\sum_Y\deg_{ {{\Lcal}_1, \dots, {\Lcal}_n}}(Y) \delta_{\xi_Y},
\]
where $Y$ ranges over the irreducible components of  {${\Xcal}_s$} and $\delta_{\xi_Y}$ is the Dirac measure in the unique point $\xi_Y$ of $\Xan$ which reduces to the generic point of $Y$ (see \cite[Proposition 2.4.4]{berkovich-book}).
\item[(d)] 
The total mass is given by $\mu(\Xan)=\deg_{L_1,\dots,L_n}(X)$. 
\end{itemize}
\end{prop}

\begin{proof} For existence, we refer to \cite[\S 3]{gubler-2007a}. Uniqueness follows from (c) alone as the existence of a simultaneous  formal $\kcirc$-model with reduced special fibre is a consequence of  \cite[Proposition 7.5]{gubler-crelle}. 
\end{proof}

\begin{thm} \label{Monge-Ampere vs Chambert-Loir}
For formally metrized line bundles $\overline{L_1}, \dots, \overline{L_n}$ on the  proper variety $X$ of dimension $n$, the Monge-Amp\`ere measure ${\rm MA}(c_1(\overline{L_1}), \dots , c_1(\overline{L_n}))$ agrees with the Chambert-Loir measure $\mu_{\overline{L_1}, \dots, \overline{L_n}}$.
\end{thm}

This theorem was first proven by Chambert-Loir and Ducros \cite[\S 6.9]{chambert-loir-ducros} for their  
Monge-Amp\`ere measures defined by a tricky approximation process with smooth metrics. Their argument uses Zariski--Riemann spaces, while we use here a more tropical approach related to our $\delta$-forms.

\begin{art} \label{notation for lemma}
In Lemma \ref{tropicalization and special fibre}, we will consider a closed subvariety $\Ucal$ of 
a torus $\T = \Spec(\kcirc[M])$ over $\kcirc$. We will use the following notation: $N$ is the dual of the 
free abelian group $M$ of finite rank. Let $U$ be the generic fibre of $\Ucal$ and let $\Ucal_s$ be the special fibre.

The tropicalization $\trop:(\T_K)^{\rm an} \rightarrow N_\R$ (resp. $\trop:\T_s^{\rm an} \rightarrow N_\R$) with respect to the valuation $v$ on $K$ (resp. the trivial valuation on $\ktilde$) leads to the tropical variety $\Trop(U)$ (resp. $\Trop(\Ucal_s)$). 

The local cone ${\rm LC}_0(\Trop(U))$ at $0$ is defined as the cone in $N_\R$ which agrees with $\Trop(U)$ in a neighbourhood of $0$. We endow it with the weights induced by {the canonical tropical weights on} $\Trop(U)$.

\end{art}

\begin{art} \rm \label{preparatory setup}
For the proof of Theorem \ref{Monge-Ampere vs Chambert-Loir}, we need a preparatory result. Let $L$ be a line bundle on the proper variety $X$ over $K$. 
We consider an algebraic $\kcirc$-model $(\Xcal,\Lcal)$ of $(X,L)$. Then we get an algebraic metric $\metr_\Lcal$ on $L$. 
We have seen in \ref{trivial ps} that the restriction $\Lcal_s$ of $\Lcal$ to the special fibre $\Xcal_s$ has a canonical metric $\metr_{\rm can}$. Note that the first metric is continuous on the Berkovich space $X^{\rm an}$ with respect to the given valuation $v$ while $\metr_{\rm can}$ is continuous on the Berkovich space $\Xcal_s^{\rm an}$ with respect to the trivial valuation on the residue field $\ktilde$. Since $\Xcal$ is assumed to be proper, we have a reduction map $\pi:X^{\rm an} \rightarrow \Xcal_s$. For $x \in X^{\rm an}$, 
$\pi(x)$ is a scheme theoretic point of $\Xcal_s$. Using the trivial valuation on the residue field of $\pi(x)$, we will view $\pi(x)$ 
as a point of $\Xcal_s^{\rm an}$. 
\end{art}

In the next lemma, we will show that $\metr_{\rm can}$ is piecewise linear in an  
{neighbourhood} of $\pi(x)$ in $\Xcal_s^{\rm an}$. This means that using a trivialization and a tropicalization, the canonical metric is induced by a piecewise linear function on the tropical variety.  It will be crucial in the proof of Theorem \ref{Monge-Ampere vs Chambert-Loir} that we can use tropically the same piecewise linear function {to describe} the formal metric $\metr_\Lcal$ in {a neighbourhood of $x$ in $X^\an$}. We make this now precise:

\begin{lem} \label{tropicalization and special fibre}
Under the setup given in \ref{preparatory setup}, we {fix an element $x\in X^\an$} and an open neighbourhood $\Vcal$ of $\pi(x)$ in $\Xcal$. Then there is an open neighbourhood $\Ucal$ of $\pi(x)$ in $\Vcal$ and a closed embedding $\Ucal \hookrightarrow \T$ into a torus $\T = \Spec(\kcirc[M])$ with the following properties (a)--(f) 
using the notation from \ref{notation for lemma}:
\begin{itemize}
\item[(a)]
{We have $0=\trop(x)$ and the weighted local cone in $0$ satisfies} 
\[{\rm LC}_0(\Trop(U))=\Trop(\Ucal_s).\]
\item[(b)] There is an open neighbourhood $\widetilde{\Omega}$ of $0$ in $N_\R$ with 
\[{\rm LC}_0(\Trop(U))  \cap \widetilde{\Omega} = \Trop(U) \cap \widetilde{\Omega}.\]
\item[(c)] There {exist a complete rational polyhedral fan $\Sigma$ on $N_\R$ 
and a continuous function $\phi:N_\R \rightarrow \R$ which is piecewise linear with 
respect to $\Sigma$ (i.e. for every $\sigma \in \Sigma$, there is 
$u_\sigma \in M$ with $\phi= u_\sigma$ on $\sigma$).}
\item[(d)] $\Ucal$ is a trivialization of $\Lcal$ with respect to a nowhere vanishing section $s \in \Gamma(\Ucal, \Lcal)$.
\item[(e)] We have $-\log \|s\|_\Lcal = \phi \circ \trop$  on a neighbourhood of $x$ in $X^{\rm an}$.
\item[(f)] We have $-\log \|s\|_{\rm can} = \phi \circ \trop$ on a neighbourhood of $\pi(x)$ in $\Xcal_s^{\rm an}$. 
\end{itemize}
If $\pi(x)$ is the generic point of an irreducible component of $\Xcal_s$, then there is an $\Ucal$ as above with (a)--(d) and the following stronger properties:
\begin{itemize}
\item[(e')] We have $-\log \|s\|_\Lcal = \phi \circ \trop$  on $\trop^{-1}(\widetilde{\Omega}) \subseteq \Uan$.
\item[(f')] The identity $-\log \|s\|_{\rm can} = \phi \circ \trop$ holds on $\Ucal_s^{\rm an}$.
\end{itemize}
\end{lem}

\begin{proof}
Let $(\Ucal_i)_{i \in I}$ be a finite affine open covering of $\Xcal$ 
such that $\Lcal$ is trivial over any $\Ucal_i$. 
The generic (resp. special) fibre of $\Ucal_i$ is denoted by 
$U_i$ (resp. $\Ucal_{i,s}$).
For every $i \in I$, we choose a nowhere vanishing {section} 
$s_i \in \Gamma(\Ucal_i,\Lcal)$. 
Let $I(x):=\{i \in I \mid \pi(x) \in \Ucal_{i,s}\}$. 
For $i \in I(x)$, let $(x_{ij})_{j \in J_i}$ be  a finite set of 
generators of the $\kcirc$-algebra $\Ocal(\Ucal_i)$. 
Replacing $x_{ij}$ by $1+x_{ij}$ if necessary, 
we may assume that these generators are invertible in $\pi(x)$. 
For $i\in I$, we have an affinoid subdomain  
$$U_i^\circ := \{z \in U_i^{\rm an} \mid |a(z)| \leq 1 \; \forall a \in \Ocal(\Ucal_i)\}= \{z \in U_i^{\rm an} \mid \pi(z) \in \Ucal_{i,s}\}$$
of $X^{\rm an}$. Using the trivial valuation on $\ktilde$, we get similarly an affinoid subdomain $\Ucal_{i,s}^\circ := \{z \in \Ucal_{i,s}^{\rm an} \mid |a(z)| \leq 1 \; \forall a \in \Ocal(\Ucal_{i,s})\}$ of $\Xcal_s^{\rm an}$. We consider $\pi(x)$ as a point of $\Xcal_s^{\rm an}$ by using the trivial absolute value on the residue field of $\pi(x)$ and 
hence we have $I(x)= \{i \in I \mid \pi(x) \in \Ucal_{i,s}^\circ \}$. 

It is easy to see that $\pi(x)$ has a very affine open neighbourhood $\Ucal$ in $\Xcal$ such that $\Ucal$ is contained in $\Ucal_i$ for every $i \in I(x)$. Very affine means that there 
is a closed embedding $\varphi:\Ucal \hookrightarrow \T$ into a torus $\T = \Spec(\kcirc[M])$. 
By shrinking $\Ucal$ and by adding new invertible functions to $\varphi$, we obtain the following properties:
\begin{itemize}
 \item[(i)] For every $i,k \in I(x)$, the invertible meromorphic function $s_i/s_k$ on $\Ucal$ is the restriction of a character $\chi^{u_{ik}}$ associated to {some} $u_{ik} \in M$.
 \item[(ii)] For every $i \in I(x)$ and every $j \in J_i$, the generator $x_{ij}$ is invertible on $\Ucal$ and equal to the restriction of a character $\chi^{u_{ij}'}$ associated to {some} $u_{ij}' \in M$.
\end{itemize}
Note that we have $0=\trop(\pi(x)) \in \Trop(\Ucal_s)$ since we use the trivial valuation on the residue field of $\pi(x)$. It follows from $\pi(x) \in \Ucal_s$ that $\trop(x)=0$. 
By definition, $\Ucal_s$ is the initial degeneration of $U$ at $0$ and hence (a) follows from \cite[Propositions 10.15, 13.7]{gubler-guide}. 
By definition of the local cone, we find an open neighbourhood $\widetilde{\Omega}$ of $0$ in $N_\R$ with (b). 

By construction, {$\Lcal$  is trivial over $\Ucal$} and we choose $s:=s_k$ for a fixed $k \in I(x)$ in (d).
For $i \in I(x)$, we define the rational cone $\sigma_{i}:=\{\omega \in N_\R \mid \langle \omega, u_{ij}' \rangle \geq 0 \; \forall j \in J_i\}$ in $N_\R$. Then (ii) yields 
\begin{equation} \label{circ and cone1}
 \Ucal_{i,s}^\circ \cap \Ucal_s^{\rm an} = \trop^{-1}(\sigma_i) \cap \Ucal_s^{\rm an}.
\end{equation}
By the Bieri--Groves theorem, $\Trop(\Ucal_s)$ is the support of a rational polyhedral fan in $N_\R$ (see \cite[Remark 3.4]{gubler-guide}). We conclude that there is a complete rational polyhedral fan $\Sigma$ on $N_\R$  and a rational polyhedral subfan $\Sigma_x$ with $|\Sigma_x|=\bigcup_{i \in I(x)} \sigma_i \cap \Trop(\Ucal_s)$ such that every cone $\sigma \in \Sigma_x$ is 
contained in $\sigma_i$ for some $i \in I(x)$. Note that $\|s_i\|_{\rm can}=1$ on $\Ucal_{i,s}^\circ$ and hence (i) shows that  
\begin{equation} \label{eq1}
 -\log\|s\|_{\rm can}= -\log |s_k/s_i| = u_{ki}\circ \trop 
\end{equation}
on  $ \Ucal_{i,s}^\circ \cap \Ucal_s^{\rm an} $. By \eqref{circ and cone1}, there is a continuous function $\phi:|\Sigma_x| \rightarrow \R$ with $\phi = u_{ki}$ on every $\sigma$. Using Remark \ref{variations of extensions}  and passing to a refinement of $\Sigma$, we easily extend $\phi$ to a continuous function on $N_\R$ satisfying (c). 
Since $\Xcal_s$ is proper over $\ktilde$, the sets $\Ucal_{i,s}^\circ$, $i \in I$, form an open covering of $\Xcal_s^{\rm an}$.  
It follows from \eqref{circ and cone1} and \eqref{eq1} that (f) holds in the neighbourhood $W:=\Ucal_s^{\rm an} \setminus \bigcup_{i \in I \setminus I(x)} \Ucal_{i,s}^\circ$ of $\pi(x)$ in $\Xcal_s^{\rm an}$.

Again (ii) shows
\begin{equation} \label{circ and cone2}
 U_i^\circ \cap U^{\rm an} = \trop^{-1}(\sigma_i) \cap U^{\rm an} .
\end{equation}
for every $i \in I(x)$. Note  that $\|s_i\|_\Lcal=1$ on $U_i^\circ$ and hence (i) shows that   
\begin{equation} \label{eq2}
 -\log\|s\|_{\Lcal}= -\log |s_k/s_i| = u_{ki}\circ \trop 
\end{equation}
on  $ U_i^\circ \cap U^{\rm an} $. Since $X$ is proper, the sets $U_i^\circ$, $i \in I$, form a compact covering of $X^{\rm an}$. It follows from (a), (b), \eqref{circ and cone2} and \eqref{eq2} that (e) holds in the neighbourhood 
$\trop^{-1}(\widetilde{\Omega}) \setminus \bigcup_{i \in I \setminus I(x)} U_i^\circ$ of $x$ in $X^{\rm an}$. This proves (e).

We assume now that $\pi(x)$ is the generic point of an irreducible component $Y$ of $\Xcal_s$. 
Then we may assume that $\Ucal_s \subseteq Y$. Let $i \in I$ with $\Ucal_{i,s} \cap Y \neq \emptyset$. Since we use the trivial valuation 
on the residue field of $\pi(x)$, we deduce easily that $\pi(x) \in \Ucal_{i,s}^\circ$. By construction, we get $W=\Ucal_s^{\rm an}$ 
proving (f'). It remains to show (e'). Let $i \in I \setminus I(x)$. By 
construction,  we have $\Ucal_{i,s} \cap Y = \emptyset$. For $y \in U_i^\circ \cap \Uan$, we have $\pi(y) \in \Ucal_{i,s}$ and hence $\pi(y) \not \in Y$. In particular, we have $y \not \in U^\circ$. Using  $\trop^{-1}(0) \cap \Uan = U^\circ$, we see that  $\trop(U_i^\circ \cap \Uan)$ is a closed subset of $\Trop(U)$ not containing $0$. By shrinking $\widetilde{\Omega}$, we may assume that $\widetilde{\Omega}$ is a neighbourhood of $0$ which is disjoint from $\trop(U_i^\circ \cap \Uan)$ for every $i \in I \setminus I(x)$. Then the above proof of (e) shows that (e') holds. 
\end{proof}

Now we are ready to prove Theorem \ref{Monge-Ampere vs Chambert-Loir}.

\begin{proof}  Let $\mu^{\rm MA}:={\rm MA}(c_1(\overline{L_1}), \dots ,c_1(\overline{L_n}))$. 
For simplicity, we assume that 
$L=L_1= \dots = L_n$ and that all metrics are induced by the same $\kcirc$-model $\Lcal$ on $\Xcal$. The general case follows 
either by the same arguments or by multilinearity. 
It is more convenient for us to work algebraically and so we use Proposition 
\ref{algebraic and formal metrics} to assume that $\Xcal$ and $\Lcal$ are 
algebraic $\kcirc$-models. There is a generically finite surjective morphism 
$\Xcal' \rightarrow \Xcal$ from a proper flat variety $\Xcal'$ over $\kcirc$ 
with reduced special fibre. 
This is a consequence of de Jong's  pluristable alteration theorem 
which works over any Henselian valuation ring 
(see \cite[Lemma 9.2]{berkovich-1999}).
Since both sides of 
the claim satisfy the projection formula, we may prove the claim for $\Xcal'$. 
This shows that we may assume that $\Xcal$ is an algebraic $\kcirc$-model of 
$X$ with reduced special fibre.

We will analyse $\mu^{\rm MA}$ 
in a neighbourhood of $x \in \Xan$. Let $\pi(x) \in \Xcal_s$ be the 
reduction of $x$. We choose a  
very affine open neighbourhood $\Ucal$ of $\pi(x)$ in $\Xcal$ as in 
Lemma \ref{tropicalization and special fibre}. We will use the closed embedding 
$\Ucal \hookrightarrow \T$ into the torus $\T$ and the notation from there. 

It follows from a theorem of Ducros \cite[Th\'eor\`eme 3.4]{ducros} that $x$ has a compact analytic neighbourhood $V$
such that the germ of $\trop(V)$ in $\trop(x)$ (considering polytopal neighbourhoods) agrees with the 
germ of $\trop(W)$ in $\trop(x)$ for every compact analytic neighbourhood $W \subseteq V$ of $x$. Using that $\trop(x)=0$, we deduce from 1) in \cite[Th\'eor\`eme 3.4]{ducros}
that the dimension of the germ is equal to the transcendence degree of $\ktilde(\pi(x))$ over $\ktilde$. 

We first assume that $\pi(x)$ is not the generic point of an irreducible 
component of $\Xcal_s$. Then the transcendence degree of $\ktilde(\pi(x))$ 
over $\ktilde$ is smaller than $n=\dim(X)$. 
Using the theorem of Ducros,  there is a compact analytic neighbourhood $V$ 
of $x$ in $X^{\rm an}$ such that $\trop(V)$ has dimension $<n$ and such that 
Lemma \ref{tropicalization and special fibre}(e) holds on $V$. 
We choose a tropical chart $(V',\varphi_{U'})$ in $x$ which is contained 
in $V$ and with $U'$ contained in the generic fibre $U$ of $\Ucal$. 
We describe $c_1(\overline L)|_{V'}$ using the function $\phi$ constructed in
Lemma \ref{tropicalization and special fibre} and the canonical affine map 
$F:N_{U',\R} \rightarrow N_\R$.
Using that $\phi$ is piecewise linear, it follows from \ref{summary of Chern delta-construction} 
that $c_1(\overline{L})|_{V'}={\trop}_{U'}^*(\beta)$ for $\beta\in AZ^{1,1}(V',\varphi_{U'})$ represented by
the $\delta$-preform $\delta_{F^*(\phi)\cdot N_{U',\R}}$ on $N_{U',\R}$.
By our construction of products and Corollary \ref{pullbackcornerlocus}, 
$\mu^{\rm MA}$ is given on $V'$ by $\beta^{\wedge n} \in  AZ^{n,n}(V',\varphi_{U'})$ represented by the $\delta$-preform
\begin{equation} \label{mu representation}
\left(\delta_{F^*(\phi)\cdot N_{U',\R}}\right)^{\wedge n} = \delta_{F^*(C)}
\end{equation}
on $N_{U',\R}$, where $C$ is the $n$-codimensional tropical cycle of $N_\R$ obtained by the $n$-fold self-intersection of the tropical divisor $\phi \cdot N_\R$.
Since $V' \subseteq V$, we have $F(\trop_{U'}(V')) \subseteq \trop(V)$ and 
hence $\dim(F(\trop_{U'}(V')))<n$. It follows   
from the definition of the pull-back and the local nature of stable tropical intersection that 
$\delta_{F^*(C) \cdot\Trop(U')}$ does not meet the open subset $\trop_{U'}(V')$ of $\Trop(U')$. A similar argument applies 
to any tropical chart compatible with $(V',\varphi_{U'})$ and hence \eqref{mu representation} yields $\beta^{\wedge n}=0$. We conclude that 
the support of $\mu^{\rm MA}$ does not meet $V'$.

Now we assume that $\pi(x)$ is the generic point of an irreducible component $Y$ of $\Xcal_s$. Then $x$ is  the unique point of $\Xan$ with reduction $\pi(x)$ (see \cite[Proposition 2.4.4]{berkovich-book}) and we write  $x=\xi_Y$. 
We may assume that the very affine open neighbourhood $\Ucal$ of $\pi(x)$ in $\Xcal$ from Lemma \ref{tropicalization and special fibre} has special fibre $\Ucal_s$ disjoint from all other irreducible components $Y'$ of $\Xcal_s$. We conclude that $\pi(\xi_{Y'}) \not \in \Ucal_s$ and hence $\trop(\xi_{Y'}) \neq 0=\trop(x)$. We may 
choose the neighbourhood $\widetilde{\Omega}$ of $0$ in $N_\R$ disjoint from all points $\trop(\xi_{Y'})$. 
We will use in the following that Lemma \ref{tropicalization and special fibre}(e') holds on the open subset 
$V:=\trop^{-1}(\widetilde{\Omega})$ of $\Xan$. Since no $\xi_{Y'}$ is contained in $V$, 
the non-generic case above shows that the restriction of  $\mu^{\rm MA}$ to $V$ is supported in $\xi_Y$. 

Now we choose a very affine open subset $U'$ contained in the generic fibre $U$ of $\Ucal$ with $x \in (U')^{\rm an}$. 
Let $F:N_{U',\R} \rightarrow N_\R$ be the canonical affine map. Then 
$$V':=\trop_{U'}^{-1}(F^{-1}(\widetilde{\Omega}))= (U')^{\rm an} \cap V.$$
is an open neighbourhood of $x$ in $\Xan$ and $(V',\varphi_{U'})$ is a tropical chart. 
Similarly as above,  $\mu^{\rm MA}$ is given on $V'$ by $\beta^{\wedge n} \in  AZ^{n,n}(V',\varphi_{U'})$ represented by the $\delta$-preform in \eqref{mu representation}. 
Since $\mu^{\rm MA}|_{V'}$ is supported in the single point $x=\xi_Y$, we conclude that the $0$-dimensional tropical cycle $F^*(C)\cdot \Trop(U')$ has only one point $\omega'$ contained in the open subset $\trop_{U'}(V')$ of $\Trop(U')$. 
In fact, we have $\omega' = \trop_{U'}(x)$ with multiplicity $\mu^{\rm MA}(V')$. 
The tropical projection formula in Proposition \ref{tropintthprop} and the 
Sturmfels--Tevelev multiplicity formula \cite[Theorem 13.17]{gubler-guide} give the identity
\[
F_*(F^*(C)\cdot \Trop(U'))= C \cdot \Trop(U)
\]
of tropical cycles on $N_\R$. Using that $\trop_{U'}(V')=F^{-1}(\trop(V)) \cap \Trop(U')$, we deduce that 
 $\mu^{\rm MA}(V')$ is equal to the multiplicity of $0=\trop(x)=F(\omega')$ in 
$C \cdot \Trop(U)$. By Lemma \ref{tropicalization and special fibre}, we conclude that $\mu^{\rm MA}(V')$ 
is equal to the tropical intersection number of $C$ with $\LC_0(\Trop(U))=\Trop(\Ucal_s)$.

We recall that $C$ is the $n$-fold self-intersection of the tropical divisor $\phi\cdot N_\R$ and we note that 
these objects are weighted tropical fans. 
Now we use Lemma \ref{tropicalization and special fibre}(f'). This shows that $\Ucal_s$ is 
a very affine chart of integration for  $c_1(\Lcal_s|_Y,\metr_{\rm can})^n$, where this $\delta$-form is 
represented by the pull-back of $\delta_C$ with respect to the canonical affine map $N_{\Ucal_s,\R} \rightarrow 
N_\R$. Note that we may perform a base change to omit the trivial 
valuation which was excluded for simplicity in our paper. 
The tropical projection formula and 
the Sturmfels--Tevelev multiplicity formula  show 
$$\int_{Y^{\rm an}}c_1(\Lcal_s|_Y,\metr_{\rm can})^n = \deg(C \cdot \Trop(\Ucal_s))$$
as above.  By  Proposition \ref{total mass of Monge-Ampere}, the left hand side is equal to 
$\deg_\Lcal(Y)$. We have seen above that the right hand side equals $\mu^{\rm MA}(V')$. 
This proves that $\mu^{\rm MA}|_V$ is a point measure  concentrated in $x=\xi_Y$ with total mass $\deg_\Lcal(Y)$. 
By Proposition \ref{Chambert-Loir's measures}(c),   the Chambert-Loir measure  $\mu_{\overline{L_1}, \dots, \overline{L_n}}$ is equal to $\mu^{\rm MA}$. 
\end{proof}

\section{Green currents} \label{GC}

In this section $X$ is an algebraic variety 
over $K$ of dimension $n$. 
We introduce Green currents for cycles on $X$. 
We define the product $g_Y * g_Z$ for a divisor $Y$ and a cycle $Z$ on 
$X$ which intersect properly.
This operation has the expected properties.

\begin{definition} \label{green current definition}
Let $Z$ be a cycle of $X$ of codimension $p$ and let $g$ be any 
$\delta$-current in $E^{p-1,p-1}(X^\an)$. 
Then we define {$\omega(Z,g):=d'd''g + \delta_Z \in E^{p,p}(X^\an)$}. 
If there is a $\delta$-form {$\omega_{Z,g} \in B^{p,p}(X^\an)$} with 
$\omega(Z,g)=[\omega_{Z,g}]$, then we call $g$ a {\it Green current} 
for the cycle $Z$. We will use often the notion $g_Z$ for such a current and 
then we set $\omega(g_Z):=\omega(Z,g_Z)$ and $\omega_Z:=\omega_{Z,g_Z}$ 
for simplicity.
\end{definition}

\begin{art} \label{star product}
Let $(L,\metr)$ be a line bundle on $X$ endowed with a $\delta$-metric and 
let $Z$ be a cycle of codimension $p$ in $X$ with any current 
$g_Z \in E^{p-1,p-1}(X^\an)$. 
We assume that $s$ is a meromorphic section of $L$ with Cartier divisor $D$ 
intersecting $Z$ properly. By the Poincar\'e--Lelong equation in Corollary 
\ref{PL for line bundles}  and  by the definition of a $\delta$-metric in \ref{properties of Chern delta-form}, 
$g_Y:= [-\log\|s\|]$ is a Green current  for the Weil divisor $Y$ associated 
to $D$ with $\omega_Y=c_1(L, \metr)$. 

If $Z$ is a prime cycle {of codimension $p$}, then we
define $g_Y \wedge \delta_Z \in E^{p,p}(\Xan)$ as the push-forward of 
$[-\log\|s\||_Z]$ with respect to the inclusion $i_Z:Z \rightarrow X$. 
In general, we proceed by linearity in the prime components of $Z$ to define 
$g_Y \wedge \delta_Z \in E^{p,p}(\Xan)$. 
This leads to the definition of the {\it $*$-product}
\[
g_Y * g_Z := g_Y \wedge \delta_Z + \omega_Y \wedge g_Z \in E^{p,p}(\Xan).
\]
\end{art}

\begin{lem} \label{Poincare Lelong and restriction}
Under the hypothesis above and if $Z$ is prime, then we have the identity
\begin{equation} \label{restriction equation1}
d'd''[-\log\|s\||_Z]=[\omega_Y|_Z] - \delta_{D\cdot Z}
\end{equation}
of $\delta$-currents on $\Zan$. 
\end{lem}

\begin{proof} This follows immediately from the Poincar\'e--Lelong equation for $s|_Z$ 
{(see Corollary \ref{PL for line bundles}}). 
We use here  $c_1(L|_Z,\metr)=c_1(L,\metr)|_Z$ which follows from 
Proposition \ref{pull-back and Chern form}.
\end{proof}

\begin{prop} \label{star product is Green current}
Under the hypothesis in \ref{star product}, we have 
$$\omega(D\cdot Z,g_Y * g_Z)=\omega_Y \wedge \omega(g_Z).$$ 
If $g_Z$ is a Green current for $Z$, then $g_Y * g_Z$ is a Green current for $D\cdot Z$.
\end{prop}

\begin{proof} 
Using Lemma \ref{Poincare Lelong and restriction} and linearity in the prime components of $Z$, we get
\begin{equation} \label{restriction equation2}
d'd''[-\log\|s\| \wedge \delta_Z]=\omega_Y \wedge \delta_Z - \delta_{D\cdot Z}
\end{equation}
and hence {\ref{first Chern delta-form properties} and Proposition \ref{properties of delta-forms}(iii) give}
\[
\omega(D\cdot Z,g_Y * g_Z)=d'd''[-\log\|s\| \wedge \delta_Z]+
d'd''(\omega_Y \wedge g_Z)+\delta_{D\cdot Z}=\omega_Y \wedge \omega(g_Z)
\] 
proving the claim.
\end{proof}

\begin{prop} \label{commutativity}
For $i=1,2$, let $L_i$ be a line bundle on $X$ with a $\delta$-metric $\metr_i$ and non-zero meromorphic section $s_i$. We assume that the associated Cartier divisors $D_1$ and $D_2$ intersect properly. Let $\eta_{Y_i}:=-\log\|s_i\|_i$ and let $g_{Y_i}=[\eta_{Y_i}]$ be the induced Green current for the Weil divisor $Y_i$ of $D_i$. Then we have the identity
\[
g_{Y_1} * g_{Y_2} - g_{Y_2} * g_{Y_1} 
= d'[\dpb\eta_{Y_1} \wedge \eta_{Y_2}] + d''[ \eta_{Y_1} \wedge \dpa \eta_{Y_2}]
\]
of $\delta$-currents on $\Xan$. 
\end{prop}

Note that the piecewise smooth forms $\dpb\eta_{Y_1} \wedge \eta_{Y_2}$ 
and $\eta_{Y_1} \wedge \dpa \eta_{Y_2}$ of degree $1$ are defined on the 
analytification of a Zariski open and dense subset of $X$. 
By \ref{psp on W} and  Proposition \ref{currents and Zariski dense}, they define $\delta$-currents on $\Xan$.

\begin{proof} 
The claim can be checked locally {on $X$. Hence we may assume that
$X$ is affine} and $L_1=L_2=O_X$. For $i=1,2$, we may view $s_i$ as a rational function $f_i$ and we have 
\begin{equation*} \label{eta representation}
\eta_{Y_i} =  -\log|f_i|-\log\|1\|_i.
\end{equation*} 
The usual partition of unity argument shows that it is enough to test the claim 
by evaluating at  $\alpha \in B^{n-1,n-1}_c(W)$ for a small open 
neighbourhood $W$ of a given point $x$ in $\Xan$. 
There are finitely many tropical charts $((V_j,\varphi_{U_j})_{j=1, \dots, m}$ 
in $W$ covering $\supp(\alpha)$ such that 
$\alpha = \trop_{U_j}^*(\alpha_j)$ on $V_j$ for {some element}
$\alpha_j \in AZ^{n-1,n-1}(V_j,\varphi_{U_j})$. 
We will use a Zariski dense very affine open subset $U$ of 
$U_1 \cap \dots \cap U_m$ which will serve as a very affine chart of 
integration for various forms. 
Now we consider the restriction of the canonical affine map $F_j:N_{U,\R} \rightarrow N_{U_j,\R}$ to $\Trop(U)$.
{Let $\Omega$ in ${\rm Trop}\,(U)$ denote the union of the  preimages of the open subsets $\Omega_j:=\trop_{U_j}(V_j)$ 
in $\Trop(U_j)$ and put $V={\rm trop}_U^{-1}(\Omega)$.
By Proposition \ref{finite sheaf property} there exists
a unique element $\alpha_U\in AZ^{n-1,n-1}(V,\varphi_U)$ such
that $\alpha|_V={\rm trop}_U^*(\alpha_U)$ and such that $\alpha_U$ coincides for all $j$ 
on the preimage of $\Omega_j$ 
with the pullback of $\alpha_j$.} 
Note that $\Omega$ is an open subset of $\Trop(U)$ and $(V,\varphi_U)$ is 
a tropical chart for $V:=\trop_{U}^{-1}(\Omega)$. 
Then $\alpha_U$ has not necessarily compact support, but we 
can extend $\alpha_U$ by zero to an element in  
{$AZ^{n-1,n-1}(U^{\an},\varphi_U)$} using 
that $\supp(\alpha) \cap \Uan$ is a closed subset of $V$. 
By abuse of notation, this extension will also be denoted by $\alpha_U$. 
Then we have $\alpha={\rm trop}_U^*(\alpha_U)$  on $\Uan$.
By shrinking $W$ and using an appropriate $U$, we may assume that 
\begin{equation*}
-\log\|1\|_i = \phi_i \circ \trop_U
\end{equation*}
on $V$ for a piecewise smooth function $\phi_i$ on $\Trop(U)$ 
and $i=1,2$. 
Since we deal with $\delta$-metrics, we may assume that there is 
a piecewise smooth extension $\tilde\phi_i$ of $\phi_i$ to $N_{U,\R}$ 
and a superform $\gamma_i$ on $N_{U,\R}$ of bidegree $(1,1)$ such 
that the first Chern {$\delta$-form} $\omega_{Y_i}$ is represented on $V$
by the $\delta$-preform
$ \gamma_i + \delta_{\tilde\phi_i \cdot N_{U,\R}}$ 
on $N_{U,\R}$ and such that $\gamma_i$ restricts to $\dpa\dpb \phi_i$ on $\Omega$ (see \ref{summary of Chern delta-construction}).
We have  
\[
\eta_{Y_i} = -\log|f_i| -\log\|1\|_i =   -\log|f_i| + \phi_i \circ \trop_U
\]
on $V$. Using bilinearity of $*$ and of $\wedge$, we may either assume that 
$\eta_{Y_i}$ is equal to $-\log\|1\|_i$ or equal to $-\log|f_i|$. 
Hence we have to consider the following four cases:

\vspace{2mm}
\noindent {Case 1: $s_1=s_2=1$.}
\vspace{2mm}

{In this case, the divisors $Y_1,Y_2$ are zero and $\eta_{Y_i}= -\log \|1\|_i$ 
for $i=1,2$ are piecewise smooth functions on $\Xan$.} 
Then we have 
\begin{equation} \label{comm2}
\langle g_{Y_1} * g_{Y_2} , \alpha \rangle = \langle \omega_{Y_1} \wedge g_{Y_2}, \alpha \rangle = \langle g_{Y_2}, \omega_{Y_1} \wedge \alpha \rangle.
\end{equation}
Recall that $g_{Y_2}$ is the current associated to  $\eta_{Y_2}$. By \ref{properties of ps forms on X}, we have $PS^{0,0}(W)=P^{0,0}(W)$ and hence $\eta_{Y_2} \alpha \in P_c^{n-1,n-1}(W)$. 
We may view it as a generalized $\delta$-form on $\Xan$   given on $\Uan$ by  $\phi_2 \alpha_U \in P^{n-1,n-1}(\Uan,\varphi_U)$. 
Since the first Chern {$\delta$-form} $\omega_{Y_1}$ is represented on $V$ by $\delta_{\tilde\phi_1 \cdot N_{U,\R}} + \gamma_1 \in P^{1,1}(N_{U,\R})$, we get
\begin{equation} \label{comm3'}
\langle g_{Y_1} * g_{Y_2} , \alpha \rangle = \int_{|\Trop(U)|}
(\delta_{\phi_1\cdot \Trop(U)} + \dpa\dpb \phi_1 )\wedge
\phi_2 \alpha_U.
\end{equation}
Here, we have used that $U$ is a very affine chart of integration for $\omega_{Y_1} \wedge \eta_{Y_2}\alpha \in P_c^{n,n}(\Xan)$. 
Recall from \ref{summary of Chern delta-construction} that the generalized $\delta$-forms $c_1(L_1,\metr_1)_{\rm res}$ and $c_1(L_1,\metr_1)_{\rm ps}$ are represented on
$V$ by $\delta_{\tilde\phi_1\cdot N_{U,\R}}$ and $\gamma_1$ 
in $ P^{1,1}(N_{U,\R})$. 
We conclude that $U$ is a very affine chart of integration for 
$c_1(L_1,\metr_1)_{\rm res}  \wedge \eta_{Y_2}\alpha$
and $c_1(L_1,\metr_1)_{\rm ps}  \wedge \eta_{Y_2}\alpha$ in $P_c^{n,n}(\Xan)$ and hence \eqref{comm3'} yields
 \begin{equation} \label{comm3}
\langle g_{Y_1} * g_{Y_2} , \alpha \rangle = \int_{|\Trop(U)|}
\delta_{\phi_1\cdot \Trop(U)} \wedge
\phi_2 \alpha_U + \int_{|\Trop(U)|} \dpa\dpb \phi_1 \wedge \phi_2 \alpha_U.
\end{equation}
Since $\alpha$ has compact support in $W$ and $\supp(\alpha)\cap \Uan \subseteq V$, it follows from Proposition \ref{delta-support on chart} and 
Corollary \ref{support corollary} that the integrands in \eqref{comm3} have compact support in $\Omega$. 
The generalization of 
Corollary \ref{support corollary} to $PSP$-forms {given} in \ref{psp on W} shows that 
$-\dpb \log\|1\|_1 \wedge \alpha$ has compact support  contained in $\Uan$. Again, we conclude that  $\dpb \phi_1 \wedge  \alpha_U$ has compact support contained in $\Omega$. 
Now Leibniz's rule and the theorem of Stokes \ref{stokesforpolyhedralcurrents} {for $\dpa$}
show
\begin{eqnarray} \label{comm4'}
\nonumber
&&\int_{|\Trop(U)|} \dpa\dpb \phi_1 \wedge \phi_2 \alpha_U \\
\label{comm4}
&=&\int_{\partial |\Trop(U)|} \dpb \phi_1 \wedge \phi_2 \alpha_U 
+ \int_{|\Trop(U)|} \dpb \phi_1 \wedge \dpa(\phi_2 \alpha_U) \\
\nonumber
&=& \int_{\partial |\Trop(U)|} \dpb \phi_1 \wedge \phi_2 \alpha_U 
+ \int_{|\Trop(U)|} \dpb \phi_1 \wedge \dpa\phi_2 \wedge \alpha_U \\
\nonumber&&+ \int_{|\Trop(U)|} \dpb \phi_1 \wedge \phi_2 d'\alpha_U.
\end{eqnarray}
Recall that $\phi_2 \alpha_U$ is a $\delta$-preform on $\Trop(U)$ and hence Remark \ref{generalization of tropicallemma} gives
\begin{equation} \label{comm5}
\int_{|\Trop(U)|}\delta_{\phi_1\cdot \Trop(U)}\wedge\phi_2 \alpha_U 
+ \int_{\partial |\Trop(U)|} 
\dpb \phi_1 \wedge \phi_2 \alpha_U = 0.
\end{equation}
Using \eqref{comm4} and \eqref{comm5} in \eqref{comm3}, we get
\begin{equation} \label{comm6}
\langle g_{Y_1} * g_{Y_2} , \alpha \rangle 
=   \int_{|\Trop(U)|} \dpb \phi_1 \wedge \dpa\phi_2 \wedge \alpha_U 
+ \int_{|\Trop(U)|} \dpb \phi_1 \wedge \phi_2 d'\alpha_U .
\end{equation}
A similar computation {where we replace \eqref{comm4} by an application of Stokes' theorem with respect to $\dpb$} shows
\begin{equation} \label{comm7}
\langle g_{Y_2} * g_{Y_1}, \alpha \rangle 
= \int_{|\Trop(U)|} \dpb \phi_1 \wedge \dpa\phi_2 \wedge \alpha_U 
- \int_{|\Trop(U)|}  \phi_1  \dpa \phi_2 \wedge d''\alpha_U
\end{equation}
Using that $U$ is a very affine chart of integration for 
$\dpb\eta_{Y_1} \wedge \eta_{Y_2}\wedge d'\alpha \in PSP_c^{n,n}(\Xan)$  and for 
$\eta_{Y_1} \wedge \dpa\eta_{Y_2}\wedge d''\alpha_U \in  PSP_c^{n,n}(\Xan)$, this proves 
the claim in the first case.

\vspace{2mm}
\noindent {Case 2: $s_1 = 1$ {and $\|1\|_2=1$}.}
\vspace{2mm}

In this case $Y_1$ is zero and $\eta_{Y_2}=-\log|f_2|$.
The following computation is similar to the one in the proof of the 
Poincar\'e--Lelong formula (see Theorem \ref{Poincare-Lelong equation}) 
and we will use the same terminology as there. 
{We have $g_{Y_2}*g_{Y_1}=0$ as $\delta_{Y_1}=0$ and 
$\omega_{Y_2}=c_1(O_X,\|\phantom a\|_2)=0$.
It remains to show} that 
\begin{equation}\label{comm13}
\omega_{Y_1} \wedge g_{Y_2}= -g_{Y_1} \wedge \delta_{Y_2} 
+ d'[\dpb\eta_{Y_1} \wedge \eta_{Y_2}] + d''[ \eta_{Y_1} \wedge \dpa\eta_{Y_2}].
\end{equation}
It is enough to check the claim locally and by linearity, 
we may assume that {$f_2$} is a regular function on $X$. 
By the first case, we may assume that {$f_2$} is non-constant.
We choose a very affine open subset $U$, the open subset $\Omega$ of 
$\Trop(U)$ and $\phi_1$ as above. We may assume that 
{$\supp(\Div(f_2)) \cap U = \emptyset$} and hence $-\log|f_2|$ is induced by an 
integral $\Gamma$-affine function {$\varphi_2$} on $\Trop(U)$. 
We use the very affine open $U$ to compute the term 
$\langle \omega_{Y_1} \wedge g_{Y_2}, \alpha \rangle$.    
Similarly as in \eqref{comm2} and \eqref{comm3'}, we deduce 
\begin{equation} \label{comm8}
\langle \omega_{Y_1} \wedge g_{Y_2} , \alpha \rangle 
= \int_{|\Trop(U)|}\delta_{\phi_1\cdot\Trop(U)}\wedge {\varphi_2} \alpha_U
+ \int_{|\Trop(U)|} \dpa\dpb\phi_1 \wedge {\varphi_2} \alpha_U.
\end{equation} 
The same computation as in \eqref{comm4'} --\eqref{comm6} yields 
\begin{equation} \label{comm9}
\langle \omega_{Y_1} \wedge g_{Y_2} , \alpha \rangle  
= \int_{|\Trop(U)|} \dpb\phi_1 \wedge \dpa\varphi_2  \wedge \alpha_U 
+ \int_{|\Trop(U)|} \dpb \phi_1 \wedge \varphi_2 d'\alpha_U.
\end{equation}
As in the first case, we have 
\begin{equation} \label{comm9'}
\int_{|\Trop(U)|} \dpb\phi_1 \wedge \varphi_2 d'\alpha_U 
= \langle d'[\dpb \eta_{Y_1} \wedge \eta_{Y_2}], \alpha \rangle . 
\end{equation}
Similarly as in the proof of the Poincar\'e--Lelong formula, we may assume 
that the support of $\alpha$ is covered by the interiors of the affinoid 
subdomains $W_j:=\trop_{U_j}^{-1}(\Delta_j)$ of the tropical chart $V_j$ for 
$j=1, \dots, m$. We set $W:=\bigcup_{j=1}^m W_j$. 
We choose $s>0$ sufficiently small with $\varphi_2 \leq -\log|s|$ 
on the compact set 
{$\supp(\dpb\phi_1 \wedge \alpha_U)$}. Since $W$ covers $\supp(\alpha)$, 
the analytic subdomain $W(s):=\{x\in W \mid |f_2(x)| \geq s\}$ 
of $W$ contains $\supp(\dpb\phi_1 \wedge \alpha)$  
and hence we have
\begin{equation} \label{comm10}
\int_{|\Trop(U)|} \dpb \phi_1 \wedge \dpa\varphi_2 \wedge \alpha_U
= \int_{\trop_U(W(s) \cap \Uan)} \dpb\phi_1 \wedge \dpa\varphi_2 \wedge \alpha_U.
\end{equation}
By the theorem of Stokes in \ref{stokesforpolyhedralcurrents}, this is equal to 
\begin{equation} \label{comm11}
\int_{\partial \trop_U(W(s) \cap \Uan)}  \phi_1 \dpa\varphi_2  \wedge \alpha_U +
\int_{\trop_U(W(s) \cap \Uan) }  \phi_1  \dpa\varphi_2  \wedge d''\alpha_U.
\end{equation}
By Corollary \ref{support corollary}, the support of $d''\alpha$ is contained in $\Uan$. 
We may assume that the compact set $\supp(d'' \alpha)$ is contained in $W(s)$. 
Using that $U$ is very affine chart of integration for 
$\eta_{Y_1} \wedge \dpa \eta_{Y_2}\wedge d''\alpha$,  we get
\begin{equation} \label{comm11'}
\int_{\trop_U(W(s) \cap \Uan)}  \phi_1  \dpa\varphi_2  \wedge d''\alpha_U 
=\bigl\langle d''[\eta_{Y_1} \wedge \dpa\eta_{Y_2}] , \alpha \bigr\rangle.
\end{equation}
Now we apply Remark \ref{PLR} with $f_2$ 
instead of $f$ and the generalized $\delta$-form 
$\eta_{Y_1}\wedge\alpha$ instead of $\alpha$ and observe that
$\varphi_2$ corresponds to $F^*(x_0)$ in \ref{PLR}. 
Then equation \eqref{PLL2} yields
\begin{equation} \label{comm12}
\int_{\partial \trop_U(W(s) \cap \Uan)}  \phi_1 \dpa\varphi_2  \wedge \alpha_U =
-\bigl\langle g_{Y_1} \wedge \delta_{Y_2}, \alpha \bigr\rangle.
\end{equation}
as $W$ covers $\supp(\alpha)$. Using \eqref{comm9'}--\eqref{comm12}    
in \eqref{comm9}, {we get \eqref{comm13}}
proving the claim in the second case.

\vspace{2mm}
\noindent {Case 3: $\|1\|_1=1$ and $s_2 = 1$.}
\vspace{2mm}

{
The formula proved in the second case yields
\[
g_{Y_2} * g_{Y_1} - g_{Y_1} * g_{Y_2} 
= d'[\dpb\eta_{Y_2} \wedge \eta_{Y_1}] + d''[ \eta_{Y_2} \wedge \dpa \eta_{Y_1}]
\]
The $(1,1)$-current on the lefthand side is clearly symmetric.
Hence the righthand side is symmetric as well and equals
\[
-d''[\dpa\eta_{Y_2} \wedge \eta_{Y_1}] - d'[ \eta_{Y_2} \wedge \dpb \eta_{Y_1}]
\]
This proves our claim in the third case.}

\vspace{2mm}
\noindent {Case 4: $\|1\|_1=1$ and $\|1\|_2=1$.} 
\vspace{2mm}

In this case $\eta_{Y_1}=-\log|f_1|$ and $\eta_{Y_2}=-\log|f_2|$.
We have to show that 
\[
g_{Y_1} \wedge \delta_{Y_2}- g_{Y_2} \wedge \delta_{Y_1} 
= d'[\dpb\eta_{Y_1} \wedge \eta_{Y_2}] + d''[ \eta_{Y_1} \wedge \dpa\eta_{Y_2}]
\]
holds. Again, we may assume that $f_1$ and $f_2$ are regular functions on $X$. 
By the previous cases, we may assume that these functions are non-constant. 
We use the same notation as above. Here, we choose the very affine open subset 
$U$ disjoint from $\supp(\Div(f_1)) \cup \supp(\Div(f_2))$. 
Then {$\varphi_1, \varphi_2$} are integral $\Gamma$-affine functions on $\Trop(U)$ 
inducing $-\log|f_1|,-\log|f_2|$ on $\Uan$. 
Going the computation in  the second case backwards, we see that
\begin{equation} \label{comm14}
\langle g_{Y_1} \wedge \delta_{Y_2}, \alpha \rangle 
= -\int_{|\Trop(U)|} \dpb\varphi_1 \wedge d' \varphi_2 \wedge \alpha_U
+ \langle d''[\eta_{Y_1} \wedge \dpa\eta_{Y_2}] , \alpha \rangle.
\end{equation}
Note here that $\dpb\varphi_1 \wedge \dpa\varphi_2 \wedge \alpha_U$ has compact support in $\Omega$. 
Indeed, it follows from Corollary \ref{support corollary}, that 
$\dpb\log |f_1| \wedge \dpa \log |f_2| \wedge \alpha$ is a well-defined 
$\delta$-form on $\Xan$ with compact support in $\Uan$ (using that the divisors intersect properly) and hence we 
get compactness in $\Omega$ from Proposition \ref{delta-support on chart}. 
Interchanging the role of $Y_1,Y_2$ and also of $d',d''$ 
and $\dpa,\dpb$ in \eqref{comm14}, we get the fourth claim. This proves the proposition.
\end{proof}

In the following, we denote the support of a cycle $Z$ (resp. of a Cartier divisor $D$) on $X$ by $|Z|$ (resp. $|D|$). 

\begin{cor}\label{cor commutativity}
Let $Z$ be a cycle of $X$ of codimension $p$   and  {let $g_Z$} be any 
$\delta$-current in $E^{p-1,p-1}(X)$.
For $i=1,2$, let $L_i$ be a line bundle on $X$ with a $\delta$-metric 
$\metr_i$ and non-zero meromorphic section $s_i$. 
Let $D_i$ denote the Cartier divisor on $X$ defined by $s_i$.
We assume that $|D_1|\cap |Z|$ and $|D_2| \cap |Z|$ have both codimension  $\geq 1$ in $|Z|$, and that $|D_1| \cap |D_2| \cap |Z|$ has codimension  $\geq 2$ in $|Z|$.  
Let $\eta_{Y_i}:=-\log\|s_i\|_i$ and let $g_{Y_i}=[\eta_{Y_i}]$ be the induced Green current for the Weil divisor $Y_i$ of $D_i$. Then we have 
\[
g_{Y_1} * (g_{Y_2} *g_Z)- g_{Y_2} * (g_{Y_1}*g_Z) 
\in d'\bigl(E^{p,p+1}(X^\an)\bigr)+d''\bigl(E^{p+1,p}(X^{\an})\bigr).
\]
\end{cor}

\proof
This follows immediately from Proposition \ref{commutativity} 
applied to the analytifications of the prime components of $Z$.
\qed

\section{Local heights of varieties} \label{Local heights of varieties}

In this section, we  study the local height of a proper variety $X$ of 
dimension $n$ over $K$ with respect to  metrized line bundles endowed 
with $\delta$-metrics. If the metrics are formal, 
then we  show that these analytically defined local 
heights agree with the ones based on divisorial intersection theory on formal model in
\cite{gubler-crelle}. In particular, they coincide with the local heights used in Arakelov theory over number fields.

\begin{art} \label{definition local height}
For $i=0, \dots , n$, let $L_i$ be a line bundle on $X$ endowed with a 
$\delta$-metric $\metr_i$ and a non-zero meromorphic section $s_i$. For 
the associated Cartier divisor  $D_i :=\Div(s_i)$, we consider the 
metrized Cartier divisor $\hat{D}_i:=(D_i,\metr_i)$, i.e. a Cartier divisor 
$D_i$ and a metric $\metr_i$ on the associated line bundle $O(D_i)$.
Recall from \ref{star product} that we obtain the Green current 
$g_{Y_i}:=[-\log\|s_i\|_i]$ for the Weil divisor $Y_i$ associated to $D_i$. 

We assume that the Cartier divisors $D_0, \dots , D_n$ intersect properly. 
Then we define {\it the local height of $X$ with respect to $\hat{D}_0, \dots, \hat{D}_n$} by 
$$\lambda_{\hat{D}_0, \dots, \hat{D}_n}(X):= g_{Y_0} * \dots * g_{Y_n}(1).$$
\end{art}

\begin{art} \label{local height of cycles}
If $Z$ is a cycle on $X$ of dimension $t$ and $\hat{D}_0, \dots ,\hat{D}_t$ are $\delta$-metrized 
Cartier divisors on $X$ with $|D_0|, \dots ,|D_t|, |Z|$ intersecting properly, then \ref{definition local height}
induces a local height $\lambda_{\hat{D}_0, \dots, \hat{D}_t}(Z)$ by linearity in the prime components of $Z$.
\end{art}

\begin{rem} \label{generalization to pseudo divisors}
The problem with this definition is that it is not functorial as the pull-back of a Cartier 
divisor is not always well-defined as a Cartier divisor. 
This problem is resolved by using pseudo-divisors 
instead of Cartier divisors (see \cite[Ch. 2]{fulton-intersection-theory}). 
We follow \cite{gubler-pisa} and define a 
{\it $\delta$-metrized pseudo-divisor} 
as a triple $(\overline L,Z,s)$ where $\overline L=(L,\metr)$
is a line bundle on $X$ equipped with a $\delta$-metric, $Z$ is
a closed subset of $X$, and $s$ is a nowhere vanishing section of $L$ over 
$X\setminus Z$.
Using the same arguments as in \cite{gubler-pisa}, we get 
a local height $\lambda_{\hat{D}_0, \dots, \hat{D}_t}(Z)$ for $\delta$-metrized pseudo-divisors which is 
well-defined under the weaker condition $|D_0| \cap  \dots  \cap |D_t| \cap |Z| = \emptyset$. 

It is straightforward to show that the local height is linear in $Z$ and multilinear in $\hat{D}_0, \dots, \hat{D}_t$. 
It follows from Corollary \ref{cor commutativity} along the arguments in
\cite{gubler-pisa} that the local height is symmetric in $\hat{D}_0, \dots, \hat{D}_t$. 
\end{rem}

The next result shows that the {\it induction formula} holds for local heights.

\begin{prop} \label{induction formula}
Let $\hat{D}_0, \dots, \hat{D}_n$ be $\delta$-metrized pseudo-divisors on $X$ with $|D_0| \cap \dots \cap |D_n| = \emptyset$.
Then  the local height $\lambda_{\hat{D}_0, \dots, \hat{D}_n}(X)$ is equal to 
$$\lambda_{\hat{D}_0, \dots, \hat{D}_{n-1}}(Y_n)- \int_{\Xan} \log\|s_n\|_n \cdot  c_1(\overline{O(D_0)}) \wedge \dots \wedge  c_1(\overline{O(D_{n-1})}),$$
where we assume that $D_n$ is a Cartier divisor with associated Weil divisor $Y_n$ and  canonical meromorphic section $s_n$ of $O(D_n)$.
\end{prop}

\begin{proof} The argument is the same as for \cite[Proposition 3.5]{gubler-pisa}. 
\end{proof}

\begin{prop} \label{functoriality}
Let $\varphi:X' \rightarrow X$ be a morphism of proper varieties over $K$ and let $\hat{D}_0, \dots ,\hat{D}_n$ be $\delta$-metrized 
pseudo-divisors on $X$ with $|D_0| \cap \dots \cap |D_n| = \emptyset$. Then the functoriality
$$\deg(\varphi)\lambda_{\hat{D}_0, \dots, \hat{D}_n}(X)= \lambda_{\varphi^*(\hat{D}_0), \dots, \varphi^*(\hat{D}_n)}(X')$$
holds.
\end{prop}

\begin{proof}
The proof relies on the induction formula in Proposition \ref{induction formula}
and the projection formula for 
integrals \eqref{integration well-definedg1}. We refer to \cite{gubler-pisa} for 
the analogous arguments in the archimedean case.
\end{proof}

The following result is called the {\it metric change formula}.

\begin{prop} \label{metrik change formula}
Suppose that the local height $\lambda(X)$ with respect to the $\delta$-metrized pseudo-divisors $\hat{D}_0, \dots, \hat{D}_n$ 
is well-defined. Let $\lambda'(X)$ be the local height of $X$ obtained by replacing the metric $\metr_0$ on $O(D_0)$ by another $\delta$-metric $\metr_0'$. Then $\rho:= \log(\metr_0'/\metr_0)$ is a piecewise smooth function on $\Xan$ and we have
\[\lambda(X)-\lambda'(X) = \int_\Xan \rho \cdot c_1(\overline{O(D_1)}) \wedge \dots \wedge  c_1(\overline{O(D_{n})}).\]
\end{prop}

\begin{proof}
This follows from linearity and symmetry of the local height in  $\hat{D}_0$ and $\hat{D}_n$  and from the induction formula in Proposition \ref{induction formula}.
\end{proof}

\begin{rem} \label{comparision with old local heights}
Now suppose that $\hat{D}_0, \dots, \hat{D}_n$ are formally metrized 
pseudo-divisors on $X$ with $|D_0| \cap \dots \cap |D_n| = \emptyset$.
Then the intersection theory of divisors on admissible formal $\kcirc$-models 
given in \cite{gubler-crelle} induces also a local height 
of $X$ (see \cite{gubler-pisa}). It also satisfies an induction formula 
involving Chambert-Loir's measures (see \cite[Remark 9.5]{gubler-pisa}).
Since the Chambert--Loir measure agrees with the Monge--Amp\`ere measure 
(see Theorem \ref{Monge-Ampere vs Chambert-Loir}), we deduce 
from the induction formula in Proposition \ref{induction formula} that the 
local height based 
on intersection theory of divisors agrees with 
$\lambda_{\hat{D}_0, \dots, \hat{D}_n}(X)$ from Remark \ref{generalization to pseudo divisors}. 
This proves in particular Theorem \ref{localalgebraicheights}
stated in the introduction.
\end{rem}

\appendix

\section{Convex Geometry}  \label{convex geometry}

In this appendix, we gather the notions from convex geometry on a finite dimensional real vector space $W$ coming with an {\it integral structure}. This means that we consider a free abelian group $N$ of rank $r$ with  $W=N_{\R}:=N \otimes_\Z \R$. Let $M:=\Hom(N,\Z)$ be the dual abelian group and let $V:=\Hom(M,\R)=M_\R$ be the dual vector space of $W$. The natural duality between $V$ and $W$ is denoted by $\langle u, \omega \rangle$. Let $\Gamma$ be a fixed subgroup of $\R$. In the applications, it is usually the value group of a non-archimedean absolute value.

\begin{art} \label{integral affine map}
Let $N'$ be another free abelian group of finite rank and 
let $F: N_\R \rightarrow N_\R'$ be an affine map. 
Then $F$ is called {\it integral $\Gamma$-affine} if 
$F = \L_F + \omega$ with  {$\omega \in N'\otimes_\Z\Gamma\subseteq N'_\R$} 
and with the associated linear map $\L_F$ induced by a 
homomorphism $N \rightarrow N'$. 
\end{art}

\begin{art} \label{polyhedron}
A {\it polyhedron} $\Delta$ in $W$ is defined as the intersection of finitely many half spaces $\{\omega \in W \mid \langle u_i, \omega \rangle \geq c_i\}$ with $u_i \in V$ and $c_i \in \R$. If we may choose all $u_i \in M$ and all $c_i \in \Gamma$, then we say that $\Delta$ is an {\it integral $\Gamma$-affine polyhedron}. A {\it face} of $\Delta$ is either $\Delta$ itself or the intersection of $\Delta$ with  the boundary of a closed halfspace containing $\Delta$.  We write $\tau \preccurlyeq \Delta$ for a face $\tau$ of $\Delta$ and we write  {$\tau \prec \Delta$} if additionally $\tau \neq \Delta$. The {\it relative interior} of $\Delta$ is defined by 
$$\relint(\Delta):= \Delta \setminus \bigcup_{\tau \prec \Delta} \tau.$$
Note that every polyhedron is convex. A {\it polytope} is a bounded polyhedron. 

A polyhedron $\Delta$ in $W$ generates an affine space $\A_\Delta$ of the same dimension. 
 {Recall that an affine space in $W$ is a translate of
a linear subspace and $\A_\Delta$ is the intersection of all
affine spaces in $W$ which contain $\Delta$.}
We denote the underlying vector space by $\L_\Delta$. If $\Delta$ is integral $\Gamma$-affine, then the integral structure of $\A_\Delta$ is given by the complete lattice $N_\Delta := N \cap \L_\Delta$ in $\L_\Delta$.
\end{art}

\begin{art} \rm \label{polyhedral complex}
A {\it polyhedral complex} $\Ccal$ in $W$ is a finite set of polyhedra such that 
\begin{itemize}
\item[(a)] $\Delta \in \Ccal \text{ $\Rightarrow$ all closed faces of $\Delta$ are in $\Ccal$}$;
\item[(b)] $\Delta, \sigma \in \Ccal \text{ $\Rightarrow$ $\Delta \cap \sigma$ is either empty or a closed face of $\Delta$ and $\sigma$}$.
\end{itemize}
The polyhedral complex $\Ccal$ is called {\it integral $\Gamma$-affine} if every $\Delta \in \Ccal$ is integral $\Gamma$-affine. The {\it support}  of $\Ccal$ is defined as $$|\Ccal|:= \bigcup_{\Delta \in
\Ccal} \Delta.$$
We say that a polyhedral complex $\Ccal$ is  {\it complete} if $|\Ccal|=W$. A {\it subdivision} of $\Ccal$ is a polyhedral complex $\Dcal$ with $|\Dcal|=|\Ccal|$ and with every $\Delta \in \Dcal$  contained in a polyhedron of $\Ccal$. This has to be distinguished from a {\it subcomplex} of $\Ccal$ which is a polyhedral complex $\Dcal$ with  $\Dcal \subseteq \Ccal$. 
\end{art}

\begin{art} \label{dimension of complex} \rm
Given a polyhedral complex $\KC$ in $N_\R$, we denote
by $\KC_n$ the subset of $n$-dimensional polyhedra in $\KC$
and by $\KC^l=\KC_{r-l}$ the subset of polyhedra in $\KC$ of codimension
$l$ in $N_\R$.
We say that a polyhedral complex $\KC$ is of {\it pure dimension $n$} 
(resp. of {\it pure codimension $l$}) if all polyhedra in $\KC$ which are maximal with 
respect to $\preccurlyeq$ lie in $\KC_n$ (resp. $\KC^l$).
Given a polyhedral complex $\KC$ of pure dimension $n$ and $m\leq n$, we denote
by $\KC_{\leq m}$ the polyhedral subcomplex of $\KC$ of pure dimension $m$ 
given by all $\sigma \in \KC$ with ${\rm dim}\,\sigma\leq m$.
We set $\KC^{\geq l}=\KC_{\leq r-l}$ if $r-l\leq n$. Recall here that $r$ is the rank of $N$.
\end{art}

\begin{definition} \label{polyhedral set}
(i) A  {\it polyhedral set} $P$ in $N_\R$
(of pure dimension $n$) is a 
finite union of  polyhedra (of pure dimension $n$). 
Equivalently, there exists a polyhedral complex $\KD$ 
(of pure dimension $n$) whose support is $P$. The polyhedral set is called {\it integral $\Gamma$-affine} 
if the above polyhedra can be chosen integral $\Gamma$-affine. 

(ii) Let $P$ be a polyhedral set in $N_\R$. A point $x\in P$ is called 
{\it regular} if there exists a polyhedron $\Delta \subseteq P$  such that $\relint(\Delta)$ is an open
neighbourhood of $x$ in $P$. We denote by $\relint(P)$ the set of regular 
points of the polyhedral set $P$.
\end{definition}

\begin{art} \label{cones}
A {\it cone} $\sigma$ in $W$ is characterized by 
$\R_{\geq 0}\cdot \sigma =\sigma$.  A cone which is a polyhedron is called a {\it  polyhedral cone}.
An integral $\R$-affine polyhedral  cone is simply called a {\it rational polyhedral cone}.
A polyhedral cone is called {\it strictly convex} if it does not contain a line.
The {\it local cone} ${\rm LC}_\omega(S)$ of $S\subseteq W$ at $\omega \in W$ is defined by 
$${\rm LC}_\omega(S):= \{ \omega' \in W \mid 
\text{$\omega + [0,\ve){\omega'} \subseteq S$ for some $\ve >0$}\}.$$
\end{art}

\begin{art} \label{fans}
A polyhedral complex $\Sigma$ consisting of strictly convex rational polyhedral cones is called
a {\it rational polyhedral fan}. The theory of 
toric varieties (see 
\cite{kempfetal, oda, fulton-toric-varieties, cox-little-schenck}) gives 
a bijective correspondence $\Sigma \mapsto Y_\Sigma$ between rational 
polyhedral fans on $N_\R$ and normal toric varieties 
over any field $K$ with open dense torus $\Spec(K[M])$ (up 
to equivariant isomorphisms restricting to the identity on the torus). 
Then $\Sigma$ is complete 
if and only if $Y_\Sigma$ is a proper variety over $K$. 

A {\it simplicial cone} in $N_\R$ is generated by a part of a basis. 
A {\it simplicial fan} is a fan formed by simplicial cones. 
A {\it smooth fan} in $N_\R$ is a rational polyhedral fan $\Sigma$  
such  that every cone $\sigma \in \Sigma$ is generated by a part 
of a basis of $N$.
In particular, a smooth fan is a simplicial fan. 
A polyhedral fan $\Sigma$ is smooth if and only if $Y_\Sigma$ is 
a smooth variety 
\cite[Ch. 1, Thm. 3.12]{cox-little-schenck}, 
\cite[2.1 Prop.]{fulton-toric-varieties}.
\end{art}


\def\cprime{$'$}

\end{document}